\documentclass[12pt, reqno]{amsart}
\usepackage{mathrsfs}
\usepackage{amsfonts}
\usepackage[centertags]{amsmath}
\usepackage{amssymb,array}
\usepackage{amsthm}
\usepackage{graphicx}
\usepackage{caption}
\usepackage{enumerate,enumitem}
\SetLabelAlign{center}{\hfill #1\hfill}
\usepackage[textwidth=16cm, hmarginratio=1:1]{geometry}
\usepackage{tikz-cd, tikz}
\usepackage{rotating}
\usepackage[all,cmtip]{xy}
\usepackage[titletoc, title]{appendix}
\usepackage{amscd}
\usepackage[colorlinks]{hyperref}
\usepackage{float}
\hypersetup{
bookmarksnumbered,
pdfstartview={FitH},
breaklinks=true,
linkcolor=blue,
urlcolor=blue,
citecolor=blue,
bookmarksdepth=2
}
\usetikzlibrary{arrows.meta}
\definecolor{weaveblue}{RGB}{15,25,170}
\definecolor{weavegreen}{RGB}{0,150,25}
\definecolor{weavecyan}{RGB}{0,160,170}

\definecolor{myblue}{RGB}{0,0,180}
\definecolor{myred}{RGB}{190,0,0}
\definecolor{mygreen}{RGB}{0,155,15}
 \definecolor{onecolor}{RGB}{34,90,170}
 \definecolor{twocolor}{RGB}{190,55,45}

\tikzset{
    weavebox/.style={
        draw=black,
        line width=1pt
    },
    bluestrand/.style={
        draw=weaveblue,
        line width=1.2pt,
        line cap=round
    },
    greenstrand/.style={
        draw=weavegreen,
        line width=1.2pt,
        line cap=round
    },
    dottedstrand/.style={
        draw=weavecyan,
        line width=1.2pt,
        densely dotted,
        line cap=round
    }
}

\newcommand{\weavebox}{
    \draw[weavebox] (0.75,0.65) rectangle (3.35,2.35);
    \node[font=\Large] at (2.05,1.50) {$\mathfrak{W}$};
}

\theoremstyle{plain}
   \newtheorem{theorem}{Theorem}[section]
   \newtheorem{proposition}[theorem]{Proposition}
   \newtheorem{lemma}[theorem]{Lemma}
   \newtheorem{corollary}[theorem]{Corollary}

\theoremstyle{definition}
   \newtheorem{definition}[theorem]{Definition}
   \newtheorem{example}[theorem]{Example}
   
   \newtheorem{remark}[theorem]{Remark}

\numberwithin{equation}{section}

\newcommand{\be}{\begin{enumerate}}

    \newcommand{\ene}{\end{enumerate}}

    \newcommand{\ZZ}{\mathbb{Z}}
    \newcommand{\QQ}{\mathbb{Q}}
    
    \newcommand{\NN}{\mathbb{N}}

    \newcommand{\CC}{\mathbb{C}}

    \newcommand{\DD}{\mathbb{D}}

    \newcommand{\qaff}{q_{\mathrm{aff}}}

    \newcommand{\hd}{\operatorname{hd}}

    \newcommand{\Id}{\operatorname{Id}}

    \newcommand{\w}{\operatorname{W}}

    \newcommand{\fg}{\operatorname{\mathfrak{g}}}

    \newcommand{\fn}{\operatorname{\mathfrak{n}}}

    \newcommand{\SL}{\operatorname{\mathbf{SL}}}

    \newcommand{\bb}{\mathbf{b}}
    \newcommand{\bfa}{\mathbf{a}}

    \newcommand{\bt}{\mathbf{t}}

    \newcommand{\cA}{\mathcal{A}}

    \newcommand{\CM}[1]{\mathcal{M}}

    \newcommand{\LambdaR}{\Lambda_{\mathrm R}}

    \newcommand{\Br}{{\operatorname{Br}}}
    \newcommand{\qbinom}[2]{\left[\begin{array}{c} #1 \\ #2 \end{array}\right]_q}

\DeclareMathOperator*{\wt}{wt}

\newlength{\mysizetiny}
\newlength{\mysizesmall}
\newlength{\mysize}
\newlength{\mysizelarge}
\begin{document}

\title{Monoidal Categorification and Quantization of Braid Varieties}
\author{Yingjin Bi}
\address{Department of Mathematics, Harbin Engineering University}
\address{Research Institute for Mathematical Sciences,
Kyoto University, 
Kyoto, Japan}
\email{yingjinbi@mail.bnu.edu.cn}
\date{} 

\begin{abstract}
Let \(G\) be a simple and simply connected algebraic group of
simply-laced type. For each positive braid word \(\beta\), and for the
complete strong duality datum attached to a
\(Q\)-datum, we construct an explicit based monoidal categorification of
the quantum cluster algebra of the braid variety \(X(\beta)\). We
identify this algebra with both a localized quantum Grothendieck ring
and a localized level-\(\geq1\) subalgebra of the bosonic extension
algebra. Under these identifications, quantum cluster monomials
correspond simultaneously to real simple modules and normalized global
basis elements. We construct the specialization homomorphism at
\(q^{1/2}=1\) and prove that it recovers \(\CC[X(\beta)]\); in
particular, the resulting integral form is a flat quantum deformation.
We also give intrinsic Lusztig parameters for cluster
variables attached to double strings, identify the corresponding
quantum grid minors, and establish generalized quantum \(T\)-systems.
\end{abstract}

\maketitle 

\tableofcontents

\section{Introduction}
Let \(G\) be a simple and simply connected algebraic group of
simply-laced type, with Cartan datum \((I,C)\) and Lie algebra
\(\fg\), and let \(\Br^+\) denote the associated positive braid
monoid. For a positive braid word
\(
    \beta=(i_1,\ldots,i_r),
\)
the braid variety \(X(\beta)\) parametrizes configurations of flags
whose successive relative positions are prescribed by the simple
reflections \(s_{i_1},\ldots,s_{i_r}\), subject to the endpoint
condition introduced in \cite[Section~3]{casals2025cluster}.

Cluster algebras were introduced by Fomin and Zelevinsky
\cite{fomin2002cluster} and have since become an important framework
connecting algebraic geometry, representation theory, and
combinatorics. A broad class of coordinate rings of algebraic
varieties admits cluster structures. Prominent examples include open
positroid varieties, reduced double Bruhat cells, open Richardson
varieties, and open double Bott--Samelson cells. These varieties may
all be realized as special cases of braid varieties. Cluster
structures on general braid varieties were constructed in
\cite{casals2025cluster,galashin2026braidII}.

A fundamental principle in the theory is that cluster monomials
should belong to distinguished canonical-type bases. For instance,
for the quantum coordinate ring \(A_q(\fn(w))\) of a unipotent
subgroup associated with a Weyl group element \(w\), quantum cluster
monomials belong to the dual canonical basis; see, for example,
\cite{qin2017triangular,kang2018monoidal}. Monoidal categorification
provides one of the principal methods for establishing such results.

The notion of monoidal categorification was introduced by
Hernandez and Leclerc \cite{hernandez2010cluster} using monoidal
subcategories of the category of finite-dimensional representations
of quantum affine algebras. One of its main consequences is that
cluster monomials are represented by classes of real simple modules
in the corresponding Grothendieck ring. In particular, a monoidal
categorification implies the positivity of the Laurent expansions of
cluster variables.

The theory was subsequently developed by Kashiwara, Kim, Oh, and
Park in a series of works, including
\cite{kashiwara2024monoidal,kashiwara2025monoidal}. On the other hand,
Hernandez and Leclerc introduced the quantum Grothendieck ring
\(
    \mathcal K_t(\mathscr C^{\ZZ})
\)
of the Hernandez--Leclerc category \(\mathscr C^{\ZZ}\) and related
it to the derived Hall algebra of the corresponding type
\cite{hernandez2015quantum}. Bosonic extension algebras were developed
subsequently and are related to these quantum Grothendieck rings by
the based isomorphisms recalled in Section~\ref{sec:bosonic-extension}.

In
\cite{oh2025pbw,kashiwara2024braid,kashiwara2025global},
Kashiwara, Kim, Oh, and Park developed global basis theory and braid
symmetries for bosonic extension algebras, in close analogy with the
corresponding theories for quantum coordinate rings of unipotent
subgroups. In particular, the bosonic extension algebra possesses
many of the structural properties of quantum unipotent coordinate
rings. Using the braid symmetries \(T_i\), they associated a
subalgebra
\(    \widehat{\cA}(\beta)
\)
with each positive braid word \(\beta\). Quantum cluster structures
on these algebras were established independently in
\cite{qin2024analogs,bi2025cluster,kashiwara2025monoidal}. In the
case corresponding to open double Bott--Samelson cells, the
resulting quantum cluster structure specializes at the classical
level to the known cluster structure on the associated braid
variety.

Cluster structures on braid varieties can be described
combinatorially in terms of Demazure weaves. A Demazure weave
\[
    \mathfrak W\colon\beta\longrightarrow\delta(\beta)
\]
determines a seed \(\mathbf s(\mathfrak W)\), where $\delta(\beta)$ is the Demazure product of $\beta$. Its cluster variables
are indexed by the trivalent vertices of \(\mathfrak W\), and its
extended exchange matrix is defined using the intersection pairing
of the corresponding Lusztig cycles. Elementary transformations of
Demazure weaves induce seed  equivalences and mutations. Consequently, different
Demazure weaves associated with the same braid word determine
mutation-equivalent seeds.

Among these weaves, the right inductive weave
\(\overrightarrow{\mathfrak W}(\beta)\) plays a particularly
important role. Its trivalent vertices are naturally indexed by the
complement of the leftmost reduced subexpression of \(\beta\)
representing its Demazure product \(\delta(\beta)\). We denote the
associated seed by
\[
    \overrightarrow{\mathbf s}(\beta)
    :=
    \mathbf s\bigl(
        \overrightarrow{\mathfrak W}(\beta)
    \bigr).
\]

It was proved in \cite{casals2025cluster} that every Demazure weave
\(\mathfrak W\) for \(\beta\) induces an isomorphism
\begin{equation}\label{eq:intro-classical-cluster}
    \CC[X(\beta)]
    \cong
    \mathcal A_{\CC}\bigl(\mathbf s(\mathfrak W)\bigr)
    =
    \mathcal U_{\CC}\bigl(\mathbf s(\mathfrak W)\bigr),
\end{equation}
where the subscript \(\CC\) denotes the complex coefficient forms
defined in Section~\ref{sec:cluster}. The resulting cluster algebra
is locally acyclic.
Cluster structures on braid varieties were also constructed from
Deodhar geometry using grid minors and, in type \(A\), from
three-dimensional plabic graphs
\cite{galashin2022braidI,galashin2026braidII}.
The two constructions give the same cluster structure by
\cite{casals2025comparing}.

Building on \cite{kashiwara2025monoidal},
Huh--Jung--Kim--Park \cite{huh2026comparison} gave
representation-theoretic realizations of braid-variety cluster
variables for braid words of the special form
\(\boldsymbol{\Delta}\beta\). Qin \cite{qin2024analogs} proved quantum
\(\mathcal A=\mathcal U\), constructed a common triangular basis, and
obtained a quasi-categorification for Lie-theoretic braid-variety
seeds. 
For an arbitrary positive braid word \(\beta\), the author
\cite[Theorem~6.8]{bi2026towards} constructed a monoidal subcategory
\(\mathscr B_{\DD}(\beta)\subset\mathscr C^{\ZZ}\) and established
the classical inclusion
\[
    \overline{\mathcal A}\bigl(
        \overrightarrow{\mathbf s}(\beta)
    \bigr)
    \subseteq
    K\bigl(\mathscr B_{\DD}(\beta)\bigr),
\]
with cluster monomials represented by simple objects. These results, however, did not
provide a full based categorification: they did not identify the
localized quantum Grothendieck ring with the upper quantum cluster
algebra or establish a bijection between its normalized simple classes
and the common triangular basis. Such a based categorification inside
a fixed Hernandez--Leclerc category, together with a flat integral
quantization of \(\CC[X(\beta)]\), remained open.

The aim of this paper is to provide this construction. We construct a
quantization of \(\CC[X(\beta)]\) as a subalgebra of the bosonic
extension algebra \(\widehat{\cA}(\beta)\) and establish its monoidal
categorification by a suitable Hernandez--Leclerc subcategory. More
precisely, we identify the quantum cluster algebra associated with the
right inductive weave with localizations of both the corresponding
quantum Grothendieck ring and the constructed subalgebra of
\(\widehat{\cA}(\beta)\). Under these identifications, quantum cluster
monomials correspond simultaneously to normalized
\((q,t)\)-characters of real simple modules and to global basis
elements. We also introduce quantum grid minors in bosonic extension
algebras, determine their Lusztig parameters, and prove that they
satisfy generalized quantum \(T\)-systems.

\subsection{Main results}

We now describe the main results of this paper. Let
\(\beta=(i_1,\ldots,i_r)\) be a positive braid word. The right
inductive weave \(\overrightarrow{\mathfrak W}(\beta)\) determines
a seed \(\overrightarrow{\mathbf s}(\beta)\) whose complex cluster algebra
is isomorphic to the coordinate ring \(\CC[X(\beta)]\). Our first
objective is to construct a canonical quantum lift
\(\overrightarrow{\mathbf s}_t(\beta)\) of this seed and to realize
the resulting quantum cluster algebra both inside a bosonic
extension algebra and as a localized quantum Grothendieck ring.

Let \(\widehat{\cA}\) be the bosonic extension algebra associated
with \(\fg\), and let \(\mathscr C^{\ZZ}\) be the
Hernandez--Leclerc category of finite-dimensional integrable
\(U_{\qaff}'(\widehat{\fg})\)-modules. In the categorification statements
below, we assume that the finite Cartan matrix is of type \(ADE\),
that the ambient quantum affine algebra is of untwisted affine type,
and that \(\DD=\DD_{\mathcal Q}\) is the complete strong duality datum
associated with a \(Q\)-datum \(\mathcal Q\). We use the standard based
isomorphism
\begin{equation*}
 \Phi_{\DD}\colon
 \mathcal K_t(\mathscr C^{\ZZ})
 \xrightarrow{\ \sim\ }
 \widehat{\cA}_{\ZZ[q^{\pm1/2}]},
\end{equation*}
where
\(t^{1/2}=q^{-1/2}\). Under this isomorphism, the classes of affine
cuspidal modules correspond to bar-invariant PBW root vectors,
whereas normalized simple classes correspond to normalized
global-basis elements.

For each position \(k\) in \(\beta\), let \(P_k^\beta\) be the PBW
root vector obtained by applying the braid symmetries associated
with the prefix \((i_1,\ldots,i_{k-1})\), and let
\(C_k^{\DD,\beta}\) be the corresponding affine cuspidal module.
With the conventions of Section~\ref{sec:bosonic-extension}, put
\(\widetilde P_k^\beta:=q^{-1/2}P_k^\beta\). Then
\[
    \Phi_{\DD}([C_k^{\DD,\beta}]_t)
    =\widetilde P_k^\beta
    =\widetilde G(\beta,E_k).
\]

\subsubsection{Quantization and monoidal categorification of braid
varieties}

The trivalent vertices of
\(\overrightarrow{\mathfrak W}(\beta)\) are naturally indexed by
the positions of \(\beta\) outside its leftmost reduced
subexpression representing \(\delta(\beta)\). One of the principal
new constructions of this paper is a family of real simple modules
\((B_v^{\DD})_{v\in V_\beta}\), indexed by these trivalent
vertices.

The module \(B_v^{\DD}\) is defined as the simple head of an ordered
tensor product involving the affine cuspidal module attached to
\(v\) and the modules associated with the preceding vertices. The
multiplicities in this tensor product are determined by the
intersection data of the corresponding Lusztig cycles. Thus the
construction of \(B_v^{\DD}\) translates the combinatorics of the
right inductive weave into the representation theory of quantum
affine algebras.

The same data determine an integral skew-symmetric matrix
\(\Lambda_\beta\), which records the quantum commutation relations
among the classes \([B_v^{\DD}]_t\). We give an explicit closed
formula for \(\Lambda_\beta\) and prove that it is compatible, in
the sense of Berenstein--Zelevinsky, with the extended exchange
matrix \(B_\beta\) of the right inductive weave. The pair
\((\Lambda_\beta,B_\beta)\) therefore defines the quantum seed
\(\overrightarrow{\mathbf s}_t(\beta)\).

Our first main result shows that the newly introduced modules
\(B_v^{\DD}\) provide a monoidal realization of this quantum seed.

\begin{theorem}[Theorem~\ref{thm:monoidal-seed-beta}]
\label{thm:monoidal-seed}
Let \(\beta\) be a positive braid word, and let \(\DD\) be the
complete strong duality datum associated with the fixed \(Q\)-datum.
Then the family
\((B_v^{\DD})_{v\in V_\beta}\), together with the exchange matrix
\(B_\beta\) and the compatible form \(\Lambda_\beta\), forms a
completely \(\Lambda\)-admissible monoidal seed.
\end{theorem}

We next pass from the initial seed to the entire quantum cluster
algebra. Let \(\mathscr C_{\DD}(\beta)\) be the monoidal subcategory
generated by the affine cuspidal modules
\(C_1^{\DD,\beta},\ldots,C_r^{\DD,\beta}\), and define its
\emph{level-\(\geq1\) truncation} by
\[
    \mathscr D_{\DD}(\beta)
    :=
    \mathscr C_{\DD}(\beta)
    \cap
    \mathscr C_{\DD}[1,\infty).
\]
This restriction to levels at least \(1\) is precisely the one needed
to compare the quantum Grothendieck ring with the coordinate ring of
the braid variety.

We localize \(\mathcal K_t(\mathscr D_{\DD}(\beta))\) by inverting
the classes of the frozen cluster-variable modules. These frozen
modules strongly commute with all simple objects of
\(\mathscr D_{\DD}(\beta)\), and hence the corresponding
localization is well defined.

\begin{theorem}[Theorem~\ref{thm:main-monoidal-categorification}]
\label{thm:monoidal-categorification}
Under the assumptions made at the beginning of this section, there is
a based algebra isomorphism
\begin{equation}\label{eq:intro-monoidal-categorification}
    \mathcal K_t\bigl(\mathscr D_{\DD}(\beta)\bigr)_{\mathrm{loc}}
    \xrightarrow{\sim}
    \mathcal A_t\bigl(
        \overrightarrow{\mathbf s}_t(\beta)
    \bigr),
\end{equation}
which sends \([B_v^{\DD}]_t\) to the quantum cluster variable indexed
by \(v\) and identifies the normalized localized simple classes with
the common triangular basis. Thus, after localization at the frozen
objects, \(\mathscr D_{\DD}(\beta)\) is a based monoidal
categorification of the quantum cluster algebra associated with
\(X(\beta)\).

In particular, every quantum cluster monomial with nonnegative frozen
exponents is represented by the normalized \((q,t)\)-character of a
real simple object of \(\mathscr D_{\DD}(\beta)\).
\end{theorem}

We then realize this quantum cluster algebra intrinsically inside
the bosonic extension algebra. Let \(\widehat{\cA}(\beta)\) be the
subalgebra of \(\widehat{\cA}\) generated by the PBW root vectors
associated with \(\beta\), and let \(\widehat{\cA}_{\geq1}\) denote
the subalgebra supported in levels at least \(1\). We define
\[
    \mathbf A(\beta)
    :=
    \widehat{\cA}_{\ZZ[q^{\pm1/2}]}(\beta)
    \cap
    \widehat{\cA}_{\geq1}.
\]
Let \(\mathbf A(\beta)_{\mathrm{loc}}\) be the localization obtained
by inverting the global basis elements corresponding to the frozen
cluster variables.

\begin{theorem}[Theorem~\ref{thm:quantization-braid-variety}]
\label{thm:bosonic-realization}
Under the preceding type \(ADE\)
hypotheses, for every positive braid word \(\beta\) there is a based algebra
isomorphism
\[\mathbf A(\beta)_{\mathrm{loc}}
\cong
\mathcal A_t(\overrightarrow{\mathbf s}_t(\beta)).\]
Under this isomorphism, every quantum cluster monomial with
nonnegative frozen exponents is, up to a power of \(q^{1/2}\), a
global-basis element of \(\mathbf A(\beta)\).
\end{theorem}

Combining the preceding two theorems, we obtain the central
identification
\begin{equation}\label{eq:intro-main-identifications}
    \mathcal K_t\bigl(\mathscr D_{\DD}(\beta)\bigr)_{\mathrm{loc}}
    \cong
    \mathcal A_t\bigl(\overrightarrow{\mathbf s}_t(\beta)\bigr)
    \cong
    \mathbf A(\beta)_{\mathrm{loc}}.
\end{equation}
The first realization provides a monoidal categorification, whereas
the second identifies the quantum cluster algebra with a localized
subalgebra of the bosonic extension algebra and relates its
distinguished basis to the global basis. We note that, by Proposition~\ref{lem:mutation-right-inductive-seed},
the category \(\mathscr D_{\DD}(\beta)\) and the algebra
\(\mathbf A(\beta)\) depend only on the positive braid element
\(b\in\Br^+\) represented by \(\beta\), and not on the choice of its
positive braid expression. Accordingly, we may denote them by
\(\mathscr D_{\DD}(b)\) and \(\mathbf A(b)\), respectively.

We further prove that specialization at \(q^{1/2}=1\) recovers the
classical coordinate ring \(\CC[X(\beta)]\). Hence
\(\mathbf A(\beta)_{\mathrm{loc}}\) is a flat
\(\ZZ[q^{\pm1/2}]\)-deformation of \(\CC[X(\beta)]\). Thus the term
\emph{quantization} in this paper refers to an intrinsic flat
deformation of the coordinate ring, rather than merely to an
embedding of its cluster variables into a noncommutative ambient
algebra.

\subsubsection{Quantum grid minors and generalized \(T\)-systems}

We next apply the preceding framework to grid minors. A double
string \(\ddot{\mathbf s}=(i_1X_1,\ldots,i_rX_r)\), with
\(X_k\in\{L,R\}\), records the successive insertion of each letter
\(i_k\) at the left or right end of the current braid word. It
determines a sequence of intermediate braid words and a final braid
word \(\beta\).

 A second principal construction of
this paper is the vector
\(\boldsymbol\delta_{k,j}^{\ddot{\mathbf s}}\in\ZZ_{\geq0}^r\), associated
with every step \(k\) and every \(j\in I\). This vector
records the Lusztig parameter of the quantum grid minor attached to
the pair \((k,j)\).

For an insertion, the vector
\(\boldsymbol\delta_{k,j}^{\ddot{\mathbf s}}\) is computed explicitly from the
leftmost reduced subexpression and the action of the Weyl group on
the fundamental weight \(\varpi_j\). Let
\(\widetilde G(\beta,\mathbf a)\) denote the normalized
global-basis element with \(\beta\)-Lusztig parameter \(\mathbf a\).
Using the newly introduced vector
\(\boldsymbol\delta_{k,j}^{\ddot{\mathbf s}}\), we define
\[D_{k,j}^{\beta,\ddot{\mathbf s}}
:=
\widetilde G
\bigl(\beta,\boldsymbol\delta_{k,j}^{\ddot{\mathbf s}}\bigr).\]
Our next main result identifies this global-basis element with the
corresponding quantum grid minor.

\begin{theorem}[Theorem~\ref{thm:quantum-grid-minor-factorization}]
\label{thm:intro-quantum-grid-minors}
Let \(\ddot{\mathbf s}\) be a double string with associated braid
word \(\beta\), and let \(k\) be a step of $\ddot{\mathbf{s}}$. Then, for every
\(j\in I\), the global-basis element
\(D_{k,j}^{\beta,\ddot{\mathbf s}}\) is a quantum lift of the
corresponding classical grid minor
\(\Delta_{k,j}^{\ddot{\mathbf s}}\). Moreover, it is a quantum
cluster monomial in
\(\mathcal A_t(\overrightarrow{\mathbf s}_t(\beta))\).
\end{theorem}

Finally, we study quantum grid minors under local transformations of
double strings. Consider two double strings that differ by the
interchange
\((\mathbf u,iL,jR,\mathbf v)
\leftrightarrow
(\mathbf u,jR,iL,\mathbf v)\).
When the two vertices involved are solid and the relevant length
condition is satisfied, this transformation is called a
\emph{solid-special mutation}. For the corresponding grid minors,
we prove a generalized quantum \(T\)-system using global-basis
positivity, their classical exchange identity, and the linear
independence of the specialized common triangular basis.

\begin{theorem}[Theorem~\ref{thm:generalized-T-system}]
\label{thm:intro-quantum-grid-T-system}
Let \(\ddot{\mathbf s}\) and \(\ddot{\mathbf s}'\) be related by a
solid-special mutation and let $k$ be the depth affected by this interchange. Then there are numbers
\(A,B\in\frac12\ZZ\) such that
\begin{equation}\label{eq:intro-quantum-grid-T-system}
\begin{aligned}
    D_{k,j}^{\beta,\ddot{\mathbf s}}
    D_{k,j}^{\beta,\ddot{\mathbf s}'}
    &=
    q^A
    D_{k-1,j}^{\beta,\ddot{\mathbf s}}
    D_{k+1,j}^{\beta,\ddot{\mathbf s}}+
    q^B
    \bigodot_{a\in I\setminus\{j\}}
    \left(
        D_{k,a}^{\beta,\ddot{\mathbf s}}
    \right)^{\odot(-c_{ja})},
\end{aligned}
\end{equation}
where \(\odot\) denotes the normalized quantum product. 
\end{theorem}

The proof has two main stages. First,
Section~\ref{sec:vector-right-weave} defines the weave vectors and
establishes their recursion. Section~\ref{sec:quantization-braid-varieties}
then identifies them with Lusztig parameters of cluster variables of right inductive weaves. The truncated-weave
extension in Lemma~\ref{lem:quantum-truncated-weave-extension} makes
this identification compatible with passing to consecutive blocks
(Theorem~\ref{thm:vector-double-inductive}); the resulting block
embeddings also give the quantum grid-minor factorization in
Theorem~\ref{thm:quantum-grid-minor-factorization}.

Second, the categorification argument is proved in
Section~\ref{sec:categorification-braid-varieties}.
We first consider the cluster structure on $\widehat{\cA}_{\ZZ[q^{\pm1/2}]}(\beta)=\overline{\cA}_t(\mathbf{s}_t(\beta))$, where $\mathbf{s}_t(\beta)$ is the seed associated with braid word $\beta$ in equation~\eqref{eq:seedbeta}. The  Lemma~\ref{lem:Qin-positive-braid-category-comparison}
identifies the common triangular basis of $\overline{\cA}_t(\mathbf{s}_t(\beta))$  with the normalized global
basis of $\widehat{\cA}_{\ZZ[q^{\pm1/2}]}(\beta)$ and the simple-object basis of $\mathcal{K}_t(\mathscr C_\DD(\beta))$, yielding the Garside base case
(Lemma~\ref{lem:Garside-base-case}). Braid-move invariance and
Demazure-descent deletion, together with the degree-zero freezing
argument of Proposition~\ref{pro:induction-upper-cluster-categorification},
then establish the main theorem for an arbitrary positive braid word.
Finally, the based identification proves flatness and identifies
the classical specialization; global-basis positivity lifts the
classical grid-minor relation to the generalized quantum \(T\)-system.

Sections~\ref{sec:cluster}--\ref{sec:braid-variety} supply the
background and conventions used in this argument: cluster seeds and
coefficient rings, Hernandez--Leclerc categories, bosonic extension
algebras, and Demazure weaves. The conventions for exchange-matrix
signs, the identification \(t^{1/2}=q^{-1/2}\), and bar normalization
are fixed there and retained throughout.

\noindent
\textbf{Acknowledgments}.
The author is grateful to Masaki Kashiwara, Ryo Fujita, and Fan Qin for helpful discussions and valuable comments concerning the results of this work. The author was supported by the China Scholarship Council (CSC), Grant No.~202506680034.

\section{Cluster algebras}\label{sec:cluster}
In this section, we fix our conventions for classical and quantum cluster
algebras. We then recall the word seeds associated with positive braid words
and describe the seed obtained by removing a leftmost reduced subexpression.

\subsection{Classical cluster algebras}
Let \(K\) be a finite index set, let \(K_{\mathrm{fr}}\subseteq K\) be
the subset of frozen vertices, and set
\(K_{\mathrm{ex}}:=K\setminus K_{\mathrm{fr}}\). For an integer \(a\),
we write \([a]_+:=\max\{a,0\}\).

An integer matrix
\(\widetilde B=(b_{ij})_{i\in K,\,j\in K_{\mathrm{ex}}}\) is called an
\emph{extended exchange matrix} if its principal part
\((b_{ij})_{i,j\in K_{\mathrm{ex}}}\) is skew-symmetric. Let
\(\mathbf X=(X_i)_{i\in K}\) be a family of algebraically independent
commuting variables, and let \(\mathcal F:=\QQ(X_i\mid i\in K)\). A
\emph{seed} in \(\mathcal F\) is a tuple
\(\mathbf t=(\mathbf X,\widetilde B,K_{\mathrm{ex}})\). Its Laurent
polynomial ring is
\(\mathcal T(\mathbf t):=\ZZ[X_i^{\pm1}\mid i\in K]\).

For \(k\in K_{\mathrm{ex}}\), the matrix mutation of \(\widetilde B\) in
direction \(k\) is the matrix
\(\mu_k(\widetilde B)=\widetilde B'=(b'_{ij})\) given by
\begin{equation}\label{eq:matrix-mutation}
    b'_{ij}=
    \begin{cases}
        -b_{ij},
        & i=k \text{ or } j=k,\\
        b_{ij}+[b_{ik}]_+[b_{kj}]_+
        -[-b_{ik}]_+[-b_{kj}]_+,
        & \text{otherwise}.
    \end{cases}
\end{equation}
The mutated cluster \(\mathbf X'=(X_i')_{i\in K}\) is defined by
\begin{equation}\label{eq:cluster-mutation}
    X_i'=
    \begin{cases}
        X_i,
        & i\neq k,\\[1mm]
        \displaystyle
        X_k^{-1}\left(
            \prod_{j\in K}X_j^{[b_{jk}]_+}
            +\prod_{j\in K}X_j^{[-b_{jk}]_+}
        \right),
        & i=k.
    \end{cases}
\end{equation}
The seed
\(\mu_k(\mathbf t):=(\mathbf X',\widetilde B',K_{\mathrm{ex}})\) is
called the \emph{mutation of \(\mathbf t\) in direction \(k\)}. In
particular, mutation leaves the frozen variables unchanged.

Two seeds are \emph{mutation-equivalent} if they are related by a finite
sequence of mutations at exchangeable vertices; we write
\(\mathbf t'\sim\mathbf t\). If
\(\mathbf t'=((X_i^{\mathbf t'})_{i\in K},
\widetilde B^{\mathbf t'},K_{\mathrm{ex}})\), set
\(\mathcal T(\mathbf t')
:=\ZZ[(X_i^{\mathbf t'})^{\pm1}\mid i\in K]\).

\begin{definition}
    The \emph{nonlocalized cluster algebra}, the \emph{cluster algebra}, and
    the \emph{upper cluster algebra} associated with \(\mathbf t\) are,
    respectively,
    \begin{align*}
        \overline{\mathcal A}(\mathbf t)
        :=\ZZ\bigl[
            X_i^{\mathbf t'}
            \mid \mathbf t'\sim\mathbf t,\ i\in K
        \bigr],\quad
        \mathcal A(\mathbf t)
        :=\overline{\mathcal A}(\mathbf t)
            [X_f^{-1}\mid f\in K_{\mathrm{fr}}],\quad
        \mathcal U(\mathbf t)
        :=\bigcap_{\mathbf t'\sim\mathbf t}
            \mathcal T(\mathbf t')
            \subseteq\mathcal F.
    \end{align*}
\end{definition}

The algebras without a coefficient subscript are defined over
\(\ZZ\). For comparison with coordinate rings of complex varieties,
we use
\begin{equation}\label{eq:complex-cluster-coefficients}
\begin{aligned}
    \overline{\mathcal A}_{\CC}(\mathbf t)
    &:=\overline{\mathcal A}(\mathbf t)\otimes_{\ZZ}\CC,\\
    \mathcal A_{\CC}(\mathbf t)
    &:=\mathcal A(\mathbf t)\otimes_{\ZZ}\CC,\\
    \mathcal U_{\CC}(\mathbf t)
    &:=\bigcap_{\mathbf t'\sim\mathbf t}
       \CC[(X_i^{\mathbf t'})^{\pm1}\mid i\in K]
       \subseteq\CC(X_i\mid i\in K).
\end{aligned}
\end{equation}
In particular, a geometric identification with \(\CC[X(\beta)]\)
always uses \(\mathcal A_{\CC}\). Grothendieck-ring
categorifications before extension of scalars use the integral forms.

The Laurent phenomenon gives
\begin{equation}\label{eq:classical-inclusions}
    \overline{\mathcal A}(\mathbf t)
    \subseteq\mathcal A(\mathbf t)
    \subseteq\mathcal U(\mathbf t)
    \subseteq\mathcal F.
\end{equation}

We associate with \(\mathbf t\) an ice quiver \(Q_{\mathbf t}\) with vertex
set \(K\) and frozen vertex set \(K_{\mathrm{fr}}\). We take no arrows
between frozen vertices, and use the convention
\begin{equation}\label{eq:quiver-convention}   b_{ij}
    =\#\{\text{arrows }j\to i\}
     -\#\{\text{arrows }i\to j\},
    \qquad i\in K,\quad j\in K_{\mathrm{ex}}.
\end{equation}

\subsection{Quantum cluster algebras}

Let \(\Lambda=(\lambda_{ij})_{i,j\in K}\) be an integral
skew-symmetric matrix. We say that \(\widetilde B\) and \(\Lambda\) are
\emph{compatible} if
\begin{equation}\label{eq:compatibility}
    \sum_{k\in K}b_{kj}\lambda_{ki}=2\delta_{ij},
    \qquad j\in K_{\mathrm{ex}},\quad i\in K.
\end{equation}
Here the factor \(2\) reflects the normalization used throughout this paper.
In matrix form, \eqref{eq:compatibility} is
\(\widetilde B^{\mathsf T}\Lambda=2I_{K_{\mathrm{ex}}}\), where the
right-hand side denotes the rectangular matrix with entries
\(2\delta_{ij}\). Since \(\Lambda\) is skew-symmetric, the same condition
may equivalently be written as
\(-\Lambda\widetilde B=2I_{K_{\mathrm{ex}}}\). Thus the compatible pair
used below is always \((\Lambda,\widetilde B)\); the minus sign appears only
when the compatibility equation is written with \(\widetilde B\) on the
right.

The quantum torus \(\mathcal T^\Lambda\) is the
\(\ZZ[t^{\pm1/2}]\)-algebra generated by
\(\widetilde X_i^{\pm1}\), \(i\in K\), subject to the relations
\(\widetilde X_i\widetilde X_j
=t^{\lambda_{ij}}\widetilde X_j\widetilde X_i\). Fix a total order on
\(K\). For \(\mathbf a=(a_i)_{i\in K}\in\ZZ^K\), its normalized
quantum monomial is
\begin{equation}\label{eq:quantum-monomial}
    \widetilde X^{\mathbf a}
    :=t^{-\frac12\sum_{i<j}a_i a_j\lambda_{ij}}
      \prod_{i\in K}^{\longrightarrow}\widetilde X_i^{a_i}.
\end{equation}
These monomials satisfy
\begin{equation}\label{eq:quantum-monomial-product}
    \widetilde X^{\mathbf a}\widetilde X^{\mathbf b}
    =t^{\frac12\Lambda(\mathbf a,\mathbf b)}
      \widetilde X^{\mathbf a+\mathbf b},
    \qquad
    \Lambda(\mathbf a,\mathbf b)
    :=\sum_{i,j\in K}a_i\lambda_{ij}b_j.
\end{equation}
Thus \(\Lambda\), with exponent-vector arguments or a seed
subscript such as \(\Lambda_\beta\), always denotes a
skew-symmetric commutation matrix and its associated bilinear form.
The categorical \(R\)-matrix degree introduced in
Section~\ref{sec:Hernandez} is denoted by \(\LambdaR(M,N)\).

A \emph{quantum seed} is a tuple
\(\mathbf t=(\widetilde{\mathbf X},\widetilde B,
\Lambda,K_{\mathrm{ex}})\), where
\(\widetilde{\mathbf X}=(\widetilde X_i)_{i\in K}\) and
\((\widetilde B,\Lambda)\) is compatible. Its entries
\(\widetilde X_i\) are called the \emph{quantum cluster variables}.

Let \(k\in K_{\mathrm{ex}}\). Define
\(E_k=(e_{ij})_{i,j\in K}\) by \(e_{ij}=\delta_{ij}\) for
\(j\neq k\), \(e_{kk}=-1\), and \(e_{ik}=[b_{ik}]_+\) for
\(i\neq k\). The mutated skew-symmetric matrix is
\(\mu_k(\Lambda):=E_k^T\Lambda E_k\); the pair
\((\mu_k(\widetilde B),\mu_k(\Lambda))\) is again compatible. Let
\[
    \mathbf a_+
    :=-\mathbf e_k+\sum_{i\in K}[b_{ik}]_+\mathbf e_i,
    \qquad
    \mathbf a_-
    :=-\mathbf e_k+\sum_{i\in K}[-b_{ik}]_+\mathbf e_i.
\]
The mutated quantum cluster
\(\mu_k(\widetilde{\mathbf X})=(\widetilde X_i')_{i\in K}\) is given by
\begin{equation}\label{eq:quantum-cluster-mutation}
    \widetilde X_i'=
    \begin{cases}
        \widetilde X^{\mathbf a_+}
        +\widetilde X^{\mathbf a_-},
        & i=k,\\
        \widetilde X_i,
        & i\neq k.
    \end{cases}
\end{equation}
The resulting quantum seed is
\begin{equation}\label{eq:quantum-seed-mutation}
    \mu_k(\mathbf t)
    :=\bigl(
        \mu_k(\widetilde{\mathbf X}),
        \mu_k(\widetilde B),
        \mu_k(\Lambda),
        K_{\mathrm{ex}}
    \bigr).
\end{equation}

Let \(\mathcal F_t\) be the skew field of fractions of
\(\mathcal T^\Lambda\). Quantum mutation realizes the quantum torus of every
seed \(\mathbf t'\sim\mathbf t\) inside \(\mathcal F_t\); we denote its
image by
\[    \mathcal T(\mathbf t')
    :=\ZZ[t^{\pm1/2}]
      \left\langle
          \widetilde X_i(\mathbf t')^{\pm1}
          \mid i\in K
      \right\rangle
      \subseteq\mathcal F_t.
\]

\begin{definition}
    The \emph{nonlocalized quantum cluster algebra}, the
    \emph{quantum cluster algebra}, and the
    \emph{upper quantum cluster algebra} associated with \(\mathbf t\) are,
    respectively,
    \begin{align*}
        \overline{\mathcal A}_t(\mathbf t)
        &:=\left\langle
            \widetilde X_i(\mathbf t')
            \mid \mathbf t'\sim\mathbf t,\ i\in K
        \right\rangle_{\ZZ[t^{\pm1/2}]\text{-alg}},\\
        \mathcal A_t(\mathbf t)
        &:=\overline{\mathcal A}_t(\mathbf t)
            [\widetilde X_f^{-1}\mid f\in K_{\mathrm{fr}}],\\
        \mathcal U_t(\mathbf t)
        &:=\bigcap_{\mathbf t'\sim\mathbf t}
            \mathcal T(\mathbf t')
            \subseteq\mathcal F_t.
    \end{align*}
\end{definition}

The quantum Laurent phenomenon yields
\begin{equation}\label{eq:quantum-inclusions}
    \overline{\mathcal A}_t(\mathbf t)
    \subseteq\mathcal A_t(\mathbf t)
    \subseteq\mathcal U_t(\mathbf t)
    \subseteq\mathcal F_t.
\end{equation}

We record the basis terminology used in the freezing argument. For
\(\mathbf g\in\ZZ^K\), an element \(Z\in\mathcal T(\mathbf t)\) is
\emph{pointed at \(\mathbf g\)} if its Laurent expansion is
\begin{equation}\label{eq:pointed-element}
    Z
    =\widetilde X^{\mathbf g}
     +\sum_{\mathbf n\in\NN^{K_{\mathrm{ex}}}\setminus\{0\}}
       c_{\mathbf n}(t)
       \widetilde X^{\mathbf g-\widetilde B\mathbf n},
    \qquad
    c_{\mathbf n}(t)\in\ZZ[t^{\pm1/2}].
\end{equation}
Here the dominance order is
\(\mathbf h\prec_{\mathbf t}\mathbf g\) when
\(\mathbf h=\mathbf g-\widetilde B\mathbf n\) for some nonzero
\(\mathbf n\in\NN^{K_{\mathrm{ex}}}\). The minus sign implements
our exchange-matrix convention \eqref{eq:quiver-convention}.
The leading monomial is normalized as in
\eqref{eq:quantum-monomial} and has coefficient \(1\).
A seed \(\mathbf t\) is \emph{injective-reachable} if there exist a seed
\(\mathbf t[-1]\) obtained from \(\mathbf t\) by mutations at exchangeable
vertices and a permutation \(\sigma\) of \(K_{\mathrm{ex}}\) such that the
exchangeable part of the \(g\)-vector of
\(\widetilde X_{\sigma(k)}(\mathbf t[-1])\), computed with respect to
\(\mathbf t\), is \(-\mathbf e_k\) for every
\(k\in K_{\mathrm{ex}}\); its frozen components are unrestricted. This is
the injective-reachability convention of
\cite{qin2020dual,qin2024analogs}. An element is
\emph{\(\mathcal M^\circ\)-pointed} if its pointed expansions are compatible
with the injective seeds in the chosen mutation class, in the precise sense
of the same references.

A \emph{common triangular basis} is a
\(\ZZ[t^{\pm1/2}]\)-basis that is bar invariant, contains all normalized
quantum cluster monomials, and is pointed and triangular with respect to
cluster multiplication in every seed of the mutation class. We use these
notions in the sense of \cite{qin2020dual,qin2024analogs}; the preceding
description fixes the conventions needed below.

\subsection{Word seeds associated with positive braid words}

Let \(G\) be a simple and simply connected algebraic group of simply-laced
type, with Cartan matrix \(C=(c_{ij})_{i,j\in I}\), and let \(W\) be its
Weyl group, with simple reflections \((s_i)_{i\in I}\). The positive braid
monoid \(\Br^+\) is generated by \(\sigma_i\), \(i\in I\), subject to
\begin{equation}\label{eq:positive-braid-relations}
    \sigma_i\sigma_j=\sigma_j\sigma_i
    \quad\text{if }c_{ij}=0,
    \qquad
    \sigma_i\sigma_j\sigma_i
    =\sigma_j\sigma_i\sigma_j
    \quad\text{if }c_{ij}=-1.
\end{equation}

Fix \(b\in\Br^+\) and a positive braid word
\(\beta=(i_1,\ldots,i_r)\) such that
\(b=\sigma_{i_1}\cdots\sigma_{i_r}\). We write
\([r]:=\{1,\ldots,r\}\) and call \(i_k\) the \emph{color} of
\(k\in[r]\). For each \(k\in[r]\), let
\[
    \begin{aligned}
        k^+
        :=\min\bigl(\{\ell>k\mid i_\ell=i_k\}
            \cup\{+\infty\}\bigr),\qquad
        k^-
        :=\max\bigl(\{\ell<k\mid i_\ell=i_k\}
            \cup\{-\infty\}\bigr).
    \end{aligned}
\]
Set \([r]_{\mathrm{fr}}:=\{k\in[r]\mid k^+=+\infty\}\) and
\([r]_{\mathrm{ex}}:=[r]\setminus[r]_{\mathrm{fr}}\).

It will also be useful to label positions by their color and occurrence
number. Set
\(n_i:=\#\{k\in[r]\mid i_k=i\}\) and
\[
    K_\beta:=\{(i,a)\mid i\in I,\ 1\leq a\leq n_i\}.
\]
We identify \([r]\) with \(K_\beta\) by sending \(k\) to
\((i_k,a(k))\), where
\(a(k):=\#\{s\leq k\mid i_s=i_k\}\). Thus, \((i,a)\) denotes the
position of the \(a\)-th occurrence of the color \(i\).

The extended exchange matrix
\(\widetilde B_\beta=(b_{st})_{s\in[r],\,t\in[r]_{\mathrm{ex}}}\) is
defined by
\begin{equation}\label{eq:B-beta}
    b_{st}=
    \begin{cases}
        1,
        & s=t^+\ \text{or}\
          \bigl(s<t<s^+<t^+\text{ and }c_{i_s i_t}=-1\bigr),\\
        -1,
        & t=s^+\ \text{or}\
          \bigl(t<s<t^+<s^+\text{ and }c_{i_s i_t}=-1\bigr),\\
        0,
        & \text{otherwise}.
    \end{cases}
\end{equation}

Let \(P=\bigoplus_{i\in I}\ZZ\varpi_i\) be the weight lattice, where
\(\varpi_i\) is the \(i\)-th fundamental weight. We normalize the
\(W\)-invariant symmetric bilinear form on \(P\) by
\((\alpha_i,\alpha_j)=c_{ij}\). For \(k\in[r]\), set
\(w_{\leq k}^\beta:=s_{i_1}\cdots s_{i_k}\), and define a skew-symmetric matrix
\(\widetilde\Lambda_\beta=(\widetilde\lambda_{st})_{s,t\in[r]}\) by
\begin{equation}\label{eq:Lambda-beta-word}
    \widetilde\lambda_{st}
    :=-\bigl(
        \varpi_{i_s}-w_{\leq s}^\beta\varpi_{i_s},
        \varpi_{i_t}+w_{\leq t}^\beta\varpi_{i_t}
    \bigr) \text{ for }s<t.
\end{equation}
The remaining entries are fixed by
\begin{equation}\label{eq:Lambda-beta-word-skew-extension}
    \widetilde\lambda_{ss}:=0,
    \qquad
    \widetilde\lambda_{ts}:=-\widetilde\lambda_{st}
    \quad(s<t).
\end{equation}
Note that
\((\widetilde B_\beta,\widetilde\Lambda_\beta)\) is a compatible pair
\cite{fujita2023isomorphisms}.

Let \(\mathbf X_\beta=(X_k)_{k\in[r]}\) and
\(\widetilde{\mathbf X}_\beta=(\widetilde X_k)_{k\in[r]}\). The
classical and quantum word seeds are
\begin{equation}\label{eq:seedbeta}
    \begin{aligned}
        \mathbf s(\beta)
        :=\bigl(\mathbf X_\beta,
            \widetilde B_\beta,[r]_{\mathrm{ex}}\bigr),\qquad
        \mathbf s_t(\beta)
        :=\bigl(\widetilde{\mathbf X}_\beta,
            \widetilde B_\beta,\widetilde\Lambda_\beta,
            [r]_{\mathrm{ex}}\bigr).
    \end{aligned}
\end{equation}
The classical word seed gives the cluster structure on the corresponding
Bott--Samelson cell constructed in \cite{shen2021cluster}. On the quantum
side, \(\overline{\mathcal A}_t(\mathbf s_t(\beta))\) is realized as a
subalgebra of the associated bosonic extension algebra; see
\cite{qin2024analogs,kashiwara2025monoidal,bi2025cluster}.

\subsection{The seed associated with the Demazure product}

The \emph{Demazure product} of a positive braid word is defined recursively
by
\begin{equation}\label{eq:demazure-product}
    \delta(\varnothing)=e,
    \qquad
    \delta(\beta i)
    =\max\{\delta(\beta),\delta(\beta)s_i\},
\end{equation}
where the maximum is taken with respect to the Bruhat order on \(W\).
Let \(m=\ell(\delta(\beta))\), and let
\(1\leq p_1<\cdots<p_m\leq r\) be the lexicographically minimal
sequence such that
\(s_{i_{p_1}}\cdots s_{i_{p_m}}=\delta(\beta)\). We call
\((i_{p_1},\ldots,i_{p_m})\) the \emph{leftmost reduced
subexpression} of \(\beta\) representing \(\delta(\beta)\).

For \(k\in[m]\), define
\[    a_k
    :=\#\{s\leq k\mid i_{p_s}=i_{p_k}\},
    \qquad
    b_k
    :=\#\bigl\{t\leq p_k
        \mid t\notin\{p_1,\ldots,p_m\},\ i_t=i_{p_k}\bigr\}.
\]

\begin{definition}\label{def:svbeta}
    Fix \(k\in[m]\) and set \(i=i_{p_k}\). Define the mutation
    sequence
    \begin{equation}\label{eq:mu-tilde-k}
        \widetilde\mu_k
        :=\begin{cases}
            \mu_{(i,n_i-a_k)}\circ\cdots\circ
            \mu_{(i,b_k+1)},
            & a_k+b_k<n_i,\\[1mm]
            \operatorname{Id},
            & a_k+b_k=n_i.
        \end{cases}
    \end{equation}
    Composition is read from right to left. Thus, when
    \(a_k+b_k<n_i\), the mutations are performed successively at
    \((i,b_k+1),(i,b_k+2),\ldots,(i,n_i-a_k)\). Set
    \begin{equation}\label{eq:M-k}
        M_k:=\widetilde\mu_k\circ\cdots\circ\widetilde\mu_1.
    \end{equation}
    We use the same notation for the induced mutation sequences on quivers
    and on compatible pairs.

    Starting from \(\widetilde{\mathbf s}_0:=\mathbf s(\beta)\), define
    \(\widetilde{\mathbf s}_k\) recursively as follows. Apply
    \(\widetilde\mu_k\) to \(\widetilde{\mathbf s}_{k-1}\), and then
    freeze the vertex
    \(
        d_k:=(i_{p_k},n_{i_{p_k}}-a_k+1).
    \)
    The same procedure applied to \(\mathbf s_t(\beta)\) gives a quantum
    seed \(\widetilde{\mathbf s}_{t,k}\).

    After the \(m\)-th step, delete the frozen vertices in
    \(
        D_\beta:=\{d_k\mid k\in[m]\},
    \)
    and set \(V_\beta:=K_\beta\setminus D_\beta\). Among the vertices in
    \(V_\beta\), freeze every vertex that either is already frozen in
    \(\widetilde{\mathbf s}_m\) or is adjacent in
    \(M_m(Q_{\mathbf s(\beta)})\) to a vertex of \(D_\beta\). Denote
    the resulting frozen set by \(V_{\beta,\mathrm{fr}}\), and put
    \(V_{\beta,\mathrm{ex}}
    :=V_\beta\setminus V_{\beta,\mathrm{fr}}\).

    Let
    \[
        \widetilde B_\beta^{(m)}
        :=M_m(\widetilde B_\beta),
        \qquad
        \widetilde\Lambda_\beta^{(m)}
        :=M_m(\widetilde\Lambda_\beta).
    \]
    By construction,
    \(\bigl(\widetilde B_\beta^{(m)}\bigr)_{
       D_\beta\times V_{\beta,\mathrm{ex}}}=0\). Hence deleting
    \(D_\beta\) preserves compatibility on the remaining exchangeable
    columns. In particular, the matrices
    \begin{equation}\label{eq:restricted-compatible-pair}
        B_{\delta(\beta),\beta}
        :=\bigl(\widetilde B_\beta^{(m)}\bigr)_{
            V_\beta\times V_{\beta,\mathrm{ex}}},
        \qquad
        \Lambda_{\delta(\beta),\beta}
        :=\bigl(\widetilde\Lambda_\beta^{(m)}\bigr)_{
            V_\beta\times V_\beta}
    \end{equation}
    form a compatible pair.

    Let \(y_v\) and \(\widetilde y_v\) denote, respectively, the
    classical and quantum cluster variables at \(v\in V_\beta\) after
    applying \(M_m\). We define
    \begin{equation}\label{eq:seed-delta-beta}
        \begin{aligned}
            \mathbf s(\delta(\beta),\beta)
            &:=\bigl(
                (y_v)_{v\in V_\beta},
                B_{\delta(\beta),\beta},
                V_{\beta,\mathrm{ex}}
            \bigr),\\
            \mathbf s_t(\delta(\beta),\beta)
            &:=\bigl(
                (\widetilde y_v)_{v\in V_\beta},
                B_{\delta(\beta),\beta},
                \Lambda_{\delta(\beta),\beta},
                V_{\beta,\mathrm{ex}}
            \bigr).
        \end{aligned}
    \end{equation}
\end{definition}

The word seeds \(\mathbf s(\beta)\) and \(\mathbf s_t(\beta)\)
have the full position set \([r]\). By contrast, the restricted
seeds \(\mathbf s(\delta(\beta),\beta)\) and
\(\mathbf s_t(\delta(\beta),\beta)\) have vertex set \(V_\beta\)
and are obtained only after the mutation--freezing--deletion
procedure in Definition~\ref{def:svbeta}. We use these symbols only
for their respective seeds; an identification later in the paper
always means an explicit seed isomorphism, including its vertex
relabeling and its number of variables.

We fix the conventions used when applying \cite{qin2024analogs}.
Write \(q_{\mathrm Q}\) for its quantum parameter and
\((B_{\mathrm Q},\Lambda_{\mathrm Q})\) for its exchange matrix and
commutation form. We use its compatibility convention
\[
 \Lambda_{\mathrm Q}B_{\mathrm Q}=-D_{\mathrm Q},
\]
where \(D_{\mathrm Q}\) is positive diagonal, as in
\cite[Section~4.1 and Example~7.8]{qin2024analogs}.
For the corresponding seed in this paper, after relabeling and
normalizing the diagonal entries to \(2\), the translation is
\begin{equation}\label{eq:Qin-quantum-convention-translation}
 B_{\mathrm Q}=-B,
 \qquad \Lambda_{\mathrm Q}=-\Lambda,
 \qquad q_{\mathrm Q}^{1/2}=t^{-1/2}.
\end{equation}
Here \(\Lambda\) is the commutation form in the parameter \(t\).
Thus \(\Lambda_{\mathrm Q}B_{\mathrm Q}=-2I\) is equivalent to
\(B^{\mathsf T}\Lambda=2I\). The normalized toric monomials are
identified without rescaling, since
\[
 q_{\mathrm Q}^{\Lambda_{\mathrm Q}(\mathbf g,\mathbf h)/2}
 =t^{\Lambda(\mathbf g,\mathbf h)/2}.
\]
The matrix sign change interchanges the two exchange monomials, so
this identification commutes with quantum mutation and with bar.

The dominance relation
\(\mathbf h=\mathbf g+B_{\mathrm Q}\mathbf n\),
\(\mathbf n\geq0\), consequently becomes
\(\mathbf h=\mathbf g-B\mathbf n\), consistently with
\eqref{eq:pointed-element}. All triangularity conditions imported
from \cite{qin2024analogs}, including conditions on coefficient
exponents, are understood after the same substitution. In
particular, its negative-power coefficient lattice
\(q_{\mathrm Q}^{-1/2}\ZZ[q_{\mathrm Q}^{-1/2}]\) becomes
\(t^{1/2}\ZZ[t^{1/2}]\). This fixes the parameter orientation in
our use of the term \emph{common triangular basis}. Under the later
identification \(t^{1/2}=q^{-1/2}\) with the bosonic parameter,
one has \(q_{\mathrm Q}^{1/2}=q^{1/2}\), and the same lattice is
\(q^{-1/2}\ZZ[q^{-1/2}]\).

We shall use the following quantum \(\mathcal A=\mathcal U\) result.

\begin{theorem}[Quantum \(\mathcal A=\mathcal U\) for braid-variety seeds]
\label{thm:quantum-A-U-braid-seeds}
For every positive braid word \(\beta\), the quantum seed
\(\mathbf s_t(\delta(\beta),\beta)\) is injective-reachable, and
\begin{equation}\label{eq:quantum-A-U-braid-seeds}
    \mathcal A_t\bigl(
        \mathbf s_t(\delta(\beta),\beta)
    \bigr)
    =
    \mathcal U_t\bigl(
        \mathbf s_t(\delta(\beta),\beta)
    \bigr).
\end{equation}
Moreover, this algebra admits an
\(\mathcal M^\circ\)-pointed common triangular basis.
\end{theorem}

\begin{proof}
Let \(\mathbf t_{\mathrm Q}(\beta)\) be Qin's braid-variety seed.
By Definition~\ref{def:svbeta} and
\cite[Lemma~7.2]{qin2024analogs}, it corresponds, after the
canonical relabeling and the exchange-matrix convention change, to
\(\mathbf s_t(\delta(\beta),\beta)\).

Qin's word seed admits a green-to-red sequence by
\cite[Proposition~8.7]{qin2024analogs}. The seed
\(\mathbf t_{\mathrm Q}(\beta)\) is obtained from it by successively
freezing mutable vertices and deleting nonessential frozen vertices
\cite[Lemma~7.2]{qin2024analogs}. After freezing, the remaining
mutable quiver is an induced subquiver, so it still admits a
green-to-red sequence by
\cite[Theorem~3.1.3]{muller2016existence}. Deleting a nonessential
frozen vertex does not change the mutable quiver and only removes an
unrestricted frozen component of the degree vectors. Hence
\(\mathbf t_{\mathrm Q}(\beta)\), and therefore
\(\mathbf s_t(\delta(\beta),\beta)\), is injective-reachable.

By \cite[Theorem~7.3]{qin2024analogs}, Qin's braid-variety seed
satisfies quantum \(\mathcal A=\mathcal U\) and admits an
\(\mathcal M^\circ\)-pointed common triangular basis. These
properties are preserved under the required change of quantization
by \cite[Corollary~4.12 and Proposition~4.18]{qin2024analogs} and equation~\eqref{eq:Qin-quantum-convention-translation}.
Thus the corresponding quantum tori and their mutation maps are
identified; the sign change of the exchange matrix merely
interchanges the two exchange monomials. This identification also
preserves the translated dominance order and hence the
\(\mathcal M^\circ\)-pointed common triangular basis. Therefore
Qin's result gives \eqref{eq:quantum-A-U-braid-seeds} and the
assertion about the basis.
\end{proof}

\section{Hernandez--Leclerc categories}\label{sec:Hernandez}

We use the categorical conventions of \cite{kashiwara2024monoidal}.
Let \(U_{\qaff}'(\widehat{\fg})\) be the quantum affine algebra associated with
the Cartan matrix \(C\), and let \(\mathscr C_{\fg}\) be the category of
finite-dimensional integrable \(U_{\qaff}'(\widehat{\fg})\)-modules. This is a
rigid monoidal category: every object \(M\in\mathscr C_{\fg}\) admits a
right dual \(\mathcal D M\) and a left dual \(\mathcal D^{-1}M\).

Here \(q_{\mathrm{aff}}\) is the fixed deformation parameter of the
quantum affine algebra. It is distinct from the formal parameter \(q\)
of the bosonic extension introduced in
Section~\ref{sec:bosonic-extension}, and from the quantum-cluster and
quantum-Grothendieck parameter \(t\). Whenever the latter two
realizations are compared, we use the coefficient-ring identification
\(t^{1/2}=q^{-1/2}\). The variable \(z\) below is reserved for spectral
parameters.

The simple objects of \(\mathscr C_{\fg}\) are classified by \(I\)-tuples
of Drinfeld polynomials \((P_i^V(z))_{i\in I}\) with constant term \(1\).
For \((i,p)\in I\times\ZZ\), let \(L(i,p)\) be the fundamental module
corresponding to
\begin{equation}\label{eq:fundamental-Drinfeld-polynomial}
    P_i(z)=1-q_{\mathrm{aff}}^p z,
    \qquad
    P_j(z)=1 \quad (j\neq i).
\end{equation}
Fix a height function \(\xi\colon I\to\ZZ\) such that
\(\lvert\xi(i)-\xi(j)\rvert=1\) whenever \(c_{ij}=-1\), and set
\begin{equation}\label{eq:I-hat-xi}
    \widehat I_\xi
    :=\{(i,p)\in I\times\ZZ\mid p-\xi(i)\in2\ZZ\}.
\end{equation}
The Hernandez--Leclerc category \(\mathscr C^{\ZZ}\) is the smallest full
subcategory of \(\mathscr C_{\fg}\) containing the modules \(L(i,p)\),
\((i,p)\in\widehat I_\xi\), and closed under tensor products, extensions,
and subquotients.

\subsection{\texorpdfstring{\(R\)-matrices and numerical invariants}
{R-matrices and numerical invariants}}

Let \(M,N\in\mathscr C^{\ZZ}\), and let \(N_z\) denote the affinization of
\(N\) with spectral parameter \(z\). If \(c_{M,N}(z)\) is a renormalizing
coefficient for the universal \(R\)-matrix, then
\(R^{\mathrm{ren}}_{M,N_z}
:=c_{M,N}(z)R^{\mathrm{univ}}_{M,N_z}\) is regular and nonzero at
\(z=1\). Its specialization defines a nonzero homomorphism
\begin{equation}\label{eq:renormalized-R-matrix}
    \mathbf r_{M,N}
    :=\left.R^{\mathrm{ren}}_{M,N_z}\right|_{z=1}
    \colon M\otimes N\longrightarrow N\otimes M.
\end{equation}
We write \(\LambdaR(M,N)\) for the integer-valued \(R\)-matrix invariant
determined by the renormalizing coefficient \(c_{M,N}(z)\), with the
normalization of \cite{kashiwara2020monoidal}, and set
\begin{equation}\label{eq:d-invariant}
    \mathfrak d(M,N)
    :=\frac{\LambdaR(M,N)+\LambdaR(N,M)}{2}
    \in\ZZ_{\geq0}.
\end{equation}

A simple module \(M\) is called \emph{real} if \(M\otimes M\) is simple.
Two simple modules \(M\) and \(N\) are said to \emph{strongly commute} if
\(M\otimes N\) is simple. We use the notation
\(M\nabla N:=\operatorname{hd}(M\otimes N)\) and
\(M\mathbin{\triangle}N:=\operatorname{soc}(M\otimes N)\). If at least
one of \(M\) and \(N\) is real, then both modules are simple and
\begin{equation}\label{eq:image-R-matrix-head-socle}
    \operatorname{Im}\mathbf r_{M,N}
    \simeq M\nabla N
    \simeq N\mathbin{\triangle}M.
\end{equation}
Under the same assumption, strong commutation is equivalent to the equality
\(\mathfrak d(M,N)=0\).

We shall repeatedly use the following cancellation property.

\begin{lemma}[{\cite[Corollary~3.13]{kang2015simplicity}}]\label{lem:cancel}
    Let \(L\) be a real simple module and \(X\) a simple module. Then
    \begin{equation}\label{eq:categorical-cancellation}
        (L\nabla X)\nabla\mathcal D L
        \simeq L\nabla(X\nabla\mathcal D L)
        \simeq X.
    \end{equation}
\end{lemma}

A sequence of simple modules \(\mathbf L=(L_1,\ldots,L_m)\) is called
\emph{normal} if the composition of renormalized \(R\)-matrices associated
with a reduced expression of the longest permutation,
\begin{equation}\label{eq:R-matrix-normal-sequence}
    \mathbf r_{\mathbf L}\colon
    L_1\otimes\cdots\otimes L_m
    \longrightarrow
    L_m\otimes\cdots\otimes L_1,
\end{equation}
is nonzero. It is called \emph{almost real} if all but at most one of its
terms are real. If \(\mathbf L\) is almost real and normal, then
\(\operatorname{Im}\mathbf r_{\mathbf L}\) is simple and is isomorphic to
\(\operatorname{hd}(L_1\otimes\cdots\otimes L_m)\).

\begin{lemma}[\cite{kashiwara2019laurent,kashiwara2023localizations}]\label{lem:normal-sequence}
    Let \(\mathbf L=(L_1,\ldots,L_m)\) be an almost real normal sequence,
    and put \(\mathbf L'=(L_2,\ldots,L_{m})\). Then
    \begin{equation}\label{eq:Lambda-normal-sequence}
        \LambdaR\bigl(L_1,\operatorname{Im}\mathbf r_{\mathbf L'}\bigr)
        =\sum_{k=2}^{m}\LambdaR(L_1,L_k).
    \end{equation}
\end{lemma}

We also recall two consequences of the functoriality of renormalized
\(R\)-matrices. Let \(L,M,N\) be simple modules, and let
\(S\) be a simple subquotient of \(M\otimes N\). Then
\begin{equation}\label{eq:R-matrix-subquotient-inequalities}
    \mathfrak d(S,L)
    \leq\mathfrak d(M,L)+\mathfrak d(N,L),
    \qquad
    \LambdaR(S,L)
    \leq\LambdaR(M,L)+\LambdaR(N,L).
\end{equation}
If \(L\) is simple and a surjection \(f\colon M\twoheadrightarrow N\)
satisfies \(\LambdaR(L,M)=\LambdaR(L,N)\), then
\begin{equation}\label{eq:R-matrix-naturality}
    (f\otimes\mathrm{id}_L)\circ\mathbf r_{L,M}=\mathbf r_{L,N}\circ(\mathrm{id}_L\otimes f)
\end{equation}
with compatible choices of the renormalized maps. We also use the
following criterion for the head of a tensor product.

\begin{lemma}[\cite{kashiwara2020monoidal}]\label{lem:head}
    Let \(M\) be a real simple module, and let \(N\) have simple head
    \(H:=\operatorname{hd}N\). If the canonical projection
    \(\pi\colon N\twoheadrightarrow H\) satisfies
    \begin{equation}\label{eq:head-criterion-R-matrix}
        (\pi\otimes\operatorname{id}_M)\circ\mathbf r_{M,N}
        =\mathbf r_{M,H}\circ(\operatorname{id}_M\otimes\pi),
    \end{equation}
    then \(M\otimes N\) has simple head and
    \(\operatorname{hd}(M\otimes N)\simeq M\nabla H\).
\end{lemma}

We next record the ordered commutation conditions used below.

\begin{definition}
    Let \(M,N\in\mathscr C^{\ZZ}\) be simple modules. The ordered pair
    \((M,N)\) is called \emph{unmixed} if
    \(\mathfrak d(\mathcal D M,N)=0\), and \emph{strongly unmixed} if
    \(\mathfrak d(\mathcal D^kM,N)=0\) for every \(k\geq1\).

    An almost real sequence \(\mathbf L=(L_1,\ldots,L_r)\) is called
    \emph{unmixed}, respectively \emph{strongly unmixed}, if every ordered
    pair \((L_i,L_j)\), \(1\leq i<j\leq r\), has the corresponding
    property.
\end{definition}

These notions depend on the ordering of the modules, and every strongly
unmixed pair is unmixed.

\begin{lemma}[{\cite[Lemma~5.3]{kashiwara2023pbw}}]
    Every unmixed sequence of real simple modules is normal.
\end{lemma}

\subsection{Duality data and root modules}

A real simple module \(L\in\mathscr C^{\ZZ}\) is called a \emph{root
module} if
\begin{equation}\label{eq:root-module}
    \mathfrak d(L,\mathcal D^kL)
    =\delta_{k,-1}+\delta_{k,1}
    \qquad (k\in\ZZ).
\end{equation}

\begin{definition}\label{def:strong-duality-datum}
    A family of root modules \(\DD=(L_i^{\DD})_{i\in I}\) is called a
    \emph{strong duality datum} associated with
    \(C=(c_{ij})_{i,j\in I}\) if
    \begin{equation}\label{eq:strong-duality-datum}
        \mathfrak d\bigl(L_i^{\DD},\mathcal D^kL_j^{\DD}\bigr)
        =-\delta_{k,0}c_{ij}
    \end{equation}
    for all \(k\in\ZZ\) and all distinct \(i,j\in I\).
\end{definition}

 For
\(-\infty\leq a\leq b\leq+\infty\), let
\(\mathscr C_{\DD}[a,b]\) be the smallest full subcategory of
\(\mathscr C^{\ZZ}\) containing \(\mathcal D^kL_i^{\DD}\) for
\(i\in I\) and \(a\leq k\leq b\), and closed under tensor products,
extensions, and subquotients. We abbreviate
\begin{equation}\label{eq:shifted-D-categories}
    \mathscr C_{\DD}:=\mathscr C_{\DD}[0,0],
    \qquad
    \mathscr C_{\DD}[k]:=\mathscr C_{\DD}[k,k].
\end{equation}

A strong duality datum \(\DD\) is called \emph{complete} if, for every
simple \(M\in\mathscr C^{\ZZ}\), there exist simple objects
\(M_k\in\mathscr C_{\DD}\), \(k\in\ZZ\), all but finitely many of which
are isomorphic to \(\mathbf1\), such that
\begin{equation}\label{eq:complete-duality-datum-factorization}
    M\simeq\operatorname{hd}\left(
        \mathop{\bigotimes_{k\in\ZZ}^{\longleftarrow}}
        \mathcal D^kM_k
    \right).
\end{equation}
Here the factors are ordered with decreasing \(k\) from left to right; thus
the tensor product in \eqref{eq:complete-duality-datum-factorization} has
the form
\(\cdots\otimes\mathcal D^2M_2\otimes\mathcal D M_1\otimes M_0
\otimes\mathcal D^{-1}M_{-1}\otimes\cdots\). In particular, the
subcategories \(\mathscr C_{\DD}[k]\), \(k\in\ZZ\), monoidally generate
\(\mathscr C^{\ZZ}\).

\begin{example}\label{ex:Q-datum}
Let \(\Delta_{\fg}\) be the Dynkin diagram associated with \(\fg\),
let \(\sigma\) be a Dynkin diagram automorphism, and let
\(\xi\colon I\to\ZZ\) be a height function compatible with
\(\sigma\). The triple
\(\mathcal Q=(\Delta_{\fg},\sigma,\xi)\) is called a
\emph{\(Q\)-datum}. In the simply-laced types \(ADE\) considered in
this paper, we take \(\sigma=\mathrm{id}\).

Associated with \(\mathcal Q\) is a bijection
\begin{equation*}
\phi_{\mathcal Q}\colon
\widehat I_{\xi}\xrightarrow{\ \sim\ }\Phi^+\times\ZZ.
\end{equation*}
For each \(i\in I\), let \((j,p)\in\widehat I_{\xi}\) be the unique
pair satisfying \(\phi_{\mathcal Q}(j,p)=(\alpha_i,0)\), and define
\(L_i^{\mathcal Q}:=L(j,p)\). Then
\begin{equation*}
\DD_{\mathcal Q}
:=
\{L_i^{\mathcal Q}\mid i\in I\}
\end{equation*}
is the complete strong duality datum associated with
\(\mathcal Q\).
\end{example}

For \(i\in I\), define a new datum
\(\mathscr S_i(\DD)=(L_j^{\mathscr S_i(\DD)})_{j\in I}\) by
\begin{equation}\label{eq:duality-datum-braid-action}
    L_j^{\mathscr S_i(\DD)}
    :=\begin{cases}
        \mathcal D L_i^{\DD},
        & j=i,\\
        L_i^{\DD}\nabla L_j^{\DD},
        & c_{ij}=-1,\\
        L_j^{\DD},
        & \text{otherwise}.
    \end{cases}
\end{equation}
By \cite[Proposition~5.9]{kashiwara2023pbw},
\(\mathscr S_i(\DD)\) is again a strong duality datum, and it is complete
whenever \(\DD\) is complete. Moreover, the induced operators on
isomorphism classes of complete duality data satisfy
\begin{equation}\label{eq:duality-datum-braid-relations}
    \mathscr S_i\mathscr S_j
    =\mathscr S_j\mathscr S_i
    \quad\text{if }c_{ij}=0,
    \qquad
    \mathscr S_i\mathscr S_j\mathscr S_i
    =\mathscr S_j\mathscr S_i\mathscr S_j
    \quad\text{if }c_{ij}=-1
\end{equation}
by \cite[Corollary~3.22]{kashiwara2025monoidal}.

Let \(\beta=(i_1,\ldots,i_r)\) be a positive braid word, and let \(\DD\)
be a strong duality datum. Define recursively
\begin{equation}\label{eq:duality-data-along-beta}
    \DD_\beta^{(0)}:=\DD,
    \qquad
    \DD_\beta^{(k)}
    :=\mathscr S_{i_k}(\DD_\beta^{(k-1)})
    \quad (1\leq k\leq r),
\end{equation}
and set
\begin{equation}\label{eq:root-modules-along-beta}
    C_k^{\DD,\beta}
    :=L_{i_k}^{\DD_\beta^{(k-1)}}
    \qquad (1\leq k\leq r).
\end{equation}
We denote by \(\mathscr C_{\DD}(\beta)\) the smallest full subcategory of
\(\mathscr C^{\ZZ}\) containing \(\mathbf1\) and the root modules
\(C_k^{\DD,\beta}\), \(1\leq k\leq r\), and closed under tensor
products, extensions, and subquotients.

For \(j\in I\) and \([a,b]\subseteq[1,r]\), set
\begin{equation}\label{eq:colored-interval-endpoints}
    a(j)^+:=\min\{s\in[a,r]\mid i_s=j\},
    \qquad
    b(j)^-:=\max\{s\in[1,b]\mid i_s=j\},
\end{equation}
whenever the relevant set is nonempty, and define
\(\{a,b]:=[a(i_b)^+,b]\) and
\([a,b\}:=[a,b(i_a)^-]\). The interval \([a,b]\) is called a
\(\beta\)-\emph{box} if \(i_a=i_b\). For a \(\beta\)-box, put
\([a,b]_\beta^\circ:=\{s\in[a,b]\mid i_s=i_a\}\) and define the
associated affine determinantial module by
\begin{equation}\label{eq:affine-determinantial-module}
    M^{\DD,\beta}[a,b]
    :=\operatorname{hd}\left(
        \mathop{\bigotimes_{s\in[a,b]_\beta^\circ}^{\longleftarrow}}
        C_s^{\DD,\beta}
    \right).
\end{equation}
Thus, the tensor factors in \eqref{eq:affine-determinantial-module} occur
in decreasing order of their positions. When \(\beta\) is fixed, we simply
write \(M^{\DD}[a,b]\).

\begin{proposition}[{\cite[Lemma~5.26 and Proposition~5.33]
    {kashiwara2025monoidal}}]\label{pro:lusztig-parameter-modules}
    Let \(\beta=(i_1,\ldots,i_r)\) be a positive braid word, and let
    \(\DD\) be a complete duality datum. For every simple module
    \(M\in\mathscr C_{\DD}(\beta)\), there exists a unique vector
    \(\bfa=(a_1,\ldots,a_r)\in\ZZ_{\geq0}^r\) such that
    \begin{equation}\label{eq:Lusztig-parameter-factorization}
        M\simeq\operatorname{hd}\left(
            (C_r^{\DD,\beta})^{\otimes a_r}\otimes\cdots\otimes
            (C_1^{\DD,\beta})^{\otimes a_1}
        \right).
    \end{equation}
    We call \(\bfa\) the \((\DD,\beta)\)-\emph{Lusztig parameter} of
    \(M\) and denote it by \(\bfa^{\DD,\beta}(M)\). If \(\DD\) is
    fixed, we simply write \(\bfa^\beta(M)\).
\end{proposition}

\subsection{Quantum Grothendieck rings}

For the remainder of this subsection, assume that
\(U_{\qaff}'(\widehat{\fg})\) is of untwisted affine type and fix a datum
\((\Delta_{\fg},\xi)\), where \(\xi\) is a height function on the
Dynkin diagram \(\Delta_{\fg}\). Set
\begin{equation}\label{eq:Y-Laurent-ring}
    \mathcal Y
    :=\ZZ[Y_{i,p}^{\pm1}\mid(i,p)\in\widehat I_\xi].
\end{equation}
Frenkel and Reshetikhin constructed an injective ring homomorphism
\begin{equation}\label{eq:q-character-homomorphism}
    \chi_q\colon K_0(\mathscr C^{\ZZ})\hookrightarrow\mathcal Y,
\end{equation}
called the \emph{\(q\)-character homomorphism}
\cite{frenkel1999qcharacters}.

Let \(\mathcal M\) be the set of Laurent monomials in \(\mathcal Y\). For
\(m\in\mathcal M\), write
\(m=\prod_{(i,p)\in\widehat I_\xi}Y_{i,p}^{u_{i,p}(m)}\). The monomial
\(m\) is called \emph{dominant} if \(u_{i,p}(m)\geq0\) for every
\((i,p)\in\widehat I_\xi\); denote the set of dominant monomials by
\(\mathcal M^+\). For each \(m\in\mathcal M^+\), let \(L(m)\) be the
unique simple module with Drinfeld polynomials
\begin{equation}\label{eq:Drinfeld-polynomials-dominant-monomial}
    P_i^{L(m)}(z)
    =\prod_{p\in\ZZ}(1-q_{\mathrm{aff}}^p z)^{u_{i,p}(m)}
    \qquad (i\in I).
\end{equation}
In particular, \(L(Y_{i,p})\simeq L(i,p)\) and
\(L(1)\simeq\mathbf1\).

Let \(A\) be the adjacency matrix of \(\Delta_{\fg}\), let \(I\) be the identity
matrix, and set \(C(z):=(z+z^{-1})I-A\). The matrix \(C(z)\) is
invertible over \(\QQ(z)\), and we write
\begin{equation}\label{eq:inverse-quantum-Cartan-matrix}
    \widetilde C(z)
    :=C(z)^{-1}
    =\sum_{k\geq0}(z+z^{-1})^{-k-1}A^k
    =\bigl(\widetilde c_{ij}(z)\bigr)_{i,j\in I}.
\end{equation}
Expanding at \(z=0\), define integers \(\widetilde c_{ij}(m)\) by
\(\widetilde c_{ij}(z)=\sum_{m\geq1}\widetilde c_{ij}(m)z^m\), and set
\(\widetilde c_{ij}(m)=0\) for \(m\leq0\). Since
\(\widetilde C(z)\) is symmetric, one has
\(\widetilde c_{ij}(m)=\widetilde c_{ji}(m)\).

For \((i,p),(j,s)\in\widehat I_\xi\), define
\begin{equation}\label{eq:N-commutation-form}
    \begin{aligned}
        \mathcal N(i,p;j,s)
        :={}&\widetilde c_{ij}(p-s-1)
             -\widetilde c_{ij}(p-s+1)\\
            &-\widetilde c_{ij}(s-p-1)
             +\widetilde c_{ij}(s-p+1).
    \end{aligned}
\end{equation}
Then \(\mathcal N(i,p;j,s)=-\mathcal N(j,s;i,p)\).

Let \(\mathcal Y_t\) be the \(\ZZ[t^{\pm1/2}]\)-algebra generated by
invertible elements \(\widetilde Y_{i,p}\),
\((i,p)\in\widehat I_\xi\), subject to
\begin{equation}\label{eq:Y-t-commutation-relation}
    \widetilde Y_{i,p}\widetilde Y_{j,s}
    =t^{\mathcal N(i,p;j,s)}
      \widetilde Y_{j,s}\widetilde Y_{i,p}.
\end{equation}
Extend \(\mathcal N\) bilinearly to \(\mathcal M\) by
\begin{equation}\label{eq:N-form-on-monomials}
    \mathcal N(m,m')
    :=\sum_{(i,p),(j,s)\in\widehat I_\xi}
      u_{i,p}(m)u_{j,s}(m')\mathcal N(i,p;j,s).
\end{equation}
For simple modules \(X\simeq L(m)\) and \(Y\simeq L(m')\), we also
write \(\mathcal N(X,Y):=\mathcal N(m,m')\).

Specialization at \(t=1\) gives a surjective ring homomorphism
\begin{equation}\label{eq:t-one-specialization}
    \operatorname{ev}_{t=1}\colon\mathcal Y_t\twoheadrightarrow\mathcal Y,
    \qquad
    t^{1/2}\longmapsto1,
    \quad
    \widetilde Y_{i,p}\longmapsto Y_{i,p}.
\end{equation}
The quantum torus \(\mathcal Y_t\) carries the bar involution determined by
\(\overline{t^{1/2}}=t^{-1/2}\),
\(\overline{\widetilde Y_{i,p}}=\widetilde Y_{i,p}\), and
\(\overline{xy}=\overline y\,\overline x\).

For each \(m\in\mathcal M^+\), the Kazhdan--Lusztig-type construction
produces a distinguished bar-invariant element
\(L_t(m)\in\mathcal Y_t\), called the \emph{\((q,t)\)-character} of
\(L(m)\). If \(M\simeq L(m)\) is simple, we write
\([M]_t:=L_t(m)\).

\begin{definition}
    The \emph{quantum Grothendieck ring} of \(\mathscr C^{\ZZ}\) is the
    \(\ZZ[t^{\pm1/2}]\)-subalgebra
    \begin{equation}\label{eq:quantum-Grothendieck-ring}
        \mathcal{K}_t(\mathscr C^{\ZZ})
        =\mathcal K_{\fg,t}
        :=\ZZ[t^{\pm1/2}]
          \bigl\langle[M]_t
          \mid M\in\operatorname{Irr}(\mathscr C^{\ZZ})\bigr\rangle
        \subseteq\mathcal Y_t.
    \end{equation}
\end{definition}

The set \(\{L_t(m)\mid m\in\mathcal M^+\}\) is a
\(\ZZ[t^{\pm1/2}]\)-basis of \(\mathcal{K}_t(\mathscr C^{\ZZ})\). A simple module
\(M\) is called \emph{quantizable} if
\(\operatorname{ev}_{t=1}([M]_t)=\chi_q([M])\).

\begin{lemma}[{\cite[Lemma~6.8]{kashiwara2025monoidal}}]\label{lem:Nm1m2}
    Let \(L(m_1)\) and \(L(m_2)\) be quantizable simple modules, at least
    one of which is real. Assume that
    \(\mathfrak d(L(m_1),L(m_2))=0\) and that \(L(m_1m_2)\) is
    quantizable. Then
    \begin{equation}\label{eq:N-equals-Lambda}
        \mathcal N(m_1,m_2)
        =\mathcal N\bigl(L(m_1),L(m_2)\bigr)
        =\LambdaR\bigl(L(m_1),L(m_2)\bigr),
    \end{equation}
    and
    \begin{equation}\label{eq:quantum-character-commuting-product}
        t^{-\mathcal N(m_1,m_2)/2}L_t(m_1)L_t(m_2)
        =L_t(m_1m_2)
        =t^{\mathcal N(m_1,m_2)/2}L_t(m_2)L_t(m_1).
    \end{equation}
\end{lemma}

\begin{remark}\label{rem:quantizable}
    If \(\fg\) is of type \(ADE\), every simple module in
    \(\mathscr C^{\ZZ}\) is quantizable by
    \cite{nakajima2004quiver}.
\end{remark}

\subsection{Monoidal categorification}

Let \(\mathscr C\) be a full subcategory of \(\mathscr C^{\ZZ}\) that
contains the monoidal unit \(\mathbf1\) and is closed under tensor products,
subquotients, and extensions. We denote its Grothendieck ring by
\(K(\mathscr C)\). We also write \(\mathcal{K}_t(\mathscr C)\) for the
\(\ZZ[t^{\pm1/2}]\)-subalgebra of \(\mathcal{K}_t(\mathscr C^{\ZZ})\) generated by
the classes \([M]_t\), where \(M\) runs over the simple objects of
\(\mathscr C\).

\subsubsection{Monoidal seeds and mutations}

\begin{definition}\label{def:monoidal-seed}
    A \emph{monoidal seed} in \(\mathscr C\) is a tuple
    \begin{equation}\label{eq:monoidal-seed}
        S=\bigl(
            \{M_i\}_{i\in J},\widetilde B;
            J,J_{\mathrm{ex}}
        \bigr),
    \end{equation}
    where \(J\) is a finite index set,
    \(J_{\mathrm{ex}}\subseteq J\) is the set of exchangeable vertices,
    \(\{M_i\}_{i\in J}\) is a pairwise strongly commuting family of real
    simple modules in \(\mathscr C\), and
    \(\widetilde B=(b_{ij})_{i\in J,\,j\in J_{\mathrm{ex}}}\) is an
    extended exchange matrix. We set
    \(J_{\mathrm{fr}}:=J\setminus J_{\mathrm{ex}}\), and call \(M_i\) the
    \(i\)-th \emph{cluster-variable module} of \(S\).
\end{definition}

The \(R\)-matrix invariants of the cluster-variable modules define the
skew-symmetric matrix
\begin{equation}\label{eq:Lambda-monoidal-seed}
    \Lambda^S=(\lambda_{ij}^S)_{i,j\in J},
    \qquad
    \lambda_{ij}^S:=\LambdaR(M_i,M_j).
\end{equation}
Indeed, the modules \(M_i\) and \(M_j\) strongly commute, and hence
\(\mathfrak d(M_i,M_j)=0\).

\begin{definition}\label{def:monoidal-mutation}
    Let \(k\in J_{\mathrm{ex}}\). We say that \(S\) \emph{admits a
    mutation in direction \(k\)} if there exists a simple module
    \(M_k'\in\mathscr C\) that strongly commutes with \(M_i\) for every
    \(i\neq k\), together with a short exact sequence
    \begin{equation}\label{eq:monoidal-mutation}
        0\longrightarrow
        \bigotimes_{b_{ik}>0}M_i^{\otimes b_{ik}}
        \longrightarrow M_k\otimes M_k'
        \longrightarrow
        \bigotimes_{b_{ik}<0}M_i^{\otimes(-b_{ik})}
        \longrightarrow0.
    \end{equation}
    Since the modules \(M_i\), \(i\neq k\), pairwise strongly commute,
    the two outer tensor products are independent of the chosen order up to
    isomorphism. If, in addition, \(\mathfrak d(M_k,M_k')=1\), then this
    mutation is called a \emph{\(\Lambda\)-mutation}.
\end{definition}

If a mutation module \(M_k'\) exists, it is unique up to isomorphism and is
real. Setting \(M_k^{(k)}:=M_k'\) and \(M_i^{(k)}:=M_i\) for \(i\neq k\),
one obtains another monoidal seed
\begin{equation}\label{eq:mutated-monoidal-seed}
    \mu_k(S)
    :=\bigl(
        \{M_i^{(k)}\}_{i\in J},
        \mu_k(\widetilde B);
        J,J_{\mathrm{ex}}
    \bigr)
\end{equation}
by \cite[Lemma~4.7]{kashiwara2024monoidal}.

\begin{definition}\label{def:Lambda-admissible}
    A monoidal seed \(S\) is called \emph{\(\Lambda\)-admissible} if it
    admits a \(\Lambda\)-mutation in every direction
    \(k\in J_{\mathrm{ex}}\). It is called \emph{completely
    \(\Lambda\)-admissible} if every monoidal seed obtained from \(S\) by
    successive \(\Lambda\)-mutations is \(\Lambda\)-admissible.
\end{definition}

If \(S\) is \(\Lambda\)-admissible, then
\begin{equation}\label{eq:monoidal-seed-compatibility}
    \sum_{\ell\in J}b_{\ell k}\lambda_{\ell i}^S
    =2\delta_{ik}
    \qquad (i\in J,\ k\in J_{\mathrm{ex}}).
\end{equation}
Thus \((\widetilde B,\Lambda^S)\) is a compatible pair with the
conventions fixed above.

The following restriction property will be used later.

\begin{lemma}[{\cite[Lemma~7.15]{kashiwara2024monoidal}}]\label{lem:subseed}
    Let
    \(S=(\{M_i\}_{i\in J},\widetilde B;J,J_{\mathrm{ex}})\) be a
    monoidal seed. Choose subsets \(J^*\subseteq J\) and
    \(J_{\mathrm{ex}}^*\subseteq J^*\cap J_{\mathrm{ex}}\), and set
    \begin{equation}\label{eq:restricted-monoidal-seed}
        S^*
        :=\bigl(
            \{M_i\}_{i\in J^*},
            \widetilde B_{J^*\times J_{\mathrm{ex}}^*};
            J^*,J_{\mathrm{ex}}^*
        \bigr).
    \end{equation}
    Assume that
    \begin{equation}\label{eq:restriction-zero-condition}
        b_{ij}=0
        \qquad
        (i\in J\setminus J^*,\ j\in J_{\mathrm{ex}}^*).
    \end{equation}
    Then the following assertions hold.

    \begin{enumerate}
        \item If \(s,j\in J_{\mathrm{ex}}^*\) and
        \(i\in J\setminus J^*\), then
        \((\mu_s(\widetilde B))_{ij}=0\). Hence mutations in directions
        belonging to \(J_{\mathrm{ex}}^*\) preserve
        \eqref{eq:restriction-zero-condition}.

        \item Suppose that \(S\) is \(\Lambda\)-admissible. Restriction
        commutes with mutation in the sense that
        \begin{equation}\label{eq:mutation-commutes-with-restriction}
            (\mu_sS)|_{(J^*,J_{\mathrm{ex}}^*)}
            =\begin{cases}
                \mu_s(S^*),
                & s\in J_{\mathrm{ex}}^*,\\
                S^*,
                & s\in J_{\mathrm{ex}}\setminus J^*.
            \end{cases}
        \end{equation}
    \end{enumerate}
    Consequently, if \(S\) is completely \(\Lambda\)-admissible, then
    \(S^*\) is also completely \(\Lambda\)-admissible.
\end{lemma}

\begin{proof}
For \(s\in J_{\mathrm{ex}}^*\), condition
\eqref{eq:restriction-zero-condition} and matrix mutation imply
\((\mu_s\widetilde B)_{ij}=0\) for
\(i\notin J^*\), \(j\in J_{\mathrm{ex}}^*\). The exchange sequence at
\(s\) therefore involves only modules indexed by \(J^*\), so it also
realizes the mutation of \(S^*\). If
\(s\in J_{\mathrm{ex}}\setminus J^*\), the restricted modules and
matrix are unchanged because \(b_{sj}=0\) for
\(j\in J_{\mathrm{ex}}^*\). This proves
\eqref{eq:mutation-commutes-with-restriction}; iteration gives the last
assertion.
\end{proof}

\subsubsection{Classical and quantum monoidal categorifications}

\begin{definition}\label{def:classical-monoidal-categorification}
The category \(\mathscr C\) is a \emph{monoidal categorification} of
\(\overline{\mathcal A}(\mathbf t)\) if it contains a completely
\(\Lambda\)-admissible monoidal seed
\(S=(\{M_i\},\widetilde B;J,J_{\mathrm{ex}})\) and there is an
isomorphism
\[
    K(\mathscr C)\xrightarrow{\sim}
    \overline{\mathcal A}(\mathbf t)
\]
identifying
\(
    [S]=(\{[M_i]\}_{i\in J},\widetilde B,J_{\mathrm{ex}})
\)
with \(\mathbf t\).
\end{definition}

A monoidal seed is called \emph{quantizable} if all of its
cluster-variable modules are quantizable. For such a seed, one has
\begin{equation}\label{eq:quantum-classes-commutation}
    [M_i]_t[M_j]_t
    =t^{\LambdaR(M_i,M_j)}[M_j]_t[M_i]_t
    \qquad (i,j\in J).
\end{equation}

\begin{definition}\label{def:quantum-monoidal-categorification}
The category \(\mathscr C\) is a \emph{monoidal categorification} of
\(\overline{\mathcal A}_t(\mathbf t)\) if it contains a completely
\(\Lambda\)-admissible monoidal seed \(S\), every seed in its mutation
class is quantizable, and there is an isomorphism
\[
    \mathcal K_t(\mathscr C)\xrightarrow{\sim}
    \overline{\mathcal A}_t(\mathbf t)
\]
identifying
\begin{equation}\label{eq:quantum-seed-from-monoidal-seed}
    [S]_t
    :=\bigl(\{[M_i]_t\}_{i\in J},\widetilde B,
             \Lambda^S,J_{\mathrm{ex}}\bigr)
\end{equation}
with \(\mathbf t\).
\end{definition}

Under this identification, every normalized quantum cluster monomial is the
\((q,t)\)-character of a real simple module in \(\mathscr C\).

An isomorphism between algebras with distinguished bar-invariant bases is
called \emph{based} if it identifies those bases. Here the distinguished
bases are the normalized simple classes in a quantum Grothendieck ring and
the normalized global basis in a bosonic extension algebra.

\begin{definition}\label{def:based-categorification}
Let \(\mathbf L\) be a common triangular basis of
\(\mathcal U_t(\bt)\). A full abelian monoidal subcategory
\(\mathscr C\subseteq\mathscr C^{\ZZ}\) is a
\emph{based categorification} of \(\mathcal U_t(\bt)\) if:
\begin{enumerate}
    \item every frozen variable is represented by a simple module
    \(M_j\) strongly commuting with every simple object of \(\mathscr C\);
    \item \(\mathscr C\) contains a completely \(\Lambda\)-admissible,
    quantizable monoidal seed representing \(\bt\), with frozen modules
    \(M_j\);
    \item there is an isomorphism
    \begin{equation}\label{eq:based-categorification-isomorphism}
        \Phi_{\mathscr C}\colon
        \mathcal{K}_t(\mathscr C)_{\mathrm{loc}}
        \xrightarrow{\ \sim\ }
        \mathcal U_t(\bt),
    \end{equation}
    localized at the frozen classes, which maps the normalized localized
    simple classes bijectively onto \(\mathbf L\).
\end{enumerate}

A localized simple class is the bar-normalized class of
\(
    [S]_t\prod_j[M_j]_t^{-d_j}
\), where \(S\) is simple and \(d_j\in\NN\). Such classes are
well defined by the strong commutation of the frozen modules.
\end{definition}

\subsubsection{Monoidal seeds associated with braid words}

Let \(\beta=(i_1,\ldots,i_r)\) be a positive braid word and let \(\DD\)
be a strong duality datum. For \(k\in[r]\), abbreviate
\begin{equation}\label{eq:initial-affine-determinantial-modules}
    M_k^{\DD,\beta}
    :=M^{\DD,\beta}\{1,k]
    =M^{\DD,\beta}[1(i_k)^+,k],
\end{equation}
where \(1(i_k)^+\) is the first occurrence of the color \(i_k\) in
\([1,k]\). Define
\begin{equation}\label{eq:word-monoidal-seed}
    S^{\DD}(\beta)
    :=\bigl(
        \{M_k^{\DD,\beta}\}_{k\in[r]},
        \widetilde B_\beta;
        [r],[r]_{\mathrm{ex}}
    \bigr).
\end{equation}

\begin{theorem}[{\cite[Theorem~9.4]{kashiwara2025monoidal}}]
\label{thm:Lambda-admissible-seed}
Let \(\DD\) be a complete duality datum and let \(\beta\) be a
positive braid word. Then \(S^{\DD}(\beta)\) is a completely
\(\Lambda\)-admissible monoidal seed in
\(\mathscr C_{\DD}(\beta)\). Moreover, there is a ring isomorphism
\begin{equation}\label{eq:word-seed-categorification}
   \Psi_\DD:\, K\bigl(\mathscr C_{\DD}(\beta)\bigr)
    \simeq
    \overline{\mathcal A}\bigl(\mathbf s(\beta)\bigr)
\end{equation}
that identifies \([S^{\DD}(\beta)]\) with the initial seed
\(\mathbf s(\beta)\). Consequently, \(\mathscr C_{\DD}(\beta)\)
provides a monoidal categorification of
\(\overline{\mathcal A}(\mathbf s(\beta))\).
\end{theorem}

When the duality datum arises from a \(Q\)-datum, the preceding
categorification admits a quantum refinement.

\begin{theorem}[{\cite[Theorem~9.12]{kashiwara2025monoidal}}]
\label{thm:word-seed-monoidal-categorification}
Assume that the ambient quantum affine algebra is of untwisted affine
type and that \(\DD=\DD_{\mathcal Q}\) is the complete strong duality
datum associated with a \(Q\)-datum \(\mathcal Q\). Then
\(\mathscr C_{\DD}(\beta)\) provides a monoidal categorification of
the quantum cluster algebra
\(\overline{\mathcal A}_t(\mathbf s_t(\beta))\). More precisely,
\begin{equation}\label{eq:quantum-word-seed-categorification}
    [S^{\DD}(\beta)]_t=\mathbf s_t(\beta),
    \qquad
    \mathcal{K}_t\bigl(\mathscr C_{\DD}(\beta)\bigr)
    \simeq
    \overline{\mathcal A}_t\bigl(\mathbf s_t(\beta)\bigr).
\end{equation}
\end{theorem}

We denote the inverse of the quantum isomorphism in
\eqref{eq:quantum-word-seed-categorification} by
\begin{equation}\label{eq:quantum-word-seed-realization}
    \Psi_{\DD}^{q,\beta}\colon
    \overline{\mathcal A}_t\bigl(\mathbf s_t(\beta)\bigr)
    \xrightarrow{\sim}
    \mathcal K_t\bigl(\mathscr C_{\DD}(\beta)\bigr).
\end{equation}
It sends each initial variable to the normalized quantum class of
the corresponding affine determinantial module and extends after
localizing the frozen variables.

We next introduce the monoidal subcategory associated with the
Demazure product. Let \(\beta\) be a braid word, and set
\begin{equation*}
    w:=\delta(\beta),\qquad
    m:=\ell(w),\qquad
    \ell:=\ell(w_0).
\end{equation*}
Choose a reduced expression \((j_1,\ldots,j_m)\) of \(w\) and a
reduced expression \((j_{m+1},\ldots,j_\ell)\) of \(w^{-1}w_0\).
Since
\(
    \ell(w)+\ell(w^{-1}w_0)=\ell(w_0),
\)
the concatenation
\((j_1,\ldots,j_\ell)\) is a reduced expression of \(w_0\).

Recall that \(j^*\in I\) is determined by
\(w_0(\alpha_j)=-\alpha_{j^*}\). Extend
\((j_1,\ldots,j_\ell)\) to the right-infinite sequence
\begin{equation}\label{eq:periodic-w0-sequence}
    \overline{\mathbf i}
    :=
    (j_1,\ldots,j_\ell,
     j_1^*,\ldots,j_\ell^*,
     j_1,\ldots,j_\ell,\ldots).
\end{equation}
Define \(\mathscr C_{\DD}^{w}\) to be the smallest full monoidal
subcategory of \(\mathscr C^{\ZZ}\) that contains
\(C_k^{\DD,\overline{\mathbf i}}\) for every \(k>m\) and is closed
under extensions, subquotients, and tensor products.
For a positive braid word \(\beta\), set
\begin{equation}\label{eq:B-D-beta-category}
    \mathscr B_{\DD}(\beta)
    :=
    \mathscr C_{\DD}(\beta)
    \cap
    \mathscr C_{\DD}^{\delta(\beta)}.
\end{equation}
This is again a full monoidal subcategory closed under extensions and
subquotients.

\begin{theorem}\label{thm:quantum-braid-seed-embedding}
Assume that the ambient quantum affine algebra is of untwisted affine
type and that \(\DD=\DD_{\mathcal Q}\) is the complete strong duality
datum associated with a \(Q\)-datum \(\mathcal Q\). Under the standing quantizability assumption, there is
a natural injective \(\ZZ[t^{\pm1/2}]\)-algebra homomorphism
\begin{equation}\label{eq:quantum-braid-seed-embedding}
    \overline{\mathcal A}_t
    \bigl(\mathbf s_t(\delta(\beta),\beta)\bigr)
    \hookrightarrow
    \mathcal{K}_t\bigl(\mathscr B_{\DD}(\beta)\bigr).
\end{equation}
Under this embedding, every quantum cluster variable is sent to the
quantum class \([M]_t\) of its cluster-variable module \(M\).
\end{theorem}

\begin{proof}
By Theorems~\ref{thm:Lambda-admissible-seed} and
\ref{thm:word-seed-monoidal-categorification},
\(S^{\DD}(\beta)\) is a completely \(\Lambda\)-admissible,
quantizable monoidal seed categorifying \(\mathbf s_t(\beta)\).
Apply the mutation sequence \(M_m\) from
Definition~\ref{def:svbeta}, and write
\[
    S^{(m)}=M_m(S^{\DD}(\beta))
\]
with cluster-variable modules \(\{M_v\}\) and exchange matrix
\(\widetilde B_\beta^{(m)}\). Let \(V_\beta\) be the vertices left
after deleting \(D_\beta\), and let \(V_{\beta,\mathrm{ex}}\) be those
remaining exchangeable after freezing. By construction,
\begin{equation}\label{eq:deleted-zero-block-proof}
    \bigl(\widetilde B_\beta^{(m)}\bigr)_{
        D_\beta\times V_{\beta,\mathrm{ex}}}
    =0.
\end{equation}
Hence Lemma~\ref{lem:subseed} shows that
\begin{equation}\label{eq:restricted-braid-monoidal-seed}
    S_{\delta(\beta),\beta}^{\DD}
    :=
    \bigl(
        \{M_v\}_{v\in V_\beta},
        \widetilde B_\beta^{(m)}
          |_{V_\beta\times V_{\beta,\mathrm{ex}}};
        V_\beta,V_{\beta,\mathrm{ex}}
    \bigr)
\end{equation}
is completely \(\Lambda\)-admissible and quantizable, and its quantum
seed is \(\mathbf s_t(\delta(\beta),\beta)\). Since restriction
commutes with mutations in \(V_{\beta,\mathrm{ex}}\), the ambient
categorification restricts to an injective homomorphism
\begin{equation}\label{eq:ambient-realization-restricted-seed}
    \Phi\colon
    \overline{\mathcal A}_t
    \bigl(\mathbf s_t(\delta(\beta),\beta)\bigr)
    \longrightarrow
    \mathcal{K}_t(\mathscr C^{\ZZ}),
    \qquad
    X_v\longmapsto[M_v]_t.
\end{equation}
Indeed, it is the restriction of the faithful realization in
Theorem~\ref{thm:word-seed-monoidal-categorification}.

By \cite[Theorem~6.8]{bi2026towards}, every mutable or frozen
cluster-variable module of this restricted seed and its mutation class
belongs to
\begin{equation}\label{eq:cluster-variable-module-in-B-category}
    \mathscr C_{\DD}(\beta)
    \cap
    \mathscr C_{\DD}^{\delta(\beta)}
    =
    \mathscr B_{\DD}(\beta).
\end{equation}
Quantizability, Lemma~\ref{lem:Nm1m2}, and the exchange sequence
\eqref{eq:monoidal-mutation} identify each quantum cluster variable
with the corresponding class \([M]_t\). Since the nonlocalized
quantum cluster algebra is generated by its mutable and frozen cluster
variables, \eqref{eq:cluster-variable-module-in-B-category} implies
\[
    \operatorname{Im}\Phi
    \subseteq
    \mathcal K_t(\mathscr B_{\DD}(\beta)).
\]
Thus \(\Phi\) corestricts to the injective homomorphism
\eqref{eq:quantum-braid-seed-embedding}.
\end{proof}

\section{Bosonic extension algebras}\label{sec:bosonic-extension}

Let \(C=(c_{ij})_{i,j\in I}\) be a symmetric Cartan matrix of simply-laced
Dynkin type, and let \(q\) be an indeterminate. For
\(m,n\in\ZZ_{\geq 0}\) with \(n\leq m\), set
\begin{equation*}
    [n]_q:=\frac{q^n-q^{-n}}{q-q^{-1}},\qquad
    [n]_q!:=\prod_{k=1}^n[k]_q,\qquad
    \qbinom{m}{n}:=\frac{[m]_q!}{[n]_q![m-n]_q!}.
\end{equation*}
Let \(U_q(\fn)\) denote the half quantum group generated by
\(\{f_i\mid i\in I\}\), and let \(\cA(\fn)\) be its graded
\(\QQ(q)\)-dual algebra. Let \(U_{\ZZ[q^{\pm1}]}(\fn)\) be the
\(\ZZ[q^{\pm1}]\)-subalgebra generated by the divided powers
\(f_i^n/[n]_q!\). Its restricted dual integral form is
\begin{equation}\label{eq:integral-form-half-quantum-group}
    \cA_{\ZZ[q^{\pm1}]}(\fn)
    :=\left\{
        \psi\in\cA(\fn)
        \ \middle|\
        \psi\bigl(U_{\ZZ[q^{\pm1}]}(\fn)\bigr)
        \subseteq\ZZ[q^{\pm1}]
      \right\}.
\end{equation}
We extend scalars by setting
\(\cA_{\ZZ[q^{\pm1/2}]}(\fn)
:=\cA_{\ZZ[q^{\pm1}]}(\fn)
\otimes_{\ZZ[q^{\pm1}]}\ZZ[q^{\pm1/2}]\).

\begin{definition}
    The \emph{bosonic extension algebra} \(\widehat{\cA}\) is the
    \(\QQ(q^{1/2})\)-algebra generated by
    \(\{f_{i,m}\mid i\in I,\ m\in\ZZ\}\), subject to the relations
    \begin{align}
        \sum_{a=0}^{1-c_{ij}}&(-1)^a
        \qbinom{1-c_{ij}}{a}
        f_{i,p}^{\,1-c_{ij}-a}f_{j,p}f_{i,p}^{\,a}
        =0,
        \label{eq:Serre}\\
        f_{i,m}f_{j,p}
        &=q^{(-1)^{p-m+1}c_{ij}}f_{j,p}f_{i,m}
        +\delta_{(i,m+1),(j,p)}(1-q^2).
        \label{eq:bosonic}
    \end{align}
    Here \eqref{eq:Serre} holds for \(i\neq j\) and \(p\in\ZZ\),
    whereas \eqref{eq:bosonic} holds for \(m<p\). The symbol
    \(\delta_{x,y}\) denotes the Kronecker delta.
\end{definition}

The algebra \(\widehat{\cA}\) is \(Q\)-graded by
\(\wt(f_{i,m})=(-1)^m\alpha_i\). Since
\eqref{eq:Serre}--\eqref{eq:bosonic} are homogeneous with respect to this
grading, one has
\(\widehat{\cA}=\bigoplus_{\alpha\in Q}\widehat{\cA}_\alpha\).

For \(-\infty\leq a\leq b\leq+\infty\), denote by
\(\widehat{\cA}[a,b]\) the subalgebra generated by the elements
\(f_{i,m}\) with \(i\in I\) and \(a\leq m\leq b\). In particular, we
write
\begin{align*}
    \widehat{\cA}_{\geq p}
    &:=\widehat{\cA}[p,+\infty] \quad(p\in\ZZ),&
    \widehat{\cA}_{<0}
    &:=\widehat{\cA}[-\infty,-1],\\
    \widehat{\cA}_{\geq0}
    &:=\widehat{\cA}[0,+\infty],&
    \widehat{\cA}_{\geq1}
    &:=\widehat{\cA}[1,+\infty],\\
    \widehat{\cA}[m]
    &:=\widehat{\cA}[m,m].&&
\end{align*}
For every finite interval \([a,b]\subset\ZZ\), ordered multiplication
induces a vector-space isomorphism
\begin{equation}\label{eq:ordered-factorization}
    \widehat{\cA}[b]\otimes\widehat{\cA}[b-1]\otimes\cdots
    \otimes\widehat{\cA}[a]
    \xrightarrow{\ \sim\ }
    \widehat{\cA}[a,b].
\end{equation}
Passing to the direct limit over finite intervals gives the restricted
ordered tensor-product decomposition
\(\widehat{\cA}\simeq
\overrightarrow{\bigotimes}_{m\in\ZZ}\widehat{\cA}[m]\). In
particular, multiplication identifies
\(\widehat{\cA}_{\geq0}\otimes\widehat{\cA}_{<0}\) with
\(\widehat{\cA}\) as a vector space.

For each \(k\in\ZZ\), extension of scalars from \(\QQ(q)\) to
\(\QQ(q^{1/2})\) yields an algebra isomorphism
\begin{equation}\label{eq:level-isomorphism}
    \varphi_k\colon
    \cA(\fn)\otimes_{\QQ(q)}\QQ(q^{1/2})
    \xrightarrow{\ \sim\ }
    \widehat{\cA}[k],
    \qquad
    \langle i\rangle\otimes1\longmapsto q^{1/2}f_{i,k},
\end{equation}
where \(\langle i\rangle\in\cA(\fn)\) is dual to
\(f_i\in U_q(\fn)\). Let
\(\widehat{\cA}_{\ZZ[q^{\pm1/2}]}\) be the
\(\ZZ[q^{\pm1/2}]\)-subalgebra generated by
\(\varphi_k(\cA_{\ZZ[q^{\pm1/2}]}(\fn))\) for all \(k\in\ZZ\).

We shall use two \(\QQ\)-algebra anti-automorphisms of
\(\widehat{\cA}\): the bar involution and the shift duality. They are
determined by
\begin{equation*}
    \overline{f_{i,m}}=f_{i,m},\qquad
    \overline{q^{1/2}}=q^{-1/2},\qquad
    \mathcal D(f_{i,m})=f_{i,m+1},\qquad
    \mathcal D(q^{1/2})=q^{-1/2}.
\end{equation*}
For a homogeneous element \(x\in\widehat{\cA}_\alpha\), define
\(c(x):=q^{(\alpha,\alpha)/2}\overline{x}\), and extend \(c\)
\(\QQ\)-linearly to \(\widehat{\cA}\).

\subsection{Extended crystals and global bases}

Let \(B(\infty)\) be the crystal basis of \(U_q(\mathfrak n)\), with
weight-zero element \(\mathbf1\). Its \emph{extended crystal} is
\begin{equation}\label{eq:extended-crystal}
    \widehat B(\infty)
    :=\left\{
        (b_k)_{k\in\ZZ}\in\prod_{k\in\ZZ}B(\infty)
        \ \middle|\
        b_k=\mathbf1\text{ for all but finitely many }k
    \right\},
\end{equation}
For \(\bb=(b_k)_{k\in\ZZ}\in\widehat B(\infty)\), define
\begin{equation}\label{eq:standard-basis-element}
    P(\bb)
    :=\overrightarrow{\prod_{k\in\ZZ}}
      \varphi_k\bigl(G^{\mathrm{up}}(b_k)\bigr)
    \in\widehat{\cA},
\end{equation}
where the product is taken in the order of
\eqref{eq:ordered-factorization}. These elements form a
\(\QQ(q^{1/2})\)-basis of \(\widehat{\cA}\).

For \(\gamma=\sum_i a_i\alpha_i\in Q\), put
\(\|\gamma\|=\sum_i|a_i|\alpha_i\). For
\(\bb,\bb'\in\widehat B(\infty)\), write \(\bb'\preceq\bb\) if
\(\|\wt(b'_k)\|\leq\|\wt(b_k)\|\) for every \(k\in \ZZ\).
This relation is a preorder. Its strict part is
\begin{equation}\label{eq:strict-extended-crystal-preorder}
    \bb'\prec\bb
    \quad\Longleftrightarrow\quad
    \bb'\preceq\bb\ \text{ and }\ \bb\not\preceq\bb'.
\end{equation}

\begin{theorem}[{\cite[Theorem~7.6]{kashiwara2025global}}]
    \label{theo:globalbasis}
    For every \(\bb\in\widehat B(\infty)\), there exists a unique
    element \(G(\bb)\in\widehat{\cA}\) such that
    \begin{equation}\label{eq:global-basis-triangularity}
        G(\bb)-P(\bb)
        \in\sum_{\bb'\prec\bb}q\ZZ[q]P(\bb'),
        \qquad
        c\bigl(G(\bb)\bigr)=G(\bb).
    \end{equation}
    The family
    \(\{G(\bb)\mid\bb\in\widehat B(\infty)\}\) is a
    \(\QQ(q^{1/2})\)-basis of \(\widehat{\cA}\), called the
    \emph{global basis}.
\end{theorem}

 We distinguish the preceding \(c\)-invariant
global basis from its bar-invariant normalization. If \(G(\bb)\) is
homogeneous of weight \(\alpha_{\bb}\), set
\begin{equation}\label{eq:normalized-global-basis-conversion}
    \widetilde G(\bb)
    :=q^{-(\alpha_{\bb},\alpha_{\bb})/4}G(\bb).
\end{equation}
Since \(c(G(\bb))=G(\bb)\), one has
\(\overline{\widetilde G(\bb)}=\widetilde G(\bb)\).
Throughout the remainder of the paper, concrete cluster variables
and normalized simple classes are identified with
\(\widetilde G\), not with the unnormalized element \(G\).

\subsection{Braid symmetries and PBW bases}

\begin{proposition}[{\cite{kashiwara2021braid}}]
    \label{pro:braidsym}
    For every \(i\in I\), there exist mutually inverse
    \(\QQ(q^{1/2})\)-algebra automorphisms
    \(T_i,T_i^*\colon\widehat{\cA}\to\widehat{\cA}\), determined by
    \begin{align}
        T_i(f_{j,m})
        &=
        \begin{cases}
            f_{j,m+\delta_{ij}}, & d(i,j)\neq1,\\[1mm]
            \displaystyle
            \frac{q^{1/2}f_{j,m}f_{i,m}
                  -q^{-1/2}f_{i,m}f_{j,m}}{q-q^{-1}},
                & d(i,j)=1,
        \end{cases}
        \label{eq:braid-symmetry-T}\\
        T_i^*(f_{j,m})
        &=
        \begin{cases}
            f_{j,m-\delta_{ij}}, & d(i,j)\neq1,\\[1mm]
            \displaystyle
            \frac{q^{1/2}f_{i,m}f_{j,m}
                  -q^{-1/2}f_{j,m}f_{i,m}}{q-q^{-1}},
                & d(i,j)=1.
        \end{cases}
        \label{eq:braid-symmetry-T-star}
    \end{align}
    Here \(d(i,j)\) denotes the distance between \(i\) and \(j\) in
    the Dynkin diagram. The families \(\{T_i\}_{i\in I}\) and
    \(\{T_i^*\}_{i\in I}\) satisfy the commutation and braid relations,
    and \(T_iT_i^*=T_i^*T_i=\Id\).
\end{proposition}

Let \(b\in\Br^+\), and choose an expression
\(\beta=(i_1,\ldots,i_r)\) of \(b\). The braid relations imply that
\(T_b:=T_{i_1}\cdots T_{i_r}\) depends only on \(b\), and not on the
chosen expression. Moreover, the braid symmetries preserve the global
basis of \(\widehat{\cA}\); see
\cite[Theorem~3.7]{kashiwara2024braid}. More precisely, they preserve
the \(c\)-invariant basis \(G\). Since the braid action preserves the
weight form, \eqref{eq:normalized-global-basis-conversion} implies
that it also preserves the bar-invariant basis \(\widetilde G\).
Below, ``fixed by \(c\)'' always refers to \(G\), whereas
``bar-invariant'' refers to \(\widetilde G\); based embeddings are
compatible with both statements through this conversion.

For \(k\in[r]\), define the PBW root vector
\begin{equation}\label{eq:PBW-root-vector}
    P_k^\beta
    :=T_{i_1}\cdots T_{i_{k-1}}\bigl(q^{1/2}f_{i_k,0}\bigr).
\end{equation}
This is the \(c\)-invariant normalization. Its bar-invariant version is
\begin{equation}\label{eq:bar-normalized-PBW-root-vector}
    \widetilde P_k^\beta
    :=q^{-1/2}P_k^\beta
    =T_{i_1}\cdots T_{i_{k-1}}(f_{i_k,0}).
\end{equation}
Let \(\widehat{\cA}(\beta)\) be the subalgebra of
\(\widehat{\cA}\) generated by \(P_1^\beta,\ldots,P_r^\beta\). By
\cite[Proposition~4.7]{kashiwara2024braid},
\begin{equation}\label{eq:PBW-subalgebra-intersection}
    \widehat{\cA}(\beta)
    =\widehat{\cA}_{\geq0}
     \cap T_b\bigl(\widehat{\cA}_{<0}\bigr).
\end{equation}
Consequently, although the generators \(P_k^\beta\) depend on the word
\(\beta\), the algebra \(\widehat{\cA}(\beta)\) depends only on the
positive braid \(b\). We use the integral form
\begin{equation}\label{eq:integral-form-PBW-subalgebra}
    \widehat{\cA}_{\ZZ[q^{\pm1/2}]}(\beta)
    :=\widehat{\cA}(\beta)
      \cap\widehat{\cA}_{\ZZ[q^{\pm1/2}]}.
\end{equation}

For \(\mathbf a=(a_1,\ldots,a_r)\in\ZZ_{\geq0}^r\), define the
normalized PBW monomial
\begin{equation}\label{eq:normalized-PBW-monomial}
    P^\beta(\mathbf a)
    :=q^{\frac12\sum_{k=1}^r a_k(a_k-1)}
      (P_r^\beta)^{a_r}\cdots(P_1^\beta)^{a_1}.
\end{equation}
We equip \(\ZZ_{\geq0}^r\) with the right lexicographic total order. For
distinct \(\mathbf a,\mathbf b\in\ZZ_{\geq0}^r\), write
\(\mathbf a\prec_{\mathrm r}\mathbf b\) if there exist
\(1\leq k_1\leq r\) such that
\begin{align}
    a_k=b_k\quad(k>k_1), \quad a_{k_1}<b_{k_1}.
    \label{eq:order2}
\end{align}

\begin{proposition}[{\cite[Lemmas~4.16 and~4.17]
    {kashiwara2024braid}}]\label{prop:PBW-global-basis}
    For every \(\mathbf a\in\ZZ_{\geq0}^r\), there exists a unique
    global basis element
    \(G(\beta,\mathbf a)\in\widehat{\cA}(\beta)\) satisfying
    \begin{equation}\label{eq:PBW-global-basis-triangularity}
        c\bigl(G(\beta,\mathbf a)\bigr)=G(\beta,\mathbf a),
        \qquad
        G(\beta,\mathbf a)
        \in P^\beta(\mathbf a)
        +\sum_{\mathbf b\prec_{\mathrm r}\mathbf a}q\ZZ[q]P^\beta(\mathbf b).
    \end{equation}
    Moreover,
    \(\{G(\beta,\mathbf a)\mid
       \mathbf a\in\ZZ_{\geq0}^r\}\)
    is the global basis of \(\widehat{\cA}(\beta)\). We call
    \(\mathbf a\) the \(\beta\)-Lusztig parameter of
    \(G(\beta,\mathbf a)\).
\end{proposition}

Let \(\alpha_\beta(\mathbf a):=\wt(G(\beta,\mathbf a))\). In
accordance with \eqref{eq:normalized-global-basis-conversion}, define
\begin{equation}\label{eq:normalized-PBW-global-basis-conversion}
    \widetilde G(\beta,\mathbf a)
    :=q^{-(\alpha_\beta(\mathbf a),
            \alpha_\beta(\mathbf a))/4}
      G(\beta,\mathbf a).
\end{equation}
Since every PBW root has squared length \(2\), the singleton
parameters satisfy
\begin{equation}\label{eq:singleton-PBW-normalization}
    G(\beta,E_k)=P_k^\beta,
    \qquad
    \widetilde G(\beta,E_k)=\widetilde P_k^\beta.
\end{equation}
For comparison of ordered products, our conventions give
\begin{equation}\label{eq:ordered-normalized-PBW-product}
    (\widetilde P_r^\beta)^{a_r}\cdots
    (\widetilde P_1^\beta)^{a_1}
    =q^{-\frac12\sum_k a_k^2}P^\beta(\mathbf a).
\end{equation}
Thus the scalar conversion changes no Lusztig parameter.
The elements \(G(\beta,\mathbf a)\) and
\(\widetilde G(\beta,\mathbf a)\) have the same
\(\beta\)-Lusztig parameter \(\mathbf a\); the former is
\(c\)-invariant and the latter is bar-invariant. Note also that the
order \(\mathbf b\prec_{\mathrm r}\mathbf a\) in
\eqref{eq:PBW-global-basis-triangularity} is weaker than the
bi-lexicographic order
\(\mathbf b\prec_{\operatorname{bi}}\mathbf a\) used in
\cite[Lemma~4.16]{kashiwara2024braid}.

The following triangularity property of multiplication will be used
repeatedly.

\begin{lemma}[{\cite[Lemma~4.9]{bi2026towards}}]
    \label{lem:global-basis-product-triangularity}
    For any \(\mathbf a,\mathbf b\in\ZZ_{\geq0}^r\), there exists
    \(A\in\frac12\ZZ\) such that
    \begin{equation}\label{eq:global-basis-product-triangularity}
        \widetilde G(\beta,\mathbf a)
        \widetilde G(\beta,\mathbf b)
        =q^A \widetilde G(\beta,\mathbf a+\mathbf b)
         +\sum_{\mathbf c\prec_{\mathrm r}\mathbf a+\mathbf b}
           g_{\mathbf c}(q)\widetilde G(\beta,\mathbf c),
        \qquad
        g_{\mathbf c}(q)\in\ZZ[q^{\pm1/2}].
    \end{equation}
\end{lemma}

An interval \([a,b]\subset[r]\) is called a \(\beta\)-box if
\(i_a=i_b\). For such a box, define
\([a,b]_\beta\in\ZZ_{\geq0}^r\) by
\begin{equation}\label{eq:beta-box-parameter}
    \bigl([a,b]_\beta\bigr)_k
    :=\begin{cases}
        1, & a\leq k\leq b\text{ and }i_k=i_a,\\
        0, & \text{otherwise}.
    \end{cases}
\end{equation}
The corresponding generalized quantum minor is
\(D^\beta[a,b]:=\widetilde G(\beta,[a,b]_\beta)\).

\begin{theorem}[
{\cite[Theorems~6.2, 6.10 and~9.7]
{kashiwara2025monoidal}}]
\label{thm:isoKt}
Let \(\beta\) be a positive braid word. Assume that the ambient quantum
affine algebra is of untwisted affine type and that
\(\DD=\DD_{\mathcal Q}\) is the complete strong duality datum associated
with a \(Q\)-datum \(\mathcal Q\). After identifying \(t^{1/2}\) with
\(q^{-1/2}\), there is a based algebra isomorphism
\begin{equation}\label{eq:Kt-bosonic-isomorphism}
    \Phi_{\DD}\colon
    \mathcal{K}_t\bigl(\mathscr C_{\DD}(\beta)\bigr)
    \xrightarrow{\ \sim\ }
    \widehat{\cA}_{\ZZ[q^{\pm1/2}]}(\beta).
\end{equation}
This isomorphism is the restriction of the ambient based isomorphism
from the quantum Grothendieck ring of \(\mathscr C^{\ZZ}\) to the
bosonic extension algebra \(\widehat{\cA}\). In particular,
\(\Phi_{\DD}\) sends the normalized quantum class of every simple
object to an element of the normalized global basis.

Consequently, every quantum cluster monomial is mapped to a
normalized global basis element of \(\widehat{\cA}(\beta)\). More
precisely, for every \(\beta\)-box \([a,b]\), one has
\begin{equation}\label{eq:determinantial-module-minor}
    \Phi_{\DD}
    \left(
        \bigl[M^{\DD,\beta}[a,b]\bigr]_t
    \right)
    =
    D^\beta[a,b]
    =
    \widetilde G\bigl(\beta,[a,b]_\beta\bigr),
\end{equation}
where the left-hand side is the normalized \((q,t)\)-character of the
affine determinantial module \(M^{\DD,\beta}[a,b]\).

Moreover, \(\Phi_{\DD}\) is compatible with the subcategories
determined by the level filtration. In particular,
\begin{equation}\label{eq:interval-compatibility-Phi}
\begin{aligned}
    \Phi_{\DD}\!\left(
        \mathcal{K}_t\bigl(
            \mathscr C_{\DD}(\beta)
            \cap
            \mathscr C_{\DD}^{\delta(\beta)}
        \bigr)
    \right)
    &=
    \widehat{\cA}_{\ZZ[q^{\pm1/2}]}(\beta)
    \cap
    T_{\delta(\beta)}\widehat{\cA}_{\geq0}.
\end{aligned}
\end{equation}
\end{theorem}

\begin{proof}
The existence of the based isomorphism
\eqref{eq:Kt-bosonic-isomorphism}, its compatibility with normalized
simple classes, and
\eqref{eq:determinantial-module-minor} follow from the cited
theorems. It remains to prove
\eqref{eq:interval-compatibility-Phi}.

Set \(w:=\delta(\beta)\), and let
\(\overline{\mathbf i}\) be the right-infinite sequence defined in
\eqref{eq:periodic-w0-sequence}. For every \(k>m=\ell(w)\), the
ambient isomorphism satisfies
\begin{equation}\label{eq:fundamental-module-PBW-vector}
    \Phi_{\DD}
    \left(
        [C_k^{\DD,\overline{\mathbf i}}]_t
    \right)
    =
    \widetilde P_k^{\overline{\mathbf i}}
    =q^{-1/2}P_k^{\overline{\mathbf i}},
\end{equation}
where \(\widetilde P_k^{\overline{\mathbf i}}\) is the bar-invariant
PBW root vector of \eqref{eq:bar-normalized-PBW-root-vector}.
By the PBW description of the level subalgebra, the elements
\(\widetilde P_k^{\overline{\mathbf i}}\), \(k>m\), generate
\(T_w\widehat{\cA}_{\geq0}\). Since
\(\mathscr C_{\DD}^{w}\) is generated by the objects
\(C_k^{\DD,\overline{\mathbf i}}\), \(k>m\), it follows that
\begin{equation}\label{eq:image-level-subcategory}
    \Phi_{\DD}
    \left(
        \mathcal{K}_t(\mathscr C_{\DD}^{w})
    \right)
    =
    T_w\widehat{\cA}_{\geq0}\cap\widehat{\cA}_{\ZZ[q^{\pm1/2}]}.
\end{equation}
On the other hand, the restricted isomorphism
\eqref{eq:Kt-bosonic-isomorphism} gives
\begin{equation}\label{eq:image-beta-category}
    \Phi_{\DD}
    \left(
        \mathcal{K}_t(\mathscr C_{\DD}(\beta))
    \right)
    =
    \widehat{\cA}_{\ZZ[q^{\pm1/2}]}(\beta).
\end{equation}

We now use the based property of \(\Phi_{\DD}\). The quantum
Grothendieck ring of each subcategory under consideration is spanned
over \(\ZZ[t^{\pm1/2}]\) by the normalized quantum classes of its
simple objects. Under \(\Phi_{\DD}\), these classes are mapped to
distinct elements of the normalized global basis. Therefore, for any
two such full subcategories \(\mathscr C_1,\mathscr C_2\), one has
\begin{equation}\label{eq:based-intersection-property}
    \Phi_{\DD}\!\left(
        \mathcal{K}_t(\mathscr C_1\cap\mathscr C_2)
    \right)
    =
    \Phi_{\DD}\!\left(\mathcal{K}_t(\mathscr C_1)\right)
    \cap
    \Phi_{\DD}\!\left(\mathcal{K}_t(\mathscr C_2)\right).
\end{equation}
Indeed, both sides are spanned by the normalized global basis
elements whose corresponding simple objects belong simultaneously
to \(\mathscr C_1\) and \(\mathscr C_2\).

Applying \eqref{eq:based-intersection-property} to
\(\mathscr C_1=\mathscr C_{\DD}(\beta)\) and
\(\mathscr C_2=\mathscr C_{\DD}^{w}\), and then using
\eqref{eq:image-level-subcategory} and
\eqref{eq:image-beta-category}, we obtain
\begin{equation*}
\begin{aligned}
    \Phi_{\DD}\!\left(
        \mathcal{K}_t\bigl(
            \mathscr C_{\DD}(\beta)
            \cap
            \mathscr C_{\DD}^{w}
        \bigr)
    \right)
    &=
    \widehat{\cA}_{\ZZ[q^{\pm1/2}]}(\beta)
    \cap
    T_w\widehat{\cA}_{\geq0}.
\end{aligned}
\end{equation*}
Since \(w=\delta(\beta)\), this is precisely
\eqref{eq:interval-compatibility-Phi}.
\end{proof}

\begin{definition}\label{def:two-level-subalgebras}
For a positive braid word \(\beta\), define side by side
\begin{align}
    \widetilde{\mathbf A}(\beta)
    &:=
    \widehat{\cA}_{\ZZ[q^{\pm1/2}]}(\beta)
    \cap
    T_{\delta(\beta)}\widehat{\cA}_{\geq0},
    \label{eq:A-beta-definition}\\
    \mathbf A(\beta)
    &:=
    \widehat{\cA}_{\ZZ[q^{\pm1/2}]}(\beta)
    \cap\widehat{\cA}_{\geq1},
    \qquad
    \mathscr D_{\DD}(\beta)
    :=
    \mathscr C_{\DD}(\beta)
    \cap\mathscr C_{\DD}[1,\infty).
    \label{eq:A-D-beta-definition}
\end{align}
The tilde distinguishes the condition defined by \(\delta(\beta)\)
from the level-\(\geq1\) condition. The latter defines
\(\mathbf A(\beta)\) and corresponds to
\(\mathscr D_{\DD}(\beta)\).
\end{definition}

\begin{corollary}\label{cor:cluster-algebra-embedding-A-beta}
Let \(\beta\) be a positive braid word. After identifying \(t^{1/2}\) with
\(q^{-1/2}\), there is an injective
\(\ZZ[q^{\pm1/2}]\)-algebra homomorphism
\begin{equation}\label{eq:cluster-algebra-embedding-A-beta}
    \overline{\mathcal A}_t
    \bigl(\mathbf s_t(\delta(\beta),\beta)\bigr)
    \hookrightarrow
    \widetilde{\mathbf A}(\beta).
\end{equation}
Under this embedding, every quantum cluster monomial is mapped to a
normalized global basis element of \(\widehat{\cA}(\beta)\) belonging
to \(\widetilde{\mathbf A}(\beta)\).
\end{corollary}

\begin{proof}
Compose the embedding of
Theorem~\ref{thm:quantum-braid-seed-embedding} with \(\Phi_{\DD}\).
By \eqref{eq:interval-compatibility-Phi} and the definition of
\(\mathscr B_{\DD}(\beta)\),
\begin{equation}\label{eq:image-B-category-A-beta}
 \Phi_{\DD}\!\left(
   \mathcal K_t(\mathscr B_{\DD}(\beta))\right)
 =\widehat{\cA}_{\ZZ[q^{\pm1/2}]}(\beta)
      \cap T_{\delta(\beta)}\widehat{\cA}_{\geq0}
 =\widetilde{\mathbf A}(\beta).
\end{equation}
This proves the claimed embedding. A quantum cluster monomial is
the normalized class of a tensor product of powers of real,
pairwise strongly commuting cluster-variable modules. This tensor
product is real and simple and belongs to the monoidal category
\(\mathscr B_{\DD}(\beta)\). The based property of
\(\Phi_{\DD}\) sends its class to a normalized global-basis
element; \eqref{eq:image-B-category-A-beta} places it in
\(\widetilde{\mathbf A}(\beta)\).
\end{proof}

\section{Braid varieties}\label{sec:braid-variety}

Throughout this section, the Cartan matrix
\(C=(c_{ij})_{i,j\in I}\) is assumed to be simply laced. We recall the
notion of braid varieties in \cite{casals2025cluster} and fix
the conventions used below.

Fix a pair of opposite Borel subgroups \(B^+\) and \(B^-\). Let \(U^+\)
and \(U^-\) be their respective unipotent radicals, and set
\(H:=B^+\cap B^-\). For flags \(F=xB^+\) and \(F'=yB^+\) in \(G/B^+\),
we write \(F\xrightarrow{w}F'\) if
\(x^{-1}y\in B^+\dot wB^+\), where \(\dot w\) is a fixed representative
of \(w\in W\). In particular, \(F\xrightarrow{s_i}F'\) means that the
two flags are in relative position \(s_i\).

We fix a pinning of \(G\). Thus, for every \(i\in I\), there is a
homomorphism \(\phi_i\colon\SL_2\to G\) satisfying
\begin{align*}
    \phi_i\!\begin{pmatrix}1&t\\0&1\end{pmatrix}&=x_i(t), &
    \phi_i\!\begin{pmatrix}1&0\\t&1\end{pmatrix}&=y_i(t), &
    \phi_i\!\begin{pmatrix}t&0\\0&t^{-1}\end{pmatrix}&=\alpha_i^\vee(t),
\end{align*}
where \(x_i(t)\) and \(y_i(t)\) are the exponentiated Chevalley
generators. We use the representatives
\begin{align*}
    \dot s_i
    &:=
    \phi_i\!\begin{pmatrix}0&-1\\1&0\end{pmatrix}, &
    \dot s_i^{-1}
    &:=
    \phi_i\!\begin{pmatrix}0&1\\-1&0\end{pmatrix}, &
    B_i(t)
    &:=
    \phi_i\!\begin{pmatrix}t&-1\\1&0\end{pmatrix}
    =x_i(t)\dot s_i.
\end{align*}
 Let
\(\beta=(i_1,\ldots,i_r)\) be a positive braid word. The braid variety associated with \(\beta\) is
\begin{equation}\label{eq:braid-variety-flags}
    X(\beta)
    :=
    \left\{
        (F_1,\ldots,F_r)\in(G/B^+)^r
        \ \middle|\
        B^+\xrightarrow{s_{i_1}}F_1
        \xrightarrow{s_{i_2}}\cdots
        \xrightarrow{s_{i_r}}F_r
        =\dot{\delta(\beta)}B^+
    \right\}.
\end{equation}

\subsection{Definition of Demazure weaves}
A \emph{Demazure weave} is a downward-oriented planar graph whose edges
are labeled by elements of \(I\). Its internal vertices encode precisely
the following local transformations of positive braid words:
\begin{enumerate}
    \item a commutation, or \(2\)-move,
    \(ik\longleftrightarrow ki\), when \(c_{ik}=c_{ki}=0\);

    \item a braid, or \(3\)-move,
    \(iji\longleftrightarrow jij\), when
    \(c_{ij}=c_{ji}=-1\);

    \item a Demazure move \(ii\longrightarrow i\).
\end{enumerate}
The corresponding vertices are four-valent, six-valent, and trivalent,
respectively.

Every generic horizontal slice of a Demazure weave determines a positive
braid word. Reading the weave from top to bottom gives a sequence of the
three local transformations above, each of which preserves the Demazure
product. Hence, if the upper boundary is labeled by a braid word \(\beta\)
and the lower boundary is reduced, the latter is a reduced expression of
\(\delta(\beta)\). We write
\begin{equation*}
    \mathfrak W\colon\beta\longrightarrow\delta(\beta),
\end{equation*}
where the notation \(\delta(\beta)\) on the right also refers to the
chosen reduced word. The three local models are displayed in
Figure~\ref{fig:local-vertices-demazure-weave}.

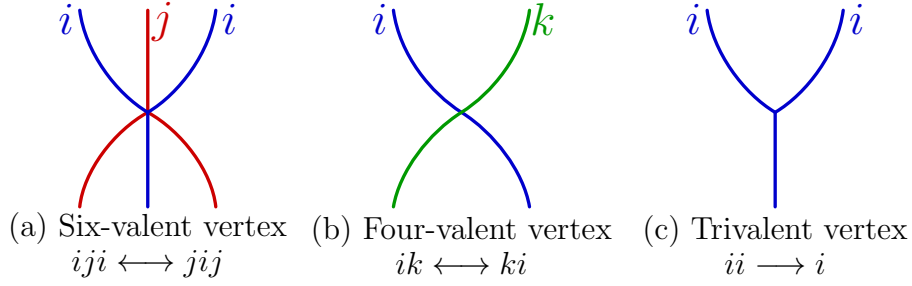
\begin{figure}[htbp]
\centering
\begin{tikzpicture}[
    line width=1.25pt,
    line cap=round,
    line join=round,
    every node/.style={font=\Large},
    myblue/.style={color=blue!80!black},
    myred/.style={color=red!80!black},
    mygreen/.style={color=green!60!black}
]

\begin{scope}[xshift=0cm]
    \coordinate (c) at (0,0);

    \draw[myred] (0,1.35) -- (c);
    \draw[myred]
        (c) .. controls (-0.45,-0.25) and (-0.85,-0.78) ..
        (-0.90,-1.25);
    \draw[myred]
        (c) .. controls (0.45,-0.25) and (0.85,-0.78) ..
        (0.90,-1.25);

    \draw[myblue]
        (-0.90,1.35)
        .. controls (-0.83,0.80) and (-0.42,0.25) ..
        (c);
    \draw[myblue]
        (0.90,1.35)
        .. controls (0.83,0.80) and (0.42,0.25) ..
        (c);
    \draw[myblue] (c) -- (0,-1.25);

    \node[myblue] at (-1.08,1.23) {$i$};
    \node[myred]  at (0.20,1.20) {$j$};
    \node[myblue] at (1.08,1.23) {$i$};

    \node[black,font=\normalsize,align=center] at (0,-1.75)
    {(a) Six-valent vertex\\[-1pt]
     \(iji\longleftrightarrow jij\)};
\end{scope}

\begin{scope}[xshift=4.15cm]
    \coordinate (c) at (0,0);

    \draw[myblue]
        (-0.90,1.35)
        .. controls (-0.82,0.76) and (-0.40,0.25) ..
        (c)
        .. controls (0.43,-0.27) and (0.82,-0.78) ..
        (0.90,-1.25);

    \draw[mygreen]
        (0.90,1.35)
        .. controls (0.82,0.76) and (0.40,0.25) ..
        (c)
        .. controls (-0.43,-0.27) and (-0.82,-0.78) ..
        (-0.90,-1.25);

    \node[myblue]  at (-1.08,1.23) {$i$};
    \node[mygreen] at (1.08,1.23) {$k$};

    \node[black,font=\normalsize,align=center] at (0,-1.75)
    {(b) Four-valent vertex\\[-1pt]
     \(ik\longleftrightarrow ki\)};
\end{scope}

\begin{scope}[xshift=8.30cm]
    \coordinate (c) at (0,0);

    \draw[myblue]
        (-0.90,1.35)
        .. controls (-0.83,0.80) and (-0.42,0.25) ..
        (c);
    \draw[myblue]
        (0.90,1.35)
        .. controls (0.83,0.80) and (0.42,0.25) ..
        (c);
    \draw[myblue] (c) -- (0,-1.25);

    \node[myblue] at (-1.08,1.23) {$i$};
    \node[myblue] at (1.08,1.23) {$i$};

    \node[black,font=\normalsize,align=center] at (0,-1.75)
    {(c) Trivalent vertex\\[-1pt]
     \(ii\longrightarrow i\)};
\end{scope}
\end{tikzpicture}

\caption{The local vertices of a Demazure weave. The six-valent
vertex represents the braid relation
\(s_is_js_i=s_js_is_j\); the four-valent vertex represents the
commutation relation \(s_is_k=s_ks_i\); and the trivalent vertex
represents the Demazure relation \(s_i\ast s_i=s_i\).}
\label{fig:local-vertices-demazure-weave}
\end{figure}

We next recall the equivalence relation on Demazure weaves. It is generated
by the following local identifications.

\begin{enumerate}
    \item[(i)]
    Suppose that
    \(\mathfrak W,\mathfrak W'\colon\gamma\to\gamma'\) consist only
    of \(2\)- and \(3\)-moves and that \(\gamma\) and \(\gamma'\)
    represent the same positive braid. Then
    \(\mathfrak W\sim\mathfrak W'\).

    \item[(ii)]
    If \(c_{ij}=c_{ji}=-1\), then the following two weaves are
    equivalent:
    \begin{equation*}
    \begin{aligned}
        \mathfrak W\colon\quad
        ijij&\longrightarrow jijj\longrightarrow jij,\\
        \mathfrak W'\colon\quad
        ijij&\longrightarrow iiji\longrightarrow iji
              \longrightarrow jij.
    \end{aligned}
    \end{equation*}

    \item[(iii)]
    If \(c_{ik}=c_{ki}=0\), then the following two weaves are
    equivalent:
    \begin{equation*}
    \begin{aligned}
        \mathfrak W\colon\quad
        iki&\longrightarrow iik\longrightarrow ik\longrightarrow ki,\\
        \mathfrak W'\colon\quad
        iki&\longrightarrow kii\longrightarrow ki.
    \end{aligned}
    \end{equation*}
\end{enumerate}

Thus, relation~(ii) moves a trivalent vertex through a six-valent vertex,
whereas relation~(iii) moves it through a four-valent vertex.

Besides these equivalences, there is a second local operation. The word
\(iii\) admits two Demazure weaves to \(i\), according to the order in
which the two Demazure moves are performed. These weaves are not
equivalent. The operation interchanging them is the \emph{weave mutation}
\(\mu\), shown in Figure~\ref{fig:weave-mutation-branches}.

\begin{figure}[htbp]
\centering
\begin{tikzpicture}[
    line cap=round,
    line join=round,
    bluestrand/.style={
        draw=blue!80!black,
        line width=1.25pt
    }
]

\coordinate (leftRoot) at (0,1);
\coordinate (leftJunction) at (-0.62,1.78);

\draw[bluestrand] (0,0) -- (leftRoot);
\draw[bluestrand]
    (leftRoot)
    .. controls (-0.25,1.28) and (-0.48,1.60) ..
    (leftJunction)
    .. controls (-0.92,2.25) and (-1.12,2.68) ..
    (-1.15,3);
\draw[bluestrand]
    (leftJunction)
    .. controls (-0.05,2.08) and (0.30,2.48) ..
    (0.30,3);
\draw[bluestrand]
    (leftRoot)
    .. controls (0.62,1.60) and (1.13,2.30) ..
    (1.18,3);

\node[font=\normalsize] at (0,-0.30) {$\mathfrak W$};

\draw[
    <->,
    black,
    line width=0.8pt,
    shorten <=2pt,
    shorten >=2pt
]
    (1.75,1.75)
    -- node[above=3pt,font=\Large] {$\mu$}
    (3.25,1.75);

\coordinate (rightRoot) at (5,1);
\coordinate (rightJunction) at (5.62,1.78);

\draw[bluestrand] (5,0) -- (rightRoot);
\draw[bluestrand]
    (rightRoot)
    .. controls (4.38,1.60) and (3.87,2.30) ..
    (3.82,3);
\draw[bluestrand]
    (rightJunction)
    .. controls (5.05,2.08) and (4.70,2.48) ..
    (4.70,3);
\draw[bluestrand]
    (rightRoot)
    .. controls (5.25,1.28) and (5.48,1.60) ..
    (rightJunction)
    .. controls (5.92,2.25) and (6.12,2.68) ..
    (6.15,3);

\node[font=\normalsize] at (5,-0.30) {$\mathfrak W'$};
\end{tikzpicture}

\caption{The weave mutation interchanging the two Demazure weaves
from \(iii\) to \(i\).}
\label{fig:weave-mutation-branches}
\end{figure}
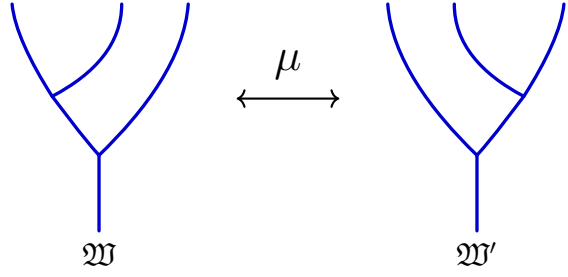

The following connectivity result will be used repeatedly.

\begin{lemma}[{\cite[Lemma~4.4]{casals2025cluster}}]
    \label{lem:connectivity-Demazure-weaves}
    Let
    \(\mathfrak W,\mathfrak W'\colon
      \beta\longrightarrow\delta(\beta)\)
    be Demazure weaves with the same fixed reduced word at the lower
    boundary. Then \(\mathfrak W\) and \(\mathfrak W'\) are connected
    by a finite sequence of weave equivalences and weave mutations.
\end{lemma}

\subsubsection{Double inductive weaves}

Following \cite[Section~6.4]{casals2025cluster}, we use double strings to
record the successive addition of letters on either side of a braid word.

\begin{definition}
    A \emph{double string} is a sequence
    \begin{equation}\label{eq:double-string}
        \ddot{\mathbf{s}}=(i_1X_1,\ldots,i_rX_r),
        \qquad i_k\in I,
        \qquad X_k\in\{\mathrm L,\mathrm R\}.
    \end{equation}
    Define braid words \(\beta^{\ddot{\mathbf{s}}}_0,\ldots,\beta^{\ddot{\mathbf{s}}}_r\) recursively by
    \(\beta^{\ddot{\mathbf{s}}}_0=\varnothing\) and
    \begin{equation}\label{eq:double-string-recursion}
        \beta^{\ddot{\mathbf{s}}}_k=
        \begin{cases}
            i_k\beta^{\ddot{\mathbf{s}}}_{k-1}, & X_k=\mathrm L,\\
            \beta^{\ddot{\mathbf{s}}}_{k-1}i_k, & X_k=\mathrm R.
        \end{cases}
    \end{equation}
    We call \(\ddot{\mathbf{s}}\) a double string for \(\beta\) if
    \(\beta^{\ddot{\mathbf{s}}}_r=\beta\).
\end{definition}

A braid word may admit several double strings; the choice between
\(\mathrm L\) and \(\mathrm R\) in the first entry has no effect. Set
\(u_k^{\ddot{\mathbf{s}}}:=\delta(\beta^{\ddot{\mathbf{s}}}_k)\) for \(0\leq k\leq r\). Since
\(\beta^{\ddot{\mathbf{s}}}_k\) is obtained by adjoining a single letter to
\(\beta_{k-1}^{\ddot{\mathbf{s}}}\), either
\(\ell(u_k^{\ddot{\mathbf{s}}})=\ell(u_{k-1}^{\ddot{\mathbf{s}}})+1\) or
\(u^{\ddot{\mathbf{s}}}_k=u^{\ddot{\mathbf{s}}}_{k-1}\). In the
latter case, we call \(k\) a \emph{solid insertion}, or a
\emph{solid step}.

\begin{definition}\label{def:double-inductive-weave}
    Let \(\ddot{\mathbf{s}}=(i_1X_1,\ldots,i_rX_r)\) be a double string for
    \(\beta\). The associated \emph{double inductive weave}
    \begin{equation*}
        \mathfrak W(\ddot{\mathbf{s}})\colon
        \beta\longrightarrow\delta(\beta)
    \end{equation*}
    is constructed recursively. Begin with the empty weave
    \(\mathfrak W_0\). Suppose that
    \(\mathfrak W_{k-1}\colon\beta^{\ddot{\mathbf{s}}}_{k-1}\to u^{\ddot{\mathbf{s}}}_{k-1}\) has already
    been constructed.

    If \(\ell(u^{\ddot{\mathbf{s}}}_k)=\ell(u^{\ddot{\mathbf{s}}}_{k-1})+1\), obtain \(\mathfrak W_k\) by
    adjoining an \(i_k\)-strand to the left of \(\mathfrak W_{k-1}\)
    when \(X_k=\mathrm L\), and to its right when
    \(X_k=\mathrm R\). The lower boundary is then the corresponding
    reduced expression of \(u^{\ddot{\mathbf{s}}}_k\).

    If \(u^{\ddot{\mathbf{s}}}_k=u^{\ddot{\mathbf{s}}}_{k-1}\), then \(s_{i_k}\) is a left descent of
    \(u^{\ddot{\mathbf{s}}}_{k-1}\) when \(X_k=\mathrm L\), and a right descent when
    \(X_k=\mathrm R\). In the first case, use braid moves to choose a
    reduced expression of \(u^{\ddot{\mathbf{s}}}_{k-1}\) beginning with \(i_k\); in the
    second case, choose one ending with \(i_k\). The newly adjoined
    \(i_k\)-strand is then merged with this boundary strand by a
    Demazure move, creating one trivalent vertex. Finally, set
    \(\mathfrak W(\ddot{\mathbf{s}}):=\mathfrak W_r\).
\end{definition}

We decorate the \(k\)-th entry by a superscript \(+\) precisely when
it increases the length of the Demazure product:
\begin{equation}\label{eq:decorated-double-string}
    i_kX_k^+
    \quad\Longleftrightarrow\quad
    \ell(u^{\ddot{\mathbf{s}}}_k)=\ell(u^{\ddot{\mathbf{s}}}_{k-1})+1.
\end{equation}
The entries for which \(u^{\ddot{\mathbf{s}}}_k=u^{\ddot{\mathbf{s}}}_{k-1}\) remain undecorated. By
Definition~\ref{def:double-inductive-weave}, these undecorated entries
are in canonical bijection with the trivalent vertices of
\(\mathfrak W(\ddot{\mathbf{s}})\).
More explicitly, set
\[
    J_{\mathrm{sol}}(\ddot{\mathbf s})
    :=\{k\in[r]\mid
          u_k^{\ddot{\mathbf s}}=u_{k-1}^{\ddot{\mathbf s}}\}.
\]
We denote the preceding bijection by
\begin{equation}\label{eq:trivalent-solid-position-bijection}
    q_{\ddot{\mathbf s}}\colon
    \mathfrak W(\ddot{\mathbf s})_3
    \xrightarrow{\sim}
    J_{\mathrm{sol}}(\ddot{\mathbf s}).
\end{equation}
Accordingly, vertices and solid positions are always related through
\(q_{\ddot{\mathbf s}}\); they are not treated as the same set.

For a braid word \(\beta=(i_1,\ldots,i_r)\), the double strings
\begin{equation*}
    (i_1\mathrm R,\ldots,i_r\mathrm R)
    \qquad\text{and}\qquad
    (i_r\mathrm L,\ldots,i_1\mathrm L)
\end{equation*}
produce the right and left inductive weaves, respectively.

\begin{example}\label{ex:double-inductive-weave-A2}
    In type \(A_2\), consider
    \begin{equation*}
        \ddot{\mathbf{s}}
        =(1\mathrm R,2\mathrm L,2\mathrm R,
          1\mathrm L,2\mathrm R).
    \end{equation*}
    The associated braid words are
    \(1,21,212,1212,12122\), and their Demazure products satisfy
    \begin{equation*}
        \delta(1)<\delta(21)<\delta(212)
        =\delta(1212)=\delta(12122)=w_0.
    \end{equation*}
    Hence the decorated double string is
    \begin{equation*}
        (1\mathrm R^+,2\mathrm L^+,2\mathrm R^+,
         1\mathrm L,2\mathrm R).
    \end{equation*}
    Its double inductive weave has two trivalent vertices, corresponding
    to the undecorated entries \(1\mathrm L\) and \(2\mathrm R\).
\end{example}

\subsection{Lusztig cycles and intersection numbers}

Let \(\mathfrak W\) be a Demazure weave. A \emph{cycle} is a map
\(C\colon E(\mathfrak W)\to\ZZ_{\geq0}\).

\begin{definition}\label{def:Lusztig-cycle}
    A cycle \(C\) is a \emph{Lusztig cycle} if its weights propagate
    through the vertices as follows.
    \begin{enumerate}
        \item At a trivalent vertex,
        \begin{equation}\label{eq:Lusztig-rule-trivalent}
            C(e_{v,\mathrm s})
            =\min\{C(e_{v,\mathrm{nw}}),C(e_{v,\mathrm{ne}})\}.
        \end{equation}
        \item At a four-valent vertex, the weights pass through along
        their strands:
        \begin{equation}\label{eq:Lusztig-rule-four-valent}
            C(e_{v,\mathrm{sw}})=C(e_{v,\mathrm{ne}}),
            \qquad
            C(e_{v,\mathrm{se}})=C(e_{v,\mathrm{nw}}).
        \end{equation}
        \item At a six-valent vertex, with the incoming and outgoing
        edges ordered from left to right, set
        \(m_v:=\min\{C(e_{v,\mathrm{nw}}),C(e_{v,\mathrm{ne}})\}\),
        and require
        \begin{align}
            C(e_{v,\mathrm{sw}})
            &=C(e_{v,\mathrm n})+C(e_{v,\mathrm{ne}})-m_v,
            \notag\\
            C(e_{v,\mathrm s})&=m_v,
            \label{eq:Lusztig-rule-six-valent}\\
            C(e_{v,\mathrm{se}})
            &=C(e_{v,\mathrm{nw}})+C(e_{v,\mathrm n})-m_v.
            \notag
        \end{align}
    \end{enumerate}
\end{definition}

For a trivalent vertex \(v\in\mathfrak W_3\), let \(\gamma_v^{\mathfrak W}\) be
zero above \(v\), assign weight one to \(e_{v,\mathrm s}\), and
propagate the weights below \(v\) by Definition~\ref{def:Lusztig-cycle}:
\begin{equation}\label{eq:Lusztig-cycle-at-v}
    \gamma_v^{\mathfrak W}(e)=0\quad(e\text{ above }v),
    \qquad
    \gamma_v^{\mathfrak W}(e_{v,\mathrm s})=1.
\end{equation}
The propagation rule is not imposed at the initial vertex \(v\). If $\mathfrak W$ is clear from the content, we simply write $\gamma_v$ for the Lusztig cycle $\gamma_v^{\mathfrak W}$.
We call \(v\) \emph{frozen} if \(\gamma_v\) is nonzero on a bottom
edge, and write
\(
    \mathfrak W_{3,\mathrm{ex}}
    :=\mathfrak W_3\setminus\mathfrak W_{3,\mathrm{fr}}.
\)

\subsubsection{Local intersection numbers}

For cycles \(C,C'\), we use the local intersection pairing of
\cite[Definitions~4.17, 4.18, and~4.23]{casals2025cluster}.

\begin{definition}\label{def:local-intersection-trivalent}
    At a trivalent vertex \(v\), let \((a_1,a',a_2)\) and
    \((b_1,b',b_2)\) be the weights of \(C\) and \(C'\) on the
    north-west, south, and north-east edges. Set
    \begin{equation}\label{eq:local-intersection-trivalent}
        \sharp_v(C\mathbin{\cdot}C')
        :=\det
        \begin{pmatrix}
            1&1&1\\
            a_1&a'&a_2\\
            b_1&b'&b_2
        \end{pmatrix}.
    \end{equation}
\end{definition}

\begin{definition}\label{def:local-intersection-six-valent}
    At a six-valent vertex \(v\), let
    \((a_1,a_2,a_3)\), \((a'_1,a'_2,a'_3)\) be the incoming and
    outgoing weights of \(C\), and define
    \((b_1,b_2,b_3)\), \((b'_1,b'_2,b'_3)\) similarly for \(C'\).
    Set
    \begin{equation}\label{eq:local-intersection-six-valent}
        \sharp_v(C\mathbin{\cdot}C')
        :=\frac12\left[
        \det\begin{pmatrix}
            1&1&1\\
            a_1&a_2&a_3\\
            b_1&b_2&b_3
        \end{pmatrix}
        -
        \det\begin{pmatrix}
            1&1&1\\
            a'_1&a'_2&a'_3\\
            b'_1&b'_2&b'_3
        \end{pmatrix}
        \right].
    \end{equation}
\end{definition}

Four-valent vertices contribute zero, and the total local intersection is
\begin{equation}\label{eq:total-local-intersection}
    \sharp_{\mathfrak W}(C\mathbin{\cdot}C')
    :=\sum_{v\in\mathfrak W_3}
      \sharp_v(C\mathbin{\cdot}C')
      +\sum_{v\in\mathfrak W_6}
      \sharp_v(C\mathbin{\cdot}C'),
\end{equation}
where \(\mathfrak W_6\) denotes the six-valent vertices. This pairing
is skew-symmetric, and it is integral on Lusztig cycles
\cite[Remark~4.22]{casals2025cluster}.

\subsubsection{Boundary intersection numbers}

Fix a horizontal slice with word
\(\eta=(i_1,\ldots,i_m)\). Label its edges \(e_1,\ldots,e_m\)
from left to right. Set
\begin{equation}\label{eq:slice-roots-coroots}
    \rho_k:=s_{i_1}\cdots s_{i_{k-1}}(\alpha_{i_k}),
    \qquad
    \rho_k^\vee
    :=s_{i_1}\cdots s_{i_{k-1}}(\alpha_{i_k}^\vee).
\end{equation}
and \(c_k=C(e_k)\), \(c'_k=C'(e_k)\).

\begin{definition}[{\cite[Definition~4.24]{casals2025cluster}}]
    \label{def:boundary-intersection}
    The \emph{boundary intersection number} of \(C\) and \(C'\) at
    the slice \(\eta\) is
    \begin{equation}\label{eq:boundary-intersection}
        \sharp_\eta(C\mathbin{\cdot}C')
        :=\frac12\sum_{k,l=1}^m
        \operatorname{sgn}(l-k)c_kc'_l
        \bigl\langle\rho_k,\rho_l^\vee\bigr\rangle,
    \end{equation}
\end{definition}

In simply-laced type this pairing is skew-symmetric.

\subsubsection{The exchange matrix}

For a Demazure weave
\(\mathfrak W\colon\beta\to\delta(\beta)\), define
\begin{equation}\label{eq:weave-exchange-entry}
    \varepsilon_{v,v'}
    :=\sharp_{\mathfrak W}
       (\gamma_v\mathbin{\cdot}\gamma_{v'})
      +\sharp_{\delta(\beta)}
       (\gamma_v\mathbin{\cdot}\gamma_{v'}).
\end{equation}
Here the second term is evaluated at the lower boundary
\(\delta(\beta)\). Both pairings are skew-symmetric, so
\(\varepsilon_{v,v'}=-\varepsilon_{v',v}\).

\begin{definition}[{\cite[Definition~4.26]{casals2025cluster}}]
    \label{def:exchange-matrix-weave}
    The matrix
    \begin{equation}\label{eq:exchange-matrix-weave}
        B(\mathfrak W)
        :=(\varepsilon_{v',v})_{v\in\mathfrak W_3, v'\in \mathfrak W_{3,\rm ex}}
    \end{equation}
    is the extended exchange matrix associated with
    \(\mathfrak W\).
\end{definition}

This transpose implements the convention in
\eqref{eq:quiver-convention}: a positive intersection entry
\(\varepsilon_{u,v}\) represents an arrow from \(u\) to \(v\)
in \cite{casals2025cluster}, whereas our matrix has
\(b_{v,u}>0\) for that arrow. Thus
\(b_{u,v}=\varepsilon_{v,u}=-\varepsilon_{u,v}\) whenever
\(v\) is mutable. Transposing the full skew-symmetric intersection
matrix before restricting the columns preserves the quiver and
commutes with mutation.

The matrix \(B(\mathfrak W)\) is integral: half-integral local
contributions occur only between frozen vertices and therefore do not
enter the extended exchange matrix
\cite[Remark~4.27]{casals2025cluster}.

\subsection{Cluster variables associated with Demazure weaves}

Let
\(\ddot{\mathbf{s}}=(i_1X_1,\ldots,i_rX_r)\) be a double string for
\(\beta\), and let
\(\mathfrak W=\mathfrak W(\ddot{\mathbf{s}})
\colon\beta\to\delta(\beta)\) be the associated double inductive
weave. Write the resulting braid word as
\(\beta=(h_1,\ldots,h_r)\).

We use the edge labeling of
\cite[Definition~4.2]{casals2025comparing}.

\begin{definition}\label{def:edge-labeling}
For an edge \(e\) of color \(i\), its \emph{edge label} is an element
of the form
\[
    g_e
    =
    B_i(f_e)\alpha_i^\vee(u_e)
    \in G,
\]
where \(f_e,u_e\in\mathbb{C}(X(\beta))\). These rational functions are
determined as follows.

\begin{enumerate}
    \item[(i)]
    Let \(e_k^{\mathrm{top}}\) be the top edge corresponding to the
    \(k\)-th letter \(h_k\) of \(\beta\). Then
    \[
        g_{e_k^{\mathrm{top}}}
        =
        B_{h_k}(z_k)\alpha_{h_k}^\vee(1).
    \]
    Equivalently,
    \(f_{e_k^{\mathrm{top}}}=z_k\) and
    \(u_{e_k^{\mathrm{top}}}=1\).

    \item[(ii)]
    The labels on the remaining edges are obtained by propagating the
    top labels downward through the local vertices of \(\mathfrak W\)
    according to the rules for framed weaves in
    \cite[Definition~5.8]{casals2025cluster}.
\end{enumerate}
\end{definition}

In particular, the edge labeling is uniquely determined by the
coordinate functions \(z_1,\ldots,z_r\) on \(X(\beta)\).
 For \(v\in\mathfrak W_3\), let
\(e_{v,\mathrm{s}}\) be the southern, or outgoing, edge incident with
\(v\).

\begin{definition}[{\cite[Definition~4.4]{casals2025comparing}}]
\label{def:weave-cluster-variable}
The \emph{weave cluster variable} associated with
\(v\in\mathfrak W_3\) is
\[
    X_v
    :=
    u_{e_{v,\mathrm{s}}}.
\]
\end{definition}

For each \(v\in\mathfrak W_3\), let \(\gamma_v\) be the corresponding
Lusztig cycle. By
\cite[Theorem~5.12(i)]{casals2025cluster}, or equivalently
\cite[Theorem~4.19]{casals2025comparing}, one has
\begin{equation}\label{eq:u-factorization}
    u_e
    =
    \prod_{v\in\mathfrak W_3}
    X_v^{\gamma_v(e)}\end{equation}
for every edge \(e\) of \(\mathfrak W\). Moreover, by \cite[equation~(25)]{casals2025cluster} we have 
\begin{equation}
X_v=f_{e_{v,\rm ne}}\prod_{v'\neq v}X_{v'}^{\gamma_{v'}(e_{v,\rm ne})+\gamma_{v'}(e_{v,\rm nw})-\gamma_{v'}(e_{v,s})}.
\end{equation}

\subsubsection{Grid minors}

We regard each fundamental weight
\(\varpi_j\) as an algebraic character of the maximal torus:
\[
    \varpi_j
    \in
    X^*(H)
    =
    \operatorname{Hom}_{\mathrm{alg.grp}}
    (H,\mathbb{C}^{\times}).
\]
Thus \(\varpi_j(h)\) is defined for every \(h\in H\).

Draw \(\mathfrak W\) inside an upward-pointing triangle
\(\mathcal T\) with horizontal base, so that, for each \(c\), the
boundary vertex labeled \(i_c\) lies strictly below every trivalent
vertex created before the \(c\)-th step. For each
\(c\in\{1,\ldots,r\}\), let \(\ell_c\) be the horizontal line passing
through the boundary vertex labeled \(i_c\), ordered from top to
bottom. These lines divide \(\mathcal T\) into \(r+1\) horizontal
strips: the strip above \(\ell_1\), the strips between
\(\ell_c\) and \(\ell_{c+1}\) for \(1\leq c<r\), and the strip below
\(\ell_r\). For \(c\in\{0,\ldots,r\}\), let \(L_c\) and \(R_c\)
denote, respectively, the leftmost and rightmost regions of
\(\mathcal T\setminus\mathfrak W\) contained in the \(c\)-th strip. In particular, \(L_0=R_0\) is the unique region containing the
apex of \(\mathcal T\), while \(L_r\) and \(R_r\) contain the
bottom-left and bottom-right vertices, respectively. Moreover,
\begin{equation}\label{eq:left-right-regions}
\begin{aligned}
    L_c=L_{c+1}
    &\quad\Longleftrightarrow\quad
    X_{c+1}=\mathrm R,\\
    R_c=R_{c+1}
    &\quad\Longleftrightarrow\quad
    X_{c+1}=\mathrm L.
\end{aligned}
\end{equation}

Represent a point of \(X(\beta)\) by a sequence of flags
\(\mathbf B=(B_0,B_1,\ldots,B_r)\), where \(B_0=B^+\) and
\(B_r=\dot{\delta(\beta)}B^+\). Following the convention of
\cite[Section~3.3.3]{casals2025comparing}, place these flags in the
top regions of \(\mathcal T\setminus\mathfrak W\), starting with
\(B_0\) in \(L_r\) and proceeding in the order prescribed by the
double string. These boundary labels extend uniquely to a flag
labeling of \(\mathfrak W\). For \(0\leq c\leq r\), let
\(B_{(c)}^L\) and \(B_{(c)}^R\) denote the flags labeling \(L_c\) and
\(R_c\), respectively.
The following example illustrates the construction.

\begin{figure}[htbp]
\centering
\begin{tikzpicture}[
    x=0.85cm,
    y=0.85cm,
    line cap=round,
    line join=round,
    every node/.style={font=\scriptsize}
]
    \def\ha{1.18}
    \def\hb{2.72}
    \def\hc{3.72}
    \def\hd{4.92}
    \def\he{5.42}
    \def\hf{5.98}
    \def\hg{6.67}

    \coordinate (b1) at ({-7*(1-\ha/7.2)},\ha);
    \coordinate (b2) at ({ 7*(1-\hb/7.2)},\hb);
    \coordinate (b3) at ({ 7*(1-\hc/7.2)},\hc);
    \coordinate (b4) at ({ 7*(1-\hd/7.2)},\hd);
    \coordinate (b5) at ({-7*(1-\he/7.2)},\he);
    \coordinate (b6) at ({ 7*(1-\hf/7.2)},\hf);
    \coordinate (b7) at ({ 7*(1-\hg/7.2)},\hg);

    \coordinate (v1) at (-1.55,1.18);
    \coordinate (v2) at ( 3.10,1.98);
    \coordinate (u)  at ( 1.92,2.88);
    \coordinate (v4) at ( 1.92,4.55);
    \coordinate (v5) at (-1.15,5.05);

    \draw[thin]
        (-7,0) -- (0,7.2) -- (7,0) -- cycle;

    \foreach \y in {\ha,\hb,\hc,\hd,\he,\hf,\hg}
    {
        \pgfmathsetmacro{\xlim}{7*(1-\y/7.2)}
        \draw[thin,dashed]
            (-\xlim,\y) -- (\xlim,\y);
    }

    \node[font=\tiny] at (-0.24,6.93) {\(L_0\)};
    \node at (-0.55,6.32) {\(L_1\)};
    \node at (-1.05,5.70) {\(L_2\)};
    \node at (-1.60,5.17) {\(L_3\)};
    \node at (-2.15,4.32) {\(L_4\)};
    \node at (-3.10,3.22) {\(L_5\)};
    \node at (-4.10,1.95) {\(L_6\)};
    \node at (-5.40,0.58) {\(L_7\)};

    \node[font=\tiny] at (0.24,6.93) {\(R_0\)};
    \node at (0.65,6.32) {\(R_1\)};
    \node at (1.40,5.70) {\(R_2\)};
    \node at (1.60,5.17) {\(R_3\)};
    \node at (2.35,4.32) {\(R_4\)};
    \node at (3.45,3.22) {\(R_5\)};
    \node at (4.10,1.95) {\(R_6\)};
    \node at (5.40,0.58) {\(R_7\)};

    \begin{scope}
    \clip (-7,0) -- (0,7.2) -- (7,0) -- cycle;

        \draw[red,thin]
            (b1)
            .. controls (-4.75,0.92) and (-3.20,1.05) ..
            (v1);
        \draw[red,thin]
            (-1.55,0) -- (v1);
        \draw[red,thin]
            (v1)
            .. controls (-0.35,1.65) and (0.55,2.55) ..
            (u);

        \draw[red,thin]
            (u)
            .. controls (2.45,2.80) and (2.65,2.25) ..
            (v2);
        \draw[red,thin]
            (3.10,0) -- (v2);
        \draw[red,thin]
            (v2)
            .. controls (3.70,2.30) and (4.15,2.35) ..
            (b2);

        \draw[red,thin]
            (u) -- (v4);
        \draw[red,thin]
            (v4)
            .. controls (1.48,4.95) and (1.15,5.55) ..
            (b6);
        \draw[red,thin]
            (v4)
            .. controls (2.20,4.78) and (2.25,4.70) ..
            (b4);

        \draw[blue,thin]
            (1.92,0) -- (u);
        \draw[blue,thin]
            (u)
            .. controls (2.55,2.90) and (3.35,3.10) ..
            (b3);

        \draw[blue,thin]
            (u)
            .. controls (0.35,3.35) and (-1.20,3.55) ..
            (v5);
        \draw[blue,thin]
            (v5)
            .. controls (-1.45,5.32) and (-1.65,5.15) ..
            (b5);
        \draw[blue,thin]
            (v5)
            .. controls (-0.45,5.65) and (0.30,6.05) ..
            (b7);

    \end{scope}
\end{tikzpicture}

\caption{The regions \(L_c\) and \(R_c\), \(c\in[1,7]\), associated
with the double string
\(\ddot{\mathbf s}=(1R,2R,1L,2R,1R,2R,2L)\).
Each dashed line passes through a boundary intersection point of the
weave, hence each depth $c$ corresponds to word $i_c$. }
\end{figure}
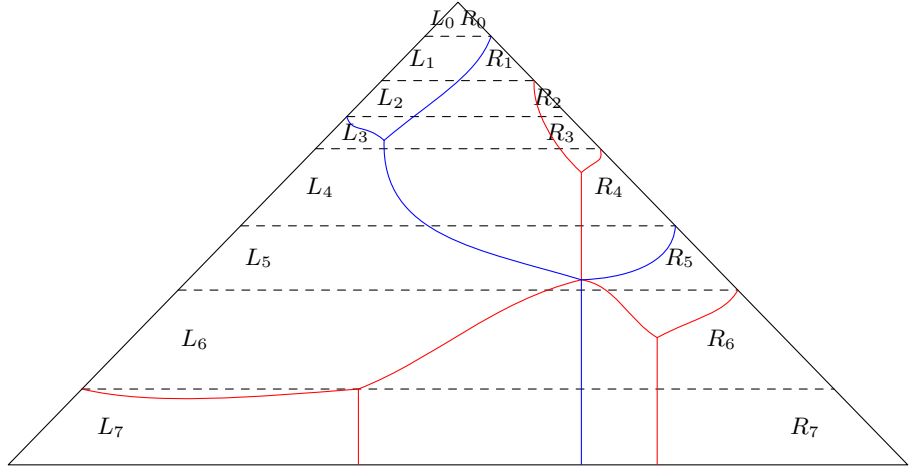

For each \(c\in\{0,\ldots,r\}\), let \(P_c^\uparrow\) denote the
uppermost path from \(L_c\) to \(R_c\), traveling from left to right
through the regions of \(\mathcal T\setminus\mathfrak W\). Let \(P_c^\downarrow\) be a horizontal path from \(L_c\) to \(R_c\),
lying in the interior of the \(c\)-th strip and sufficiently close to
its lower boundary so that it passes below every internal vertex of
the strip, whether trivalent, four-valent, or six-valent. We require
\(P_c^\downarrow\) to avoid all vertices and to meet the crossed edges
transversely. This convention also fixes the isotopy class of
\(P_c^\downarrow\) when the strip contains no trivalent vertex. For
\(c=r\), the lower boundary is the base of \(\mathcal T\). If
\(L_c=R_c\), the path \(P_c^\downarrow\) is understood to be empty.

If a path \(P\) crosses the edges \(e_1,\ldots,e_m\), listed in the
order in which they are encountered, define its path product by
\[
    Z(P):=g_{e_1}\cdots g_{e_m}.
\]
For an empty path, we set \(Z(P)=1\).

By \cite[Lemma~4.24]{casals2025comparing},
\(Z(P_c^\uparrow)\) is the representative of the relative-position
coset determined by \(B_{(c)}^L\) and \(B_{(c)}^R\), while
\(Z(P_c^\downarrow)\) differs from it only by right multiplication by
an element of \(U^+\).

Suppose that \(P_c^\downarrow\) crosses the edges
\(e_1,\ldots,e_m\), whose colors give a reduced expression
\[
    \mathbf j=(j_1,\ldots,j_m),
    \qquad
    w_c^{\mathfrak W}
    =
    s_{j_1}\cdots s_{j_m}.
\]

For \(1\leq a\leq m\), define the corresponding inversion coroot by
\begin{equation}\label{eq:inversion-coroots}
    \chi_a^{\mathbf j}
    :=
    s_{j_m}\cdots s_{j_{a+1}}
    \bigl(\alpha_{j_a}^\vee\bigr),
\end{equation}
where an empty product is understood to be the identity. Since
\(\mathbf j\) is reduced, every
\(\chi_a^{\mathbf j}\) is a positive coroot.

Proposition~4.25 of \cite{casals2025comparing} expresses the Cartan
component determined by the slice at depth \(c\) as
\begin{equation}\label{eq:hc-plus}
    h_c^+
    =
    \prod_{a=1}^m
    \chi_a^{\mathbf j}(u_{e_a})
    \in H.
\end{equation}
For \(c=0\), this product is understood to be the identity element
of \(H\).

\begin{definition}\label{def:grid-minors}
For \(c\in\{1,\ldots,r\}\) and \(j\in I\), the \emph{positive grid
minor} associated with \((c,j)\) and \(\mathfrak W\) is
\begin{equation}\label{eq:def-grid-minor}
    \Delta_{c,j}^{\mathfrak W}
    :=
    \varpi_j(h_c^+).
\end{equation}
\end{definition}

Under the natural identification between the weave and Deodhar
charts, Definition~\ref{def:grid-minors} agrees with the grid minors
in \cite[Section~2.6]{galashin2026braidII}; see also
\cite[Definition~4.20 and Proposition~4.25]
{casals2025comparing}.

Fix \(c\), and retain the notation
\(e_1,\ldots,e_m\) and \(\mathbf j=(j_1,\ldots,j_m)\) introduced
above. For \(v\in\mathfrak W_3\), set
\begin{equation}\label{eq:def-gamma-vc}
    \gamma_{v,c}
    :=
    \sum_{a=1}^m
    \gamma_v(e_a)\chi_a^{\mathbf j}.
\end{equation}
Thus \(\gamma_{v,c}\) is a nonnegative integral combination of
positive coroots.

\begin{lemma}\label{lem:grid-minors}
For every \(c\in\{1,\ldots,r\}\) and \(j\in I\), one has
\begin{equation}\label{eq:grid-minor-cluster-monomial}
    \Delta_{c,j}^{\mathfrak W}
    =
    \prod_{v\in\mathfrak W_3}
    X_v^{\langle\varpi_j,\gamma_{v,c}\rangle}.
\end{equation}
In particular, \(\Delta_{c,j}^{\mathfrak W}\) is a monomial in the
weave cluster variables.
\end{lemma}

\begin{proof}
Applying \(\varpi_j\) to \eqref{eq:hc-plus} and using
\eqref{eq:u-factorization}, we obtain
\begin{align*}
    \Delta_{c,j}^{\mathfrak W}
    =
    \prod_{a=1}^m
    u_{e_a}^{\langle\varpi_j,\chi_a^{\mathbf j}\rangle} 
    =
    \prod_{a=1}^m\prod_{v\in\mathfrak W_3}
    X_v^{\gamma_v(e_a)
        \langle\varpi_j,\chi_a^{\mathbf j}\rangle} 
    =
    \prod_{v\in\mathfrak W_3}
    X_v^{
        \left\langle
            \varpi_j,
            \sum_{a=1}^m
            \gamma_v(e_a)\chi_a^{\mathbf j}
        \right\rangle} =
    \prod_{v\in\mathfrak W_3}
    X_v^{\langle\varpi_j,\gamma_{v,c}\rangle}.
\end{align*}
All exponents are nonnegative integers because the
\(\chi_a^{\mathbf j}\) are positive coroots and
\(\gamma_v(e_a)\in\ZZ_{\geq0}\).
\end{proof}

\begin{lemma}\label{lem:grid-minor-equation}
Let
\(
    \ddot{\mathbf s}
    =
    (i_1X_1,\ldots,i_rX_r)
\)
be a double string. For \(k\in[r]\), write
\(
    \beta_k^{\ddot{\mathbf s}}
    =
    (j_1,\ldots,j_k)
\)
for the braid word obtained after the first \(k\) insertions, and set
\[
    \ddot{\mathbf s}_k
    :=
    (j_1R,\ldots,j_kR,
      i_{k+1}X_{k+1},\ldots,i_rX_r).
\]
Thus the first \(k\) steps of \(\ddot{\mathbf s}_k\) form the right
inductive string of \(\beta_k^{\ddot{\mathbf s}}\), while its
remaining steps agree with those of \(\ddot{\mathbf s}\). Then, for
every \(j\in I\),
\[
    \Delta_{k,j}^{\mathfrak W(\ddot{\mathbf s})}
    =
    \Delta_{k,j}^{\mathfrak W(\ddot{\mathbf s}_k)}.
\]
\end{lemma}

\begin{proof}
Set
\(
    w_k:=\delta(\beta_k^{\ddot{\mathbf s}}).
\)
Denote the upper paths at depth \(k\) in
\(\mathfrak W(\ddot{\mathbf s})\) and
\(\mathfrak W(\ddot{\mathbf s}_k)\) by
\(P_k^\uparrow\) and \(\widehat P_k^\uparrow\), respectively.

The two double strings have the same initial variables $z_i$, which we
identify term by term. Recall that each edge element \(g_e\) is a
fixed function of these initial variables by Definition ~\ref{def:edge-labeling}(1). By the construction of
the upper paths at depth \(k\), the edges crossed by
\(P_k^\uparrow\) and \(\widehat P_k^\uparrow\) correspond in the
same order, and the elements attached to corresponding edges are
equal. Hence their path products coincide:
\begin{equation}\label{eq:grid-minor-common-path-product}
    Z(P_k^\uparrow)
    =
    Z(\widehat P_k^\uparrow).
\end{equation}

By \cite[Definition~4.20 and Lemma~4.24]
{casals2025comparing}, there are unique elements
\(h_k^+,\widehat h_k^+\in H\) such that
\[
    Z(P_k^\uparrow)
    \in U^+\dot w_kh_k^+U^+,
    \qquad
    Z(\widehat P_k^\uparrow)
    \in U^+\dot w_k\widehat h_k^+U^+.
\]
Equation \eqref{eq:grid-minor-common-path-product} and the uniqueness
of the \(H\)-factor imply
\(    h_k^+=\widehat h_k^+.
\)
Therefore,
\[
    \Delta_{k,j}^{\mathfrak W(\ddot{\mathbf s})}
    =
    \varpi_j(h_k^+)
    =
    \varpi_j(\widehat h_k^+)
    =
    \Delta_{k,j}^{\mathfrak W(\ddot{\mathbf s}_k)},
\]
as required.
\end{proof}

\section{Vectors for double inductive weaves}
\label{sec:vector-right-weave}
In this section, we introduce the \(\ZZ_{\geq0}^r\)-vectors associated
with the trivalent vertices of double inductive weaves. In
Section~\ref{sec:quantization-braid-varieties}, we identify them with
the Lusztig parameters of the corresponding cluster variables.

\subsection{Vectors associated with a right inductive weave}
Let $\beta=(i_1,\cdots, i_r)$ be a braid word, and let $\overrightarrow{\mathfrak{W}}(\beta)$ be the right inductive weave for the double string $(i_1R,\cdots i_r R)$.

If \(v\in\overrightarrow{\mathfrak W}(\beta)_3\) is the
trivalent vertex associated with the entry \(i_k\mathrm R\), define
\(q_v^\beta:=k\). These positions give a total order on the trivalent
vertices:
\begin{equation}\label{eq:order-right-inductive-vertices}
    v'<v
    \quad\Longleftrightarrow\quad
    q_{v'}^\beta<q_v^\beta.
\end{equation}

For each
\(v\in\overrightarrow{\mathfrak W}(\beta)_3\), let
\(E_{q_v^\beta}\in\ZZ^r\) be the standard basis vector at position \(q_v^\beta\):
\begin{equation}\label{eq:standard-vector-Ev}
    (E_{q_v^\beta})_k=
    \begin{cases}
        1,&k=q_v^\beta,\\
        0,&k\neq q_v^\beta.
    \end{cases}
\end{equation}

\begin{definition}\label{def:eta-vector-right-inductive}
    For \(v\in\overrightarrow{\mathfrak W}(\beta)_3\), define
    \(\eta_v^\beta\in\ZZ_{\geq0}^r\) recursively by
    \begin{equation}\label{eq:etav}
        \eta_v^\beta
        :=E_{q_v^\beta}+
        \sum_{\substack{
            u\in\overrightarrow{\mathfrak W}(\beta)_3\\
            u<v
        }}
        \gamma_u(e_{v,\mathrm{nw}})\eta_u^\beta,
    \end{equation}
    where \(e_{v,\mathrm{nw}}\) is the north-west incoming edge of
    \(v\). The sum is empty when \(v\) is the minimal trivalent
    vertex.
\end{definition}

The recursion is well defined because the order in
\eqref{eq:order-right-inductive-vertices} is total. Since every
coefficient \(\gamma_u(e_{v,\mathrm{nw}})\) is a nonnegative integer,
equation~\eqref{eq:etav} also shows inductively that
\(\eta_v^\beta\in\ZZ_{\geq0}^r\). We suppress the superscript
\(\beta\) when no confusion is likely. For \(\mathbf a=(a_1,\ldots,a_r)\in\ZZ^r\), write
\(\operatorname{supp}(\mathbf a):=\{k\in[r]\mid a_k\neq0\}\).

\begin{lemma}\label{lem:maxindex}
    For every \(v\in\overrightarrow{\mathfrak W}(\beta)_3\),
    \begin{equation}\label{eq:maximal-support-eta}
        \max\operatorname{supp}(\eta_v^\beta)=q_v^\beta,
        \qquad
        (\eta_v^\beta)_{q_v^\beta}=1.
    \end{equation}
\end{lemma}

\begin{proof}
    We argue by induction on the total order
    \eqref{eq:order-right-inductive-vertices}. If \(v\) is minimal,
    then \(\eta_v^\beta=E_{q_v^\beta}\), and both assertions follow from
    \eqref{eq:standard-vector-Ev}.

    Now let \(v\) be nonminimal and assume the result for every
    \(u<v\). The induction hypothesis gives
    \(\operatorname{supp}(\eta_u^\beta)
      \subseteq\{1,\ldots,q_u^\beta\}\), while
    \eqref{eq:order-right-inductive-vertices} gives
    \(q_u^\beta<q_v^\beta\). Hence every summand in
    \eqref{eq:etav}, other than \(E_{q_v^\beta}\), is supported in
    \(\{1,\ldots,q_v^\beta-1\}\). Consequently, these summands make
    no contribution at position \(q_v^\beta\) or at any larger
    position.

    On the other hand, \(E_{q_v^\beta}\) has coefficient \(1\) at
    \(q_v^\beta\) and vanishes above it. Therefore
    \((\eta_v^\beta)_{q_v^\beta}=1\) and
    \((\eta_v^\beta)_k=0\) for every \(k>q_v^\beta\), proving
    \eqref{eq:maximal-support-eta}.
\end{proof}

We next define the mutation rule for tuples of integer vectors.

\begin{definition}\label{def:order-vector}
For
distinct vectors \(\mathbf a=(a_1,\ldots,a_r)\) and
\(\mathbf b=(b_1,\ldots,b_r)\) in \(\ZZ^r\), write
\(\mathbf a\prec_{\mathrm r}\mathbf b\) if there exists
\(q\in[1,r]\) such that
\begin{equation}\label{eq:right-lexicographic-order}
    a_k=b_k\quad(k>q),
    \qquad
    a_q<b_q.
\end{equation}
Equivalently, if \(q\) is the largest index at which \(\mathbf a\)
and \(\mathbf b\) differ, then
\(\mathbf a\prec_{\mathrm r}\mathbf b\) precisely when
\(a_q<b_q\). Thus \(\prec_{\mathrm r}\) is a strict total order on
\(\ZZ^r\). We denote its reflexive closure by
\(\preceq_{\mathrm r}\) and the larger of two vectors by
\(\max_{\mathrm r}\).
\end{definition}
We introduce the mutation of $\ZZ^r$-vectors associated with Demazure weaves.
\begin{definition}\label{def:mutlusztig}
    Let \(\mathfrak W\) be a Demazure weave, let
    \(B(\mathfrak W)=(b_{u,w})\) be its
    exchange matrix, and let
    \(\boldsymbol\xi=(\xi_u)_{u\in\mathfrak W_3}\) be a tuple in
    \((\ZZ^r)^{\mathfrak W_3}\). Suppose that \(v\in\mathfrak W_3\)
    is mutable and that \(\mathfrak W'\) is obtained by mutation at
    \(v\). Denote by \(v^+\) the new trivalent vertex replacing
    \(v\), and identify
    \(\mathfrak W_3\setminus\{v\}\) with
    \(\mathfrak W'_3\setminus\{v^+\}\).

    The extended exchange matrix has columns indexed only by mutable
    vertices. We therefore use the column indexed by \(v\), including
    all mutable and frozen rows, and set
    \begin{equation}\label{eq:two-vector-mutation-sums}
        A_v^+(\boldsymbol\xi)
        :=\sum_{w\in\mathfrak W_3}[b_{w,v}]_+\xi_w,
        \qquad
        A_v^-(\boldsymbol\xi)
        :=\sum_{w\in\mathfrak W_3}[-b_{w,v}]_+\xi_w,
    \end{equation}
    where \([a]_+:=\max\{a,0\}\). We define
    \(\mu_v(\boldsymbol\xi)
      =(\xi'_u)_{u\in\mathfrak W'_3}\) by
    \begin{equation}\label{eq:mutation-integer-vectors}
        \xi'_u=
        \begin{cases}
            \xi_u, & u\neq v^+,\\[1mm]
            \max_{\mathrm{r}}
            \bigl\{A_v^+(\boldsymbol\xi),
                   A_v^-(\boldsymbol\xi)\bigr\}-\xi_v,
                & u=v^+.
        \end{cases}
    \end{equation}
\end{definition}

\begin{example}\label{ex:right-inductive-weaves}
 In type \(A_2\), consider the right string
\begin{equation*}
\mathbf{s}=(1R^{+},1R,2R^{+},2R,1R^{+},2R,1R).
\end{equation*}
The four undecorated letters correspond to the trivalent vertices
\(v_1,v_2,v_3,v_4\), occurring at positions \(2,4,6,7\), respectively.
The associated vectors are
\begin{align*}
\eta_{v_1}&=E_2, &
\eta_{v_2}&=E_4,\\
\eta_{v_3}&=E_6+\eta_{v_1}+\eta_{v_2}
           =E_6+E_2+E_4, &
\eta_{v_4}&=E_7+\eta_{v_2}
           =E_7+E_4.
\end{align*}
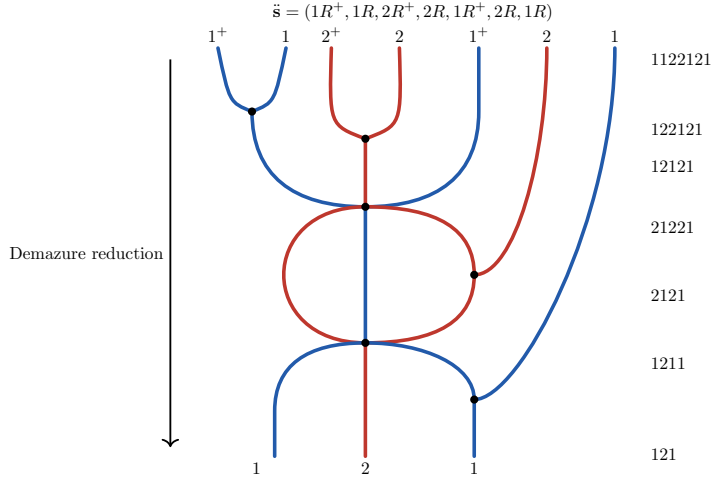
\begin{tikzpicture}[
    scale=0.60,
    transform shape,
    x=1cm,
    y=1cm,
    weave/.style={
        line width=1.35pt,
        line cap=round,
        line join=round
    },
    one/.style={
        weave,
        draw=onecolor
    },
    two/.style={
        weave,
        draw=twocolor
    },
    vertex/.style={
        circle,
        fill=black,
        inner sep=1.8pt
    },
    boundary/.style={
        font=\small
    },
    slice/.style={
        font=\small,
        anchor=west
    }
]

\coordinate (t1) at (0.75,9);
\coordinate (t2) at (2.25,9);
\coordinate (t3) at (3.25,9);
\coordinate (t4) at (4.75,9);
\coordinate (t5) at (6.50,9);
\coordinate (t6) at (8.00,9);
\coordinate (t7) at (9.50,9);

\coordinate (v1) at (1.50,7.60);
\coordinate (v2) at (4.00,7.00);
\coordinate (v3) at (6.40,4.00);
\coordinate (v4) at (6.40,1.25);

\coordinate (h1) at (4.00,5.50);
\coordinate (h2) at (4.00,2.50);

\coordinate (p5) at (6.50,7.60);

\coordinate (u4) at (2.00,1.00);

\draw[one]
    (t1) .. controls (1.00,7.80) .. (v1);

\draw[one]
    (t2) .. controls (2.00,7.80) .. (v1);

\draw[one]
    (v1)
    .. controls (1.50,6.20) and (2.40,5.50) ..
    (h1);

\draw[two]
    (t3) .. controls (3.20,7.30) ..(v2);

\draw[two]
    (t4) .. controls (4.80,7.30) .. (v2);

\draw[two]
    (v2) -- (h1);

\draw[one]
    (t5) -- (p5)
    .. controls (6.50,6.20) and (5.60,5.50) ..
    (h1);


\draw[two]
    (h1)
    .. controls (1.60,5.50) and (1.60,2.50) ..
    (h2);

\draw[one]
    (h1) -- (h2);

\draw[two]
    (h1)
    .. controls (5.60,5.50) and (6.40,5.00) ..
    (v3);

\draw[two]
    (t6)
    .. controls (8.00,6.60) and (7.30,4.00) ..
    (v3);

\draw[two]
    (v3)
    .. controls (6.40,3.00) and (5.60,2.50) ..
    (h2);


\draw[one]
    (h2)
    .. controls (2.70,2.50) and (2.00,2.00) ..
    (u4)
    -- (2.00,0);

\draw[two]
    (h2) -- (4.00,0);

\draw[one]
    (h2)
    .. controls (5.30,2.50) and (6.40,2.00) ..
    (v4)
    -- (6.40,0);

\draw[one]
    (t7)
    .. controls (9.50,5.20) and (7.40,1.25) ..
    (v4);

\foreach \p in {v1,v2,h1,v3,h2,v4}
    \node[vertex] at (\p) {};

\foreach \x/\lab in {
    0.75/{1^{+}},
    2.25/1,
    3.25/{2^{+}},
    4.75/2,
    6.50/{1^{+}},
    8.00/2,
    9.50/1
}{
    \node[boundary,above] at (\x,9) {$\lab$};
}

\foreach \x/\lab in {
    1.60/1,
    4.00/2,
    6.40/1
}{
    \node[boundary,below] at (\x,0) {$\lab$};
}

\draw[
    ->,
    line width=.7pt
]
    (-0.30,8.75) -- (-0.30,0.20)
    node[
        midway,
        left,
        font=\small
    ]
    {Demazure reduction};

\node[slice] at (10.15,8.75) {$1122121$};
\node[slice] at (10.15,7.20) {$122121$};
\node[slice] at (10.15,6.40) {$12121$};
\node[slice] at (10.15,5.05) {$21221$};
\node[slice] at (10.15,3.55) {$2121$};
\node[slice] at (10.15,2.05) {$1211$};
\node[slice] at (10.15,0.05) {$121$};

\node[font=\small] at (5.05,9.80)
    {$\ddot{\mathbf{s}}
      =(1R^{+},1R,2R^{+},2R,1R^{+},2R,1R)$};

\end{tikzpicture}
\end{example}

\subsection{Mutation sequences for the
\texorpdfstring{\(\eta\)}{eta}-vectors}

Let \(\beta=(i_1,\ldots,i_r)\) be a braid word. Its right inductive
weave is encoded by the double string
\begin{equation}\label{eq:right-inductive-double-string}
    \overrightarrow{\mathfrak W}(\beta)
    =(i_1\mathrm R,i_2\mathrm R,\ldots,i_r\mathrm R).
\end{equation}
The choice of side in the first entry is immaterial, so we may replace
\(i_1\mathrm R\) by \(i_1\mathrm L\). Moving this entry successively
to the right produces
\begin{equation}\label{eq:shifted-right-inductive-weave}
    \mathfrak W_1
    =(i_2\mathrm R,\ldots,i_r\mathrm R,i_1\mathrm L).
\end{equation}

Set \(i:=i_1\). At an intermediate stage, suppose that
\(i\mathrm L\) is immediately followed by \(j\mathrm R\). Let
\(\beta_k\) be the preceding part of the double string and put
\(u_k:=\delta(\beta_k)\). The change in the double strings gives
exactly five possibilities:
\begin{enumerate}[label=\textup{(\arabic*)}]
    \item \((\beta_k,i\mathrm L^+,j\mathrm R^+,\ldots)
    \longleftrightarrow
    (\beta_k,j\mathrm R^+,i\mathrm L^+,\ldots)\);

    \item \((\beta_k,i\mathrm L^+,j\mathrm R,\ldots)
    \longleftrightarrow
    (\beta_k,j\mathrm R,i\mathrm L^+,\ldots)\);

    \item \((\beta_k,i\mathrm L,j\mathrm R,\ldots)
    \longleftrightarrow
    (\beta_k,j\mathrm R,i\mathrm L,\ldots)\), where
    \(\ell(s_i u_k s_j)=\ell(u_k)-2\);

    \item \((\beta_k,i\mathrm L^+,j\mathrm R,\ldots)
    \longleftrightarrow
    (\beta_k,j\mathrm R^+,i\mathrm L,\ldots)\);

    \item \((\beta_k,i\mathrm L,j\mathrm R,\ldots)
    \longleftrightarrow
    (\beta_k,j\mathrm R,i\mathrm L,\ldots)\), where
    \(\ell(s_i u_k s_j)=\ell(u_k)\).
\end{enumerate}

By \cite[Section~6.4]{casals2025cluster}, the transformations in
Cases~(1)--(3) induce canonical identifications of the trivalent
vertices and preserve the corresponding exchange matrix. In Case~(4),
the vertex associated with \(j\mathrm R\) is replaced by the vertex
associated with \(i\mathrm L\). In Case~(5), the two double inductive
weaves differ by mutation at the vertex associated with
\(i\mathrm L\), up to weave equivalences.

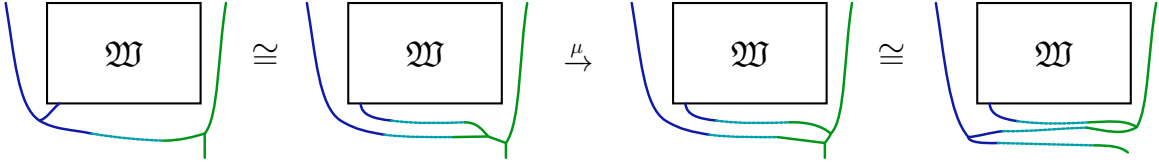
\begin{figure}[htbp]
\centering

\resizebox{0.96\linewidth}{!}{%
\begin{tikzpicture}

\begin{scope}
    \weavebox

    \draw[bluestrand]
        (0.05,2.35)
        .. controls (0.20,1.45) and (0.25,0.60) ..
        (0.62,0.36)
        .. controls (0.88,0.20) and (1.18,0.21) ..
        (1.50,0.14);

    \draw[bluestrand]
        (0.62,0.36)
        .. controls (0.75,0.39) and (0.84,0.56) ..
        (0.96,0.65);

    \draw[dottedstrand]
        (1.50,0.14)
        .. controls (1.92,0.07) and (2.35,0.03) ..
        (2.72,0.02);

    \draw[greenstrand]
        (2.72,0.02)
        .. controls (3.02,0.02) and (3.25,0.09) ..
        (3.42,0.14)
        .. controls (3.66,0.42) and (3.64,1.50) ..
        (3.78,2.35);

    \draw[greenstrand] (3.42,0.14) -- (3.42,-0.27);
\end{scope}

\node[font=\Large] at (4.45,1.45) {$\cong$};

\begin{scope}[xshift=5.1cm]
    \weavebox

    \coordinate (secondUpper) at (3.12,0.09);
    \coordinate (secondLower) at (3.42,-0.02);

    \draw[bluestrand]
        (0.05,2.35)
        .. controls (0.20,1.45) and (0.25,0.60) ..
        (0.62,0.36)
        .. controls (0.88,0.20) and (1.08,0.17) ..
        (1.36,0.12);

    \draw[bluestrand]
        (0.96,0.65)
        .. controls (0.97,0.45) and (1.18,0.39) ..
        (1.48,0.36);

    \draw[dottedstrand]
        (1.48,0.36)
        .. controls (1.90,0.30) and (2.38,0.32) ..
        (2.72,0.33);

    \draw[dottedstrand]
        (1.36,0.12)
        .. controls (1.76,0.07) and (2.18,0.08) ..
        (2.52,0.08);

    \draw[greenstrand]
        (2.72,0.33)
        .. controls (2.94,0.33) and (3.03,0.19) ..
        (secondUpper);

    \draw[greenstrand]
        (2.52,0.08) -- (secondUpper);

    \draw[greenstrand]
        (secondUpper) -- (secondLower);

    \draw[greenstrand]
        (3.78,2.35)
        .. controls (3.64,1.52) and (3.68,0.34) ..
        (secondLower);

    \draw[greenstrand]
        (secondLower) -- (3.42,-0.27);
\end{scope}

\node[font=\large] at (9.75,1.45) {$\xrightarrow{\mu}$};

\begin{scope}[xshift=10.6cm]
    \weavebox

    \draw[bluestrand]
        (0.05,2.35)
        .. controls (0.20,1.45) and (0.25,0.60) ..
        (0.62,0.36)
        .. controls (0.88,0.20) and (1.08,0.17) ..
        (1.36,0.12);

    \draw[bluestrand]
        (0.96,0.65)
        .. controls (0.97,0.45) and (1.18,0.39) ..
        (1.48,0.36);

    \draw[dottedstrand]
        (1.48,0.36)
        .. controls (1.90,0.30) and (2.35,0.32) ..
        (2.72,0.33);

    \draw[dottedstrand]
        (1.36,0.12)
        .. controls (1.76,0.07) and (2.18,0.08) ..
        (2.52,0.08);

    \coordinate (thirdFirst)  at (3.42,0.14);
    \coordinate (thirdSecond) at (3.32,-0.02);

    \draw[greenstrand]
        (2.72,0.33)
        .. controls (3.03,0.34) and (3.27,0.27) ..
        (thirdFirst);

    \draw[greenstrand]
        (3.78,2.35)
        .. controls (3.64,1.52) and (3.68,0.38) ..
        (thirdFirst)
        .. controls (3.40,0.08) and (3.35,0.02) ..
        (thirdSecond);

    \draw[greenstrand]
        (2.52,0.08) -- (thirdSecond);

    \draw[greenstrand]
        (thirdSecond) -- (3.32,-0.27);
\end{scope}

\node[font=\Large] at (15.05,1.45) {$\cong$};

\begin{scope}[xshift=15.75cm]
    \weavebox

    \coordinate (fourthLeft)  at (0.60,0.08);
    \coordinate (fourthRight) at (3.45,0.25);

    \draw[bluestrand]
        (0.05,2.35)
        .. controls (0.20,1.40) and (0.27,0.46) ..
        (fourthLeft);

    \draw[bluestrand]
        (0.96,0.65)
        .. controls (0.97,0.45) and (1.17,0.39) ..
        (1.43,0.35);

    \draw[bluestrand]
        (fourthLeft)
        .. controls (0.76,0.09) and (0.98,0.16) ..
        (1.18,0.18);

    \draw[bluestrand]
        (fourthLeft)
        .. controls (0.60,-0.01) and (0.85,-0.03) ..
        (1.10,-0.02);

    \draw[dottedstrand]
        (1.43,0.35)
        .. controls (1.86,0.29) and (2.30,0.31) ..
        (2.64,0.33);

    \draw[greenstrand]
        (2.64,0.33)
        .. controls (2.96,0.34) and (3.25,0.38) ..
        (fourthRight);

    \draw[greenstrand]
        (fourthRight)
        .. controls (3.65,0.56) and (3.64,1.55) ..
        (3.78,2.35);

    \draw[dottedstrand]
        (1.18,0.18)
        .. controls (1.65,0.17) and (2.18,0.23) ..
        (2.60,0.24);

    \draw[greenstrand]
        (2.60,0.24)
        .. controls (2.92,0.19) and (3.22,0.10) ..
        (fourthRight);

    \draw[dottedstrand]
        (1.10,-0.02)
        .. controls (1.58,-0.03) and (2.20,-0.05) ..
        (2.72,-0.06);

    \draw[greenstrand]
        (2.72,-0.06)
        .. controls (3.00,-0.06) and (3.18,-0.08) ..
        (3.30,-0.18);
\end{scope}

\end{tikzpicture}
}
\caption{The Case~\textup{(5)} transformation. The first and last
panels correspond to the double strings
\((\beta_k,i\mathrm L,j\mathrm R,\ldots)\) and
\((\beta_k,j\mathrm R,i\mathrm L,\ldots)\), respectively. The
middle two panels differ by a weave mutation.}
\label{fig:weave-mutation}
\end{figure}

As illustrated in Figure~\ref{fig:weave-mutation}, the transformation
in Case~(5) factors as
\begin{equation}\label{eq:case5}
    \mathfrak W(\beta_k,i\mathrm L,j\mathrm R,\ldots)
    \ \sim\
    \mathfrak W_-
    \overset{\mu}{\longleftrightarrow}
    \mathfrak W_+
    \ \sim\
    \mathfrak W(\beta_k,j\mathrm R,i\mathrm L,\ldots),
\end{equation}
where \(\mathfrak W_-\) and \(\mathfrak W_+\) are the second and
third local weaves in the figure, and \(\sim\) denotes weave
equivalence. In particular, Case~(5) changes the exchange matrix by a
single matrix mutation.

Let \(K_\beta^{\mathrm{mut}}\) be the set of positions at which
Case~\textup{(5)} occurs as
\(i_1\mathrm L\) is moved to the right. Write
\begin{equation}\label{eq:K-beta-I-beta}
    K_\beta^{\mathrm{mut}}
    =\{k_1<\cdots<k_n\}\subseteq\{2,\ldots,r\},
    \qquad
    \mathbf i_\beta^{\mathrm{mut}}
    :=(i_{k_1},\ldots,i_{k_n}).
\end{equation}

For \(1\leq t\leq n\), define
\begin{equation}\label{eq:intermediate-right-inductive-weave}
    \mathfrak W^{(t)}
    =(i_2\mathrm R,\ldots,i_{k_t-1}\mathrm R,
      i_1\mathrm L,i_{k_t}\mathrm R,\ldots,i_r\mathrm R),
\end{equation}
where the initial subsequence is understood to be empty when \(k_t=2\),
so that \(i_1\mathrm L\) lies immediately to the left of
\(i_{k_t}\mathrm R\). Set
\begin{equation}\label{eq:last-intermediate-right-inductive-weave}
    \mathfrak W^{(n+1)}
    :=(i_2\mathrm R,\ldots,i_r\mathrm R,i_1\mathrm L)
    =\mathfrak W_1.
\end{equation}

The weave \(\mathfrak W^{(1)}\) is equivalent to
\(\overrightarrow{\mathfrak W}(\beta)\), because only
transformations of Types~(1)--(4) occur before the first index
\(k_1\). For \(1\leq t\leq n\), the passage from
\(\mathfrak W^{(t)}\) to \(\mathfrak W^{(t+1)}\) consists of the
mutation occurring at \(k_t\), followed by transformations of
Types~(1)--(4) up to the next index in
\(K_\beta^{\mathrm{mut}}\).

Transporting the vertex labels through these equivalences and
mutations, we identify the trivalent vertices of each
\(\mathfrak W^{(t)}\) with those of
\(\overrightarrow{\mathfrak W}(\beta)\). We denote the resulting
exchange matrix by
\begin{equation}\label{eq:intermediate-exchange-matrix}
    B^{(t)}:=B(\mathfrak W^{(t)})
    =\bigl(\varepsilon_{v',v}^{(t)}\bigr),
    \qquad 1\leq t\leq n+1.
\end{equation}

We now describe how the tuple of vectors associated with the right
inductive weave changes along the mutation sequence leading to
\((i_2\mathrm R,\ldots,i_r\mathrm R,i_1\mathrm L)\). We begin with the local identity governing the coefficients of the
Lusztig cycles under a weave mutation.

\begin{lemma}\label{lem:gammamutation}
Let \(\mathfrak W_-\) and \(\mathfrak W_+\) be the two local
Demazure weaves in equation \eqref{eq:case5}, related
by mutation at the upper trivalent vertex \(v_1\) of
\(\mathfrak W_-\) as in Figure~\ref{fig:weave-mutation-branches}. Let \(v_2\) be the lower trivalent vertex of
\(\mathfrak W_-\), and let \(v_1^+\) be the trivalent vertex of
\(\mathfrak W_+\) replacing \(v_1\). Let
\(\varepsilon^{\mathfrak W_-}_{u,w}\) denote the full
skew-symmetric intersection matrix of the trivalent vertices. When
\(w\) is mutable, our exchange-matrix convention gives
\(b_{u,w}=\varepsilon^{\mathfrak W_-}_{w,u}
=-\varepsilon^{\mathfrak W_-}_{u,w}\).
Then, for every trivalent vertex \(v'<v_1\),
\begin{equation}\label{eq:gamma-local-mutation}
    \gamma_{v'}^{\mathfrak W_+}
        (e_{v_1^+,\mathrm{nw}})
    =
    \gamma_{v'}^{\mathfrak W_-}
        (e_{v_2,\mathrm{nw}})
    +[\varepsilon^{\mathfrak W_-}_{v_1,v'}]_+.
\end{equation}
\end{lemma}

\begin{proof}
Let \(a,b,c\), from left to right, be the weights with which
\(\gamma_{v'}^{\mathfrak W_-}\) enters the local configuration in Figure~\ref{fig:weave-mutation-branches}.
Because \(v'<v_1\), the rightmost incoming branch extends to the
upper boundary without meeting the support of
\(\gamma_{v'}^{\mathfrak W_-}\); hence \(c=0\). The tropical
Lusztig rule at \(v_1\) and the corresponding rule on the mutated
weave give
\begin{equation}\label{eq:local-cycle-weights-before-after}
    \gamma_{v'}^{\mathfrak W_-}
        (e_{v_2,\mathrm{nw}})=\min\{a,b\},
    \qquad
    \gamma_{v'}^{\mathfrak W_+}
        (e_{v_1^+,\mathrm{nw}})=b.
\end{equation}

It remains to identify the exchange-matrix contribution. The cycle
\(\gamma_{v_1}^{\mathfrak W_-}\) is the short cycle joining
\(v_1\) to \(v_2\). Its local intersection with
\(\gamma_{v'}^{\mathfrak W_-}\) at \(v_1\) is
\begin{equation}\label{eq:local-intersection-v1}
    \sharp_{v_1}
    \bigl(
        \gamma_{v_1}^{\mathfrak W_-}
        \mathbin{\cdot}
        \gamma_{v'}^{\mathfrak W_-}
    \bigr)
    =
    \begin{vmatrix}
        1&1&1\\
        0&1&0\\
        a&\min\{a,b\}&b
    \end{vmatrix}
    =b-a.
\end{equation}
The local contribution at \(v_2\) is zero because \(c=0\).
Moreover, the short cycle has no support outside this local
configuration and does not reach the southern boundary. Hence all
other local and boundary contributions vanish, and therefore
\(\varepsilon^{\mathfrak W_-}_{v_1,v'}=b-a\). Combining this
equality with
\eqref{eq:local-cycle-weights-before-after}, we obtain
\begin{align*}
    \gamma_{v'}^{\mathfrak W_-}
        (e_{v_2,\mathrm{nw}})
    +[\varepsilon^{\mathfrak W_-}_{v_1,v'}]_+
    &=\min\{a,b\}+[b-a]_+\\
    &=b
     =\gamma_{v'}^{\mathfrak W_+}
        (e_{v_1^+,\mathrm{nw}}),
\end{align*}
which proves \eqref{eq:gamma-local-mutation}.
\end{proof}

We next record an invariance property for the boundary weights of a
Lusztig cycle associated with a frozen vertex. Fix an expression of
\(\delta(\beta)\), and let
\(e_1^{\mathrm{out}},\ldots,e_m^{\mathrm{out}}\) be the southern
boundary edges of a Demazure weave \(\mathfrak W\), ordered from
left to right. For a cycle \(C\) on \(\mathfrak W\), define
\begin{equation}\label{eq:output-vector-cycle}
    \operatorname{out}_{\mathfrak W}(C)
    :=\bigl(C(e_1^{\mathrm{out}}),\ldots,
            C(e_m^{\mathrm{out}})\bigr).
\end{equation}

\begin{lemma}\label{lem:output}
Let \(\mathfrak U,\mathfrak U'\colon
\beta\longrightarrow\delta(\beta)\) be two Demazure weaves with the
same fixed expression of \(\delta(\beta)\) on their southern
boundaries. Consider a sequence
\begin{equation}\label{eq:sequence-local-weave-moves}
    \mathfrak U=\mathfrak U^{[0]}
    \longrightarrow\mathfrak U^{[1]}
    \longrightarrow\cdots
    \longrightarrow\mathfrak U^{[N]}=\mathfrak U'
\end{equation}
of weave equivalences and weave mutations. Let \(v^{[0]}\) be a
frozen vertex of \(\mathfrak U^{[0]}\), and transport it through
\eqref{eq:sequence-local-weave-moves}: under an equivalence, take the
corresponding trivalent vertex, while under a mutation, take the
unchanged vertex corresponding to it. Then every \(v^{[k]}\) is
frozen and
\begin{equation}\label{eq:output-vector-invariance}
    \operatorname{out}_{\mathfrak U^{[k]}}
    \bigl(\gamma_{v^{[k]}}^{\mathfrak U^{[k]}}\bigr)
    =
    \operatorname{out}_{\mathfrak U}
    \bigl(\gamma_{v^{[0]}}^{\mathfrak U}\bigr)
    \qquad(0\leq k\leq N).
\end{equation}
In particular, the cycles at the two ends of the sequence have the
same output vector.
\end{lemma}

\begin{proof}
We argue by induction on the number of local moves. It suffices to
consider a single move
\(\mathfrak U^{[k]}\to\mathfrak U^{[k+1]}\) and prove that it
preserves the output vector of the cycle being tracked.

If the move occurs strictly above \(v^{[k]}\), then the associated
cycle vanishes on every edge above its origin, whereas its outgoing
edge at \(v^{[k]}\) has weight \(1\). Thus the cycles on the two
sides agree below \(v^{[k]}\), and their output vectors coincide.
We may therefore assume that the move occurs below \(v^{[k]}\), or
that \(v^{[k]}\) belongs to the local configuration.

Suppose first that the move is an equivalence involving only four-
and six-valent vertices. The incoming weights of the restricted cycle
at the top of the local configuration are unchanged. The invariance
of the tropical Lusztig rules under braid moves
\cite[Lemma~4.10]{casals2025cluster} then implies equality of the
local outgoing weights. For the equivalence involving a trivalent
vertex, the same conclusion follows from
\cite[Example~4.12]{casals2025cluster} when the tracked vertex is not
the trivalent vertex involved in the move. If it is that vertex, its
local outgoing vector is \((0,0,1)\) on both sides.

Finally, suppose that the move is a weave mutation. A frozen vertex
cannot be the mutation vertex
\cite[Lemma~4.37(1)]{casals2025cluster}, so the transported vertex is
well defined. If its cycle enters the mutation configuration with
weights \(a,b,c\), then the outgoing weights on the two sides are
\(\min\{\min\{a,b\},c\}\) and
\(\min\{a,\min\{b,c\}\}\), respectively. Both equal
\(\min\{a,b,c\}\). If the tracked vertex is the lower trivalent
vertex in the mutation configuration, its associated cycle instead
originates inside the configuration; in this case, its unique local
outgoing edge has weight \(1\) on both sides. Thus every possible
local move preserves the output vector.

Since \(v^{[k]}\) is frozen, its output vector is nonzero. The same is
therefore true after the move, so \(v^{[k+1]}\) is also frozen.
Induction on \(k\) proves both assertions.
\end{proof}

Recall that
\(K_\beta^{\mathrm{mut}}=\{k_1<\cdots<k_n\}\) is the set of positions at which
Case~\textup{(5)} occurs while \(i_1\mathrm L\) is moved to the
right, and that
\(\mathbf i_\beta^{\mathrm{mut}}=(i_{k_1},\ldots,i_{k_n})\). We use
\(\mu_{k_s}\) as shorthand for the vector mutation associated with
the corresponding weave mutation. Set
\begin{equation}\label{eq:successive-eta-tuples}
    \boldsymbol\eta^{(0)}
    :=(\eta_v^\beta)_
       {v\in\overrightarrow{\mathfrak W}(\beta)_3},
    \qquad
    \boldsymbol\eta^{(t)}
    :=\mu_{k_t}\circ\cdots\circ\mu_{k_1}
      (\boldsymbol\eta^{(0)})
      \quad(1\leq t\leq n).
\end{equation}
The component indexed by \(v\) is denoted by \(\eta_v^{(t)}\),
where vertex labels are transported along the local equivalences and
mutations. We use the same convention for the order of the vertices.
For every intermediate weave \(\mathfrak U\), we retain the fixed
coordinate lattice \(\ZZ^r\) indexed by the original positions of
\(\beta\). Let \(v_0\) denote the trivalent vertex corresponding to
the left insertion \(i_1\mathrm L\). The remaining trivalent vertices
inherit labels from the original solid right steps via the transported
labeling map
\begin{equation}\label{eq:transported-vertex-position-map}
    q_{\mathfrak U}\colon
    \mathfrak U_3\setminus\{v_0\}
    \longrightarrow [r],
\end{equation}
where \(q_{\mathfrak U}(v)=k\) if \(v\) is transported from the
trivalent vertex associated with the original solid step
\(i_k\mathrm R\). Accordingly, we set
\begin{equation}\label{eq:transported-coordinate-vector}
    E_v^{\mathfrak U}
    :=
    E_{q_{\mathfrak U}(v)}
    \in\ZZ^r
    \qquad
    \bigl(v\in\mathfrak U_3\setminus\{v_0\}\bigr).
\end{equation}

For example, consider
\[
    \ddot{\mathbf s}
    =
    (1\mathrm R^+,2\mathrm R^+,1\mathrm R^+,
      2\mathrm R,1\mathrm L,2\mathrm R),
\]
and let \(v_1<v_2<v_3\) be its three trivalent vertices. The
exceptional vertex is \(v_0=v_2\), corresponding to
\(1\mathrm L\), whereas the other two vertices retain the transported
labels
\[
    q_{\mathfrak U}(v_1)=5,
    \qquad
    q_{\mathfrak U}(v_3)=6.
\]
\begin{lemma}\label{lem:partial-eta}
Fix \(1\leq t\leq n\) and \(k_t<p\leq r\). Consider the partial
double inductive weave
\begin{equation}\label{eq:partial-double-inductive-weave}
    \mathfrak W_{t,p}
    =(i_2\mathrm R,\ldots,i_{k_t}\mathrm R,
      i_1\mathrm L,
      i_{k_t+1}\mathrm R,\ldots,i_{p-1}\mathrm R)
\end{equation}
and the partial right inductive weave
\(\overrightarrow{\mathfrak W}_p
=(i_1\mathrm R,\ldots,i_{p-1}\mathrm R)\). The decorations
determined by the relevant Demazure products are suppressed. Let
\(A_t\) and \(B_t\) be the cluster variables attached, respectively,
to \(i_1\mathrm L\) and \(i_{k_t}\mathrm R\) in
\(\mathfrak W_{t,p}\). Then the following statements hold.

\begin{enumerate}[label=\textup{(\arabic*)}]
    \item Under the identifications induced by the local weave
    equivalences, \(A_t\) is the cluster variable attached to
    \(i_{k_t}\mathrm R\) in
    \(\overrightarrow{\mathfrak W}_p\).

    \item Let \(A_0\) be the cluster variable attached to
    \(i_1\mathrm L\) immediately before the first mutation. At the
    \(t\)-th step,
    \begin{equation}\label{eq:Bt-mutation-At-minus-one}
        B_t=\mu_{k_t}(A_{t-1}).
    \end{equation}
    Equivalently, \(B_t\) is the new cluster variable introduced at
    the \(t\)-th step of the seed-mutation sequence
    \(\mu_{k_t}\circ\cdots\circ\mu_{k_1}\).

    \item Suppose that appending \(i_p\mathrm R\) produces a
    trivalent vertex \(v\), so that \(q_v^\beta=p\). Then
    \begin{equation}\label{eq:partial-eta-recursion}
        \eta_v^{(t)}=\eta_v^{(0)}
        =E_{q_v^\beta}+
        \sum_{v'<v}
        \gamma_{v'}^{\mathfrak W_{t,p}}
            (e_{v,\mathrm{nw}})\eta_{v'}^{(t)}.
    \end{equation}
    Here \(e_{v,\mathrm{nw}}\) is viewed as the corresponding
    southern boundary edge of the partial weave
    \(\mathfrak W_{t,p}\).
\end{enumerate}
\end{lemma}

\begin{proof}
We first prove \textup{(1)} and \textup{(2)}. Immediately before the
\(t\)-th mutation, suppose that the variable attached to
\(i_1\mathrm L\) is \(A_{t-1}\). The Case~\textup{(5)}
transformation consists of a mutation followed by weave
equivalences as in Figure \ref{fig:weave-mutation}. Under this transformation from $\mathfrak W_-$ to $\mathfrak W_+$ in equation~\eqref{eq:case5}, \(A_{t-1}\) mutates to
\(B_t\), which becomes attached to \(i_{k_t}\mathrm R\) in $\mathfrak W_{t,p}$. The
unchanged variable previously attached to \(i_{k_t}\mathrm R\) is
transported to \(i_1\mathrm L\); this is precisely \(A_t\). Starting
with \(A_0\) and arguing inductively proves both assertions.

We now prove \textup{(3)}. Since \(p>k_t\), the vertex \(v\) is not
one of the first \(t\) mutation vertices. Vector mutation changes
only the component at the mutation vertex, and hence
\(\eta_v^{(t)}=\eta_v^{(0)}\). The defining recursion in the partial
right inductive weave is
\begin{equation}\label{eq:eta-recursion-partial-right-weave}
    \eta_v^{(0)}
    =E_{q_v^\beta}+
    \sum_{v'<v}
    \gamma_{v'}^{\overrightarrow{\mathfrak W}_p}
        (e_{v,\mathrm{nw}})\eta_{v'}^{(0)}.
\end{equation}

If \(v'<v\) is frozen in either partial weave, Lemma~\ref{lem:output}
transports its frozen status and its output vector through all the
local moves. In particular, it is not one of the mutation vertices,
so \(\eta_{v'}^{(t)}=\eta_{v'}^{(0)}\), and
\begin{equation}\label{eq:frozen-summand-partial-weaves}
    \gamma_{v'}^{\mathfrak W_{t,p}}
        (e_{v,\mathrm{nw}})\eta_{v'}^{(t)}
    =
    \gamma_{v'}^{\overrightarrow{\mathfrak W}_p}
        (e_{v,\mathrm{nw}})\eta_{v'}^{(0)}.
\end{equation}
If \(v'\) is mutable, its Lusztig cycle vanishes on every southern
boundary edge. Both coefficients in
\eqref{eq:frozen-summand-partial-weaves} are then zero, so the same
identity remains valid. Summing over \(v'<v\) and substituting into
\eqref{eq:eta-recursion-partial-right-weave} gives
\eqref{eq:partial-eta-recursion}.
\end{proof}

We now establish the recursive description of the mutated vectors.

\begin{proposition}\label{pro:eta-recursion}
For \(1\leq t\leq n\), let
\(v_1^{(t)}\in(\mathfrak W^{(t)})_3\) be the vertex at which
\(\mu_{k_t}\) is performed, and let
\(v_1^{(t)+}\in(\mathfrak W^{(t+1)})_3\) be the vertex replacing it.
Let \(v_2^{(t)}\) be the lower trivalent vertex in the local mutation
configuration of Figure~\ref{fig:weave-mutation-branches}. All other
vertices are identified in the natural way. Then, for every
\(v\in(\mathfrak W^{(t+1)})_3\setminus\{v_2^{(t)}\}\),
\begin{equation}\label{eq:etavt}
    \eta_v^{(t)}
    =E_{v}^{\mathfrak W^{(t+1)}}+
    \sum_{v'<v}
    \gamma_{v'}^{\mathfrak W^{(t+1)}}
        (e_{v,\mathrm{nw}})\eta_{v'}^{(t)},
\end{equation}
Here \(E_v^{\mathfrak W^{(t+1)}}\) is defined by
\eqref{eq:transported-coordinate-vector}.
Moreover,
\begin{equation}\label{eq:exceptional-lower-vector}
    \eta_{v_2^{(t)}}^{(t)}
    =\eta_{v_2^{(t)}}^{(0)}
    =\eta_{v_2^{(t)}}^\beta.
\end{equation}
\end{proposition}

\begin{proof}
Note that $v_2^{(t)}$ corresponds to $i_1L$ in $\mathfrak W^{(t+1)}$ by Lemma~\ref{lem:partial-eta}. We proceed by induction on \(t\). We suppress the weave equivalences
occurring immediately before and after the local mutation in
\eqref{eq:case5}; these equivalences preserve the relevant boundary
weights by Lemma~\ref{lem:output}. Thus the \(t\)-th step may be
treated as the single local mutation from \(\mathfrak W^{(t)}\) to
\(\mathfrak W^{(t+1)}\).

For later use, put
\(
q_v^{(t)}:=q_{\mathfrak W^{(t+1)}}(v)
\)
for \(v\in(\mathfrak W^{(t+1)})_3\). The transported vertex order
and the original solid-position order agree:
\[
   v'<u \quad\Longleftrightarrow\quad
   q_{v'}^{(t)}<q_u^{(t)}.
\]
Suppose that \eqref{eq:etavt} has been established for \(u\) and
for every trivalent vertex preceding \(u\). We prove inductively
that, for every \(v'\leq u\) for which that recursion holds,
\[
    \max\operatorname{supp}\bigl(\eta_{v'}^{(t)}\bigr)
       =q_{v'}^{(t)},
    \qquad
    \bigl(\eta_{v'}^{(t)}\bigr)_{q_{v'}^{(t)}}=1.
\]
Indeed, if \(v'<u\), the induction hypothesis shows that every
nonzero coordinate of \(\eta_{v'}^{(t)}\) is at most
\(q_{v'}^{(t)}<q_u^{(t)}\). Hence, in \eqref{eq:etavt} for \(u\),
every term other than \(E_u^{\mathfrak W^{(t+1)}}\) is supported
strictly below \(q_u^{(t)}\). The latter unit vector has coefficient
one at \(q_u^{(t)}\), and all Lusztig-cycle weights are nonnegative.
It follows that
\begin{equation}\label{eq:leading-coordinate-mutated-eta}
    \max\operatorname{supp}\bigl(\eta_u^{(t)}\bigr)=q_u^{(t)},
    \qquad
    \bigl(\eta_u^{(t)}\bigr)_{q_u^{(t)}}=1.
\end{equation}

We shall prove \eqref{eq:etavt},
\eqref{eq:exceptional-lower-vector}, and this triangular consequence
simultaneously; no recursion is asserted for the exceptional lower
vertex.

Fix \(t\), and abbreviate
\(v_1=v_1^{(t)}\), \(v_1^+=v_1^{(t)+}\), and
\(v_2=v_2^{(t)}\). We consider four classes of vertices.

\smallskip
\noindent
\emph{Vertices above \(v_1\).}
Let \(v<v_1\). The mutation leaves the component indexed by \(v\)
unchanged, and the two weaves coincide above the local mutation
configuration. Thus
\begin{equation}\label{eq:unchanged-data-above-mutation}
    \eta_v^{(t)}=\eta_v^{(t-1)},
    \qquad
    E_v^{\mathfrak W^{(t+1)}}=E_v^{\mathfrak W^{(t)}},
\end{equation}
and the coefficients
\(\gamma_{v'}(e_{v,\mathrm{nw}})\) are unchanged for all \(v'<v\).
For \(t>1\), the unique exceptional vertex at the preceding stage is
the vertex labelled by \(i_1\mathrm L\), namely \(v_1\), so the
induction hypothesis applies to every \(v<v_1\). When \(t=1\), the
same recursion follows directly from the defining recursion for
\(\eta^\beta\) and the equivalences of Types~\textup{(1)}--\textup{(4)}
preceding the first mutation. Hence \eqref{eq:etavt} holds for all
vertices above \(v_1\).

\smallskip
\noindent
\emph{The mutated vertex \(v_1^+\).}
Set
\begin{equation}\label{eq:positive-negative-vector-sums}
    P_t:=\sum_{u\in(\mathfrak W^{(t)})_3}
       [\varepsilon_{v_1,u}^{(t)}]_+\eta_u^{(t-1)},
    \qquad
    N_t:=\sum_{u\in(\mathfrak W^{(t)})_3}
       [\varepsilon_{u,v_1}^{(t)}]_+\eta_u^{(t-1)}.
\end{equation}
By \eqref{eq:exchange-matrix-weave}, these are respectively the
positive and negative sums for the exchange column indexed by
\(v_1\).
The local configuration and
\cite[Lemma~4.33]{casals2025cluster} give
\(\varepsilon_{v_1,v_2}^{(t)}=1\). Moreover, \(v_2\) is the unique
vertex following \(v_1\) that is adjacent to it. Consequently,
\begin{align}
    P_t&=\eta_{v_2}^{(t-1)}+
        \sum_{u<v_1}[\varepsilon_{v_1,u}^{(t)}]_+
        \eta_u^{(t-1)},\label{eq:Pt-expansion}\\
    N_t&=\sum_{u<v_1}[-\varepsilon_{v_1,u}^{(t)}]_+
        \eta_u^{(t-1)}.\label{eq:Nt-expansion}
\end{align}

We now determine which of these two vectors is selected by the
mutation rule. The vertex \(v_2\), which corresponds to the entry
\(i_{k_t}\mathrm R\), has not been mutated during the first
\(t-1\) steps. Thus
\(\eta_{v_2}^{(t-1)}=\eta_{v_2}^{(0)}\). Put
\(q:=q_{v_2}^\beta\). By
Lemma~\ref{lem:maxindex} and Lemma~\ref{lem:partial-eta}(1), this vector has
coefficient \(1\) at position \(q\) and vanishes at every position
larger than \(q\). On the other hand,
\eqref{eq:leading-coordinate-mutated-eta} shows that every vector
\(\eta_u^{(t-1)}\) with \(u<v_1\) is supported in positions strictly
smaller than \(q\). It follows from
\eqref{eq:Pt-expansion}--\eqref{eq:Nt-expansion} that
\begin{equation}\label{eq:right-order-comparison-Pt-Nt}
    (P_t)_k=(N_t)_k=0\quad(k>q),
    \qquad
    (P_t)_q=1>0=(N_t)_q.
\end{equation}
By the definition of the right lexicographic order in
\eqref{eq:right-lexicographic-order}, we have
\(N_t\prec_{\mathrm r}P_t\). Therefore the vector-mutation rule gives
\begin{equation}\label{eq:mutated-vector-Pt}
    \eta_{v_1^+}^{(t)}=P_t-\eta_{v_1}^{(t-1)}.
\end{equation}

We next use the recursion for \(\eta_{v_2}^{(t-1)}\). When \(t=1\),
this is the defining recursion on the partial right inductive weave.
When \(t>1\), it follows from Lemma~\ref{lem:partial-eta} applied with
the indices \(t-1\) and \(p=k_t\). Since
\(\gamma_{v_1}^{\mathfrak W^{(t)}}
(e_{v_2,\mathrm{nw}})=1\), we obtain
\begin{equation}\label{eq:recursion-at-v2-before-mutation}
    \eta_{v_2}^{(t-1)}
    =E_{v_2}^{\mathfrak W^{(t)}}+
     \eta_{v_1}^{(t-1)}+
     \sum_{u<v_1}
     \gamma_u^{\mathfrak W^{(t)}}
        (e_{v_2,\mathrm{nw}})\eta_u^{(t-1)}.
\end{equation}
Substituting \eqref{eq:recursion-at-v2-before-mutation} into
\eqref{eq:Pt-expansion} and then into
\eqref{eq:mutated-vector-Pt} yields
\begin{equation}\label{eq:mutated-vector-before-gamma-identity}
    \eta_{v_1^+}^{(t)}
    =E_{v_2}^{\mathfrak W^{(t)}}+
     \sum_{u<v_1}
     \left(
       \gamma_u^{\mathfrak W^{(t)}}
          (e_{v_2,\mathrm{nw}})
       +[\varepsilon_{v_1,u}^{(t)}]_+
     \right)\eta_u^{(t-1)}.
\end{equation}
Only the component at \(v_1\) changes under the mutation, so
\(\eta_u^{(t)}=\eta_u^{(t-1)}\) for \(u<v_1\). By
Lemma~\ref{lem:gammamutation}, the coefficient in parentheses in
\eqref{eq:mutated-vector-before-gamma-identity} equals
\(\gamma_u^{\mathfrak W^{(t+1)}}
(e_{v_1^+,\mathrm{nw}})\). Finally, the local mutation identifies the
corresponding labels, giving
\(E_{v_2}^{\mathfrak W^{(t)}}
=E_{v_1^+}^{\mathfrak W^{(t+1)}}\) by Figure~\ref{fig:weave-mutation}. Hence
\begin{equation}\label{eq:eta-recursion-mutated-vertex}
    \eta_{v_1^+}^{(t)}
    =E_{v_1^+}^{\mathfrak W^{(t+1)}}+
     \sum_{u<v_1^+}
     \gamma_u^{\mathfrak W^{(t+1)}}
       (e_{v_1^+,\mathrm{nw}})\eta_u^{(t)},
\end{equation}
which is \eqref{eq:etavt} for the mutated vertex.

\smallskip
\noindent
\emph{The lower vertex \(v_2\).}
The vertex \(v_2\) corresponds to the entry
\(i_{k_t}\mathrm R\) before the \(t\)-th mutation, and it has not
been mutated at any earlier step. The \(t\)-th mutation is performed
at \(v_1\), not at \(v_2\). Therefore
\(\eta_{v_2}^{(t)}=\eta_{v_2}^{(0)}=\eta_{v_2}^\beta\), proving
\eqref{eq:exceptional-lower-vector}. No recursive formula is asserted
for this exceptional vertex.

\smallskip
\noindent
\emph{Vertices below \(v_2\).}
Let \(v>v_2\), and put
\(p=q_v^\beta\). The portion of \(\mathfrak W^{(t+1)}\) above \(v\)
is, up to weave equivalence, the partial double inductive weave
\(\mathfrak W_{t,p}\). Lemma~\ref{lem:partial-eta} therefore gives
\begin{equation}\label{eq:eta-recursion-below-local-mutation}
    \eta_v^{(t)}=\eta_v^{(0)}
    =E_v^{\mathfrak W^{(t+1)}}+
     \sum_{v'<v}
     \gamma_{v'}^{\mathfrak W^{(t+1)}}
       (e_{v,\mathrm{nw}})\eta_{v'}^{(t)},
\end{equation}
which is precisely \eqref{eq:etavt}.

These four cases exhaust the trivalent vertices of
\(\mathfrak W^{(t+1)}\). The triangular property
\eqref{eq:leading-coordinate-mutated-eta} follows for every vertex
to which the newly established recursion applies. This completes the
simultaneous induction.
\end{proof}

\begin{corollary}\label{cor:etavn}
Suppose $K_\beta^{\mathrm{mut}}\neq\emptyset$, let
\(
    \mu_{k_n}\circ\cdots\circ\mu_{k_1}
\)
be the mutation sequence induced by the transformations of
Case~\textup{(5)} relating
\(\overrightarrow{\mathfrak W}(\beta)\) to
\begin{equation}\label{eq:final-shifted-double-inductive-weave}
    \mathfrak W_1
    =
    (i_2\mathrm R,\ldots,i_r\mathrm R,i_1\mathrm L)
    =
    \mathfrak W^{(n+1)}.
\end{equation}
The final step \(i_1\mathrm L\) is solid. Let
\(v_{\mathrm{last}}\) be the corresponding trivalent vertex of
\(\mathfrak W_1\). Then, for every
\(
    v\in
    (\mathfrak W_1)_3\setminus\{v_{\mathrm{last}}\},
\)
one has
\begin{equation}\label{eq:final-eta-recursion}
    \eta_v^{(n)}
    =
    E_v^{\mathfrak W_1}
    +
    \sum_{v'<v}
    \gamma_{v'}^{\mathfrak W_1}
        (e_{v,\mathrm{nw}})
    \eta_{v'}^{(n)}.
\end{equation}
\end{corollary}

\begin{proof}
When the transformation of Case~\textup{(5)} first occurs, the
moving step \(i_1\mathrm L\) is solid. Its solid status is preserved
by all subsequent transformations, so it determines the final
trivalent vertex \(v_{\mathrm{last}}\) of \(\mathfrak W_1\).

Since
\(\mathfrak W^{(n+1)}=\mathfrak W_1\), the exceptional vertex
\(v_2^{(n)}\) in Proposition~\ref{pro:eta-recursion} is precisely
\(v_{\mathrm{last}}\). The result therefore follows immediately
from Proposition~\ref{pro:eta-recursion} with \(t=n\).
\end{proof}

\subsection{Vectors for right inductive weaves}
We use \(\Delta_{\fg}\) for the Dynkin diagram,
\(\Delta\in\Br^+\) for the positive braid lift of \(w_0\), and
\(\boldsymbol\Delta=(j_1,\ldots,j_\ell)\) for a chosen reduced word
representing that braid element. Thus
\(\boldsymbol\Delta\beta\) denotes a concatenated word, not a
product involving the Dynkin diagram. We first consider this case.
Let \(\beta=(i_1,\ldots,i_r)\) be a braid word, and set
\begin{equation}\label{eq:w-delta-beta}
    w:=\delta(\beta),
    \qquad
    m:=\ell(w),
    \qquad
    \ell:=\ell(w_0).
\end{equation}
Let \((j_1,\ldots,j_m)\) be the sequence of letters selected by the
leftmost reduced subexpression of \(\beta\) representing \(w\).
Recall that \(i^*\in I\) is determined by
\begin{equation}\label{eq:star-involution-simple-roots}
    w_0(\alpha_i)=-\alpha_{i^*}.
\end{equation}
The assignment \(i\mapsto i^*\) is an involution. If
\(w^*:=w_0ww_0\), then \((j_1^*,\ldots,j_m^*)\) is a reduced
expression of \(w^*\).

Choose a reduced expression
\((h_1,\ldots,h_{\ell-m})\) of \(w_0w^{-1}\). Since
\begin{equation}\label{eq:completion-to-w0}
    w^*(w_0w^{-1})=w_0,
    \qquad
    \ell(w^*)+\ell(w_0w^{-1})
    =m+(\ell-m)=\ell,
\end{equation}
the concatenation
\begin{equation}\label{eq:adapted-reduced-expression-Delta}
    \boldsymbol{\Delta}
    :=(j_1^*,\ldots,j_m^*,h_1,\ldots,h_{\ell-m})
\end{equation}
is a reduced expression of \(w_0\). Accordingly, we write
\begin{equation}\label{eq:word-Delta-beta}
    \boldsymbol{\Delta}\beta
    =(j_1^*,\ldots,j_m^*,h_1,\ldots,h_{\ell-m},
      i_1,\ldots,i_r).
\end{equation}

Consider the right inductive weave
\(\overrightarrow{\mathfrak W}(\boldsymbol{\Delta}\beta)\).
Every letter in the \(\beta\)-block produces a trivalent vertex,
because the preceding reduced block represents \(w_0\) and
\(\delta(w_0s_i)=w_0\) for every \(i\in I\). For a trivalent
vertex \(v\) of color \(i\), let \(v^-\) denote the last trivalent
vertex of color \(i\) preceding \(v\), when it exists. Otherwise,
set
\begin{equation}\label{eq:empty-predecessor-convention}
    v^-:=\varnothing,
    \qquad
    \eta_{\varnothing}^{\boldsymbol{\Delta}\beta}:=0.
\end{equation}

The natural embedding
\(\iota_{\boldsymbol\Delta}\colon\ZZ^r\hookrightarrow
\ZZ^{\ell+r}\) of the \(\beta\)-coordinate lattice into the
\(\boldsymbol{\Delta}\beta\)-coordinate lattice sends
\(E_k\) to \(E_{\ell+k}\), where \(\ell=\ell(w_0)\). We write
\begin{equation}\label{eq:embedded-beta-coordinate-after-Delta}
    \widehat E_k
    :=\iota_{\boldsymbol\Delta}(E_k)=E_{\ell+k}
\end{equation}
and never identify \(E_k\) and \(\widehat E_k\) literally.

\begin{remark}
We use \(E_{q_v^\beta}\), rather than
\(\widehat E_{q_v^\beta}\), because we record Lusztig parameters in
the original \(\beta\)-coordinate lattice. Indeed, the braid symmetry
\(T_{w_0}\) sends a global-basis element of
\(\widehat{\cA}(\beta)\) with \(\beta\)-Lusztig parameter
\(\mathbf a\) to a global-basis element of
\(\widehat{\cA}(\boldsymbol{\Delta}\beta)\) whose
\(\boldsymbol{\Delta}\beta\)-Lusztig parameter is the zero extension
\(\iota_{\boldsymbol{\Delta}}(\mathbf a)\). Moreover,
\[
    T_{w_0}\bigl(\widehat{\cA}(\beta)\bigr)
    =
    \widehat{\cA}(\boldsymbol{\Delta}\beta)
    \cap \widehat{\cA}_{\geq1}.
\]
Thus \(\widehat E_{q_v^\beta}
=\iota_{\boldsymbol{\Delta}}(E_{q_v^\beta})\) is the coordinate
vector obtained after applying \(T_{w_0}\). Since we suppress this
braid symmetry and work directly with the original
\(\beta\)-Lusztig parameters, we use \(E_{q_v^\beta}\).
\end{remark}

\begin{lemma}\label{lem:eta-Delta-beta-predecessor}
For every trivalent vertex \(v\) created by a letter in the \(\beta\)-block
of \(\overrightarrow{\mathfrak W}(\boldsymbol{\Delta}\beta)\), one has
\begin{equation}\label{eq:eta-Delta-beta-predecessor}
    \eta_v^{\boldsymbol{\Delta}\beta}
    =E_{q_v^\beta}
     +\eta_{v^-}^{\boldsymbol{\Delta}\beta}.
\end{equation}
Equivalently, the second summand in
\eqref{eq:eta-Delta-beta-predecessor} is absent when \(v\) has no
preceding trivalent vertex of the same color.
\end{lemma}

\begin{proof}
Fix a trivalent vertex \(u\) of color \(i\), and consider its
Lusztig cycle
\(\gamma_u^{\overrightarrow{\mathfrak W}
(\boldsymbol{\Delta}\beta)}\). Immediately above \(u\), the
corresponding horizontal slice has the form \((\mathbf a,i,i)\),
where \((\mathbf a,i)\) is a reduced expression of \(w_0\). The
first occurrence of \(i\) is the final letter of this reduced
expression, whereas the second is the newly appended letter that
creates \(u\).

The strand corresponding to the final letter \(i\) of
\((\mathbf a,i)\) is labelled by
\begin{equation}\label{eq:simple-root-label-at-u}
    w_0(-\alpha_i)=\alpha_{i^*}.
\end{equation}
Consequently, \(\gamma_u\) starts on the unique strand labelled by
the simple root \(\alpha_{i^*}\).

Every reduced horizontal slice below \(u\) represents \(w_0\).
Indeed, if \(\beta_{\leq p}:=(i_1,\ldots,i_p)\), then
\begin{equation}\label{eq:Demazure-product-after-Delta}
    \delta(\boldsymbol{\Delta}\beta_{\leq p})=w_0
    \qquad(0\leq p\leq r).
\end{equation}
The root labels on each such slice therefore form the set of positive
roots, so there is a unique strand labelled by
\(\alpha_{i^*}\). By
\cite[Lemma~4.39]{casals2025cluster}, a strand carrying a simple-root
label cannot enter a six-valent vertex through its middle edge, and
its label is preserved as it passes through four- and six-valent
vertices. Together with the tropical Lusztig rules, this implies that
\(\gamma_u\) has weight \(1\) on the distinguished strand and
weight \(0\) on every other strand. In particular, the cycle does
not bifurcate, in agreement with
\cite[Lemma~4.42]{casals2025cluster}.

Suppose that the distinguished strand next meets a trivalent vertex
\(x\) of color \(h\). The strand entering \(x\) from the reduced
\(w_0\)-part is labelled by \(\alpha_{h^*}\). Hence the strand
supporting \(\gamma_u\) enters \(x\) if and only if
\(\alpha_{i^*}=\alpha_{h^*}\), or equivalently, \(h=i\). At such a
vertex, the other incoming edge has weight \(0\), and the trivalent
Lusztig rule gives \(\min\{1,0\}=0\). Thus \(\gamma_u\) terminates
at the next trivalent vertex of color \(i\); before reaching that
vertex, it passes all trivalent vertices of other colors without
changing its distinguished weight.

It follows that, for every \(u<v\),
\begin{equation}\label{eq:gamma-coefficient-same-color-predecessor}
    \gamma_u^{\overrightarrow{\mathfrak W}
        (\boldsymbol{\Delta}\beta)}(e_{v,\mathrm{nw}})
    =
    \begin{cases}
        1,&u=v^-,\\
        0,&u\neq v^-.
    \end{cases}
\end{equation}
When \(v^-\) does not exist, the first case is absent. Substituting
\eqref{eq:gamma-coefficient-same-color-predecessor} into the defining
recursion for \(\eta_v^{\boldsymbol{\Delta}\beta}\) gives
\eqref{eq:eta-Delta-beta-predecessor}.
\end{proof}

We next define the mutation sequence that moves the initial block
\((j_1^*,\ldots,j_m^*)\) of \(\boldsymbol{\Delta}\) from the left
of the right-insertion string to the right as left insertions. For
\(0\leq s\leq m\), let \(\mathfrak W_{[s]}\) be the double
inductive weave encoded by
\begin{equation}\label{eq:intermediate-Delta-beta-weaves}
\begin{aligned}
    \mathfrak W_{[s]}=\bigl(&
      j_{s+1}^*\mathrm R,\ldots,j_m^*\mathrm R,
      h_1\mathrm R,\ldots,h_{\ell-m}\mathrm R,
      i_1\mathrm R,\ldots,i_r\mathrm R,j_s^*\mathrm L,\ldots,j_1^*\mathrm L\bigr).
\end{aligned}
\end{equation}
Here and below, an empty block is omitted. Thus
\begin{equation}\label{eq:initial-final-Delta-beta-weaves}
\begin{aligned}
    \mathfrak W_{[0]}
    &=
    \overrightarrow{\mathfrak W}
       (\boldsymbol{\Delta}\beta),\\
    \mathfrak W_{[m]}
    &=
    (h_1\mathrm R,\ldots,h_{\ell-m}\mathrm R,
      i_1\mathrm R,\ldots,i_r\mathrm R,
      j_m^*\mathrm L,\ldots,j_1^*\mathrm L).
\end{aligned}
\end{equation}
For \(1\leq s\leq m\), let \(T_s\) be the mutation sequence that
transforms \(\mathfrak W_{[s-1]}\) into
\(\mathfrak W_{[s]}\) by moving \(j_s^*\mathrm R\) through the
remaining right-insertion block and converting it into
\(j_s^*\mathrm L\). Set
\begin{equation}\label{eq:mutation-sequence-T}
    T:=T_m\circ\cdots\circ T_1.
\end{equation}
By \cite[Lemma~3.11]{bi2026towards}, the mutation sequence \(T\)
agrees with the \(m\)-step mutation sequence \(M_m\) in
\eqref{eq:M-k}.
The weave equivalences occurring between mutations are understood to
transport vertex labels in the natural way. The block
\((h_1\mathrm R,\ldots,h_{\ell-m}\mathrm R)\) remains fixed
throughout and may therefore be suppressed from the double strings
when no confusion can arise.

Set
\begin{equation}\label{eq:initial-eta-tuple-Delta-beta}
    \boldsymbol\eta^{[0]}
    :=\left(\eta_v^{\boldsymbol{\Delta}\beta}\right)_
       {v\in\overrightarrow{\mathfrak W}
       (\boldsymbol{\Delta}\beta)_3},
\end{equation}
and define recursively
\begin{equation}\label{eq:successively-mutated-eta-Delta-beta}
    \boldsymbol\eta^{[s]}
    :=T_s(\boldsymbol\eta^{[s-1]})
    \quad(1\leq s\leq m),
    \qquad
    \boldsymbol\eta':=\boldsymbol\eta^{[m]}.
\end{equation}
We denote the component of \(\boldsymbol\eta'\) indexed by \(v\)
by \(\eta'_v\).

The choice of the leftmost reduced subexpression identifies the
trivalent vertices in the \(\beta\)-block of \(\mathfrak W_{[m]}\)
with those of \(\overrightarrow{\mathfrak W}(\beta)\). To prove it, we first introduce the following Lemma.

For \(x,y\in W\), choose a reduced expression
\[
    y=s_{a_1}\cdots s_{a_p}
\]
and define the \emph{Demazure product} \(x*y\) recursively by
\[
    x_0:=x,
    \qquad
    x_d:=
    \begin{cases}
        x_{d-1}s_{a_d},
        &\ell(x_{d-1}s_{a_d})>\ell(x_{d-1}),\\
        x_{d-1},
        &\ell(x_{d-1}s_{a_d})<\ell(x_{d-1}),
    \end{cases}
\]
and \(x*y:=x_p\). This definition is independent of the chosen
reduced expression of \(y\). If \(\beta\) and \(\gamma\) are
positive words, then
\begin{equation}\label{eq:Demazure-product-concatenation}
    \delta(\beta\gamma)
    =
    \delta(\beta)*\delta(\gamma).
\end{equation}
Indeed, both sides are obtained by applying the same Demazure
recursion successively to the letters of \(\beta\) and then to
those of \(\gamma\).

\begin{lemma}[Invariance under deletion of the reduced \(h\)-block]
\label{lem:Lusztig-cycle-h-block-invariance}
Retain the notation of
\eqref{eq:initial-final-Delta-beta-weaves}. Deleting from \(\mathfrak W_{[m]}\) the strands introduced by the
reduced block
\[
    (h_1\mathrm R,\ldots,h_{\ell-m}\mathrm R)
\]
canonically identifies the remaining \(\beta\)-block with
\(\overrightarrow{\mathfrak W}(\beta)\). Under this identification:

\begin{enumerate}
\item the trivalent vertices in the \(\beta\)-block are identified
with the vertices of
\(\overrightarrow{\mathfrak W}(\beta)_3\);

\item the northwest edge of each such vertex \(v\) is identified with
\(e_{v,\mathrm{nw}}\) in
\(\overrightarrow{\mathfrak W}(\beta)\);

\item for every
\(u,v\in\overrightarrow{\mathfrak W}(\beta)_3\), one has
\begin{equation}\label{eq:Lusztig-cycle-h-block-invariance}
    \gamma_u^{\mathfrak W_{[m]}}
        (e_{v,\mathrm{nw}})
    =
    \gamma_u^{\overrightarrow{\mathfrak W}(\beta)}
        (e_{v,\mathrm{nw}}).
\end{equation}
\end{enumerate}
Here the same symbols \(u,v\), and \(e_{v,\mathrm{nw}}\) are used for
the corresponding vertices and edges under this identification.
\end{lemma}

\begin{proof}
Put
\[
    \boldsymbol h:=(h_1,\ldots,h_{\ell-m}),
    \qquad
    w:=\delta(\beta),
    \qquad
    x:=\delta(\boldsymbol h)=w_0w^{-1}.
\]
For \(0\leq k\leq r\), set
\(
    \beta_k:=(i_1,\ldots,i_k),\)
    and 
    \(u_k:=\delta(\beta_k).
\)
Since \(u_k\) is below \(w\) in the right weak order, there exists
\(y_k\in W\) such that
\[
    w=u_ky_k,
    \qquad
    \ell(w)=\ell(u_k)+\ell(y_k).
\]
Moreover,
\[
    xw=w_0,
    \qquad
    \ell(x)+\ell(w)=\ell(w_0).
\]
It follows that
\[
\begin{aligned}
    \ell(x)+\ell(w)
    &=\ell(xu_ky_k)\\
    &\leq\ell(xu_k)+\ell(y_k)\\
    &\leq\ell(x)+\ell(u_k)+\ell(y_k)
     =\ell(x)+\ell(w).
\end{aligned}
\]
Hence
\[
    \ell(xu_k)=\ell(x)+\ell(u_k),
\]
and therefore \(x*u_k=xu_k\). By
\eqref{eq:Demazure-product-concatenation},
\begin{equation}\label{eq:Demazure-h-prefix}
    \delta(\boldsymbol h\beta_k)
    =
    \delta(\boldsymbol h)*\delta(\beta_k)
    =
    xu_k.
\end{equation}
Consequently,
\begin{align*}
    \delta(\beta_k)=\delta(\beta_{k-1})
    &\iff u_k=u_{k-1}\\
    &\iff xu_k=xu_{k-1}\\
    &\iff
    \delta(\boldsymbol h\beta_k)
    =
    \delta(\boldsymbol h\beta_{k-1}),
\end{align*}
where the middle equivalence follows from cancellation in \(W\).

Thus the trivalent positions in the \(\beta\)-block of
\(\mathfrak W_{[m]}\) are precisely those of
\(\overrightarrow{\mathfrak W}(\beta)\). By the definition of double inductive weave in Definition~\ref{def:double-inductive-weave}, the weave
\(\mathfrak W_{[m]}\) can be obtained from
\(\overrightarrow{\mathfrak W}(\beta)\) by adjoining the strands
corresponding to
\(h_1,\ldots,h_{\ell-m}\). These additional strands run straight from
the northern to the southern boundary and contain no trivalent
vertices. Consequently, deleting them does not alter the local configurations in the
\(\beta\)-block. It therefore identifies the corresponding trivalent
vertices, their northwest edges, and the restrictions of their
Lusztig cycles to the \(\beta\)-block. Hence
\[
    \gamma_u^{\mathfrak W_{[m]}}
        (e_{v,\mathrm{nw}})
    =
    \gamma_u^{\overrightarrow{\mathfrak W}(\beta)}
        (e_{v,\mathrm{nw}})
\]
for all
\(u,v\in\overrightarrow{\mathfrak W}(\beta)_3\). This proves
\eqref{eq:Lusztig-cycle-h-block-invariance}.
\end{proof}

\begin{theorem}\label{thm:eta-after-Delta-mutation}
For every
\(v\in\overrightarrow{\mathfrak W}(\beta)_3\), one has
\begin{equation}\label{eq:eta-after-Delta-mutation}
    \eta'_v
    =
    E_{q_v^\beta}
    +
    \sum_{\substack{
        v'\in\overrightarrow{\mathfrak W}(\beta)_3\\
        v'<v}}
    \gamma_{v'}^{\overrightarrow{\mathfrak W}(\beta)}
        (e_{v,\mathrm{nw}})\eta'_{v'}.
\end{equation}
Here the coordinate lattice of the \(\beta\)-block is identified with
its natural coordinate sublattice in the lattice associated with
\(\boldsymbol{\Delta}\beta\).
\end{theorem}

\begin{proof}
Let
\[
    \mathfrak W_{[0]},\mathfrak W_{[1]},\ldots,
    \mathfrak W_{[m]}
\]
be the sequence of weaves occurring in
\eqref{eq:initial-final-Delta-beta-weaves}, where \(T_s\) transforms
\(\mathfrak W_{[s-1]}\) into \(\mathfrak W_{[s]}\). If \(m=0\),
the assertion is the defining recursion for the \(\eta\)-vectors, so
assume \(m>0\).

Apply Corollary~\ref{cor:etavn} successively to
\(j_1^*,\ldots,j_m^*\). At the \(s\)-th step, its only exceptional
vertex is the new final trivalent vertex associated with
\(j_s^*\mathrm L\). This vertex lies after the entire \(\beta\)-block.
Consequently, it does not occur among the predecessors of a vertex
\(v\in\overrightarrow{\mathfrak W}(\beta)_3\). Induction on \(s\)
therefore gives
\begin{equation}\label{eq:eta-recursion-at-stage-s}
    \eta_v^{[s]}
    =
     E_{q_v^\beta}
    +
    \sum_{\substack{
        v'<v\\
        v'\ \text{in the \(\beta\)-block}}}
    \gamma_{v'}^{\mathfrak W_{[s]}}
        (e_{v,\mathrm{nw}})
    \eta_{v'}^{[s]}
\end{equation}
for every vertex \(v\) in the \(\beta\)-block.
By Lemma~\ref{lem:Lusztig-cycle-h-block-invariance} and \(\eta_v^{[m]}=\eta'_v\), we obtain
\eqref{eq:eta-after-Delta-mutation}.
\end{proof}

\subsection{Vectors of grid minors}
\label{subsec:vectors-grid-minors}

We retain the notation introduced above. Fix a braid word
\(\beta=(i_1,\ldots,i_r)\), let
\(\beta_c=(i_1,\ldots,i_c)\), and set
\[
J_\beta
=
\bigl\{
c\in[1,r]
\mid
\delta(\beta_c)=\delta(\beta_{c-1})
\bigr\}.
\]
Thus, \(J_\beta\) is the set of solid positions of the right inductive
weave \(\overrightarrow{\mathfrak W}(\beta)\). We use the
specialization of the bijection
\eqref{eq:trivalent-solid-position-bijection} to pass between a
trivalent vertex and its corresponding element of \(J_\beta\).

For \(e\in J_\beta\) and \(e\leq c\leq r\), write
\begin{equation}
 \bigl\{
 d\mid e<d\leq c,\ d\notin J_\beta
 \bigr\}
 =
 \{d_1<\cdots<d_t\},
 \qquad
 \sigma_{e,c}
 :=
 s_{i_{d_1}}\cdots s_{i_{d_t}}.
 \label{eq:sigma-ec}
\end{equation}
By convention, \(\sigma_{e,c}=\mathrm{id}\) when the set on the
left-hand side is empty. Notice that
\(\sigma_{e,c}=\sigma_{e,c-1}s_{i_c}\) if \(c\notin J_\beta\), whereas
\(\sigma_{e,c}=\sigma_{e,c-1}\) if \(c\in J_\beta\).

We next recall the coroot associated with a Lusztig cycle at a
horizontal slice. Fix \(c\in[1,r]\), and choose a reduced expression
\[
w_c:=\delta(\beta_c)=s_{j_1}\cdots s_{j_m}.
\]
Let \(\mathsf e_1^{(c)},\ldots,\mathsf e_m^{(c)}\) be the edges
intersected by the corresponding horizontal path, ordered from left
to right. For \(1\leq a\leq m\), recall the inversion coroot
\begin{equation}
 \chi_a^{(c)}
 :=
 s_{j_m}\cdots s_{j_{a+1}}
 \bigl(\alpha_{j_a}^{\vee}\bigr).
 \label{eq:inversion-coroot-c}
\end{equation}
For a trivalent vertex \(v\leq c\), recall
\eqref{eq:def-gamma-vc}:
\begin{equation}
 \gamma_{v,c}
 :=
 \sum_{a=1}^{m}
 \gamma_v\bigl(\mathsf e_a^{(c)}\bigr)\chi_a^{(c)}
 \in Q^\vee.
 \label{eq:gamma-vc}
\end{equation}

\begin{lemma}
\label{lem:gamma-vc-well-defined}
The coroot \(\gamma_{v,c}\) is independent of the chosen reduced
expression of \(w_c\).
\end{lemma}

\begin{proof}
By Matsumoto's theorem, any two reduced expressions of \(w_c\) are
related by a sequence of braid moves. Since \(G\) is simply laced, it
suffices to consider commutation moves and braid moves of the form
\[
    iji\longleftrightarrow jij.
\]

A commutation move interchanges two orthogonal simple reflections.
The corresponding two inversion coroots and the two cycle
coefficients are interchanged simultaneously, so their contribution
to \(\gamma_{v,c}\) is unchanged.

Consider a \(3\)-move
\[
    \mathbf j=\mathbf x(iji)\mathbf y
    \quad\longleftrightarrow\quad
    \mathbf j'=\mathbf x(jij)\mathbf y,
\]
where \(i\) and \(j\) are adjacent. Write
\(
    \mathbf y=(y_1,\ldots,y_q)
\)
and let
\(
    u:=s_{y_q}\cdots s_{y_1}.
\)
Thus \(u\) is the common Weyl group action on the coroot lattice
coming from the suffix \(\mathbf y\).

Before applying \(u\), the three inversion coroots associated with
the substring \(iji\) are
\[
    \bigl(
        \alpha_j^\vee,\,
        \alpha_i^\vee+\alpha_j^\vee,\,
        \alpha_i^\vee
    \bigr),
\]
whereas those associated with \(jij\) are
\[
    \bigl(
        \alpha_i^\vee,\,
        \alpha_i^\vee+\alpha_j^\vee,\,
        \alpha_j^\vee
    \bigr).
\]
Consequently, the actual inversion coroot triples in the full words
\(\mathbf j\) and \(\mathbf j'\) are respectively
\[
    \bigl(
        u(\alpha_j^\vee),\,
        u(\alpha_i^\vee+\alpha_j^\vee),\,
        u(\alpha_i^\vee)
    \bigr)
\]
and
\[
    \bigl(
        u(\alpha_i^\vee),\,
        u(\alpha_i^\vee+\alpha_j^\vee),\,
        u(\alpha_j^\vee)
    \bigr).
\]

Let \((a_1,a_2,a_3)\) be the values of the Lusztig cycle on the three
affected edges. Under the braid move, they become
\[
\begin{aligned}
    b_1=a_2+a_3-\min(a_1,a_3),\quad
    b_2=\min(a_1,a_3),\quad
    b_3=a_1+a_2-\min(a_1,a_3).
\end{aligned}
\]
In particular,
\[
    b_1+b_2=a_2+a_3,
    \qquad
    b_2+b_3=a_1+a_2.
\]
Hence
\begin{align*}
    &a_1\alpha_j^\vee
     +a_2(\alpha_i^\vee+\alpha_j^\vee)
     +a_3\alpha_i^\vee\\
    &\qquad
     =(a_2+a_3)\alpha_i^\vee
       +(a_1+a_2)\alpha_j^\vee\\
    &\qquad
     =b_1\alpha_i^\vee
       +b_2(\alpha_i^\vee+\alpha_j^\vee)
       +b_3\alpha_j^\vee.
\end{align*}
Applying the common linear action \(u\) to this equality gives
\begin{align*}
    &a_1u(\alpha_j^\vee)
     +a_2u(\alpha_i^\vee+\alpha_j^\vee)
     +a_3u(\alpha_i^\vee)\\
    &\qquad
     =b_1u(\alpha_i^\vee)
       +b_2u(\alpha_i^\vee+\alpha_j^\vee)
       +b_3u(\alpha_j^\vee).
\end{align*}
Thus the contribution of the three affected positions to
\(\gamma_{v,c}\) is unchanged.

The coefficients outside these three positions are unchanged by the
local tropical braid move. The inversion coroots associated with
positions in the suffix \(\mathbf y\) are plainly unchanged. Those
associated with positions in the prefix \(\mathbf x\) are also
unchanged, because their defining suffix actions contain respectively
\[
    s_is_js_i
    \quad\text{and}\quad
    s_js_is_j,
\]
which are equal by the braid relation. Therefore every contribution
outside the affected triple is unchanged as well.

Hence \(\gamma_{v,c}\) is invariant under a \(3\)-move, and therefore
under every change of reduced expression.
\end{proof}

For \(v\in J_\beta\), recall the vector 
\(\eta_v^\beta\) in Definition \ref{def:eta-vector-right-inductive}. For \(e\in[1,r]\), we write
\[
 \eta_{v,e}:=(\eta_v^\beta)_e.
\]
The recursive definition of \(\eta_v^\beta\) implies that
\(\eta_{v,e}=0\) whenever \(q_v^\beta<e\).

\begin{lemma}
\label{lem:solid-step-gamma}
Suppose that \(c+1\in J_\beta\). Then, for every \(e\in J_\beta\)
with \(e\leq c\),
\begin{equation}
 \sum_{\substack{v\in J_\beta\\ e\leq v\leq c}}
 \gamma_{v,c}\eta_{v,e}
 =
 \sum_{\substack{v\in J_\beta\\ e\leq v\leq c+1}}
 \gamma_{v,c+1}\eta_{v,e}.
 \label{eq:solid-step-gamma}
\end{equation}
\end{lemma}

\begin{proof}
Set \(i=i_{c+1}\). Since \(c+1\in J_\beta\), we have
\(\delta(\beta_{c+1})=\delta(\beta_c)\), and \(i\) is a right descent
of \(w_c\). By Lemma~\ref{lem:gamma-vc-well-defined}, we may choose a
reduced expression of \(w_c\) ending in \(s_i\). Let
\(\chi_m^{(c)}=\alpha_i^\vee\) be its final inversion coroot.

The local trivalent-vertex rule in the right inductive weave gives
\begin{equation}
 \gamma_{v,c}
 =
 \gamma_{v,c+1}
 +
 \gamma_v(e_{c+1,\mathrm{nw}})\chi_m^{(c)}
 \quad (v\leq c),
 \qquad
 \gamma_{c+1,c+1}=\chi_m^{(c)}.
 \label{eq:gamma-solid-local}
\end{equation}
On the other hand, taking the \(e\)-th component in the recursion for
\(\eta_{c+1}^\beta\) yields
\begin{equation}
 \eta_{c+1,e}
 =
 \sum_{\substack{v\in J_\beta\\ e\leq v\leq c}}
 \gamma_v(e_{c+1,\mathrm{nw}})\eta_{v,e}.
 \label{eq:eta-solid-local}
\end{equation}
Substituting \eqref{eq:gamma-solid-local} and
\eqref{eq:eta-solid-local}, we obtain
\begin{align*}
 \sum_{\substack{v\in J_\beta\\e\leq v\leq c}}
 \gamma_{v,c}\eta_{v,e}
 &=
 \sum_{\substack{v\in J_\beta\\e\leq v\leq c}}
 \gamma_{v,c+1}\eta_{v,e}
 +
 \chi_m^{(c)}
 \sum_{\substack{v\in J_\beta\\e\leq v\leq c}}
 \gamma_v(e_{c+1,\mathrm{nw}})\eta_{v,e}                 \\
 &=
 \sum_{\substack{v\in J_\beta\\e\leq v\leq c}}
 \gamma_{v,c+1}\eta_{v,e}
 +
 \gamma_{c+1,c+1}\eta_{c+1,e},
\end{align*}
which is precisely \eqref{eq:solid-step-gamma}.
\end{proof}

The following identity is the main ingredient in the computation of
the $\beta$-Lusztig parameters of quantum grid minors.

\begin{theorem}
\label{thm:root-identity}
For \(c\in [1,r], e\in J_\beta\) and \(e\leq c\), one has
\begin{equation}\label{eq:root-identity}
 \sum_{\substack{v\in J_\beta\\e\leq v\leq c}}
 \gamma_{v,c}\eta_{v,e}
 =
 \sigma_{e,c}^{-1}\bigl(\alpha_{i_e}^{\vee}\bigr).
\end{equation}
\end{theorem}

\begin{proof}
We proceed by induction on \(c-e\). If \(c=e\), then
\(\eta_{e,e}=1\), \(\gamma_{e,e}=\alpha_{i_e}^\vee\), and
\(\sigma_{e,e}=\mathrm{id}\), so the result follows.

Suppose first that \(c\notin J_\beta\). Then
\[
 \delta(\beta_c)=\delta(\beta_{c-1})s_{i_c},
 \qquad
 \ell\bigl(\delta(\beta_c)\bigr)
 =
 \ell\bigl(\delta(\beta_{c-1})\bigr)+1.
\]
No new trivalent vertex is introduced at this step. Moreover, the
inversion coroots associated with the preceding slice are transformed
by \(s_{i_c}\). Hence
\(\gamma_{v,c}=s_{i_c}(\gamma_{v,c-1})\) for every \(v\in J_\beta \text{ and }v\leq c-1\).
The induction hypothesis therefore gives
\begin{align*}
 \sum_{\substack{v\in J_\beta\\e\leq v\leq c}}
 \gamma_{v,c}\eta_{v,e}
 &=
 s_{i_c}
 \sum_{\substack{v\in J_\beta\\e\leq v\leq c-1}}
 \gamma_{v,c-1}\eta_{v,e}                               \\
 &=
 s_{i_c}\sigma_{e,c-1}^{-1}
 \bigl(\alpha_{i_e}^\vee\bigr)                           \\
 &=
 \sigma_{e,c}^{-1}
 \bigl(\alpha_{i_e}^\vee\bigr),
\end{align*}
where the last equality follows from
\(\sigma_{e,c}=\sigma_{e,c-1}s_{i_c}\).

Suppose next that \(c\in J_\beta\). Applying
Lemma~\ref{lem:solid-step-gamma} with \(c-1\) in place of \(c\), we
obtain
\[
 \sum_{\substack{v\in J_\beta\\e\leq v\leq c}}
 \gamma_{v,c}\eta_{v,e}
 =
 \sum_{\substack{v\in J_\beta\\e\leq v\leq c-1}}
 \gamma_{v,c-1}\eta_{v,e}.
\]
The required identity now follows from the induction hypothesis and
the equality \(\sigma_{e,c}=\sigma_{e,c-1}\).
\end{proof}

\begin{definition}
\label{def:delta-beta-cj}
For \(c\in[1,r]\) and \(j\in I\), define an integral vector
\(\boldsymbol\delta_{c,j}^\beta\in\ZZ^r\) by
\begin{equation}
 \bigl(\boldsymbol\delta_{c,j}^\beta\bigr)_e
 =
 \begin{cases}
 \bigl\langle
 \sigma_{e,c}\varpi_j,\alpha_{i_e}^\vee
 \bigr\rangle,
 & e\in J_\beta\text{ and }e\leq c,\\[2mm]
 0,
 & \text{otherwise}.
 \end{cases}
\label{eq:delta-beta-cj}
\end{equation}
\end{definition}

We now transport the preceding degree vector to an arbitrary double
string. Let
\[
 \ddot{\mathbf s}
 =
 (i_1X_1,\ldots,i_rX_r)
\]
be a double string, and fix \(c\in[1,r]\).  Let
\(\beta_c^{\ddot{\mathbf s}}\) be the braid word
obtained after the first \(c\) insertions. Since subsequent insertions
are made only at the two ends, \(\beta_c^{\ddot{\mathbf s}}\) occurs
as a consecutive subword of the final braid word
\(\beta=\beta_r^{\ddot{\mathbf s}}\); thus,
\begin{equation}
 \beta
 =
 \gamma_c\,
 \beta_c^{\ddot{\mathbf s}}\,
 \eta_c
 \label{eq:beta-block-decomposition}
\end{equation}
for suitable braid words \(\gamma_c\) and \(\eta_c\).

Let
\[
 \iota_c^{\ddot{\mathbf s}}
 \colon
 \ZZ^c\longrightarrow\ZZ^r
\]
be the embedding obtained by placing a vector in the coordinates
corresponding to the block \(\beta_c^{\ddot{\mathbf s}}\) in
\eqref{eq:beta-block-decomposition} and assigning zero to the
\(\gamma_c\)- and \(\eta_c\)-coordinates.

\begin{definition}
\label{def:delta-double-string}
For an insertion \(i_cX_c\) of a double string $\ddot{\mathbf{s}}$, define
\begin{equation}
 \boldsymbol\delta_{c,j}^{\ddot{\mathbf s}}
 :=
 \iota_c^{\ddot{\mathbf s}}
 \left(
 \boldsymbol\delta_{c,j}^{\beta_c^{\ddot{\mathbf s}}}
 \right)
 \in\ZZ_{\geq0}^r.
 \label{eq:delta-double-string}
\end{equation}
\end{definition}

\begin{example}
\label{ex:delta-braid-i-box}
Let \(\boldsymbol\Delta=(j_1,\ldots,j_\ell)\) be a reduced word
for \(w_0\), let \(\beta=(i_1,\ldots,i_r)\), and let \([a,b]\)
be an \(i\)-box in \(\beta\). Write
\(\mathbf p=(j_1R^+,\ldots,j_\ell R^+)\). Consider
\begin{equation}\label{eq:double-string-i-box}
 \ddot{\mathbf s}_{[a,b]}
 =\bigl(
   \mathbf p,\,
   i_aR,\ldots,i_bR,\,
   i_{a-1}^*L,\ldots,i_1^*L,\,
   i_{b+1}R,\ldots,i_rR
  \bigr),
\end{equation}
with empty blocks omitted. Its actual final word is
\begin{equation}\label{eq:i-box-actual-word}
    \omega_{[a,b]}
    =(i_1^*,\ldots,i_{a-1}^*)\,
      \boldsymbol\Delta\,(i_a,\ldots,i_r).
\end{equation}
The stars on the left insertions are required by
\(\sigma_{i^*}\Delta=\Delta\sigma_i\), where \(\sigma_i\) is
the positive braid generator. Thus \(\omega_{[a,b]}\) and
\(\boldsymbol\Delta\beta\) represent the same braid element;
they need not be the same word.

All steps after \(\mathbf p\) are solid. At depth
\(d:=\ell+b-a+1\), the truncated word is
\(\boldsymbol\Delta(i_a,\ldots,i_b)\). Its reduced prefix has
length \(\ell\), and the intervening operators \(\sigma_{e,d}\)
are the identity for the solid positions. Therefore
Definition~\ref{def:delta-double-string} gives
\begin{equation}\label{eq:i-box-vector}
    [a,b]_\phi
    :=\sum_{\substack{a\leq t\leq b\\ i_t=i}}E_{\ell+t},
    \qquad \phi(t):=\ell+t,
\end{equation}
and
\begin{equation}\label{eq:i-box-degree}
    \boldsymbol\delta_{d,i}^{\ddot{\mathbf s}_{[a,b]}}
    =[a,b]_\phi
    \quad\text{in the \(\omega_{[a,b]}\)-coordinate lattice.}
\end{equation}
Indeed, the later left block has length \(a-1\), so the coordinate
\(\ell+(t-a+1)\) in the truncated word becomes \(\ell+t\)
in the final word.

To compare this vector with the coordinates of
\(\boldsymbol\Delta\beta\), choose braid moves entirely within
the prefix
\((i_1^*,\ldots,i_{a-1}^*)\boldsymbol\Delta\), transforming it
to \(\boldsymbol\Delta(i_1,\ldots,i_{a-1})\). Denote the induced
PBW transition map by
\(\Psi_{\boldsymbol\Delta\beta}^{\omega_{[a,b]}}\).
All coordinates of \([a,b]_\phi\) in this prefix vanish. The
local PBW transition rules preserve the zero vector and fix the
coordinates outside the move, so
\begin{equation}\label{eq:i-box-PBW-transition}
    \Psi_{\boldsymbol\Delta\beta}^{\omega_{[a,b]}}
       ([a,b]_\phi)
    =\iota_{\boldsymbol\Delta}([a,b]_\beta).
\end{equation}
Here the transition map identifies PBW parameters of the same
global-basis element, as in
\cite[Theorem~4.12]{bi2026towards}. The corresponding quantum
identity is established after
Theorem~\ref{thm:quantum-grid-minor-factorization} below.
\end{example}

\section{Quantization of braid varieties}\label{sec:quantization-braid-varieties}
In this section, we construct a quantization of braid varieties in terms of bosonic extension algebras. 
We first recall the relation between the word seed and the seed
associated with the right inductive weave. Let
\(\beta=(i_1,\ldots,i_r)\) be a positive braid word, and let
\(\boldsymbol{\Delta}\) be the reduced expression of \(w_0\) fixed
above. For each \(k\in[r]\), let \(v_k\) be the trivalent vertex of
\(
    \overrightarrow{\mathfrak W}
    (\boldsymbol{\Delta}\beta)
\)
created by the letter \(i_k\) in the \(\beta\)-block. Let
\(1(i_k)^+\) denote the first occurrence of the color \(i_k\) in
\([1,k]\), and abbreviate
\begin{equation}\label{eq:initial-beta-box}
    \{1,k]_\beta
    :=
    [1(i_k)^+,k]_\beta.
\end{equation}
Under the identification of
\cite{kashiwara2025monoidal,casals2025cluster}, the cluster variable
attached to \(v_k\) is
\begin{equation}\label{eq:right-weave-variable-initial-minor}
    D_k^\beta
    :=
    D^\beta[1(i_k)^+,k]
    =
    \widetilde G\bigl(\beta,\{1,k]_\beta\bigr).
\end{equation}
Consequently, the right inductive weave seed
\(\overrightarrow{\mathbf s}
(\boldsymbol{\Delta}\beta)\) is isomorphic to the word seed
\(\mathbf s(\beta)\) under the explicit bijection
\(v_k\leftrightarrow k\) for \(1\leq k\leq r\). In particular,
both seeds have exactly \(r\) variables; the statement is a seed
isomorphism, not a literal identification of their vertex sets.
We shall also use the following identification.

\begin{lemma}[{\cite[Theorem~3.12]{bi2026towards}}]
\label{lem:right-inductive-seed-restriction}
For every positive braid word \(\beta\), one has
\begin{equation}\label{eq:right-inductive-seed-identification}
    \mathbf s(\delta(\beta),\beta)
    =
    \overrightarrow{\mathbf s}(\beta),
\end{equation}
where \(\mathbf s(\delta(\beta),\beta)\) is the seed introduced in
Definition~\ref{def:svbeta}, and
\(\overrightarrow{\mathbf s}(\beta)\) is the seed associated with
the right inductive weave
\(\overrightarrow{\mathfrak W}(\beta)\).
The equality in \eqref{eq:right-inductive-seed-identification} is
under the canonical bijection between \(V_\beta\) and the trivalent
vertices of this weave and includes the frozen set, exchange matrix,
and, in the quantum version, the compatible form.
\end{lemma}

\subsection{Global basis elements for cluster variables}
The compatible quantum seed corresponding to
\eqref{eq:right-inductive-seed-identification} will be denoted by
\(
    \overrightarrow{\mathbf s}_t(\beta)
    :=
    \mathbf s_t(\delta(\beta),\beta).
\)
 By
Corollary~\ref{cor:cluster-algebra-embedding-A-beta}, every quantum
cluster variable of
\(\overrightarrow{\mathbf s}_t(\beta)\) is identified with a
normalized global basis element of
\(
    \widetilde{\mathbf A}(\beta)
    =
    \widehat{\cA}_{\ZZ[q^{\pm1/2}]}(\beta)
    \cap
    T_{\delta(\beta)}\widehat{\cA}_{\geq0}.
\)
For \(v\in\overrightarrow{\mathfrak W}(\beta)_3\), let \(X_v\)
denote the corresponding quantum cluster variable, and write
\begin{equation}\label{eq:Av-global-basis-element}
    A_v
    :=
    (\Phi_{\DD}\circ\iota)(X_v)
    \in\widetilde{\mathbf A}(\beta).
\end{equation}
Thus \(A_v\) is a normalized global basis element of
\(\widehat{\cA}(\beta)\). We denote its \(\beta\)-Lusztig parameter by
\begin{equation}\label{eq:Lusztig-parameter-Av}
    \mathbf a^\beta(A_v)
    \in\ZZ_{\geq0}^r.
\end{equation}

\begin{lemma}[Mutation of Lusztig parameters]
\label{lem:Lusztig-parameter-mutation}
Let \(\mathfrak W\) be a weave with exchange matrix
\(
    B(\mathfrak W)=(b_{u,w}),
\)
and let \(A'_u\) be the normalized global basis element corresponding
to the cluster variable at \(u\). For an exchangeable vertex
\(v\in\mathfrak W_3\), set
\[
    A''_v:=\mu_v(A'_v),
    \qquad
    A''_u:=A'_u\quad (u\neq v).
\]
Then
\begin{equation}\label{eq:Lusztig-parameter-mutated-variable}
\begin{aligned}
    \mathbf a^\beta(A''_v)
    &=
    \max_{\mathrm r}
    \left\{
        \sum_{w\in\mathfrak W_3}
            [b_{w,v}]_+
            \mathbf a^\beta(A'_w),
        \sum_{w\in\mathfrak W_3}
            [-b_{w,v}]_+
            \mathbf a^\beta(A'_w)
    \right\}
    -\mathbf a^\beta(A'_v).
\end{aligned}
\end{equation}
Here \(\max_{\mathrm r}\) is taken with respect to the right
lexicographic order in Definition \ref{def:order-vector}. Thus the family
\(\{\mathbf a^\beta(A''_u)\}_u\) is obtained from
\(\{\mathbf a^\beta(A'_u)\}_u\) by the vector-mutation rule of
Definition~\ref{def:mutlusztig}.
\end{lemma}

\begin{proof}
After fixing an order of the factors, the quantum exchange relation at
\(v\) can be written as
\begin{equation}\label{eq:exchange-relation-global-basis-elements}
    q^{c_0}A'_vA''_v
    =
    q^{c_+}
    \prod_{w\in\mathfrak W_3}
        (A'_w)^{[b_{w,v}]_+}
    +
    q^{c_-}
    \prod_{w\in\mathfrak W_3}
        (A'_w)^{[-b_{w,v}]_+},
\end{equation}
where \(c_0,c_+,c_-\in\frac12\ZZ\). Changing the order of the
factors only changes these powers of \(q^{1/2}\).

Let \(\operatorname{lp}_\beta\) denote the largest
\(\beta\)-Lusztig parameter with respect to the right lexicographic
order. By
Lemma~\ref{lem:global-basis-product-triangularity},
\[
    \operatorname{lp}_\beta(A'_vA''_v)
    =
    \mathbf a^\beta(A'_v)
    +
    \mathbf a^\beta(A''_v).
\]
The leading parameters of the two exchange monomials on the
right-hand side of
\eqref{eq:exchange-relation-global-basis-elements} are
\[
    \mathbf p_v^+
    :=
    \sum_{w\in\mathfrak W_3}
        [b_{w,v}]_+\,
        \mathbf a^\beta(A'_w),
    \qquad
    \mathbf p_v^-
    :=
    \sum_{w\in\mathfrak W_3}
        [-b_{w,v}]_+\,
        \mathbf a^\beta(A'_w).
\]

We claim that
\[
    \mathbf p_v^+\neq\mathbf p_v^-.
\]
Indeed, suppose that
\(\mathbf p_v^+=\mathbf p_v^-=\mathbf p\). Then the two exchange
monomials have leading terms
\[
    q^{d_+}\widetilde G(\beta,\mathbf p)
    \qquad\text{and}\qquad
    q^{d_-}\widetilde G(\beta,\mathbf p),
\]
respectively, for some \(d_+,d_-\in\frac12\ZZ\). Consequently, the
coefficient of \(\widetilde G(\beta,\mathbf p)\) on the right-hand
side of
\eqref{eq:exchange-relation-global-basis-elements} would be
\(
    q^{c_++d_+}+q^{c_-+d_-}.
\)
The two contributions cannot cancel. On the other hand, Lemma~\ref{lem:global-basis-product-triangularity}
shows that the leading coefficient of the left-hand side is a single
power of \(q^{1/2}\). This is impossible in
\(\ZZ[q^{\pm1/2}]\). Hence
\(\mathbf p_v^+\neq\mathbf p_v^-\).

Therefore,
\[
    \mathbf a^\beta(A'_v)
    +
    \mathbf a^\beta(A''_v)
    =
    \max_{\mathrm r}
    \left\{
        \mathbf p_v^+,\mathbf p_v^-
    \right\}.
\]
Solving for \(\mathbf a^\beta(A''_v)\) gives
\eqref{eq:Lusztig-parameter-mutated-variable}. Since mutation leaves
\(A'_u\) unchanged for \(u\neq v\), all other Lusztig parameters are
unchanged as well.
\end{proof}

\begin{proposition}
\label{pro:Lusztig-parameter-right-weave-variable}
For every
\(v\in\overrightarrow{\mathfrak W}(\beta)_3\), the
\(\beta\)-Lusztig parameter of \(A_v\) satisfies
\begin{equation}
\label{eq:Lusztig-parameter-right-weave-recursion}
    \mathbf a^\beta(A_v)
    =
    E_{q_v^\beta}
    +
    \sum_{\substack{
        v'\in\overrightarrow{\mathfrak W}(\beta)_3\\
        v'<v
    }}
    \gamma_{v'}(e_{v,\mathrm{nw}})
    \mathbf a^\beta(A_{v'}).
\end{equation}
Consequently,
\[
    \mathbf a^\beta(A_v)=\eta_v^\beta.
\]
\end{proposition}

\begin{proof}
Let
\[
    \boldsymbol\eta^{\boldsymbol{\Delta}\beta}
    :=
    \left(
        \eta_{v_k}^{\boldsymbol{\Delta}\beta}
    \right)_{k\in[r]}
\]
be the tuple associated with the \(\beta\)-block of
\(\overrightarrow{\mathfrak W}
(\boldsymbol{\Delta}\beta)\).
Lemma~\ref{lem:eta-Delta-beta-predecessor} and
\eqref{eq:right-weave-variable-initial-minor} give
\[
    \eta_{v_k}^{\boldsymbol{\Delta}\beta}
    =
    \{1,k]_\beta
    =
    \mathbf a^\beta(D_k^\beta)
    \qquad(k\in[r]).
\]
Thus, in the initial seed, the vectors
\(\eta_{v_k}^{\boldsymbol{\Delta}\beta}\) coincide with the Lusztig
parameters of the corresponding normalized global basis elements.

Apply the mutation sequence \(T\) from
\eqref{eq:mutation-sequence-T} simultaneously to the seed and to the
tuple \(\boldsymbol\eta^{\boldsymbol{\Delta}\beta}\), and set
\[
    \boldsymbol\eta'
    :=
    T\left(
        \boldsymbol\eta^{\boldsymbol{\Delta}\beta}
    \right).
\]
By Lemma~\ref{lem:Lusztig-parameter-mutation}, Lusztig parameters
transform according to the same vector-mutation rule as the
\(\eta\)-vectors. Induction along \(T\) therefore shows that, at every
stage, the vector attached to a cluster variable is its
\(\beta\)-Lusztig parameter.

After completing \(T\), delete the vertices specified in
Definition~\ref{def:svbeta}. By
Lemma~\ref{lem:right-inductive-seed-restriction}, the resulting seed
is \(\overrightarrow{\mathbf s}_t(\beta)\), and the variables at the
surviving vertices are unchanged. Hence
\[
    \eta'_v
    =
    \mathbf a^\beta(A_v)
    \qquad
    \left(
        v\in\overrightarrow{\mathfrak W}(\beta)_3
    \right).
\]

Finally, Theorem~\ref{thm:eta-after-Delta-mutation} gives
\[
    \eta'_v
    =
    E_{q_v^\beta}
    +
    \sum_{\substack{
        v'\in\overrightarrow{\mathfrak W}(\beta)_3\\
        v'<v
    }}
    \gamma_{v'}(e_{v,\mathrm{nw}})\eta'_{v'}.
\]
Substituting
\(\eta'_v=\mathbf a^\beta(A_v)\) proves
\eqref{eq:Lusztig-parameter-right-weave-recursion}. This is precisely
the recursion defining \(\eta_v^\beta\); therefore
\(\mathbf a^\beta(A_v)=\eta_v^\beta\).
\end{proof}

Let \(\ddot{\mathbf s}=(i_1X_1,\ldots,i_rX_r)\) be a double string,
and put \(\beta:=\beta_r^{\ddot{\mathbf s}}\). Its quantum seed is
mutation equivalent to \(\overrightarrow{\mathbf s}_t(\beta)\).
Through the fixed embedding
\eqref{eq:cluster-algebra-embedding-A-beta}, let
\(A_v^{\ddot{\mathbf s}}\) denote the normalized global basis element
corresponding to the cluster variable at
\(v\in\mathfrak W(\ddot{\mathbf s})_3\).

\begin{definition}\label{def:vector-double-inductive}
The vector attached to \(v\) is defined intrinsically by
\begin{equation}\label{eq:intrinsic-double-string-vector}
    \eta_v^{\ddot{\mathbf s}}
    :=\mathbf a^\beta\bigl(A_v^{\ddot{\mathbf s}}\bigr)
    \in\ZZ_{\geq0}^r.
\end{equation}
The notation \(\eta_v^\beta\) continues to denote the combinatorial
right-inductive vector of
Definition~\ref{def:eta-vector-right-inductive}. Proposition
\ref{pro:Lusztig-parameter-right-weave-variable} proves that it
agrees with \eqref{eq:intrinsic-double-string-vector} for a right
inductive string.
\end{definition}

We shall also use compatibility of global bases with consecutive
subwords. If \(\beta=\gamma\alpha\kappa\), let
\(\iota_{\gamma,\alpha,\kappa}\colon
\ZZ^{\ell(\alpha)}\to\ZZ^{\ell(\beta)}\) insert zero coordinates in
the \(\gamma\)- and \(\kappa\)-blocks.

\begin{lemma}[Global basis under consecutive-block embeddings]
\label{lem:global-basis-consecutive-block}
Let
\(
    \beta=\gamma\alpha\kappa
\)
be a decomposition into positive braid words, and let
\[
    T_\gamma\colon
    \widehat{\cA}(\alpha)
    \longrightarrow
    \widehat{\cA}(\beta)
\]
be the braid-symmetry embedding associated with \(\gamma\). Then, for every
\(\mathbf a\in\ZZ_{\geq0}^{|\alpha|}\),
\begin{equation}\label{eq:global-basis-consecutive-block}
    T_\gamma\bigl(\widetilde G(\alpha,\mathbf a)\bigr)
    =
    \widetilde G\bigl(
        \beta,
        \iota_{\gamma,\alpha,\kappa}(\mathbf a)
    \bigr).
\end{equation}
\end{lemma}
\begin{proof}
The braid symmetry \(T_\gamma\) identifies the PBW root vectors of
\(\alpha\) with those in the \(\alpha\)-block of \(\beta\). Hence
\[
    T_\gamma\bigl(P^\alpha(\mathbf a)\bigr)
    =
    P^\beta\bigl(
        \iota_{\gamma,\alpha,\kappa}(\mathbf a)
    \bigr).
\]
The \(c\)-invariant triangular characterization in
Proposition~\ref{prop:PBW-global-basis}, together with the global-basis
compatibility of the braid symmetries, identifies the corresponding
unnormalized elements \(G\). Since \(T_\gamma\) preserves the weight
form, applying the conversion
\eqref{eq:normalized-PBW-global-basis-conversion} gives
\eqref{eq:global-basis-consecutive-block} for \(\widetilde G\).
\end{proof}

\begin{lemma}[Extension of mutations from a truncated weave]
\label{lem:quantum-truncated-weave-extension}
Let \(\mathfrak U\) be a truncated double inductive weave, with
vertex set \(J\), and obtain \(\mathfrak W\) by restoring the
remaining insertions. Denote its vertex set by \(K\supseteq J\)
and the mutable vertices of \(\mathfrak U\) by \(J_{\mathrm{ex}}\).
Then
\begin{equation}\label{eq:truncated-weave-zero-columns}
 b^{\mathfrak W}_{uv}=
 \begin{cases}
 b^{\mathfrak U}_{uv},&u\in J,\\
 0,&u\in K\setminus J,
 \end{cases}
 \qquad v\in J_{\mathrm{ex}}.
\end{equation}
Restrict the full based toric frame to \(\ZZ^J\), and equip it
with the induced form \(\Lambda_{J\times J}\). This form is
compatible with \(B(\mathfrak U)\), and every sequence of
mutations at vertices mutable in the truncated weave is the
restriction of the same quantum mutations in the full weave.
Vertices frozen in the truncated weave retain their variables
throughout this sequence, even if they are mutable in the full weave.
\end{lemma}

\begin{proof}
By \cite[Definition~4.35]{casals2025cluster}, a mutable Lusztig
cycle of \(\mathfrak U\) has zero southern output. Its extension
vanishes below the truncation slice. In particular, every
vertex of \(J_{\mathrm{ex}}\) remains mutable in \(\mathfrak W\),
so \(J_{\mathrm{ex}}\subseteq K_{\mathrm{ex}}\).

Conversely, a cycle born
in the restored tail is zero above that slice. Hence restoring the
tail adds no local intersection involving the former cycle, and
its boundary intersection is zero both before and after extension.
The exchange-matrix formula \eqref{eq:weave-exchange-entry} gives
\eqref{eq:truncated-weave-zero-columns}.

For \(v\in J_{\mathrm{ex}}\) and \(u\in J\), compatibility of
the full seed now gives
\[
 \sum_{z\in J}b^{\mathfrak U}_{zv}\lambda_{zu}
 =\sum_{z\in K}b^{\mathfrak W}_{zv}\lambda_{zu}
 =2\delta_{uv}.
\]
The two exchange monomials in
\eqref{eq:quantum-cluster-mutation} have support in \(J\), so
their normalization uses only \(\Lambda_{J\times J}\).
Matrix mutation preserves the zero block, as in
Lemma~\ref{lem:subseed}, and mutation of the compatible form
commutes with restriction for the same reason. Induction proves
the assertion for the sequence. Equivalences only relabel the
corresponding frames. No vertex frozen in \(\mathfrak U\) is
mutated \cite[Lemmas~4.36--4.37]{casals2025cluster}.
\end{proof}

The induced form in this lemma is the restriction of the chosen
full toric frame; no identification with a separately chosen
quantization of the truncated weave is assumed. Whenever the frames
are identified by the based embedding of
Lemma~\ref{lem:global-basis-consecutive-block}, that embedding also
identifies these forms, since it preserves products.

\begin{theorem}\label{thm:vector-double-inductive}
Let
\(
    \ddot{\mathbf s}
    =(i_1X_1,\ldots,i_rX_r)
\)
be a double string with final word
\(\beta=\beta_r^{\ddot{\mathbf s}}\), and let \(v\) be the
trivalent vertex created by a solid step \(i_kX_k\). Set
\(
    \beta_k
    :=\beta_k^{\ddot{\mathbf s}}
    =(j_1,\ldots,j_k),
\)
and write
\[
    \beta=\gamma_k\beta_k\kappa_k
\]
for the consecutive-block decomposition determined by
\(\ddot{\mathbf s}\). Let
\[
    \iota_k^{\ddot{\mathbf s}}
    :=
    \iota_{\gamma_k,\beta_k,\kappa_k}
    \colon\ZZ^k\hookrightarrow\ZZ^r
\]
be the corresponding zero-extension map; explicitly,
\[
    \iota_k^{\ddot{\mathbf s}}(E_a)
    =
    E_{|\gamma_k|+a}
    \qquad(1\leq a\leq k).
\]
Then there exists
\(
    v'\in
    \overrightarrow{\mathfrak W}(\beta_k)_3
\)
such that
\begin{equation}\label{eq:eta-double-string}
    \eta_v^{\ddot{\mathbf s}}
    =
    \iota_k^{\ddot{\mathbf s}}
    \bigl(\eta_{v'}^{\beta_k}\bigr).
\end{equation}
\end{theorem}

\begin{proof}
Consider the truncated string
\[
    \ddot{\mathbf s}_{\leq k}
    :=
    (i_1X_1,\ldots,i_kX_k).
\]
Since \(i_kX_k\) is its final solid step, \(v\) is frozen in
\(\mathfrak W(\ddot{\mathbf s}_{\leq k})\). By
\cite[Lemma~4.4]{casals2025cluster}, this weave is related to
\(\overrightarrow{\mathfrak W}(\beta_k)\) by weave mutations and
equivalences. Frozen vertices are preserved by equivalences and are
never mutated
\cite[Lemmas~4.36 and~4.37(1)]{casals2025cluster}. Hence \(v\)
is transported to a frozen vertex
\[
    v'\in\overrightarrow{\mathfrak W}(\beta_k)_3.
\]

Restore the remaining insertions and define
\[
    \ddot{\mathbf s}'
    :=
    (j_1\mathrm R,\ldots,j_k\mathrm R,
      i_{k+1}X_{k+1},\ldots,i_rX_r).
\]
Lemma~\ref{lem:quantum-truncated-weave-extension} extends the
mutations of the truncated weave to the full weave. Since \(v\)
and \(v'\) are not mutated, their cluster variables are identified:
\begin{equation}\label{eq:transport-v-to-v-prime}
    A_v^{\ddot{\mathbf s}}
    =
    A_{v'}^{\ddot{\mathbf s}'}.
\end{equation}

It remains to compute the parameters of the vertices belonging to
the initial \(\beta_k\)-block. Write
\[
    \gamma_k=(g_1,\ldots,g_h),
    \qquad h:=|\gamma_k|,
\]
and, for \(0\leq t\leq h\), consider
\[
    \mathbf r_t
    :=
    \bigl(
       g_{t+1}\mathrm R,\ldots,g_h\mathrm R,
       j_1\mathrm R,\ldots,j_k\mathrm R,
       \kappa_k\mathrm R,
       g_t\mathrm L,\ldots,g_1\mathrm L
    \bigr),
\]
where \(\kappa_k\mathrm R\) means that all letters of
\(\kappa_k\) are inserted successively on the right. Every
\(\mathbf r_t\) has final word \(\beta\), and
\(\mathbf r_0\) is the right inductive string of \(\beta\).

Starting with \(\mathbf r_0\), move
\(g_1,\ldots,g_h\) successively from the initial right block to
later left insertions. At each step, apply
Corollary~\ref{cor:etavn} and then
Lemma~\ref{lem:quantum-truncated-weave-extension}. The exceptional
vertex in Corollary~\ref{cor:etavn} is the newly moved left
insertion, so it does not belong to the retained \(\beta_k\)-block.
Consequently, for every
\(w\in\overrightarrow{\mathfrak W}(\beta_k)_3\),
\begin{equation}\label{eq:Lusztig-parameter-Arh}
\begin{aligned}
    \mathbf a^\beta\bigl(A_w^{\mathbf r_h}\bigr)
    =
    E_{h+q_w^{\beta_k}}\quad+
    \sum_{\substack{
        u\in\overrightarrow{\mathfrak W}(\beta_k)_3\\
        u<w
    }}
    \gamma_u^{\overrightarrow{\mathfrak W}(\beta_k)}
        (e_{w,\mathrm{nw}})
    \mathbf a^\beta\bigl(A_u^{\mathbf r_h}\bigr).
\end{aligned}
\end{equation}
Here \(h+q_w^{\beta_k}\) is the position in \(\beta\) corresponding
to the position \(q_w^{\beta_k}\) in the consecutive
\(\beta_k\)-block. The transformations from \(\mathbf r_h\) to
\(\ddot{\mathbf s}'\) are supported in the tail
\(
    \bigl(
       \kappa_k\mathrm R,
       g_h\mathrm L,\ldots,g_1\mathrm L
    \bigr).
\)
They do not mutate the variables belonging to the initial
\(\beta_k\)-block. Therefore
\[
    A_w^{\mathbf r_h}
    =
    A_w^{\ddot{\mathbf s}'}
    \qquad
    \bigl(
      w\in\overrightarrow{\mathfrak W}(\beta_k)_3
    \bigr).
\]
Thus \eqref{eq:Lusztig-parameter-Arh} becomes
\begin{equation}\label{eq:double-string-block-recursion}
\begin{aligned}
    \mathbf a^\beta\bigl(A_w^{\ddot{\mathbf s}'}\bigr)
    =
    E_{h+q_w^{\beta_k}} 
    \quad+
    \sum_{u<w}
    \gamma_u^{\overrightarrow{\mathfrak W}(\beta_k)}
        (e_{w,\mathrm{nw}})
    \mathbf a^\beta\bigl(A_u^{\ddot{\mathbf s}'}\bigr).
\end{aligned}
\end{equation}

On the other hand, applying
\(\iota_k^{\ddot{\mathbf s}}\) to the defining recursion
\eqref{eq:etav} for \(\eta_w^{\beta_k}\) gives
\[
\begin{aligned}
    \iota_k^{\ddot{\mathbf s}}
        \bigl(\eta_w^{\beta_k}\bigr)
    =
    E_{h+q_w^{\beta_k}} 
    \quad+
    \sum_{u<w}
    \gamma_u^{\overrightarrow{\mathfrak W}(\beta_k)}
        (e_{w,\mathrm{nw}})
    \iota_k^{\ddot{\mathbf s}}
        \bigl(\eta_u^{\beta_k}\bigr).
\end{aligned}
\]
Comparing this recursion with
\eqref{eq:double-string-block-recursion} and inducting on the
ordered vertices \(w\) yields
\[
    \mathbf a^\beta
        \bigl(A_w^{\ddot{\mathbf s}'}\bigr)
    =
    \iota_k^{\ddot{\mathbf s}}
        \bigl(\eta_w^{\beta_k}\bigr).
\]
Taking \(w=v'\), using
\eqref{eq:transport-v-to-v-prime}, and recalling
Definition~\ref{def:vector-double-inductive}, we obtain equation~\eqref{eq:eta-double-string}.
This proves the theorem.
\end{proof}
\begin{corollary}[Frozen vertices at an arbitrary depth]
\label{cor:frozen-block-transport}
Fix \(1\leq c\leq r\), and write
\(\beta=\gamma_c\beta_c^{\ddot{\mathbf s}}\kappa_c\).
A weave transformation of the truncated string to the right
inductive weave induces a bijection
\[
 \tau_c\colon
 \operatorname{Fr}\mathfrak W(\ddot{\mathbf s}_{\leq c})
 \longrightarrow
 \operatorname{Fr}\overrightarrow{\mathfrak W}
     (\beta_c^{\ddot{\mathbf s}}).
\]
For every frozen trivalent vertex \(v\) of the truncated weave,
\begin{equation}\label{eq:frozen-block-transport}
\begin{aligned}
 A_v^{\ddot{\mathbf s}}
 =T_{\gamma_c}
     \bigl(A_{\tau_c(v)}^{\beta_c^{\ddot{\mathbf s}}}\bigr),\qquad
 \eta_v^{\ddot{\mathbf s}}
 =\iota_c^{\ddot{\mathbf s}}
     \bigl(\eta_{\tau_c(v)}^{\beta_c^{\ddot{\mathbf s}}}\bigr).
\end{aligned}
\end{equation}
The corresponding Lusztig cycles have the same southern output,
and hence the same coroot \(\gamma_{v,c}\).
\end{corollary}

\begin{proof}
Set
\(
    \alpha:=\beta_c^{\ddot{\mathbf s}}.
\)
By Lemma~\ref{lem:output}, There is a bijection
\[
    \tau_c\colon
    \operatorname{Fr}\mathfrak W(\ddot{\mathbf s}_{\leq c})
    \longrightarrow
    \operatorname{Fr}\overrightarrow{\mathfrak W}(\alpha).
\]

Write \(\alpha=(j_1,\ldots,j_c)\), and set
\[
    \ddot{\mathbf s}^{[c]}
    :=
    (j_1\mathrm R,\ldots,j_c\mathrm R,
      i_{c+1}X_{c+1},\ldots,i_rX_r).
\]
Restore the remaining insertions at every stage of the preceding
transformation. Lemma~\ref{lem:quantum-truncated-weave-extension}
extends all mutations to the corresponding full quantum seeds,
while equivalences merely relabel the based toric frames. Since the
vertices under consideration are frozen in the truncated weave,
they are never mutated. Therefore
\begin{equation}\label{eq:frozen-variable-depth-c-transport}
    A_v^{\ddot{\mathbf s}}
    =
    A_{\tau_c(v)}^{\ddot{\mathbf s}^{[c]}}.
\end{equation}

The block-recursion argument in the proof of
Theorem~\ref{thm:vector-double-inductive}, applied with \(k=c\),
actually proves that, for every
\(w\in\overrightarrow{\mathfrak W}(\alpha)_3\),
\(    \mathbf a^\beta
       \bigl(A_w^{\ddot{\mathbf s}^{[c]}}\bigr)
    =
    \iota_c^{\ddot{\mathbf s}}
       \bigl(\eta_w^\alpha\bigr).
\)
Taking \(w=\tau_c(v)\), using
\eqref{eq:frozen-variable-depth-c-transport}, and recalling
Definition~\ref{def:vector-double-inductive}, we obtain
\[
    \eta_v^{\ddot{\mathbf s}}
    =
    \iota_c^{\ddot{\mathbf s}}
       \bigl(\eta_{\tau_c(v)}^\alpha\bigr).
\]

Hence Lemma~\ref{lem:global-basis-consecutive-block} gives
\[
\begin{aligned}
    T_{\gamma_c}
       \bigl(A_{\tau_c(v)}^\alpha\bigr)
    =
    \widetilde G\left(
       \beta,
       \iota_c^{\ddot{\mathbf s}}
          \bigl(\eta_{\tau_c(v)}^\alpha\bigr)
    \right)
    =
    \widetilde G
       \bigl(\beta,\eta_v^{\ddot{\mathbf s}}\bigr)
     =
    A_v^{\ddot{\mathbf s}}.
\end{aligned}
\]
This proves both identities in
\eqref{eq:frozen-block-transport}.

Finally, Lemma~\ref{lem:output} yields
\(
 \operatorname{out}_{\mathfrak W(\ddot{\mathbf s}_{\leq c})}
     (\gamma_v)
 =
 \operatorname{out}_{\overrightarrow{\mathfrak W}(\alpha)}
     (\gamma_{\tau_c(v)}).
\)
Because the same southern reduced word of \(\delta(\alpha)\) is used
on both sides, the corresponding inversion coroots
\(\chi_a^{(c)}\) are identical. Thus, by
\eqref{eq:gamma-vc},
\[
    \gamma_{v,c}
    =
    \gamma_{\tau_c(v),c}.
\]
This proves the final assertion.
\end{proof}

Consequently, every global basis element arising as a cluster
variable in a double-string seed is characterized by a vector
associated with a right inductive weave. In particular,
Theorem~\ref{thm:vector-double-inductive} determines the
\(\beta\)-Lusztig parameters of all such elements.
\subsection{Quantum grid minors}
We first construct the specialization used below. Put
\(
    R:=\ZZ[q^{\pm1/2}],
\)
identify \(t^{1/2}=q^{-1/2}\), and regard \(\CC\) as an
\(R\)-algebra via \(q^{1/2}\mapsto1\). Let
\(\mathcal T_q\) be the based quantum torus of the right inductive
quantum seed, and let
\(
    \mathcal T_1
    :=
    \CC[x_v^{\pm1}\mid v\in V_\beta]
\)
be the Laurent torus of the corresponding classical seed. Formula
\eqref{eq:quantum-monomial-product} shows that
\begin{equation}\label{eq:initial-torus-specialization}
    \operatorname{sp}_{\mathcal T}\colon
    \mathcal T_q\otimes_R\CC
    \longrightarrow
    \mathcal T_1,
    \qquad
    \widetilde X^{\mathbf g}\otimes1
    \longmapsto
    x^{\mathbf g},
\end{equation}
is a well-defined \(\CC\)-algebra homomorphism.

\begin{lemma}
\label{lem:specialization-cluster-variables}
Let \(R=\ZZ[q^{\pm1/2}]\), and regard \(\CC\) as an \(R\)-algebra
via \(q^{1/2}\mapsto1\). Specialization at \(q^{1/2}=1\) induces a
natural \(\CC\)-algebra isomorphism
\begin{equation}\label{eq:cluster-specialization-map}
    \operatorname{sp}^{\mathrm{cl}}_\beta\colon
    \mathcal A_t\bigl(
        \overrightarrow{\mathbf s}_t(\beta)
    \bigr)\otimes_R\CC
    \xrightarrow{\sim}
    \mathcal A_{\CC}\bigl(
        \overrightarrow{\mathbf s}(\beta)
    \bigr)
    \xrightarrow{\sim}
    \CC[X(\beta)].
\end{equation}
More precisely, if \(\Sigma\) is any seed mutation equivalent to
\(\overrightarrow{\mathbf s}_t(\beta)\), then
\begin{equation}\label{eq:specialization-arbitrary-cluster-variable}
    \operatorname{sp}^{\mathrm{cl}}_\beta
    \bigl(X_{v;\Sigma}\otimes1\bigr)
    =
    x_{v;\Sigma}
\end{equation}
for every vertex \(v\) of \(\Sigma\), where \(X_{v;\Sigma}\) and
\(x_{v;\Sigma}\) are the corresponding quantum and classical cluster
variables. Under the bosonic realization, the global-basis element
representing \(X_{v;\Sigma}\) therefore specializes to
\(x_{v;\Sigma}\).
\end{lemma}

\begin{proof}
The specialization homomorphism
\(\operatorname{sp}_{\mathcal T}\) on the initial quantum torus sends
each normalized quantum monomial to the corresponding classical
Laurent monomial. Under this homomorphism, the two normalized
monomials in the quantum exchange relation specialize to the two
monomials in the classical exchange relation. Hence specialization
commutes with mutation.

Induction along an arbitrary mutation sequence therefore gives
\(    \operatorname{sp}_{\mathcal T}
    \bigl(X_{v;\Sigma}\otimes1\bigr)
    =
    x_{v;\Sigma}.
\)
Consequently, \(\operatorname{sp}_{\mathcal T}\) induces a
\(\CC\)-algebra homomorphism
\[
    \operatorname{sp}^{\mathrm{cl}}_\beta\colon
    \mathcal A_t\bigl(
        \overrightarrow{\mathbf s}_t(\beta)
    \bigr)\otimes_R\CC
    \longrightarrow
    \mathcal A_{\CC}\bigl(
        \overrightarrow{\mathbf s}(\beta)
    \bigr).
\]
Its image contains every classical cluster variable and every inverse
of a frozen variable occurring in the localization. It is therefore
surjective.

By Theorem~\ref{thm:quantum-A-U-braid-seeds}, the quantum cluster
algebra admits a common triangular basis \(\mathbf L\). In particular,
\[
    \mathcal A_t\bigl(
        \overrightarrow{\mathbf s}_t(\beta)
    \bigr)
    =
    \bigoplus_{L\in\mathbf L}RL,
\]
so the elements \(L\otimes1\), \(L\in\mathbf L\), form a
\(\CC\)-basis after specialization.

For each \(L\in\mathbf L\), its Laurent expansion in the initial
quantum torus is pointed:
\[
    L
    =
    \widetilde X^{\mathbf g(L)}
    +
    \sum_{\mathbf h\prec_{\mathbf t}\mathbf g(L)}
        c_{L,\mathbf h}(q^{1/2})
        \widetilde X^{\mathbf h},
\]
where the sum is finite, the degrees \(\mathbf g(L)\) are pairwise
distinct, and \(\prec_{\mathbf t}\) is the dominance order of the
initial seed defined in \cite[Definition~2.3]{qin2024analogs}.
Therefore
\[
    \operatorname{sp}^{\mathrm{cl}}_\beta(L)
    =
    x^{\mathbf g(L)}
    +
    \sum_{\mathbf h\prec_{\mathbf t}\mathbf g(L)}
        c_{L,\mathbf h}(1)x^{\mathbf h}.
\]

Suppose that
\[
    \sum_{L\in F}
        a_L\operatorname{sp}^{\mathrm{cl}}_\beta(L)
    =
    0
\]
for a finite subset \(F\subset\mathbf L\) and coefficients
\(a_L\in\CC\). Choose \(L_0\in F\) with \(a_{L_0}\neq0\) such that
\(\mathbf g(L_0)\) is maximal among the degrees
\(\mathbf g(L)\) with \(a_L\neq0\). The monomial
\(x^{\mathbf g(L_0)}\) cannot be the leading term of another summand,
because the pointed degrees are distinct. It cannot occur as a lower
term of another summand either: otherwise
\[
    \mathbf g(L_0)\prec_{\mathbf t}\mathbf g(L)
\]
for some \(L\in F\) with \(a_L\neq0\), contradicting the maximality
of \(\mathbf g(L_0)\). Hence the coefficient of
\(x^{\mathbf g(L_0)}\) in the relation is \(a_{L_0}\), a
contradiction.

Thus the specialized elements
\[
    \left\{
        \operatorname{sp}^{\mathrm{cl}}_\beta(L)
        \,\middle|\,
        L\in\mathbf L
    \right\}
\]
are linearly independent. Therefore
\(\operatorname{sp}^{\mathrm{cl}}_\beta\) is injective and hence an
isomorphism. Composing it with the classical braid-variety cluster
isomorphism gives \eqref{eq:cluster-specialization-map}.
\end{proof}
We are now ready to define quantum grid minors. Let
\(\ddot{\mathbf s}\) be a double string with associated positive braid
word
\(
    \beta=\beta_r^{\ddot{\mathbf s}}.
\)
The elements
\(
    \left\{
        A_v^{\ddot{\mathbf s}}
        \,\middle|\,
        v\in\mathfrak W(\ddot{\mathbf s})_3
    \right\}
\)
belong to a common quantum cluster and hence pairwise
\(q\)-commute. Consequently, for every
\(
    \mathbf d=(d_v)_
    {v\in\mathfrak W(\ddot{\mathbf s})_3}
    \in
    \ZZ_{\geq0}^{\mathfrak W(\ddot{\mathbf s})_3},
\)
their product has a unique normalization which is a normalized
global-basis element. We denote this element by
\begin{equation}\label{eq:normalized-product-global-basis}
    \bigodot_{
        v\in\mathfrak W(\ddot{\mathbf s})_3
    }
    \left(A_v^{\ddot{\mathbf s}}\right)^{\odot d_v}
    :=
    \widetilde G\left(
        \beta,
        \sum_{
            v\in\mathfrak W(\ddot{\mathbf s})_3
        }
        d_v\,
        \mathbf a^\beta
        \left(A_v^{\ddot{\mathbf s}}\right)
    \right).
\end{equation}

\begin{definition}
\label{def:quantum-grid-minors}
For every \(j\in I\), set \(D_{0,j}^{\beta,\ddot{\mathbf s}}:=1\).
For every depth \(1\leq c\leq r\), define the \emph{quantum grid
minor} by
\begin{equation}
 D_{c,j}^{\beta,\ddot{\mathbf s}}
 :=
 \bigodot_{v\in\mathfrak W(\ddot{\mathbf s})_3}
 A_v^{\odot\langle\varpi_j,\gamma_{v,c}\rangle}
 \in\widetilde{\mathbf A}(\beta).
 \label{eq:def-quantum-grid-minor}
\end{equation}
The exponent vector is nonnegative by Lemma~\ref{lem:grid-minors}, so
this is a quantum cluster monomial. Under the specialization
\eqref{eq:cluster-specialization-map}, the toric frame specializes to
the classical cluster torus and Lemma~\ref{lem:grid-minors} gives
\(D_{c,j}^{\beta,\ddot{\mathbf s}}\mapsto
\Delta_{c,j}^{\ddot{\mathbf s}}\).
When the ambient braid word is clear, we simply write
\(D_{c,j}^{\ddot{\mathbf s}}\).
\end{definition}

Example~\ref{ex:delta-braid-i-box} will identify the shifted
determinantial element \(T_{w_0}(D^\beta[a,b])\) with a quantum
grid minor. We record the equality explicitly after proving the
parameter formula below.
\begin{lemma}[Quantum transport of grid minors]
\label{lem:quantum-grid-minor-transport}
Let
\(
    \beta
    =
    \gamma_c\beta_c^{\ddot{\mathbf s}}\eta_c
\)
be the block decomposition in
\eqref{eq:beta-block-decomposition}. Under the based embedding of
Lemma~\ref{lem:global-basis-consecutive-block}, one has
\begin{equation}\label{eq:quantum-grid-minors}
    D_{c,j}^{\beta,\ddot{\mathbf s}}
    =
    T_{\gamma_c}\left(
        \bigodot_{
            v\in
            \overrightarrow{\mathfrak W}
            (\beta_c^{\ddot{\mathbf s}})_3
        }
        \left(
            A_v^{\beta_c^{\ddot{\mathbf s}}}
        \right)^{
            \odot\langle\varpi_j,\gamma_{v,c}\rangle
        }
    \right).
\end{equation}
\end{lemma}

\begin{proof}
Only frozen vertices of the truncated weave at depth \(c\) have
nonzero southern cycle output and can contribute to
\eqref{eq:def-quantum-grid-minor}. Apply
Corollary~\ref{cor:frozen-block-transport} to identify these vertices
with the frozen vertices of the right inductive weave of
\(\beta_c^{\ddot{\mathbf s}}\). Under its bijection \(\tau_c\),
\[
    A_v^{\ddot{\mathbf s}}
    =T_{\gamma_c}
       \bigl(A_{\tau_c(v)}^{\beta_c^{\ddot{\mathbf s}}}\bigr),
    \qquad
    \gamma_{v,c}=\gamma_{\tau_c(v),c}.
\]
Thus both the factors and their multiplicities agree with those
in \eqref{eq:quantum-grid-minors}. Since \(T_{\gamma_c}\) is a
based algebra embedding, it also preserves their normalized
product, proving the formula.
\end{proof}

\begin{theorem}
\label{thm:quantum-grid-minor-factorization}
Let \(c\) be an insertion in \(\ddot{\mathbf s}\). Then, for
every \(j\in I\),
\begin{equation}
\label{eq:quantum-grid-minor-factorization}
    D_{c,j}^{\beta,\ddot{\mathbf s}}
    =
    \widetilde G\bigl(
        \beta,\boldsymbol\delta_{c,j}^{\ddot{\mathbf s}}
    \bigr).
\end{equation}
In particular, the vector
\(\boldsymbol\delta_{c,j}^{\ddot{\mathbf s}}\) defined in
Definition~\ref{def:delta-double-string} is the
\(\beta\)-Lusztig parameter of
\(D_{c,j}^{\beta,\ddot{\mathbf s}}\).
\end{theorem}

\begin{proof}
Set \(\beta_c:=\beta_c^{\ddot{\mathbf s}}\). By
Definition~\ref{def:quantum-grid-minors}, the quantum grid minor
\(D_{c,j}^{\beta,\ddot{\mathbf s}}\) is a normalized global basis element. By
Lemma~\ref{lem:global-basis-product-triangularity}, its leading PBW
parameter is the sum of those of its factors.

Using \eqref{eq:quantum-grid-minors},
Proposition~\ref{pro:Lusztig-parameter-right-weave-variable}, and
Lemma~\ref{lem:global-basis-consecutive-block}, we obtain
\[
\begin{aligned}
    \bfa^\beta\bigl(D_{c,j}^{\beta,\ddot{\mathbf s}}\bigr)
    &=
    \sum_{
        v\in\overrightarrow{\mathfrak W}(\beta_c)_3
    }
    \langle\varpi_j,\gamma_{v,c}\rangle
    \iota_c^{\ddot{\mathbf s}}
    \left(\eta_v^{\beta_c}\right)                                      \\
    &=
    \iota_c^{\ddot{\mathbf s}}\left(
        \sum_{e\in J_{\beta_c}}
        \sum_{
            v\in\overrightarrow{\mathfrak W}(\beta_c)_3
        }
        \langle\varpi_j,\gamma_{v,c}\rangle
        \eta_{v,e}^{\beta_c}E_e
    \right)                                                            \\
    &\stackrel{\eqref{eq:root-identity}}{=}
    \iota_c^{\ddot{\mathbf s}}\left(
        \sum_{e\in J_{\beta_c}}
        \left\langle
            \varpi_j,
            \sigma_{e,c}^{-1}(\alpha^\vee_{i_e})
        \right\rangle E_e
    \right)                                                            \\
    &=
    \boldsymbol\delta_{c,j}^{\ddot{\mathbf s}}
\end{aligned}
\]
where $E_e$ refers to the unit vector for $e$. The last equality follows from
Definition~\ref{def:delta-double-string}. Since
\(D_{c,j}^{\beta,\ddot{\mathbf s}}\) is a global basis element, its
\(\beta\)-Lusztig parameter uniquely determines it. This proves
\eqref{eq:quantum-grid-minor-factorization}.
\end{proof}

\begin{corollary}[Determinantial elements as grid minors]
\label{cor:determinantial-grid-identification}
With the words, double string, and depth \(d\) of
Example~\ref{ex:delta-braid-i-box}, one has, in the common bosonic
algebra of the braid element \(\Delta\beta\),
\begin{equation}\label{eq:determinantial-grid-identification}
    D_{d,i}^{\omega_{[a,b]},\ddot{\mathbf s}_{[a,b]}}
    =T_{w_0}\bigl(D^\beta[a,b]\bigr).
\end{equation}
\end{corollary}
\begin{proof}
Theorem~\ref{thm:quantum-grid-minor-factorization} identifies the
left-hand side with
\(\widetilde G(\omega_{[a,b]},[a,b]_\phi)\).
The PBW transition in \eqref{eq:i-box-PBW-transition} identifies
this element with
\(\widetilde G(\boldsymbol\Delta\beta,
\iota_{\boldsymbol\Delta}([a,b]_\beta))\).
By Lemma~\ref{lem:global-basis-consecutive-block}, the latter is
\(T_{w_0}(\widetilde G(\beta,[a,b]_\beta))\), as required.
\end{proof}

\subsection{New subalgebras of bosonic extension algebras}

In this subsection, we introduce a new subalgebra of the bosonic
extension algebra \(\widehat{\cA}(\beta)\) and the corresponding
monoidal subcategory of \(\mathscr C_{\DD}(\beta)\). We begin with
the following lemma.

\begin{lemma}\label{lem:global-basis-extension-by-zero}
Let \(\beta=(i_1,\ldots,i_r)\) and let \(\gamma\) be a positive braid
word. Under the natural inclusion
\(\widehat{\cA}(\beta)\subseteq\widehat{\cA}(\beta\gamma)\), one has
\begin{equation}\label{eq:global-basis-extension-by-zero}
    \widetilde G(\beta,\mathbf a)
    =\widetilde G\bigl(\beta\gamma,(\mathbf a,\boldsymbol 0)\bigr)
    \qquad
    \bigl(\mathbf a\in\ZZ_{\geq0}^{r}\bigr).
\end{equation}
\end{lemma}

\begin{proof}
The first \(r\) PBW root vectors for \(\beta\gamma\) are precisely those
for \(\beta\). Hence
\(P^{\beta\gamma}(\mathbf a,\boldsymbol 0)=P^\beta(\mathbf a)\).
If
\(\mathbf b\prec_{\mathrm r}(\mathbf a,\boldsymbol 0)\) with
\(\mathbf b\in\ZZ_{\geq0}^{r+\ell(\gamma)}\), then every coordinate of
\(\mathbf b\) belonging to the \(\gamma\)-block is zero: otherwise the
rightmost nonzero such coordinate would make \(\mathbf b\) larger, rather
than smaller, in the right lexicographic order. Thus
\(\mathbf b=(\mathbf c,\boldsymbol 0)\) with
\(\mathbf c\prec_{\mathrm r}\mathbf a\). Hence the
\(c\)-invariance and triangularity conditions characterizing the
corresponding unnormalized elements \(G\) are identical. Uniqueness
in Proposition~\ref{prop:PBW-global-basis} identifies those elements;
their weights agree, so
\eqref{eq:normalized-PBW-global-basis-conversion} gives the stated
identity for \(\widetilde G\).
\end{proof}

Recall the paired definitions of \(\mathbf A(\beta)\) and
\(\mathscr D_{\DD}(\beta)\) from
\eqref{eq:A-D-beta-definition}.

\begin{corollary}\label{cor:right-weave-cluster-variables-level-one}
Let $\beta$ be a positive braid word, and assume that the ambient quantum affine algebra is of untwisted affine
type and that \(\DD=\DD_{\mathcal Q}\) is the complete strong duality
datum associated with a \(Q\)-datum \(\mathcal Q\). After identifying $t^{1/2}$ with $q^{-1/2}$, every
quantum cluster variable $X_v$ of
$\overline{\mathcal A}_t(\overrightarrow{\mathbf s}_t(\beta))$ is mapped by
$\Phi_{\DD}\circ\iota$ to a normalized global basis element
\begin{equation*}
    A_v:=(\Phi_{\DD}\circ\iota)(X_v)
    \in\widehat{\cA}_{\ZZ[q^{\pm1/2}]}(\beta)
       \cap\widehat{\cA}_{\geq1}
    =\mathbf A(\beta).
\end{equation*}
Consequently,
\begin{align}
    (\Phi_{\DD}\circ\iota)\!
    \left(
        \overline{\mathcal A}_t
        (\overrightarrow{\mathbf s}_t(\beta))
    \right)
    &\subseteq \mathbf A(\beta),
    \label{eq:right-weave-cluster-algebra-in-A-beta}\\
    \iota\!
    \left(
        \overline{\mathcal A}_t
        (\overrightarrow{\mathbf s}_t(\beta))
    \right)
    &\subseteq
    \mathcal{K}_t\bigl(\mathscr D_{\DD}(\beta)\bigr).
    \label{eq:right-weave-cluster-algebra-in-D-beta}
\end{align}
Here, as throughout, frozen cluster variables are not inverted.
\end{corollary}

\begin{proof}
Set $w:=\delta(\beta)$, and choose a reduced expression $\gamma$ of
$w^{-1}w_0$, regarded as a positive braid word. Since
\begin{equation*}
    \ell(w)+\ell(w^{-1}w_0)=\ell(w_0),
\end{equation*}
we have $\delta(\beta\gamma)=w_0$.

The right inductive weave
$\overrightarrow{\mathfrak W}(\beta\gamma)$ is obtained from
$\overrightarrow{\mathfrak W}(\beta)$ by successively adjoining the
right insertions associated with the letters of $\gamma$. Every prefix of
\(\gamma\) increases the Coxeter length from \(w\) toward \(w_0\); hence
these insertions create no new trivalent vertices. Thus
$\overrightarrow{\mathfrak W}(\beta)$ identifies naturally with the
entire trivalent-vertex set of
$\overrightarrow{\mathfrak W}(\beta\gamma)$. The Lusztig cycles and their
intersection numbers on the retained \(\beta\)-block are unchanged, so the
classical and quantum seeds are identified on this common vertex set. In
particular, for every
trivalent vertex
$v\in\overrightarrow{\mathfrak W}(\beta)_3$, the Lusztig vector
attached to $v$ is unchanged on the $\beta$-block:
\begin{equation}\label{eq:eta-beta-gamma-restriction}
    \eta_v^{\beta\gamma}
    =
    \bigl(\eta_v^\beta,0,\ldots,0\bigr),
\end{equation}
where the right-hand side is understood via the natural embedding of the
two coordinate lattices.

By Lemma~\ref{lem:global-basis-extension-by-zero},
Proposition~\ref{pro:Lusztig-parameter-right-weave-variable}, and
\eqref{eq:eta-beta-gamma-restriction}, the global basis elements corresponding to the cluster
variable at \(v\) are literally the same normalized global basis element,
whether \(v\) is regarded as a vertex of
\(\overrightarrow{\mathfrak W}(\beta)\) or of
\(\overrightarrow{\mathfrak W}(\beta\gamma)\).

Let $X$ be any quantum cluster variable of
$\overline{\mathcal A}_t(\overrightarrow{\mathbf s}_t(\beta))$, and choose a
mutation sequence producing $X$ from the initial seed. Under the seed
identification just established, applying the
same sequence to
$\overrightarrow{\mathbf s}_t(\beta\gamma)$ produces a cluster
variable $X^+$. At each
step, the corresponding exchange relations agree on the retained
vertices. It follows, equivalently by the associated Lusztig-parameter mutation rule in equation~\eqref{eq:Lusztig-parameter-mutated-variable}, that $X$ and
$X^+$ are labeled by the same normalized global basis element in the
common ambient bosonic extension algebra.

Since $\delta(\beta\gamma)=w_0$,
Corollary~\ref{cor:cluster-algebra-embedding-A-beta} applied to
$\beta\gamma$ gives
\begin{equation*}
    (\Phi_{\DD}\circ\iota)(X^+)
    \in
    \widehat{\cA}_{\ZZ[q^{\pm1/2}]}(\beta\gamma)
    \cap T_{w_0}\widehat{\cA}_{\geq0}
    \stackrel{(*)}{=}
    \widehat{\cA}_{\ZZ[q^{\pm1/2}]}(\beta\gamma)
    \cap\widehat{\cA}_{\geq1}.
\end{equation*}
The equality \((*)\) follows from
\(T_{w_0}\widehat{\cA}_{\geq0}=\widehat{\cA}_{\geq1}\) by
\cite[Lemma~4.4]{oh2025pbw}.
The preceding identification implies that
$(\Phi_{\DD}\circ\iota)(X)$ also belongs to
$\widehat{\cA}_{\geq 1}$. On the other hand, it lies in
$\widehat{\cA}_{\ZZ[q^{\pm1/2}]}(\beta)$ by Corollary~\ref{cor:cluster-algebra-embedding-A-beta}.
Therefore
\begin{equation*}
    (\Phi_{\DD}\circ\iota)(X)
    \in
    \widehat{\cA}_{\ZZ[q^{\pm1/2}]}(\beta)\cap\widehat{\cA}_{\geq 1}
    =\mathbf A(\beta).
\end{equation*}
Since the nonlocalized quantum cluster algebra is generated by its
quantum cluster variables, including the frozen ones, this proves
\eqref{eq:right-weave-cluster-algebra-in-A-beta}.

Finally, the compatibility of $\Phi_{\DD}$ with the level filtration
gives
\begin{equation}\label{eq:Phi-D-level-one-intersection}
    \Phi_{\DD}\!
    \left(
        \mathcal{K}_t\bigl(\mathscr D_{\DD}(\beta)\bigr)
    \right)
    =
    \widehat{\cA}_{\ZZ[q^{\pm1/2}]}(\beta)\cap\widehat{\cA}_{\geq 1}
    =
    \mathbf A(\beta).
\end{equation}
Because $\Phi_{\DD}$ is injective, the inclusion
\eqref{eq:right-weave-cluster-algebra-in-A-beta}, together with
\eqref{eq:Phi-D-level-one-intersection}, yields
\eqref{eq:right-weave-cluster-algebra-in-D-beta}.
\end{proof}

\begin{example}\label{ex:normalized-seven-letter-word}
Continue with Example~\ref{ex:right-inductive-weaves}. Consider the
double string
\begin{equation*}
\ddot{\mathbf s}
=(1R^+,1R,2R^+,2R,1R^+,2R,1R),
\end{equation*}
whose underlying braid word is
\(\beta=(1,1,2,2,1,2,1)\). The cluster variables associated with the
first two trivalent vertices are
\begin{equation*}
A_{v_1}=\widetilde P_2^\beta=q^{-1/2}P_2^\beta,
\qquad
A_{v_2}=\widetilde P_4^\beta=q^{-1/2}P_4^\beta.
\end{equation*}
The remaining two cluster variables are the global basis elements
\begin{equation*}
A_{v_3}=\widetilde G(\beta,E_6+E_4+E_2),
\qquad
A_{v_4}=\widetilde G(\beta,E_7+E_4),
\end{equation*}
whose leading PBW monomials are \(P_6^\beta P_4^\beta P_2^\beta\)
and \(P_7^\beta P_4^\beta\), respectively.

Corollary~\ref{cor:right-weave-cluster-variables-level-one} directly gives
\(A_{v_j}\in\mathbf A(\beta)\) for \(1\leq j\leq4\). The following
calculations are only consistency checks and are not used to infer
membership of a global-basis element from its individual PBW factors.
Clearly \(P_2^\beta,P_4^\beta\in\widehat{\cA}_{\geq1}\). Using the braid
relation \(T_2T_1T_2=T_1T_2T_1\), we obtain
\begin{align*}
P_7^\beta
 &=T_{112212}\bigl(q^{1/2}f_{1,0}\bigr)
  =T_{112121}\bigl(q^{1/2}f_{1,0}\bigr)
  \in\widehat{\cA}_{\geq1}.
\end{align*}
Similarly, let \(b=\widetilde G(\beta,E_6+E_4)\). Its leading PBW
term is \(q^{-1/2}P_6^\beta P_4^\beta\). Indeed, the braid move
\(212\mapsto121\) at positions \(4,5,6\) gives
\(\beta'=(1,1,2,1,2,1,1)\) and sends the local parameter
\((1,0,1)\) to \((0,1,0)\). Thus
\begin{equation*}
b=\widetilde G(\beta',E_5)
 =T_{1121}(f_{2,0})
 =T_{1212}(f_{2,0})
 \in\widehat{\cA}_{\geq1}.
\end{equation*}
The coefficients in the leading terms of \(A_{v_3}\) and
\(A_{v_4}\) are \(1\): both parameters have total weight zero,
whereas \(E_4+E_6\) has squared weight \(2\). This agrees with
\eqref{eq:normalized-PBW-global-basis-conversion}.
Notice, however, that \(P_6^\beta\notin\widehat{\cA}_{\geq1}\).
Thus, membership of \(A_{v_3}\) in
\(\widehat{\cA}_{\geq1}\) cannot be deduced from the individual PBW
factors alone; it follows from the based level-one intersection proved in
the corollary.
\end{example}

\section{Categorification of braid varieties}
\label{sec:categorification-braid-varieties}
In this section, we construct an explicit monoidal seed associated
with a braid variety, determine its quantum commutation matrix, and
establish its monoidal categorification and quantization.
Throughout this section, let
\(\beta=(i_1,\ldots,i_r)\) be a positive braid word, and fix a
complete duality datum \(\DD\). To simplify the notation, we write
\begin{equation*}
    C_k^\beta:=C_k^{\DD,\beta},
\end{equation*}
and suppress \(\DD\) from the notation for the simple modules
constructed below.

\subsection{The monoidal seed}

Let
\begin{equation*}
    \overrightarrow{\mathfrak W}(\beta)
    \colon
    \beta\longrightarrow\delta(\beta)
\end{equation*}
be the right inductive weave associated with \(\beta\), and
recall from Lemma~\ref{lem:right-inductive-seed-restriction} that the
index set \(V_\beta=K_\beta\setminus D_\beta\) defined in
Definition~\ref{def:svbeta} is canonically identified with
\(\overrightarrow{\mathfrak W}(\beta)_3\). We use \(V_\beta\) for
this common set; this is an identification, not a second definition.
We equip \(V_\beta\) with the order inherited from the right inductive
construction. For \(v\in V_\beta\), let \(q_v\in[r]\) be the position
of the letter of \(\beta\) that produces the trivalent vertex \(v\).

For \(u,v\in V_\beta\) with \(u<v\), set
\begin{equation}\label{eq:duv-Lusztig-cycle}
    d_{u,v}
    :=
    \gamma_u^{\overrightarrow{\mathfrak W}(\beta)}
    (e_{v,\mathrm{nw}}).
\end{equation}
We recursively define a family of simple modules
\(\{B_v\}_{v\in V_\beta}\). Suppose that \(B_u\) has been defined for
every \(u<v\), and set
\begin{equation}\label{eq:Mv-Bv-definition}
    M_v
    :=
    \overleftarrow{\bigotimes}_{u<v}
    B_u^{\otimes d_{u,v}},
    \qquad
    B_v
    :=
    C_{q_v}^\beta\nabla M_v
    =
    \operatorname{hd}
    \bigl(C_{q_v}^\beta\otimes M_v\bigr),
\end{equation}
where the tensor factors are arranged in decreasing order of the
vertices. The empty tensor product is understood to be the monoidal
unit. In particular, if \(v\) is the minimal vertex of \(V_\beta\),
then \(B_v=C_{q_v}^\beta\).

\begin{proposition}\label{pro:LusztigparaBv}
For every \(v\in V_\beta\), one has
\begin{equation}\label{eq:Lusztig-parameter-Bv-Av}
    \mathbf a^\beta(B_v)=\mathbf a^\beta(A_v)=\eta^\beta_v.
\end{equation}
Consequently, \(B_v\) is isomorphic to the cluster-variable module
corresponding to \(X_v\). In particular, every \(B_v\) is real and the
family \(\{B_v\}_{v\in V_\beta}\) is pairwise strongly commuting.
\end{proposition}

\begin{proof}
Set \(\eta_v:=\mathbf a^\beta(A_v)\). By
Proposition~\ref{pro:Lusztig-parameter-right-weave-variable}, these
vectors satisfy
\begin{equation}\label{eq:eta-v-right-weave-recursion}
    \eta_v
    =
    E_{q_v}
    +
    \sum_{u<v}d_{u,v}\eta_u.
\end{equation}
We prove \eqref{eq:Lusztig-parameter-Bv-Av} by induction on \(v\).

Suppose first that \(v\) is minimal. Then \(B_v=C_{q_v}^\beta\), and
hence
\begin{equation*}
    \mathbf a^\beta(B_v)=E_{q_v}=\eta_v,
\end{equation*}
where the second equality follows from
\eqref{eq:eta-v-right-weave-recursion}.

Now let \(v\) be nonminimal, and assume that
\(\mathbf a^\beta(B_u)=\eta_u\) for every \(u<v\). By the induction
hypothesis, each \(B_u\) is the cluster-variable module corresponding
to \(X_u\). Since the variables \(\{X_u\mid u<v\}\) belong to the same
cluster, the modules \(\{B_u\mid u<v\}\) are pairwise strongly
commuting. It follows that \(M_v\) is simple and
\begin{align}
    \mathbf b
    &:=
    \mathbf a^\beta(M_v)
    =
    \sum_{u<v}d_{u,v}\mathbf a^\beta(B_u)
    \nonumber\\
    &=
    \sum_{u<v}
    \gamma_u^{\overrightarrow{\mathfrak W}(\beta)}
    (e_{v,\mathrm{nw}})\eta_u.
    \label{eq:Lusztig-parameter-Mv}
\end{align}
By Lemma~\ref{lem:maxindex}, the support of \(\mathbf b\) is contained
in \([1,q_v-1]\); equivalently, \(b_k=0\) for \(k\geq q_v\). The PBW
classification of simple modules gives
\begin{equation}\label{eq:PBW-realization-Mv}
    M_v
    \simeq
    \operatorname{hd}\left(
        (C_{q_v-1}^\beta)^{\otimes b_{q_v-1}}
        \otimes\cdots\otimes
        (C_1^\beta)^{\otimes b_1}
    \right).
\end{equation}

By \cite[Theorem~5.16]{kashiwara2025monoidal}, the affine cuspidal
modules \(C_a^\beta\) and \(C_b^\beta\) are strongly unmixed whenever
\(a>b\). Therefore, the ordered sequence
\begin{equation*}
    \left(
        C_{q_v}^\beta,
        (C_{q_v-1}^\beta)^{\otimes b_{q_v-1}},
        \ldots,
        (C_1^\beta)^{\otimes b_1}
    \right)
\end{equation*}
is normal; see
\cite[Proposition~1.14]{kashiwara2025monoidal}. Lemma \ref{lem:normal-sequence}, equation~\eqref{eq:R-matrix-naturality}, Lemma~\ref{lem:head},
and the naturality of the normalized \(R\)-matrices consequently give
\begin{equation}\label{eq:PBW-realization-Bv}
    B_v
    \simeq
    \operatorname{hd}\left(
        C_{q_v}^\beta
        \otimes
        (C_{q_v-1}^\beta)^{\otimes b_{q_v-1}}
        \otimes\cdots\otimes
        (C_1^\beta)^{\otimes b_1}
    \right).
\end{equation}
It follows that
\begin{equation*}
    \mathbf a^\beta(B_v)=E_{q_v}+\mathbf b.
\end{equation*}
Combining this identity with
\eqref{eq:Lusztig-parameter-Mv} and
\eqref{eq:eta-v-right-weave-recursion}, we obtain
\begin{equation*}
    \mathbf a^\beta(B_v)
    =
    E_{q_v}+\sum_{u<v}d_{u,v}\eta_u
    =
    \eta_v
    =
    \mathbf a^\beta(A_v).
\end{equation*}

Finally, the modules \(B_v\), for \(v\in V_\beta\), correspond to
variables in a common cluster. They are therefore real and pairwise
strongly commuting.
\end{proof}

\subsubsection{The quantum commutation matrix}

For \(k\in[r]\), define
\begin{equation*}
    \alpha_k^\beta
    :=
    s_{i_1}\cdots s_{i_{k-1}}(\alpha_{i_k}).
\end{equation*}
For \(a,b\in[r]\), set
\begin{equation}\label{eq:lambda-beta-root-pairing}
    \lambda_{a,b}^\beta
    :=
    \begin{cases}
        (\alpha_a^\beta,\alpha_b^\beta),
            & a<b,\\
        0,
            & a=b,\\
        -(\alpha_a^\beta,\alpha_b^\beta),
            & a>b.
    \end{cases}
\end{equation}
The matrix
\(
    L^\beta
    :=
    \bigl(\lambda_{a,b}^\beta\bigr)_{a,b\in[r]}
\)
is skew-symmetric. Moreover, by \cite[Proposition~5.9]{kashiwara2025monoidal}, the \(R\)-matrix invariants of the affine
cuspidal modules satisfy
\begin{equation}\label{eq:LambdaCa}
    \LambdaR(C_a^\beta,C_b^\beta)
    =
    \lambda_{a,b}^\beta
    \qquad (a>b).
\end{equation}

For \(\eta,\eta'\in\ZZ_{\geq0}^r\), define the skew-symmetric
bilinear form
\begin{equation}\label{eq:E-beta-definition}
    E^\beta(\eta,\eta')
    :=
    \eta^{\mathsf T}L^\beta\eta'
    =
    \sum_{a,b=1}^r
    \eta_a\eta_b'\lambda_{a,b}^\beta.
\end{equation}
Thus
\(E^\beta(\eta,\eta')=-E^\beta(\eta',\eta)\).

\begin{proposition}\label{pro:LambdaBv}
For any \(v,w\in V_\beta\), one has
\begin{equation}\label{eq:Lambda-Bv-Bw-E-beta}
    \LambdaR(B_v,B_w)
    =
    E^\beta\bigl(
        \eta_v^\beta,
        \eta_w^\beta
    \bigr).
\end{equation}
\end{proposition}

\begin{proof}
For each \(v\in V_\beta\), set
\(\eta_v:=\eta_v^\beta=\mathbf a^\beta(A_v)\). Then
\eqref{eq:eta-v-right-weave-recursion} becomes
\begin{equation}\label{eq:eta-B-recursion}
    \eta_v
    =
    E_{q_v}
    +
    \sum_{u<v}d_{u,v}\eta_u.
\end{equation}

We prove the statement by strong induction on the larger of the two
vertices. Fix \(v\in V_\beta\), and assume that
\eqref{eq:Lambda-Bv-Bw-E-beta} holds for every pair of vertices whose
larger vertex is strictly smaller than \(v\). Let \(w<v\).

By Proposition~\ref{pro:LusztigparaBv}, the family
\(\{B_u\}_{u\in V_\beta}\) is pairwise strongly commuting. Moreover,
\cite[Lemma~2.24]{kashiwara2023pbw} implies that
\((C_{q_v}^\beta,M_v,B_w)\) is a normal sequence. Since
\(B_v=C_{q_v}^\beta\nabla M_v\) and
\(M_v=\overleftarrow{\bigotimes}_{u<v}
B_u^{\otimes d_{u,v}}\), the additivity of the \(R\)-matrix invariant
along normal sequences gives
\begin{equation}\label{eq:Lambda-recursion}
    \LambdaR(B_v,B_w)
    =
    \LambdaR(C_{q_v}^\beta,B_w)
    +
    \sum_{u<v}d_{u,v}\LambdaR(B_u,B_w).
\end{equation}

By the construction of the right inductive weave, \(w<v\) implies
\(q_w<q_v\). Proposition~\ref{pro:LusztigparaBv} and the PBW
classification of simple modules yield
\begin{equation}\label{eq:PBW-realization-Bw}
    B_w
    \simeq
    \operatorname{hd}\left(
        (C_{q_w}^\beta)^{\otimes(\eta_w)_{q_w}}
        \otimes\cdots\otimes
        (C_1^\beta)^{\otimes(\eta_w)_1}
    \right).
\end{equation}
The corresponding ordered sequence of affine cuspidal modules is
normal. Therefore, using \eqref{eq:LambdaCa}, we obtain
\begin{align}
    \LambdaR(C_{q_v}^\beta,B_w)
    &=
    \sum_{a=1}^{q_w}
    (\eta_w)_a
    \LambdaR(C_{q_v}^\beta,C_a^\beta)
    \nonumber\\
    &=
    \sum_{a=1}^{q_w}
    (\eta_w)_a\lambda_{q_v,a}^\beta
    =
    E^\beta(E_{q_v},\eta_w).
    \label{eq:Lambda-root-B}
\end{align}

For every \(u<v\) with \(u\neq w\), both \(u\) and \(w\) are strictly
smaller than \(v\). Hence the induction hypothesis gives
\(\LambdaR(B_u,B_w)=E^\beta(\eta_u,\eta_w)\). If \(u=w\), both sides
vanish: \(B_w\) is real, so \(\LambdaR(B_w,B_w)=0\), whereas the
skew-symmetry of \(E^\beta\) gives
\(E^\beta(\eta_w,\eta_w)=0\).

Substituting these identities and \eqref{eq:Lambda-root-B} into
\eqref{eq:Lambda-recursion}, we find
\begin{align*}
    \LambdaR(B_v,B_w)
    &=
    E^\beta(E_{q_v},\eta_w)
    +
    \sum_{u<v}d_{u,v}E^\beta(\eta_u,\eta_w)\\
    &=
    E^\beta\left(
        E_{q_v}
        +
        \sum_{u<v}d_{u,v}\eta_u,
        \eta_w
    \right)
    =
    E^\beta(\eta_v,\eta_w),
\end{align*}
where the last equality follows from
\eqref{eq:eta-B-recursion}. This proves the assertion for \(w<v\).

If \(v<w\), applying the preceding case to \((w,v)\) gives
\(\LambdaR(B_w,B_v)=E^\beta(\eta_w,\eta_v)\). Since \(B_v\) and \(B_w\)
strongly commute, the \(R\)-matrix invariant is skew-symmetric on this
pair. Hence
\begin{align*}
    \LambdaR(B_v,B_w)
    &=
    -\LambdaR(B_w,B_v)
    =
    -E^\beta(\eta_w,\eta_v)\\
    &=
    E^\beta(\eta_v,\eta_w).
\end{align*}
Finally, when \(v=w\), both sides are zero. This completes the
induction.
\end{proof}

We now assume that \(\delta(\beta)=w\). Set
\begin{equation*}
    N:=r-\ell(w),
    \qquad
    V_\beta=\{v_1<\cdots<v_N\}.
\end{equation*}
Since
\(
    \operatorname{supp}(\eta_{v_j})
    \subseteq
    \{q_{v_1},\ldots,q_{v_j}\},
\)
we identify the sublattice
\begin{equation*}
    \bigoplus_{j=1}^N\ZZ E_{q_{v_j}}
    \subseteq\ZZ^r
\end{equation*}
with \(\ZZ^N\). Under this identification, we continue to write
\(\eta_{v_j}\) for the image of \(\eta_{v_j}\), and denote by \(E_j\)
the \(j\)-th standard basis vector of \(\ZZ^N\).

Define the strictly upper-triangular matrix
\begin{equation*}
    A_\beta
    =
    (a_{ij})_{1\leq i,j\leq N},
    \qquad
    a_{ij}
    :=
    \begin{cases}
        \gamma_{v_i}^{\overrightarrow{\mathfrak W}(\beta)}
        (e_{v_j,\mathrm{nw}}),
            & i<j,\\
        0,
            & i\geq j.
    \end{cases}
\end{equation*}
Let
\begin{equation*}    H_\beta
    :=
    \begin{pmatrix}
        \eta_{v_1}&\cdots&\eta_{v_N}
    \end{pmatrix}.
\end{equation*}
The recursion \eqref{eq:eta-B-recursion} is equivalent to
\( H_\beta=I_N+H_\beta A_\beta.
\)
Since \(A_\beta\) is nilpotent, \(I_N-A_\beta\) is
invertible over \(\ZZ\), and
\begin{equation}\label{eq:H-beta-inverse}
    H_\beta=(I_N-A_\beta)^{-1}.
\end{equation}

Let
\begin{equation*}
    \widehat L^\beta
    :=
    \bigl(
        \lambda_{q_{v_i},q_{v_j}}^\beta
    \bigr)_{1\leq i,j\leq N}
\end{equation*}
be the principal submatrix of \(L^\beta\) indexed by
\(\{q_{v_1},\ldots,q_{v_N}\}\). Proposition~\ref{pro:LambdaBv} and
\eqref{eq:H-beta-inverse} give
\begin{align}
    \LambdaR(B_{v_i},B_{v_j})
    &=
    \eta_{v_i}^{\mathsf T}
    \widehat L^\beta\eta_{v_j}
    \nonumber\\
    &=
    E_i^{\mathsf T}
    (I_N-A_\beta)^{-\mathsf T}
    \widehat L^\beta
    (I_N-A_\beta)^{-1}E_j.
    \label{eq:LambdaBvbv'}
\end{align}
We therefore define
\begin{equation}\label{eq:Lambda-beta-definition}
    \Lambda_\beta
    :=
    (I_N-A_\beta)^{-\mathsf T}
    \widehat L^\beta
    (I_N-A_\beta)^{-1}.
\end{equation}
Equivalently,
\begin{equation*}
    (\Lambda_\beta)_{ij}
    =
    \LambdaR(B_{v_i},B_{v_j}).
\end{equation*}

\begin{proposition}\label{pro:compatible-Lambda-beta}
Let \(\beta\) be a positive braid word, and let \(B_\beta\) be the
extended exchange matrix of
\(\overrightarrow{\mathbf s}_t(\beta)\). Then
\((\Lambda_\beta,B_\beta)\) is a compatible pair; equivalently,
\begin{equation*}
    (-\Lambda_\beta)B_\beta=2I
\end{equation*}
on the exchangeable columns.
\end{proposition}

\begin{proof}
As in equation \eqref{eq:word-monoidal-seed}, let
\(
    S^{\DD}(\beta)
\)
be the completely \(\Lambda\)-admissible and quantizable monoidal seed
associated with the right inductive weave
\(\overrightarrow{\mathfrak W}(\boldsymbol{\Delta}\beta)\) . The corresponding quantum seed is $\mathbf{s}_t(\beta)$. Apply the
mutation sequence \(T\) equation~\eqref{eq:mutation-sequence-T} that moves the letters of
\(\boldsymbol{\Delta}\) from the left boundary to the final left
insertions.
Let
\(
    \widetilde{\mathcal S}
    =
    \bigl(
        \{\widetilde M_i\}_{i\in K},
        \widetilde B;
        K,K_{\mathrm{ex}}
    \bigr)
\)
be the monoidal seed associated with the resulting double inductive
weave $\mathfrak W_{[m]}$. Let
\(D\subseteq K\) be the set of vertices corresponding to
\(i_{p_m}^*L,\ldots,i_{p_1}^*L\), and set \(J:=K\setminus D\). Write
\begin{equation*}
    \widetilde B
    =(\widetilde b_{ik})_{
       i\in K,\;k\in K_{\mathrm{ex}}},
    \qquad
    \widetilde\Lambda
    =(\widetilde\lambda_{ij})_{i,j\in K},
    \qquad
    \widetilde\lambda_{ij}
    :=\LambdaR(\widetilde M_i,\widetilde M_j).
\end{equation*}

The seed \(\widetilde{\mathcal S}\) is obtained by mutations from a
completely \(\Lambda\)-admissible monoidal seed and is therefore
completely \(\Lambda\)-admissible. Under our convention for the
\(R\)-matrix invariant, its compatibility relation is
\begin{equation}\label{eq:full-seed-compatibility}
    \bigl((-\widetilde\Lambda)\widetilde B\bigr)_{jk}
    =2\delta_{jk}
    \qquad
    (j\in K,\;k\in K_{\mathrm{ex}}).
\end{equation}

The seed \(\overrightarrow{\mathbf s}_t(\beta)\) is obtained from
\(\widetilde{\mathcal S}\) by deleting the vertices in \(D\) and
freezing every retained vertex adjacent to \(D\) by Lemma
\ref{lem:right-inductive-seed-restriction} and the definition of
\(\mathbf s_t(\delta(\beta),\beta)\). Let
\(J_{\mathrm{ex}}\subseteq J\) be the resulting set of exchangeable
vertices.  We have
\begin{equation}\label{eq:B-beta-restriction}
    B_\beta
    =\widetilde B_{J\times J_{\mathrm{ex}}},
    \qquad
    \widetilde b_{dk}=0
    \quad
    (d\in D,\;k\in J_{\mathrm{ex}}).
\end{equation}
By Lemma~\ref{lem:Lusztig-cycle-h-block-invariance}, after the vertices in \(D\) are deleted, the remaining
undecorated vertices are exactly those created by the
\((i_1\mathrm R,\ldots,i_r\mathrm R)\)-block. By the
right-inductive construction, this retained set is naturally identified
with \(V_\beta\); hence \(J=V_\beta\) under this identification.

By Proposition~\ref{pro:LusztigparaBv}, the cluster-variable modules
\(\widetilde M_j\), \(j\in J\), are precisely the modules
\(B_v\), \(v\in V_\beta\). Consequently,
\(\Lambda_\beta=\widetilde\Lambda_{J\times J}\). For
\(j\in J\) and \(k\in J_{\mathrm{ex}}\), skew-symmetry of
\(\widetilde\Lambda\), together with
\eqref{eq:full-seed-compatibility} and
\eqref{eq:B-beta-restriction}, gives
\begin{align*}
    \bigl((-\Lambda_\beta)B_\beta\bigr)_{jk}
    &=-\sum_{i\in J}
      \widetilde\lambda_{ji}\widetilde b_{ik}\\
    &=\sum_{i\in J}
      \widetilde\lambda_{ij}\widetilde b_{ik}\\
    &=\sum_{i\in K}
      \widetilde\lambda_{ij}\widetilde b_{ik}
      =2\delta_{jk}.
\end{align*}
Therefore, \((\Lambda_\beta,B_\beta)\) is a compatible pair.
\end{proof}

\begin{theorem}\label{thm:monoidal-seed-beta}
Let \(\beta\) be a positive braid word, and let \(\DD=\DD_{\mathcal Q}\) be the duality datum associated with a \(Q\)-datum \(\mathcal Q\).
Denote by \(V_{\beta,\mathrm{ex}}\subseteq V_\beta\) the exchangeable
vertex set of the right-inductive seed. Then
\begin{equation}\label{eq:monoidal-seed-beta-definition}
    \mathcal T^{\DD}(\beta)
    :=
    \left(
        \{B_v\}_{v\in V_\beta},
        B_\beta;
        V_\beta,V_{\beta,\mathrm{ex}}
    \right)
\end{equation}
is a completely \(\Lambda\)-admissible and quantizable monoidal seed
in
\begin{equation*}
    \mathscr D_{\DD}(\beta)
    =
    \mathscr C_{\DD}(\beta)
    \cap\mathscr C_{\DD}[1,\infty).
\end{equation*}
Here \(\mathcal T^{\DD}(\beta)\) denotes the restricted
right-inductive monoidal seed indexed by \(V_\beta\);
Its associated quantum seed is
\begin{equation}\label{eq:quantum-seed-associated-Bv}
    [\mathcal T^{\DD}(\beta)]_t
    =
    \left(
        \{[B_v]_t\}_{v\in V_\beta},
        B_\beta, \Lambda_\beta,
        V_{\beta,\mathrm{ex}}
    \right)
    =
    \overrightarrow{\mathbf s}_t(\beta).
\end{equation}
\end{theorem}

\begin{proof}
Consider the monoidal seed
\begin{equation*}
    \widetilde{\mathcal S}
    :=
    T\bigl(S^{\DD}(\beta)\bigr).
\end{equation*}
Complete \(\Lambda\)-admissibility and quantizability are preserved by
mutation, so \(\widetilde{\mathcal S}\) has both properties.

As in the proof of Proposition~\ref{pro:compatible-Lambda-beta}, the right
inductive seed for \(\beta\) is obtained from
\(\widetilde{\mathcal S}\) by deleting the final left-insertion
vertices and freezing every retained vertex adjacent to a deleted
vertex. Consequently, the vanishing condition
\begin{equation*}
    \widetilde b_{dk}=0
    \qquad
    (d\in D,\ k\in V_{\beta,\mathrm{ex}})
\end{equation*}
holds. Lemma~\ref{lem:subseed} therefore shows that the restricted
monoidal seed is completely \(\Lambda\)-admissible.

By Corollary \ref{cor:right-weave-cluster-variables-level-one}, every cluster-variable module
reachable from this restricted seed belongs to
\begin{equation*}
    \mathscr C_{\DD}(\beta)
    \cap\mathscr C_{\DD}[1,\infty)
    =
    \mathscr D_{\DD}(\beta).
\end{equation*}
Thus the restricted seed and all of its successive mutations are
realized inside \(\mathscr D_{\DD}(\beta)\).

Proposition~\ref{pro:LusztigparaBv} identifies the cluster-variable
module attached to \(v\in V_\beta\) with \(B_v\), while the exchange
matrix of the restricted seed is \(B_\beta\). Moreover,
\eqref{eq:LambdaBvbv'} gives
\begin{equation*}
    \LambdaR(B_{v_i},B_{v_j})
    =
    (\Lambda_\beta)_{ij}.
\end{equation*}
Hence the restricted monoidal seed is precisely
\(\mathcal T^{\DD}(\beta)\). With the quantum commutation convention
used in this paper, its commutation matrix is
\(\Lambda_\beta\). Proposition~\ref{pro:compatible-Lambda-beta} shows that
this matrix is compatible with \(B_\beta\), and therefore the
associated quantum seed is exactly
\(\overrightarrow{\mathbf s}_t(\beta)\). This proves the theorem.
\end{proof}

\subsection{Monoidal categorifications}

In this subsection, we assume that the Cartan matrix is of type
\(A\), \(D\), or \(E\), and impose the following standing assumption:
\begin{equation}\label{assumption}
\begin{gathered}
\text{\bfseries Assumption.}\quad
\text{\(U_{\qaff}'(\widehat{\fg})\) is of untwisted affine type,}
\text{ and }\DD=\DD_{\mathcal Q}\\
\text{ is the complete strong duality datum associated}
\text{with a }Q\text{-datum }\mathcal Q.
\end{gathered}
\end{equation}
We establish a monoidal categorification of the cluster algebra
associated with a braid variety. The proof has three steps.
First, a descent argument derives the case of \(\beta\) from that of
\(\beta i\). Second, we prove that two positive braid words
\(\beta\) and \(\beta'\) representing the same element
\(b\in\operatorname{Br}^{+}\) give rise to the same cluster algebra
and categorification. Finally, we establish the based monoidal
categorification for braid words of the special form
\(\boldsymbol{\Delta}\beta\). Combining these three steps proves the
main theorem.

\subsubsection{Inductive step}
We begin with a
lemma concerning the frozen cluster-variable modules arising from a
Demazure descent.

If \(J\) is a subset of the exchangeable vertices of a seed \(\mathbf t\),
we write \(\operatorname{Fr}_J\mathbf t\) for the seed obtained by keeping
the cluster and compatible form unchanged and replacing the exchangeable
set \(K_{\mathrm{ex}}\) by \(K_{\mathrm{ex}}\setminus J\). Thus
\(\operatorname{Fr}_J\) means \emph{freezing}.

\begin{lemma}\label{lem:commuting-factors-last-vertex}
Let \(\beta=(i_1,\ldots,i_r)\) be a positive braid word satisfying
\(\delta(\beta i)=\delta(\beta)\), and let \(v\) be the trivalent
vertex of \(\overrightarrow{\mathfrak W}(\beta i)\) corresponding to
the final letter \(i\). Suppose that \(B_v\) strongly commutes with
every simple object of \(\mathscr D_{\DD}(\beta i)\). Then every
\(v'\in\overrightarrow{\mathfrak W}(\beta)_3\) satisfying
\[
    \gamma_{v'}^{\overrightarrow{\mathfrak W}(\beta i)}
    (e_{v,\mathrm{nw}})\neq0
\]
has the property that \(B_{v'}\) strongly commutes with every simple
object of \(\mathscr D_{\DD}(\beta)\).
\end{lemma}

\begin{proof}
We use the natural inclusion
\(
    \mathscr D_{\DD}(\beta)
    \subseteq
    \mathscr D_{\DD}(\beta i).
\)
For \(v'<v\), set
\(
    d_{v',v}
    :=
    \gamma_{v'}^{\overrightarrow{\mathfrak W}(\beta i)}
    (e_{v,\mathrm{nw}})
\)
and define
\[
    M
    :=
    \overleftarrow{\bigotimes}_{v'<v}
    B_{v'}^{\otimes d_{v',v}}.
\]
The modules \(B_{v'}\) occurring in this product are pairwise strongly
commuting real simple modules. Hence \(M\) is real and simple. By the
recursive construction of \(B_v\),
\[
    B_v=L\nabla M,
    \qquad
    L:=C_{r+1}^{\DD,\beta i}.
\]
Let \(\mathcal DL\) be the right dual of \(L\). Since \(L\) is real,
Lemma~\ref{lem:cancel} gives
\[
    B_v\nabla\mathcal DL
    =
    (L\nabla M)\nabla\mathcal DL
    \simeq M.
\]

Let \(N\) be a simple object of \(\mathscr D_{\DD}(\beta)\). By Proposition~\ref{pro:lusztig-parameter-modules}, there exists
\(\mathbf n=(n_1,\ldots,n_r)\in\ZZ_{\geq0}^r\) such that
\[
    N
    \simeq
    \operatorname{hd}\left(
        (C_r^{\DD,\beta})^{\otimes n_r}
        \otimes\cdots\otimes
        (C_1^{\DD,\beta})^{\otimes n_1}
    \right).
\]
By \cite[Theorem~5.16]{kashiwara2025monoidal}, the sequence
\[
    \bigl(
        C_{r+1}^{\DD,\beta i},
        C_r^{\DD,\beta},\ldots,C_1^{\DD,\beta}
    \bigr)
\]
is strongly unmixed. Its stability under taking heads
\cite[Proposition~1.14]{kashiwara2025monoidal} implies that
\((L,N)\) is strongly unmixed, and therefore
\[
    \mathfrak d(\mathcal DL,N)=0.
\]
Moreover, \(N\in\mathscr D_{\DD}(\beta i)\), so the hypothesis gives
\(\mathfrak d(B_v,N)=0\). Using the cancellation isomorphism and the
subadditivity of \(\mathfrak d\), we obtain
\[
\begin{aligned}
    0
    \leq\mathfrak d(M,N)
    &=
    \mathfrak d(B_v\nabla\mathcal DL,N)\\
    &\leq
    \mathfrak d(B_v,N)
    +
    \mathfrak d(\mathcal DL,N)
    =
    0.
\end{aligned}
\]
Thus \(\mathfrak d(M,N)=0\).

Because \(M\) is real, \(\mathfrak d(M,N)=0\) implies that
\(M\otimes N\) is simple. Fix \(v'<v\) with
\(d_{v',v}>0\). The factors occurring in \(M\) strongly commute
pairwise, so they may be reordered and one copy of \(B_{v'}\) may be
placed next to \(N\): up to the harmless grading normalization,
\[
    M\otimes N\simeq M'\otimes(B_{v'}\otimes N)
\]
for a nonzero tensor product \(M'\) of the remaining factors.
If \(B_{v'}\otimes N\) were not simple, it would have a nonzero
proper submodule with nonzero quotient. Tensoring this short exact
sequence by \(M'\) remains exact, and both the submodule and quotient
remain nonzero because tensoring finite-dimensional modules with a
nonzero module is faithful on underlying vector spaces. This would
give a nonzero proper submodule of \(M\otimes N\), contradicting its
simplicity. Hence \(B_{v'}\otimes N\) is simple. Since \(B_{v'}\)
is real, it strongly commutes with \(N\). As \(N\) was arbitrary,
the result follows.
\end{proof}

Let \(V_{\beta,\mathrm{fr}}\) be the set of frozen vertices of the
right-inductive quantum seed
\(\overrightarrow{\mathbf s}_t(\beta)\), and let \(B_v\) be the
cluster-variable module associated with \(v\in V_\beta\). Define
\begin{equation}\label{eq:localized-Kt-D-beta}
    S_\beta
    :=
    \bigl\langle
        [B_v]_t
        \mid v\in V_{\beta,\mathrm{fr}}
    \bigr\rangle,
    \qquad
    \mathcal{K}_t\bigl(\mathscr D_{\DD}(\beta)\bigr)_{\mathrm{loc}}
    :=
    S_\beta^{-1}
    \mathcal{K}_t\bigl(\mathscr D_{\DD}(\beta)\bigr).
\end{equation}
If every frozen module strongly commutes with every simple object of
\(\mathscr D_{\DD}(\beta)\), then their classes are normal and
pairwise \(t\)-commuting. Hence \(S_\beta\) is an Ore multiplicative
set, and the localization in \eqref{eq:localized-Kt-D-beta} is well
defined.
A cluster monomial with nonnegative frozen exponents is represented
by the corresponding tensor product in the unlocalized category.
Frozen inverses are adjoined formally in this localization; such a
localized expression does not itself specify a tensor product with
negative multiplicities.

In the Demazure-descent case, let
\(\beta i=(i_1,\ldots,i_r,i)\) satisfy
\(\delta(\beta i)=\delta(\beta)\), and let \(v_0\) be the trivalent
vertex of \(\overrightarrow{\mathfrak W}(\beta i)\) associated with
the final letter \(i\). Set
\begin{equation}\label{eq:J-beta-i-fr-definition}
    J_{\beta,i,\mathrm{fr}}
    :=
    \left\{
        v\in V_{\beta i,\rm ex}
        \,\middle|\,
        \gamma_v^{\overrightarrow{\mathfrak W}(\beta i)}
        (e_{v_0,\mathrm{nw}})
        \neq 0
    \right\}.
\end{equation}
These are exactly the mutable vertices that must be frozen before
deleting \(v_0\).

\begin{lemma}[Demazure-descent deletion]
\label{lem:Demazure-descent-seed-deletion}
With the preceding notation, one has
\begin{equation}\label{eq:Demazure-descent-frozen-set}
    V_{\beta i}=V_\beta\sqcup\{v_0\},
    \qquad
    V_{\beta i,\mathrm{fr}}\cup J_{\beta,i,\mathrm{fr}}
    =
    V_{\beta,\mathrm{fr}}\cup\{v_0\}.
\end{equation}
Moreover,
\begin{equation}\label{eq:Demazure-descent-seed-deletion}
    \operatorname{Fr}_{J_{\beta,i,\mathrm{fr}}}
    \overrightarrow{\mathbf s}_t(\beta i)
    \setminus\{v_0\}
    =
    \overrightarrow{\mathbf s}_t(\beta)
\end{equation}
as quantum seeds. Finally,
\begin{equation}\label{eq:Demazure-descent-zero-row}
    b_{v_0k}=0
\end{equation}
for every vertex \(k\) that remains mutable after freezing.
\end{lemma}

\begin{proof}
The vertex identities, the restriction of the exchange matrix, and
\eqref{eq:Demazure-descent-zero-row} follow from the right-inductive
version of
\cite[Lemma~4.47 and Remark~4.48]{casals2025cluster}. Indeed, \(v_0\)
is the unique new frozen vertex, and its mutable neighbors are
precisely the vertices in \(J_{\beta,i,\mathrm{fr}}\). Deleting
\(v_0\) and freezing these neighbors therefore recovers the quiver of
\(\overrightarrow{\mathbf s}_t(\beta)\).

It remains only to compare the compatible forms. Appending the final
letter \(i\) does not change the modules attached to the vertices in
\(V_\beta\), so
\[
    B_v^{\beta i}=B_v^\beta
    \qquad(v\in V_\beta).
\]
Hence, for \(u,v\in V_\beta\),
\[
    (\Lambda_{\beta i})_{uv}
    =
    \LambdaR(B_u^{\beta i},B_v^{\beta i})
    =
    \LambdaR(B_u^\beta,B_v^\beta)
    =
    (\Lambda_\beta)_{uv}.
\]
Since freezing leaves the compatible form unchanged, this proves
\eqref{eq:Demazure-descent-seed-deletion}.
\end{proof}

\begin{lemma}\label{lem:upper-skew-Laurent-extension}
Let
\(\widetilde{\mathbf s}_t
:=\operatorname{Fr}_{J}
\overrightarrow{\mathbf s}_t(\beta i)\)
be the seed obtained by freezing the vertices in \(J\). Suppose that
\(v_0\) is frozen in \(\widetilde{\mathbf s}_t\) and satisfies
\(b_{v_0k}=0\) for every exchangeable vertex \(k\). Assume further
that deleting \(v_0\) from \(\widetilde{\mathbf s}_t\) gives
\(\overrightarrow{\mathbf s}_t(\beta)\).

Set \(x_{v_0}:=[B_{v_0}]_t\). Then conjugation by \(x_{v_0}\)
restricts to an algebra automorphism
\(\sigma\) of
\(\mathcal U_t(\overrightarrow{\mathbf s}_t(\beta))\), given by
\(\sigma(a)=x_{v_0}ax_{v_0}^{-1}\). Moreover, there is an algebra
isomorphism
\begin{equation}\label{eq:upper-skew-Laurent-extension}
    \mathcal U_t\bigl(
        \operatorname{Fr}_{J}
        \overrightarrow{\mathbf s}_t(\beta i)
    \bigr)
    \simeq
    \mathcal U_t\bigl(
        \overrightarrow{\mathbf s}_t(\beta)
    \bigr)
    [x_{v_0}^{\pm1};\sigma].
\end{equation}
Here the skew Laurent extension is determined by
\(x_{v_0}a=\sigma(a)x_{v_0}\).
\end{lemma}

\begin{proof}
Write \(\mathbf s_t:=\overrightarrow{\mathbf s}_t(\beta)\). Since
\(b_{v_0k}=0\) for every exchangeable vertex \(k\), mutation in any
exchangeable direction preserves the zero row indexed by \(v_0\).
Thus \(v_0\) remains frozen and is not adjacent to any exchangeable
vertex throughout the mutation class of
\(\widetilde{\mathbf s}_t\). Deleting \(v_0\) therefore induces a
bijection between the mutation classes of
\(\widetilde{\mathbf s}_t\) and \(\mathbf s_t\).

Let \(\mathbf s_t'\) be mutation equivalent to \(\mathbf s_t\), and
let \(\widetilde{\mathbf s}_t'\) be the corresponding seed obtained
by adjoining \(v_0\). Denote the associated based quantum tori by
\(\mathcal T(\mathbf s_t')\) and
\(\mathcal T(\widetilde{\mathbf s}_t')\), respectively. Since
\(x_{v_0}\) does not occur in any exchange relation and
quasi-commutes with every cluster variable of \(\mathbf s_t'\),
conjugation by \(x_{v_0}\) preserves
\(\mathcal T(\mathbf s_t')\). Consequently,
\begin{equation}\label{eq:cluster-torus-skew-extension}
    \mathcal T(\widetilde{\mathbf s}_t')
    =
    \mathcal T(\mathbf s_t')
    [x_{v_0}^{\pm1};\sigma].
\end{equation}

The upper quantum cluster algebra is the intersection of the quantum
tori over the corresponding mutation class. Hence
\(\mathcal U_t(\widetilde{\mathbf s}_t)\) is the intersection of the
algebras
\(\mathcal T(\mathbf s_t')[x_{v_0}^{\pm1};\sigma]\), where
\(\mathbf s_t'\sim\mathbf s_t\).

Every element of the ambient skew Laurent extension has a unique
finite expansion \(\sum_{d\in\ZZ}a_dx_{v_0}^d\). Such an element
belongs to every algebra
\(\mathcal T(\mathbf s_t')[x_{v_0}^{\pm1};\sigma]\) if and only if
each coefficient \(a_d\) belongs to every
\(\mathcal T(\mathbf s_t')\). Therefore,
\(\mathcal U_t(\widetilde{\mathbf s}_t)
=\mathcal U_t(\mathbf s_t)[x_{v_0}^{\pm1};\sigma]\), which proves
\eqref{eq:upper-skew-Laurent-extension}.
\end{proof}

\begin{lemma}\label{lem:degree-zero-triangular-basis}
In the setting of Lemma~\ref{lem:upper-skew-Laurent-extension}, suppose that
\(\widetilde{\mathcal U}
=\mathcal U_t(\overrightarrow{\mathbf s}_t(\beta))
[x_{v_0}^{\pm1};\sigma]\) has a common triangular basis
\(\widetilde{\mathbf L}\) that is pointed at the frozen seed $\operatorname{Fr}_{J}
\overrightarrow{\mathbf s}_t(\beta i)$. Give
\(x_{v_0}\) degree \(1\) and every variable of
\(\overrightarrow{\mathbf s}_t(\beta)\) degree \(0\). Then every element
of \(\widetilde{\mathbf L}\) is homogeneous, and
\begin{equation}\label{eq:degree-zero-common-triangular-basis}
    \mathbf L_0
    :=\{L\in\widetilde{\mathbf L}\mid\deg_{v_0}(L)=0\}
\end{equation}
is a common triangular basis of
\(\mathcal U_t(\overrightarrow{\mathbf s}_t(\beta))\).
\end{lemma}

\begin{proof}
The row of the extended exchange matrix indexed by \(v_0\) is zero.
Consequently, subtracting \(\widetilde B\mathbf n\) in the pointed
expansion \eqref{eq:pointed-element} does not change the
\(v_0\)-coordinate. Every Laurent monomial in a pointed basis element
therefore has the same \(x_{v_0}\)-degree, so that basis element is
homogeneous.

The skew Laurent extension is the direct sum of its homogeneous components,
and its degree-zero component is exactly
\(\mathcal U_t(\overrightarrow{\mathbf s}_t(\beta))\). Since
\(\widetilde{\mathbf L}\) is a homogeneous basis, its degree-zero elements
form a basis of this component. Bar invariance, pointedness, and the
triangular multiplication rules restrict to degree zero. Moreover, the
degree-zero subset contains all normalized cluster monomials of the deleted
seed. Hence it is a common triangular basis of the degree-zero upper quantum
cluster algebra.
\end{proof}

Let \(\mathscr T\) be a full, replete monoidal subcategory of
\(\mathscr C_{\DD}(\beta i)\), closed under extensions and
subquotients. Since the normalized simple classes form bases of the
corresponding quantum Grothendieck rings, and distinct simple objects
of \(\mathscr T\) remain distinct in
\(\mathscr C_{\DD}(\beta i)\), the inclusion induces an injective
algebra homomorphism
\[
    \mathcal K_t(\mathscr T)
    \hookrightarrow
    \mathcal K_t\bigl(\mathscr C_{\DD}(\beta i)\bigr).
\]

For
\(\mathbf a=(a_1,\ldots,a_{r+1})\in\ZZ_{\geq0}^{r+1}\), put
\(
    P_{\DD}^{\beta i}(\mathbf a)
    :=
    [C_{r+1}^{\DD,\beta i}]_t^{a_{r+1}}
    \cdots
    [C_1^{\DD,\beta i}]_t^{a_1}.
\)
The elements \(P_{\DD}^{\beta i}(\mathbf a)\) form a PBW basis of
\(\mathcal K_t(\mathscr C_{\DD}(\beta i))\). Hence every nonzero
\(f\in\mathcal K_t(\mathscr T)\) has a unique expansion
\[
    f
    =
    \sum_{\mathbf a}
        h_{\mathbf a}(t)P_{\DD}^{\beta i}(\mathbf a).
\]
Define
\(
    \operatorname{supp}_{\mathrm{PBW}}(f)
    :=
    \left\{
        \mathbf a\in\ZZ_{\geq0}^{r+1}
        \,\middle|\,
        h_{\mathbf a}(t)\neq0
    \right\}.
\)
Let \(\operatorname{lt}_{\mathrm{PBW}}(f)\) be the maximal element of
this finite set with respect to the right lexicographic order
\(\prec_{\mathrm r}\). If
\(
    \operatorname{lt}_{\mathrm{PBW}}(f)
    =
    (a_1,\ldots,a_{r+1}),
\)
set
\[
    \nu_{r+1}(f):=a_{r+1},
    \qquad
    \nu_{r+1}(0):=-\infty.
\]

PBW triangularity
\cite[Lemma~5.5]{oh2025pbw}, applied in the ambient quantum
Grothendieck ring, gives
\begin{equation}\label{eq:PBW-last-coordinate-multiplicativity}
    \nu_{r+1}(fg)
    =
    \nu_{r+1}(f)+\nu_{r+1}(g)
\end{equation}
for all nonzero
\(f,g\in\mathcal K_t(\mathscr T)\). Moreover, for every
\(L\in\operatorname{Irr}(\mathscr T)\),
\begin{equation}\label{eq:PBW-degree-simple-class}
    \nu_{r+1}([L]_t)
    =
    \bigl(\mathbf a^{\beta i}(L)\bigr)_{r+1}.
\end{equation}

Let \(J_{\mathrm{fr}}^{\mathrm{loc}}\) be the set of frozen vertices
whose cluster-variable modules are to be inverted in the
localization under consideration. For
\(p\in J_{\mathrm{fr}}^{\mathrm{loc}}\), denote the corresponding
frozen module by
\(
    F_p\in\operatorname{Irr}(\mathscr T)
\)
and put \(x_p:=[F_p]_t\). Assume that every \(F_p\) strongly commutes
with every simple object of \(\mathscr T\). Define
\begin{equation}\label{eq:frozen-Ore-set}
    S
    :=
    \left\langle
        x_p
        \,\middle|\,
        p\in J_{\mathrm{fr}}^{\mathrm{loc}}
    \right\rangle_{\mathrm{mult}}
    \subseteq\mathcal K_t(\mathscr T).
\end{equation}
Thus \(S\) is generated precisely by the normalized classes of the
frozen modules required in the localization.

For every \(p\in J_{\mathrm{fr}}^{\mathrm{loc}}\) and
\(L\in\operatorname{Irr}(\mathscr T)\), strong commutation gives
\[
    x_p[L]_t
    =
    t^{\lambda_{p,L}}[L]_t x_p
\]
for some \(\lambda_{p,L}\in\ZZ\). Hence each \(x_p\) is normal in
\(\mathcal K_t(\mathscr T)\): explicitly, extend
\[
    \sigma_p([L]_t):=t^{\lambda_{p,L}}[L]_t
\]
linearly from the normalized simple basis. Associativity and the
preceding commutation identities give
\(x_pa=\sigma_p(a)x_p\) and show that \(\sigma_p\) is an algebra
automorphism. Notice that the exponents
\(\lambda_{p,L}\) may depend on \(L\); no common exponent is being
assigned to a sum of simple classes. The elements \(x_p\) pairwise
\(t\)-commute, and
\(\mathcal K_t(\mathscr T)\) is a domain because it embeds into the
PBW domain
\(\mathcal K_t(\mathscr C_{\DD}(\beta i))\). Therefore \(S\) is a
regular normal Ore set, and we may form
\[
    \mathcal K_t(\mathscr T)_{\mathrm{loc}}
    :=
    \mathcal K_t(\mathscr T)S^{-1}
    \simeq
    S^{-1}\mathcal K_t(\mathscr T).
\]

For a nonzero right fraction \(fs^{-1}\), where
\(f\in\mathcal K_t(\mathscr T)\) and \(s\in S\), define
\begin{equation}\label{eq:PBW-valuation-Ore-extension}
    \nu_{r+1}(fs^{-1})
    :=
    \nu_{r+1}(f)-\nu_{r+1}(s).
\end{equation}
To see that this is well defined, suppose that
\(    fs^{-1}=gt^{-1}.
\)
By the Ore property and regularity, there exist \(u,v\in S\) such
that
\(
    su=tv\),
    \(    fu=gv.
\)
Applying
\eqref{eq:PBW-last-coordinate-multiplicativity} gives
\[
\begin{aligned}
    \nu_{r+1}(s)+\nu_{r+1}(u)
    &=
    \nu_{r+1}(t)+\nu_{r+1}(v),\\
    \nu_{r+1}(f)+\nu_{r+1}(u)
    &=
    \nu_{r+1}(g)+\nu_{r+1}(v).
\end{aligned}
\]
Subtracting the first equality from the second yields
\[
    \nu_{r+1}(f)-\nu_{r+1}(s)
    =
    \nu_{r+1}(g)-\nu_{r+1}(t).
\]
Thus \eqref{eq:PBW-valuation-Ore-extension} is independent of the
chosen right-fraction representative.

Finally, the normal automorphisms associated with the generators
\(x_p\) act diagonally, by powers of \(t\), on the normalized simple
basis. Distinct simple basis elements have distinct PBW leading
exponents by the PBW parametrization, so this diagonal action cannot
create cancellation at the maximal exponent. It therefore preserves
PBW leading exponents. Together with
\eqref{eq:PBW-last-coordinate-multiplicativity}, this shows that
\[
    \nu_{r+1}(ab)
    =
    \nu_{r+1}(a)+\nu_{r+1}(b)
\]
for all nonzero
\(a,b\in\mathcal K_t(\mathscr T)_{\mathrm{loc}}\).

\begin{proposition}
\label{pro:induction-upper-cluster-categorification}
Let \(\beta\) be a positive braid word and \(i\in I\). Suppose that
\(\mathscr D_{\DD}(\beta i)\) is a based categorification of
\(\mathcal U_t(\overrightarrow{\mathbf s}_t(\beta i))\). Then
\(\mathscr D_{\DD}(\beta)\) is a based categorification of
\(\mathcal U_t(\overrightarrow{\mathbf s}_t(\beta))\). In
particular, there is an isomorphism
\begin{equation}\label{eq:categorification-beta-i}
    \mathcal{K}_t\bigl(
        \mathscr D_{\DD}(\beta)
    \bigr)_{\mathrm{loc}}
    \xrightarrow{\ \sim\ }
    \mathcal U_t\bigl(
        \overrightarrow{\mathbf s}_t(\beta)
    \bigr).
\end{equation}
Every quantum cluster monomial with nonnegative frozen exponents is
the normalized \((q,t)\)-character of a real simple object of
\(\mathscr D_{\DD}(\beta)\). Arbitrary Laurent powers of frozen
variables are represented in the corresponding monoidal
localization.
\end{proposition}

\begin{proof}
Set \(w:=\delta(\beta)\). We consider separately the two cases in
the definition of the Demazure product.

Suppose first that \(\ell(ws_i)=\ell(w)+1\). Then
\(\delta(\beta i)=ws_i\), and the final letter \(i\) belongs to the
distinguished reduced subword. Hence it produces no trivalent
vertex, and the seeds
\(\overrightarrow{\mathbf s}_t(\beta i)\) and
\(\overrightarrow{\mathbf s}_t(\beta)\) are canonically identified,
including their frozen vertex sets.

Corollary~\ref{cor:right-weave-cluster-variables-level-one} places every
cluster variable of \(\overrightarrow{\mathbf s}_t(\beta)\), and the
inverse of every frozen variable, in
\(\mathcal{K}_t(\mathscr D_{\DD}(\beta))_{\mathrm{loc}}\). Hence
Theorem~\ref{thm:quantum-A-U-braid-seeds} gives
\begin{equation}\label{eq:length-increase-upper-in-Kt}
    \mathcal U_t(\overrightarrow{\mathbf s}_t(\beta))
    \subseteq
    \mathcal{K}_t(\mathscr D_{\DD}(\beta))_{\mathrm{loc}}.
\end{equation}

For the reverse inclusion, let \(S\) be simple in
\(\mathscr D_{\DD}(\beta)\). It is also simple in
\(\mathscr D_{\DD}(\beta i)\). By the assumed based categorification,
\([S]_t\) is an element of the common-triangular basis of \( \mathcal{U}_t(\overrightarrow{\mathbf{s}}_t(\beta i))\). Note that $\mathcal{U}_t(\overrightarrow{\mathbf{s}}_t(\beta i))=\mathcal{U}_t(\overrightarrow{\mathbf{s}}_t(\beta))$, therefore this basis element belongs to
\(\mathcal U_t(\overrightarrow{\mathbf s}_t(\beta))\). The same holds after
the allowed frozen shifts. Thus the reverse inclusion to
\eqref{eq:length-increase-upper-in-Kt} holds, and the restricted localized
simple basis is exactly the common triangular basis. Together with
Theorem~\ref{thm:monoidal-seed-beta}, this proves the based
categorification in the length-increasing case.

Suppose now that \(\ell(ws_i)=\ell(w)-1\), or equivalently,
\(\delta(\beta i)=w\). Let \(v_0\) be the trivalent vertex associated
with the final letter \(i\), and set
\(J:=J_{\beta,i,\mathrm{fr}}\). Let \(\mathscr T\) be the full
subcategory of \(\mathscr D_{\DD}(\beta i)\) consisting of objects
whose simple composition factors strongly commute with \(B_v\) for
every \(v\in J\).

The vertex \(v_0\) is frozen in
\(\overrightarrow{\mathbf s}_t(\beta i)\). By the assumed based
categorification, \(B_{v_0}\) strongly commutes with every simple
object of \(\mathscr D_{\DD}(\beta i)\). Lemma
\ref{lem:commuting-factors-last-vertex} therefore implies
\(\mathscr D_{\DD}(\beta)\subseteq\mathscr T\).

Let
\(\operatorname{Fr}_{J}
\overrightarrow{\mathbf s}_t(\beta i)\)
be the seed obtained by freezing the vertices in \(J\).
Theorem~\ref{thm:quantum-A-U-braid-seeds} shows that
\(\overrightarrow{\mathbf s}_t(\beta)\) is injective-reachable and has a
common triangular basis. Lemma
\ref{lem:Demazure-descent-seed-deletion} shows that the frozen seed is
obtained from this seed by adjoining the isolated frozen vertex
\(v_0\); hence it is injective-reachable. The modules indexed by \(J\) are
pairwise strongly commuting cluster-variable modules. Let
\(\mathbf L_{\beta i}\) be the common triangular basis supplied by
the assumed based categorification. Since
\(\mathbf L_{\beta i}\) is triangular with respect to the seed
\(\overrightarrow{\mathbf s}_t(\beta i)\),
\cite[Remark~5.9]{qin2024analogs} gives Condition~\textup{(T)} for
the cluster-variable modules \(B_v\), \(v\in J\). By
\cite[Remark~5.7]{qin2024analogs}, Condition~\textup{(T)} implies
Condition~\textup{(C)}. Consequently, the definition of
\(\mathscr T\) agrees with Qin's commuting subcategory associated
with \(J\), and \cite[Lemmas~5.11--5.12]{qin2024analogs} shows that
\(\mathscr T\) is full, abelian, and monoidal.

Finally, \(\mathbf L_{\beta i}\) is
\(\mathcal M^\circ\)-pointed by the assumed based
categorification and Theorem~\ref{thm:quantum-A-U-braid-seeds}, and the frozen seed $\operatorname{Fr}_{J}
        \overrightarrow{\mathbf s}_t(\beta i)$ is injective-reachable. Therefore
\cite[Theorem~5.17]{qin2024analogs}, applied to \(J\), gives a based
categorification
\begin{equation}\label{eq:Kt-T-frozen-upper}
    \mathcal{K}_t(\mathscr T)_{\mathrm{loc}}
    \xrightarrow{\ \sim\ }
    \mathcal U_t\bigl(
        \operatorname{Fr}_{J}
        \overrightarrow{\mathbf s}_t(\beta i)
    \bigr).
\end{equation}
The basis on the right-hand side is
\(\operatorname{Fr}_{J}(\mathbf L_{\beta i})\). By
\cite[Theorem~3.17]{qin2024analogs}, freezing preserves the common
triangular basis. Thus this is a common triangular basis of the
frozen upper algebra, as required for the degree-zero restriction
in Lemma~\ref{lem:degree-zero-triangular-basis}.
By \eqref{eq:Demazure-descent-frozen-set}, the frozen vertices used in
this localization form exactly the set
\(V_{\beta i,\mathrm{fr}}\cup J
=V_{\beta,\mathrm{fr}}\cup\{v_0\}\).
The corresponding frozen modules strongly commute with every simple
object of \(\mathscr T\): this holds by the definition of
\(\mathscr T\) for the vertices in \(J\), by the assumed based
categorification for the inherited frozen vertices, and by the universal
commutation of \(B_{v_0}\). Consequently,
the PBW last-coordinate valuation $\nu_{r+1}$ above applies to this
localization. In particular, the function \(\nu_{r+1}\) used below
is well defined and multiplicative on all Ore fractions.

By Lemma~\ref{lem:Demazure-descent-seed-deletion}, deleting \(v_0\)
from
\(\operatorname{Fr}_{J}
\overrightarrow{\mathbf s}_t(\beta i)\)
gives \(\overrightarrow{\mathbf s}_t(\beta)\), and
\eqref{eq:Demazure-descent-zero-row} supplies the zero-row hypothesis.
Lemma
\ref{lem:upper-skew-Laurent-extension} therefore identifies the
right-hand side of \eqref{eq:Kt-T-frozen-upper} with
\(\mathcal U_t(\overrightarrow{\mathbf s}_t(\beta))
[x_{v_0}^{\pm1};\sigma]\).

We next characterize the simple objects of \(\mathscr T\) that
belong to \(\mathscr D_{\DD}(\beta)\). Compatibility of the PBW
parametrizations for \(\beta\) and \(\beta i\) gives
\begin{equation}\label{eq:Irr-D-beta-in-T}
    \operatorname{Irr}\bigl(
        \mathscr D_{\DD}(\beta)
    \bigr)
    =
    \left\{
        S\in\operatorname{Irr}(\mathscr T)
        \,\middle|\,
        a_{r+1}^{\beta i}(S)=0
    \right\},
\end{equation}
where \(a_{r+1}^{\beta i}(S)\) denotes the \((r+1)\)-st component of the Lusztig parameter \(\mathbf a^{\beta i}(S)\) of the simple module \(S\), computed with respect to the word \(\beta i\).
Indeed, the inclusion from left to right follows from
\(\mathscr D_{\DD}(\beta)\subseteq\mathscr T\). Conversely, if
\(S\in\operatorname{Irr}(\mathscr T)\) has zero
\((r+1)\)-st PBW coordinate, then the Lusztig parametrization implies
\(S\in\mathscr C_\DD(\beta)\). Since
\(\mathscr T\subseteq\mathscr D_{\DD}(\beta i)\), we also have
\(S\in\mathscr C_{\DD}[1,\infty)\), and hence
\(S\in\mathscr D_{\DD}(\beta)\). In particular,
\(\nu_{r+1}([S]_t)=0\) for every
\(S\in\operatorname{Irr}(\mathscr D_{\DD}(\beta))\).

By Corollary
\ref{cor:right-weave-cluster-variables-level-one} and
Theorem~\ref{thm:quantum-A-U-braid-seeds}, we have
\begin{equation}\label{eq:upper-contained-localized-Kt}
    \mathcal U_t\bigl(
        \overrightarrow{\mathbf s}_t(\beta)
    \bigr)
    \subseteq
    \mathcal{K}_t\bigl(
        \mathscr D_{\DD}(\beta)
    \bigr)_{\mathrm{loc}}.
\end{equation}
Because the right lexicographic order compares the
\((r+1)\)-st coordinate first, every PBW term occurring in the class
of a simple object of \(\mathscr D_{\DD}(\beta)\) has zero final
coordinate: its leading coordinate is zero by
\eqref{eq:Irr-D-beta-in-T}.
The same is true for the frozen modules used in the
localization. It follows that
\(\nu_{r+1}(x)=0\) for every nonzero
\(x\in\mathcal U_t(\overrightarrow{\mathbf s}_t(\beta))\).
On the other hand, Proposition~\ref{pro:LusztigparaBv} gives
\(\nu_{r+1}(x_{v_0})=1\).

Let \(S\) be a simple object of
\(\mathscr D_{\DD}(\beta)\subseteq\mathscr T\). Under the based
categorification \eqref{eq:Kt-T-frozen-upper}, its normalized class
\([S]_t\) is pointed with respect to
\(\operatorname{Fr}_{J}
\overrightarrow{\mathbf s}_t(\beta i)\). Thus its Laurent expansion
has the form
\begin{equation*}
    [S]_t
    =
    X^{\mathbf g}
    +
    \sum_{\mathbf n\in\NN^{V_{\beta,\mathrm{ex}}}\setminus\{0\}}
    c_{\mathbf n}(t)
    X^{\mathbf g-\widetilde B\mathbf n},
\end{equation*}
where \(\mathbf n\) is supported on the exchangeable vertices. Since
\(b_{v_0k}=0\) for every exchangeable vertex \(k\), the
\(v_0\)-coordinate of \(\widetilde B\mathbf n\) is zero. Therefore,
all Laurent monomials occurring in \([S]_t\) have the same
\(x_{v_0}\)-degree.

By the uniqueness of the skew Laurent expansion in
Lemma~\ref{lem:upper-skew-Laurent-extension}, there exist a nonzero
element
\(x\in\mathcal U_t(\overrightarrow{\mathbf s}_t(\beta))\)
and an integer \(d\in\ZZ\) such that
\begin{equation}\label{eq:S-skew-Laurent-homogeneous}
    [S]_t=x\,x_{v_0}^{d}.
\end{equation}
Applying \(\nu_{r+1}\) gives
\begin{equation*}\nu_{r+1}([S]_t)
=\nu_{r+1}(x)+d\nu_{r+1}(x_{v_0})=d.\end{equation*}
Since \(\nu_{r+1}([S]_t)=0\), we obtain \(d=0\). Therefore,
\([S]_t\in
\mathcal U_t(\overrightarrow{\mathbf s}_t(\beta))\).

The classes of simple objects form a basis of the quantum
Grothendieck ring, and the inverses of the frozen classes belong to
the upper quantum cluster algebra. It follows that
\begin{equation}\label{eq:Kt-D-beta-contained-upper}
    \mathcal{K}_t\bigl(
        \mathscr D_{\DD}(\beta)
    \bigr)_{\mathrm{loc}}
    \subseteq
    \mathcal U_t\bigl(
        \overrightarrow{\mathbf s}_t(\beta)
    \bigr).
\end{equation}
Together with \eqref{eq:upper-contained-localized-Kt}, this proves
the algebra isomorphism \eqref{eq:categorification-beta-i}.

It remains to verify the basis and strong-commutation conditions.
The frozen vertices of
\(\overrightarrow{\mathbf s}_t(\beta)\) consist of the vertices in
\(V_{\beta i,\mathrm{fr}}\setminus\{v_0\}\), together with the
vertices in \(J\). The modules associated with the former strongly
commute with every simple object of
\(\mathscr D_{\DD}(\beta)\) by the assumed based categorification of
\(\mathscr D_{\DD}(\beta i)\). For the
vertices in \(J\), the same conclusion follows from
Lemma~\ref{lem:commuting-factors-last-vertex}.

Finally, choose the usual nonredundant localized simple basis of
\(\mathcal K_t(\mathscr D_{\DD}(\beta))_{\mathrm{loc}}\). By
\eqref{eq:Irr-D-beta-in-T} and the calculation
\eqref{eq:S-skew-Laurent-homogeneous}, each of its elements belongs to
the degree-zero common triangular basis in
Lemma~\ref{lem:degree-zero-triangular-basis}. The algebra equality
\eqref{eq:categorification-beta-i} shows that these localized simple
classes span the entire degree-zero upper quantum cluster algebra.
They are therefore a spanning subset of a basis and must coincide with
that basis. This proves the basedness assertion.

By Corollary
\ref{cor:right-weave-cluster-variables-level-one} and
Theorem~\ref{thm:monoidal-seed-beta}, the variables in every cluster
are represented by a pairwise strongly commuting family of real
simple objects. Their normalized tensor products are therefore real
and simple. This proves that
\(\mathscr D_{\DD}(\beta)\) is a based categorification of
\(\mathcal U_t(\overrightarrow{\mathbf s}_t(\beta))\).
\end{proof}

\subsubsection{Braid words related by braid moves}
In this subsection, we prove that any two positive braid words
representing the same element \(b\in\operatorname{Br}^{+}\) define
the same cluster algebra and that their associated monoidal seeds are
mutation equivalent. 

\begin{proposition}\label{lem:mutation-right-inductive-seed}
Under the assumption~\ref{assumption}, suppose that \(\beta\) and \(\beta'\)
are related by a single commutation move or a single simply-laced
braid move. Then
\begin{equation}\label{eq:braid-move-category-equality}
    \mathscr D_{\DD}(\beta)
    =
    \mathscr D_{\DD}(\beta').
\end{equation}
Moreover, there is a relabeling
\(
    \rho\colon V_\beta\xrightarrow{\sim}V_{\beta'}
\)
such that one of the following holds:
\begin{enumerate}
\item one has
\(
    \rho\!\left(
        \overrightarrow{\mathbf s}_t(\beta)
    \right)
    =
    \overrightarrow{\mathbf s}_t(\beta');
\)
in this case,
\[
    B_u^\beta\simeq B_{\rho(u)}^{\beta'},
    \qquad
    \rho(\Lambda_\beta)=\Lambda_{\beta'};
\]
\item there is a mutable vertex \(k\in V_\beta\) such that
\begin{equation}\label{eq:braid-move-quantum-seed}
    \rho\!\left(
        \mu_k\,
        \overrightarrow{\mathbf s}_t(\beta)
    \right)
    =
    \overrightarrow{\mathbf s}_t(\beta').
\end{equation}
In this case,
\[
    \mu_k(B_k^\beta)
    \simeq
    B_{\rho(k)}^{\beta'},
    \qquad
    B_u^\beta
    \simeq
    B_{\rho(u)}^{\beta'}
    \quad(u\neq k),
\]
and
\begin{equation}\label{eq:braid-move-compatible-form}
    \rho\!\left(
        \mu_k(\Lambda_\beta)
    \right)
    =
    \Lambda_{\beta'}.
\end{equation}
\end{enumerate}
\end{proposition}

\begin{proof}
Since \(\beta\) and \(\beta'\) represent the same positive braid
element \(b\), \cite[Corollary~5.34]{kashiwara2025monoidal} gives
\[
    \mathscr C_{\DD}(\beta)
    =
    \mathscr C_{\DD}(b)
    =
    \mathscr C_{\DD}(\beta').
\]
Intersecting these categories with the fixed level subcategory
\(\mathscr C_{\DD}[1,\infty)\) yields
\eqref{eq:braid-move-category-equality}. Hence the
expression-independent based isomorphism \(\Phi_{\DD}\) realizes the
two quantum Grothendieck rings in the same based algebra
\(\widehat{\cA}(b)\).

By \cite[Proposition~4.46 and Figures~9--10]{casals2025cluster}, the
two right-inductive ice quivers are either identified by a relabeling
or related, up to relabeling, by a single mutation at the
distinguished local vertex. In the first case, the corresponding
Lusztig cycles are identified. The PBW transition formula in
\cite[Theorem~4.12]{bi2026towards} therefore identifies their
Lusztig parameters and normalized global-basis elements. Since
\(\Phi_{\DD}\) is based, the corresponding simple modules are
isomorphic. Their quasi-commutation exponents also agree, proving the
first case.

We now consider the nontrivial \(3\)-move. After reindexing, write
the affected subwords as
\[
    (i_k,i_{k+1},i_{k+2})
    =
    (i,j,i)
    \longleftrightarrow
    (j,i,j).
\]
Let \(\gamma\) be the prefix before these positions. The nontrivial
case is \(\delta(\gamma iji)=\delta(\gamma)\), so all three
positions are solid. Denote their vertices temporarily by
\(v_0,v_1,v_2\), with
\(q_{v_r}^{\beta}=k+r\) for \(r=0,1,2\), and use primed symbols
for \(\beta'\). The correspondence of vertex labels is
\begin{equation}\label{eq:braid-move-vertex-labels}
    \rho(v_0)=v_0',\qquad
    \rho(v_1)=v_2',\qquad
    \rho(v_2)=v_1',
\end{equation}
with all other positional labels fixed. In the formulas below we
abbreviate \(v_r\) by \(k+r\); \(E_{k+r}\) always denotes a
PBW coordinate, whereas \(\mu_k\) means mutation at \(v_0\).

Here is the local calculation specifying these labels as in \cite[Figure~9-10]{casals2025cluster}. They are Demazure weaves: $(i_{k-3}i_{k-2}i_{k-1}i_ki_{k+1}i_{k+2})=(ijiiji)\to (iji)$. Read the last three
incoming edge weights at $(i_{k-3}i_{k-2}i_{k-1})$ of an earlier cycle \(\gamma_u\), \(u<k\),
as \((a_u,b_u,c_u)\). On replacing this reduced suffix by \(jij\),
these weights become
\[
 (a_u',b_u',c_u')
 =(b_u+c_u-p_u,\,p_u,\,a_u+b_u-p_u),
 \qquad p_u:=\min(a_u,c_u).
\]
Apply the six-valent rule
\eqref{eq:Lusztig-rule-six-valent} between successive insertions
and the trivalent rule at each solid step. The northwest edge
coefficients in the three recursions are
\begin{equation}\label{eq:braid-local-cycle-table}
\begin{array}{c|ccc}
 &k&k+1&k+2\\ \hline
 \gamma_u^\beta(e_{\bullet,\mathrm{nw}})
     &c_u&a_u+b_u&b_u\\
 \gamma_u^{\beta'}(e_{\bullet,\mathrm{nw}})
     &a_u+b_u-p_u&b_u+c_u&p_u
\end{array}
\end{equation}
For example, before the move the successive weight triples of an
earlier cycle, just after each solid reduction, are
\((a_u,b_u,0)\), \((b_u,0,0)\), and \((0,0,0)\).
Their next braid transforms give exactly the first row of the
table; applying the same computation to \((a_u',b_u',c_u')\)
gives the second row.

The three cycles created in this disk have southern triples
\((0,0,0)\), \((1,0,0)\), and \((0,0,1)\), respectively, in
the \(iji\) convention. The braid transformation to the common
\(jij\) boundary interchanges the last two triples. This gives
\eqref{eq:braid-move-vertex-labels}. The first cycle has zero
output and is the mutable vertex involved in the local move.
With the exchange-matrix sign convention of
\eqref{eq:quiver-convention}, the local intersection formulas give
\begin{equation}\label{eq:braid-local-exchange-column}
    \varepsilon_{k,u}=a_u-c_u\ (u<k),\qquad
    \epsilon_{k,k+1}=-1,\quad \epsilon_{k,k+2}=1,\quad
    \epsilon_{k,u}=0\ (u>k+2).
\end{equation}
Indeed, in the order \((\gamma_k,\gamma_u)\), the contributions
at the three solid reductions and the two intervening braid
vertices are, successively, \(-c_u,a_u,0,0,0\).
The short cycle has zero boundary output, so
\(\varepsilon_{k,u}=a_u-c_u\).
The same local determinants give
\(\varepsilon_{k,k+1}=-1\) and
\(\varepsilon_{k,k+2}=1\).
These statements fix the vertex correspondence independently of
the labels used in a picture of the reduced-word strands.
By the local weave move of
\cite[Proposition~4.46]{casals2025cluster}, with this correspondence,
\begin{equation}\label{eq:braid-move-exchange-matrix}
    \rho\!\left(\mu_k(B_\beta)\right)=B_{\beta'}.
\end{equation}
Set \(v:=\rho(k)=k\).

Let
\[
    \Psi_{\beta'}^\beta\colon
    \ZZ_{\geq0}^{|\beta|}
    \longrightarrow
    \ZZ_{\geq0}^{|\beta'|}
\]
be the PBW transition map. On the affected coordinates
\(k,k+1,k+2\), it is given by
\begin{equation}\label{eq:braid-move-PBW-transition}
    (a_1,a_2,a_3)
    \longmapsto
    \bigl(
        a_2+a_3-p,\,
        p,\,
        a_1+a_2-p
    \bigr),
    \qquad
    p:=\min\{a_1,a_3\},
\end{equation}
and it fixes all remaining coordinates. We claim that
\begin{equation}\label{eq:Psibetabeta'}
    \Psi_{\beta'}^\beta(\eta_u^\beta)
    =
    \eta_{\rho(u)}^{\beta'}
    \qquad (u\neq k),
\end{equation}
Here \(\rho\) is the vertex relabeling
\eqref{eq:braid-move-vertex-labels}; in positional notation it
exchanges \(k+1\) and \(k+2\) and fixes \(k\).
We first verify this claim for vertices preceding the affected
position. If \(q_u^\beta<k\), then Lemma~\ref{lem:maxindex} gives
\[
    \operatorname{supp}(\eta_u^\beta)
    \subseteq
    \{1,\ldots,k-1\}.
\]
Thus all affected coordinates of \(\eta_u^\beta\) vanish, so
\(\Psi_{\beta'}^\beta\) fixes this vector. Since the two weaves agree
before the local braid-move disk, we obtain
\begin{equation}\label{eq:earlier-eta-under-braid-transition}
    \Psi_{\beta'}^\beta(\eta_u^\beta)
    =
    \eta_{u}^{\beta'}
    \qquad (u<k).
\end{equation}

We next identify the variable obtained by mutation at \(k\). Set
\[
    A_k'
    :=
    \mu_k(A_k^\beta),
    \qquad
    \xi_k
    :=
    \mathbf a^\beta(A_k').
\]
By Lemma~\ref{lem:Lusztig-parameter-mutation},
\begin{equation}\label{eq:braid-move-mutated-parameter}
    \xi_k
    =
    \max_{\mathrm r}
    \left\{
        \sum_u[\epsilon_{k,u}]_+\eta_u^\beta,\,
        \sum_u[-\epsilon_{k,u}]_+\eta_u^\beta
    \right\}
    -
    \eta_k^\beta.
\end{equation}
Here we have used
\(
    b_{u,k}=\epsilon_{u,k}=-\epsilon_{k,u}
\)
in Definition~\ref{def:mutlusztig}.
We now determine this maximum. The exchange column in
\eqref{eq:braid-local-exchange-column} shows that the larger leading
parameter is
\[
    \eta_{k+2}^\beta
    +
    \sum_{u<k}[\epsilon_{k,u}]_+\eta_u^\beta.
\]
Indeed, the two parameters agree in every coordinate greater than
\(k+2\), while the displayed parameter has coefficient \(1\) at
\(k+2\) and the other has coefficient \(0\). It is therefore larger
with respect to \(\prec_{\mathrm r}\). Consequently,
\eqref{eq:braid-move-mutated-parameter} becomes
\begin{equation}\label{eq:explicit-mutated-parameter}
    \xi_k
    =
    \eta_{k+2}^\beta
    +
    \sum_{u<k}[\epsilon_{k,u}]_+\eta_u^\beta
    -
    \eta_k^\beta.
\end{equation}

The defining recursion for \(\eta^\beta\) and the first row of
\eqref{eq:braid-local-cycle-table} give
\[
    \eta_{k+2}^\beta
    =
    \eta_k^\beta
    +
    E_{k+2}
    +
    \sum_{u<k}
        \gamma_u^\beta(e_{k+2,\mathrm{nw}})
        \eta_u^\beta.
\]
Substituting this identity into
\eqref{eq:explicit-mutated-parameter}, we obtain
\begin{equation}\label{eq:mutated-parameter-before-transition}
    \xi_k
    =
    E_{k+2}
    +
    \sum_{u<k}
        \left(
            \gamma_u^\beta(e_{k+2,\mathrm{nw}})
            +
            [\epsilon_{k,u}]_+
        \right)
        \eta_u^\beta.
\end{equation}

For an earlier vertex \(u<k\), equations
\eqref{eq:braid-local-cycle-table} and
\eqref{eq:braid-local-exchange-column} give
\[
    \gamma_u^\beta(e_{k+2,\mathrm{nw}})=b_u,
    \qquad \epsilon_{k,u}=a_u-c_u,
    \qquad
    \gamma_u^{\beta'}(e_{v,\mathrm{nw}})
       =a_u+b_u-\min(a_u,c_u).
\]
Therefore,
\begin{align*}
    \gamma_u^\beta(e_{k+2,\mathrm{nw}})
    +
    [\epsilon_{k,u}]_+
    &=
    b_u+[a_u-c_u]_+\\
    &=
    b_u+a_u-\min\{a_u,c_u\}\\
    &=
    \gamma_{u}^{\beta'}(e_{v,\mathrm{nw}}).
\end{align*}
Thus
\begin{equation}\label{eq:mutated-parameter-local-cycle-form}
    \xi_k
    =
    E_{k+2}
    +
    \sum_{u<k}
        \gamma_{u}^{\beta'}(e_{v,\mathrm{nw}})
        \eta_u^\beta.
\end{equation}

The local parameter \((0,0,1)\) becomes \((1,0,0)\) under
\eqref{eq:braid-move-PBW-transition}. Hence
\[
    \Psi_{\beta'}^\beta(E_{k+2})=E_k.
\]
Although \(\Psi_{\beta'}^\beta\) is only piecewise linear, no
additivity is required here: the summation term in
\eqref{eq:mutated-parameter-local-cycle-form} is supported in
\(\{1,\ldots,k-1\}\), where the transition map is the identity.
Using \eqref{eq:earlier-eta-under-braid-transition}, we obtain
\begin{align*}
    \Psi_{\beta'}^\beta(\xi_k)
    &=
    E_k
    +
    \sum_{u<k}
        \gamma_{u}^{\beta'}(e_{v,\mathrm{nw}})
        \eta_{u}^{\beta'}\\
    &=
    \eta_k^{\beta'}.
\end{align*}
The last equality is precisely the defining recursion
\eqref{eq:etav} for \(k\). Hence
\begin{equation}\label{eq:braid-move-mutated-eta-transition}
    \Psi_{\beta'}^\beta(\xi_k)
    =
    \eta_{k}^{\beta'}.
\end{equation}

By \cite[Theorem~4.12]{bi2026towards}, the two PBW parametrizations of
the global basis satisfy
\[
    \widetilde G(\beta,\mathbf a)
    =
    \widetilde G\left(
        \beta',
        \Psi_{\beta'}^\beta(\mathbf a)
    \right).
\]
It follows that
\begin{align*}
    \mu_k(A_k^\beta)
    &=
    \widetilde G(\beta,\xi_k)\\
    &=
    \widetilde G\left(
        \beta',
        \Psi_{\beta'}^\beta(\xi_k)
    \right)\\
    &=
    \widetilde G\left(
        \beta',
        \eta_{k}^{\beta'}
    \right)\\
    &=
    A_{k}^{\beta'}.
\end{align*}

It remains to prove \eqref{eq:Psibetabeta'} for the vertices at and
after the remaining affected positions. First, the defining recursion
in the \(\beta'\)-weave gives
\begin{align*}
    \eta_{k+2}^{\beta'}
    &=
    E_{k+2}
    +
    \eta_k^{\beta'}
    +
    \sum_{u<k}
        \gamma_{u}^{\beta'}(e_{k+2,\mathrm{nw}})
        \eta_{u}^{\beta'}\\
    &=
    E_{k+2}
    +
    E_k
    +
    \sum_{u<k}
        \left(
            \gamma_{u}^{\beta'}(e_{k,\mathrm{nw}})
            +
            \gamma_{u}^{\beta'}(e_{k+2,\mathrm{nw}})
        \right)
        \eta_u^\beta,
\end{align*}
where the second equality follows from
\eqref{eq:earlier-eta-under-braid-transition} and the defining
recursion for \(\eta_v^{\beta'}\).

Let \(b_u'=\min\{a_u,c_u\}\) and
\(c_u'=a_u+b_u-b_u'\). A direct local calculation gives
\[
    \gamma_{u}^{\beta'}(e_{v,\mathrm{nw}})
    +
    \gamma_{u}^{\beta'}(e_{k+2,\mathrm{nw}})
    =
    b_u'+c_u'
    =
    b_u+a_u
    =
    \gamma_u^\beta(e_{k+1,\mathrm{nw}}).
\]
Therefore,
\begin{equation}\label{eq:etak+2}
    \eta_{k+2}^{\beta'}
    =
    E_{k+2}
    +
    E_k
    +
    \sum_{u<k}
        \gamma_u^\beta(e_{k+1,\mathrm{nw}})
        \eta_u^\beta
    =
    \Psi_{\beta'}^\beta(\eta_{k+1}^\beta).
\end{equation}

Similarly, if \(a_u'=b_u+c_u-b_u'\), then
\begin{align}
    \eta_{k+1}^{\beta'}
    &=
    E_{k+1}
    +
    \sum_{u<k}
        (a_u'+b_u')
        \eta_u^\beta \notag\\
    &=
    E_{k+1}
    +
    \sum_{u<k}
        (b_u+c_u)
        \eta_u^\beta \notag\\
    &=
    \Psi_{\beta'}^\beta(\eta_{k+2}^\beta).
    \label{eq:etak+1}
\end{align}

Finally, let \(u>k+2\). We make explicit why the piecewise
linearity of \(\Psi_{\beta'}^\beta\) causes no difficulty. Let
\[
    H
    :=
    \left\{
        \mathbf a\in\ZZ^{|\beta|}
        \,\middle|\,
        a_k=a_{k+2}
    \right\}.
\]
If the affected coordinates of \(\mathbf a\in H\) are
\((x,y,x)\), then
\eqref{eq:braid-move-PBW-transition} gives
\[
    (x,y,x)\longmapsto(y,x,y).
\]
Consequently, the restriction of
\(\Psi_{\beta'}^\beta\) to \(H\) is linear.

By Lemma~\ref{lem:maxindex}, the affected coordinates of
\(\eta_w^\beta\) vanish for \(w<k\). Moreover, the local recursions
used in \eqref{eq:etak+2} and \eqref{eq:etak+1} give
\[
    \left.\eta_{k+1}^\beta\right|_{\{k,k+1,k+2\}}
    =(0,1,0),
    \qquad
    \left.\eta_{k+2}^\beta\right|_{\{k,k+1,k+2\}}
    =(1,0,1).
\]
Thus all these vectors belong to \(H\).

We now proceed by induction on \(u>k+2\). The cycle
\(\gamma^\beta_k\) has weight \(1\) at its origin. The local calculation above
transports it to the incoming edge at \(k+2\), where it meets
the fresh edge of weight \(0\). The trivalent rule gives
\(\min(1,0)=0\), so this cycle $\gamma^\beta_{k}$ has zero weight on every edge
leaving the disk. Subsequent Lusztig rules preserve the zero
output. It therefore contributes nothing to any later recursion:
\[
    \gamma_k^\beta(e_{u,\mathrm{nw}})=0.
\]
Hence
\begin{equation}\label{eq:later-eta-recursion-braid-move}
    \eta_u^\beta
    =
    E_{q_u^\beta}
    +
    \sum_{\substack{w<u\\w\neq k}}
        \gamma_w^\beta(e_{u,\mathrm{nw}})
        \eta_w^\beta.
\end{equation}
Here \(q_u^\beta\notin\{k,k+1,k+2\}\), so
\(E_{q_u^\beta}\in H\). By the induction hypothesis, every vector
occurring in the sum also belongs to \(H\). Therefore
\(\eta_u^\beta\in H\), and
\(\Psi_{\beta'}^\beta\) may be applied linearly to
\eqref{eq:later-eta-recursion-braid-move}.

The identification of the Lusztig cycles outside the local
braid-move disk gives
\[
    q_{u}^{\beta'}=q_u^\beta,
    \qquad
    \gamma_w^\beta(e_{u,\mathrm{nw}})
    =
    \gamma_{\rho(w)}^{\beta'}
        (e_{u,\mathrm{nw}})
    \qquad(w\neq k).
\]
Using \eqref{eq:earlier-eta-under-braid-transition},
\eqref{eq:etak+2}, \eqref{eq:etak+1}, and the induction hypothesis,
we therefore obtain
\begin{align*}
    \Psi_{\beta'}^\beta(\eta_u^\beta)
    &=
    E_{q_{\rho(u)}^{\beta'}}
    +
    \sum_{\substack{w<u\\w\neq k}}
        \gamma_{\rho(w)}^{\beta'}
            (e_{u,\mathrm{nw}})
        \Psi_{\beta'}^\beta(\eta_w^\beta)\\
    &=
    E_{q_{\rho(u)}^{\beta'}}
    +
    \sum_{\substack{w<u\\w\neq k}}
        \gamma_{\rho(w)}^{\beta'}
            (e_{u,\mathrm{nw}})
        \eta_{\rho(w)}^{\beta'}\\
    &=
    \eta_{u}^{\beta'}.
\end{align*}
Thus all the parameters occurring in this induction lie on the
common linearity wall \(a_k=a_{k+2}\), so no unjustified additivity
of the piecewise-linear map is used. This completes the proof of
\eqref{eq:Psibetabeta'}.

For every \(u\neq k\), equation~\eqref{eq:Psibetabeta'} and the PBW
transition formula now give
\[
    A_u^\beta
    =
    \widetilde G(\beta,\eta_u^\beta)
    =
    \widetilde G\left(
        \beta',
        \Psi_{\beta'}^\beta(\eta_u^\beta)
    \right)
    =
    A_{\rho(u)}^{\beta'}.
\]
Since \(\Phi_{\DD}\) is based, these identities imply
\[
    \mu_k(B_k^\beta)
    \simeq
    B_{\rho(k)}^{\beta'},
    \qquad
    B_u^\beta
    \simeq
    B_{\rho(u)}^{\beta'}
    \quad (u\neq k).
\]
Thus the mutated toric frame is identified, after relabeling, with
the toric frame of
\(\overrightarrow{\mathbf s}_t(\beta')\).

Finally, quantum mutation transforms the quasi-commutation matrix of
the initial toric frame into \(\mu_k(\Lambda_\beta)\). Since the
mutated toric-frame elements coincide, after relabeling, with those
of \(\overrightarrow{\mathbf s}_t(\beta')\), their
quasi-commutation matrices agree. Hence
\[
    \rho\!\left(\mu_k(\Lambda_\beta)\right)
    =
    \Lambda_{\beta'}.
\]
Together with \eqref{eq:braid-move-exchange-matrix}, this proves
\eqref{eq:braid-move-quantum-seed} and
\eqref{eq:braid-move-compatible-form}. The reverse move
\(jij\to iji\) follows from the same argument after interchanging
\(i\) and \(j\), or equivalently by applying the inverse PBW
transition map.
\end{proof}

\subsubsection{The initial case $\boldsymbol{\Delta}\beta$} In this section, we will prove the based categorification for the braid word $\boldsymbol{\Delta}\beta$.
\begin{lemma}\label{lem:commuting-with-PBW-heads}
Let \(M\) be a real simple object of \(\mathscr C_{\DD}(\kappa)\).
If \(M\) strongly commutes with every affine cuspidal module
\(C_k^{\DD,\kappa}\), then it strongly commutes with every simple
object of \(\mathscr C_\DD(\kappa)\).
\end{lemma}

\begin{proof}
Let \(N\) be simple and write its PBW factorization as in
Proposition~\ref{pro:lusztig-parameter-modules}. The ordered sequence
of cuspidal factors is normal. Subadditivity of \(\mathfrak d\) under
taking the simple head gives
\[
 0\leq\mathfrak d(M,N)
 \leq\sum_k a_k\,
       \mathfrak d(M,C_k^{\DD,\kappa})=0.
\]
Thus \(\mathfrak d(M,N)=0\). Since \(M\) is real, this is equivalent
to simplicity of \(M\otimes N\), hence to strong commutation.
\end{proof}

\begin{lemma}[The common triangular basis of a word algebra]
\label{lem:Qin-positive-braid-category-comparison}
Under Assumption~\ref{assumption}, let
\(\DD'=\DD_{\mathcal Q'}\) be associated with a \(Q\)-datum and let
\(\eta=(i_1,\ldots,i_r)\) be a positive braid word.  Identify
\(t^{1/2}=q^{-1/2}\), and consider the isomorphism
\begin{equation}\label{eq:word-bosonic-based-realization}
 \Xi_{\DD'}^\eta
 :=\Phi_{\DD'}\circ\Psi_{\DD'}^{q,\eta}\colon
 \overline{\mathcal A}_t\bigl(\mathbf s_t(\eta)\bigr)
 \xrightarrow{\sim}
 \widehat{\cA}_{\ZZ[q^{\pm1/2}]}(\eta).
\end{equation}
Here $\Psi_{\DD'}^{q,\eta}$ refers to \eqref{eq:quantum-word-seed-realization}.
Let \(\overline{\mathbf L}_\eta\) be the common triangular basis of
the word algebra with its frozen variables not inverted.  Then
\begin{equation}\label{eq:word-common-triangular-global-basis}
 \Xi_{\DD'}^\eta\bigl(\overline{\mathbf L}_\eta\bigr)
 =\{\widetilde G(\eta,\mathbf a)
       \mid\mathbf a\in\ZZ_{\geq0}^r\}.
\end{equation}
Consequently \(\Psi_{\DD'}^{q,\eta}\) identifies
\(\overline{\mathbf L}_\eta\) with all normalized simple classes
in \(\mathcal K_t(\mathscr C_{\DD'}(\eta))\), and the same assertion
holds for their nonredundant bases after localization at the frozen
variables.
\end{lemma}

\begin{proof}
\noindent\emph{Quantum conventions.}
We use Qin's parameter \(q_{\mathrm Q}=q=t^{-1}\)
\cite{qin2024analogs}. For his positive-word seed, the dictionary is
\begin{equation}\label{eq:Qin-word-quantum-dictionary}
 B^{\mathrm Q}=-\widetilde B_\eta,
 \qquad
 \Lambda^{\mathrm Q}=-\widetilde\Lambda_\eta,
 \qquad q_{\mathrm Q}^{1/2}=q^{1/2}=t^{-1/2}.
\end{equation}
This is \eqref{eq:Qin-quantum-convention-translation} for the word
seed.  The exchange-matrix identity follows entry by entry from
\cite[Lemma~6.3]{qin2024analogs} and \eqref{eq:B-beta}; the
compatible-form convention is that of
\cite[Section~4.1]{qin2024analogs}.  We choose the quantum word
seed of \cite[Section~6.2]{qin2024analogs} with this compatible
form; its standard-basis and Kazhdan--Lusztig statements apply to
this quantization.  Thus the toric products agree:
\(q_{\mathrm Q}^{\Lambda^{\mathrm Q}_{uv}/2}
 =t^{\widetilde\lambda_{uv}/2}\).
Initial variables are identified by their positions in \(\eta\).
This gives an identification of quantum word algebras and their
mutation maps, with the dominance-order translation of
Section~\ref{sec:cluster}.

\smallskip
\noindent\emph{Interval variables.}
For every \(\eta\)-box \([a,b]\), let \(W[a,b]\) denote Qin's
interval variable \cite[Definition~8.3]{qin2024analogs}. We claim
that
\begin{equation}\label{eq:Qin-interval-bosonic-comparison}
\begin{aligned}
    \Psi_{\DD'}^{q,\eta}\bigl(W[a,b]\bigr)
    =
    \bigl[M^{\DD',\eta}[a,b]\bigr]_t,\qquad
    \Xi_{\DD'}^\eta\bigl(W[a,b]\bigr)
    =
    D^\eta[a,b].
\end{aligned}
\end{equation}

Recall that
\[
    1(i_k)^+
    =
    \min\{s\leq k\mid i_s=i_k\}.
\]
Under the seed identification fixed above, Qin's initial interval
labeling gives
\[
    \widetilde X_k
    =
    W[1(i_k)^+,k].
\]
By the definition of
\(\Psi_{\DD'}^{q,\eta}\), we therefore have
\[
    \Psi_{\DD'}^{q,\eta}
       \bigl(W[1(i_k)^+,k]\bigr)
    =
    \bigl[
       M^{\DD',\eta}[1(i_k)^+,k]
    \bigr]_t.
\]
Thus the first identity in
\eqref{eq:Qin-interval-bosonic-comparison} holds for the initial
interval variables.

By
\cite[Definition~8.3 and Proposition~8.6]{qin2024analogs},
all the remaining variables \(W[a,b]\) are obtained from these
initial variables by successive box mutations. On the categorical
side, \cite[Lemma~8.12]{kashiwara2025monoidal} identifies the same
box moves with monoidal mutations whose new variables are the
normalized classes of the corresponding affine determinantial
modules \(M^{\DD',\eta}[a,b]\). These mutations are admissible by
Theorem~\ref{thm:Lambda-admissible-seed} and
\cite[Corollary~8.13]{kashiwara2025monoidal}.

Since \(\Psi_{\DD'}^{q,\eta}\) identifies the initial quantum seeds
and intertwines quantum mutation, induction along the box-mutation
sequence gives
\[
    \Psi_{\DD'}^{q,\eta}\bigl(W[a,b]\bigr)
    =
    \bigl[M^{\DD',\eta}[a,b]\bigr]_t
\]
for every \(\eta\)-box \([a,b]\). The equality is exact, rather
than only up to a power of \(q\), because
\eqref{eq:Qin-word-quantum-dictionary} identifies both the initial
toric frames and their compatible forms.

Applying \(\Phi_{\DD'}\) and using
\eqref{eq:determinantial-module-minor}, we obtain
\[
\begin{aligned}
    \Xi_{\DD'}^\eta\bigl(W[a,b]\bigr)
    =
    \Phi_{\DD'}\left(
       \Psi_{\DD'}^{q,\eta}\bigl(W[a,b]\bigr)
    \right)
    =
    \Phi_{\DD'}\left(
       \bigl[M^{\DD',\eta}[a,b]\bigr]_t
    \right)
    =
    D^\eta[a,b].
\end{aligned}
\]
This proves \eqref{eq:Qin-interval-bosonic-comparison}.
In particular, putting \(W_k:=W[k,k]\), we have
\[
    M^{\DD',\eta}[k,k]=C_k^{\DD',\eta},
    \qquad
    D^\eta[k,k]
    =
    \widetilde G(\eta,E_k)
    =
    \widetilde P_k^\eta.
\]
Hence
\begin{equation}\label{eq:Qin-singleton-bosonic-comparison}
\begin{aligned}
    \Psi_{\DD'}^{q,\eta}(W_k)
    =[C_k^{\DD',\eta}]_t,\qquad
    \Xi_{\DD'}^\eta(W_k)
    =\widetilde P_k^\eta.
\end{aligned}
\end{equation}

\smallskip
\noindent\emph{Standard monomials and bar normalization.}
We now compare normalizations.  Let \(f_k\) be the initial toric
degree vectors, set \(f_{-\infty}=0\), and write
\[
 \beta_k:=f_k-f_{k^-},\qquad
 g(\mathbf a):=\sum_{k=1}^r a_k\beta_k.
\]
By \cite[Lemma~8.4]{qin2024analogs}, \(\deg W_k=\beta_k\).
With \(\alpha_k^\eta=s_{i_1}\cdots s_{i_{k-1}}\alpha_{i_k}\),
the word-form formula \eqref{eq:Lambda-beta-word} gives
\begin{equation}\label{eq:Qin-root-degree-form}
 \Lambda^{\mathrm Q}(\beta_k,\beta_l)
 =-(\alpha_k^\eta,\alpha_l^\eta)\qquad(k<l).
\end{equation}
One can obtain this identity by taking the successive same-color
differences in both arguments of \eqref{eq:Lambda-beta-word},
using
\(\varpi_{i_k}-w_{\leq k}^\eta\varpi_{i_k}
 =\sum_{s\leq k,\ i_s=i_k}\alpha_s^\eta\), and then applying
\eqref{eq:Qin-word-quantum-dictionary}.

For \(\mathbf a\in\ZZ_{\geq0}^r\) put
\[
 s(\mathbf a):=\sum_{k<l}a_ka_l
                         (\alpha_k^\eta,\alpha_l^\eta),
 \qquad
 \nu(\mathbf a):=\operatorname{wt}P^\eta(\mathbf a),
 \qquad
 E(\mathbf a):=q^{-(\nu(\mathbf a),\nu(\mathbf a))/4}
                         P^\eta(\mathbf a).
\]
The normalized standard monomial of
\cite[Section~8.3]{qin2024analogs} is
\(\mathbf M(\mathbf a)=[W_1^{a_1}\cdots W_r^{a_r}]\), where the
brackets normalize its leading toric coefficient to one.
Equations \eqref{eq:Qin-singleton-bosonic-comparison} and
\eqref{eq:Qin-root-degree-form} imply the exact identity
\begin{equation}\label{eq:Qin-standard-bar-PBW}
 \Xi_{\DD'}^\eta\bigl(\mathbf M(\mathbf a)\bigr)
 =q^{s(\mathbf a)/2}
       (\widetilde P_1^\eta)^{a_1}\cdots
       (\widetilde P_r^\eta)^{a_r}
 =\overline{E(\mathbf a)}.
\end{equation}
For the last equality, the squared length of every PBW root is two,
so \((\nu(\mathbf a),\nu(\mathbf a))
 =2\sum_k a_k^2+2s(\mathbf a)\).
Equation~\eqref{eq:ordered-normalized-PBW-product} consequently gives
\(E(\mathbf a)=q^{-s(\mathbf a)/2}
 (\widetilde P_r^\eta)^{a_r}\cdots
 (\widetilde P_1^\eta)^{a_1}\).
Applying the bar anti-involution proves the equality, including its
scalar.

\smallskip
\noindent\emph{Triangular-basis uniqueness and localization.}
All terms in the PBW expansion of \(G(\eta,\mathbf a)\) have the
same weight.  Thus Proposition~\ref{prop:PBW-global-basis} and
\eqref{eq:normalized-PBW-global-basis-conversion} imply
\[
 \widetilde G(\eta,\mathbf a)
 \in E(\mathbf a)+
    \sum_{\mathbf b\prec_{\mathrm r}\mathbf a}q\ZZ[q]E(\mathbf b).
\]
Since \(\widetilde G(\eta,\mathbf a)\) is bar-invariant, applying
bar yields
\begin{equation}\label{eq:Qin-KL-from-PBW}
 \widetilde G(\eta,\mathbf a)
 \in\overline{E(\mathbf a)}+
    \sum_{\mathbf b\prec_{\mathrm r}\mathbf a}
       q^{-1}\ZZ[q^{-1}]\,\overline{E(\mathbf b)}.
\end{equation}
The right lexicographic order is Qin's reverse lexicographic order.
Theorem~8.10(4) of \cite{qin2024analogs}, using its
Kazhdan--Lusztig characterization with
\(\mathbf m=q^{-1/2}\ZZ[q^{-1/2}]\), now applies to
\eqref{eq:Qin-standard-bar-PBW} and \eqref{eq:Qin-KL-from-PBW}.
Here \(\mathbf m\) is the negative-power ideal defined in
\cite[Definition~2.6]{qin2024analogs} and used in the proof of that
theorem.  Uniqueness of the bar-invariant unitriangular basis gives
\[
 \Xi_{\DD'}^\eta
   \bigl(\mathbf L_{g(\mathbf a)}\bigr)
 =\widetilde G(\eta,\mathbf a).
\]
These are all basis elements by
\cite[Lemma~8.11]{qin2024analogs}, proving
\eqref{eq:word-common-triangular-global-basis}.

Finally \(\Phi_{\DD'}\) identifies normalized simple classes with
the normalized global basis by Theorem~\ref{thm:isoKt}.
This proves the categorical assertion.  By \cite[Lemma~5.12]{kashiwara2025monoidal} and
Lemma~\ref{lem:commuting-with-PBW-heads}, frozen variables
correspond to real simple objects strongly commuting with every
simple object of this word category.  The isomorphisms therefore extend across their Ore
localizations.  Multiplying the two bases by normalized frozen
Laurent monomials and identifying duplicate elements proves the
localized assertion.
\end{proof}

\begin{remark}
The categorical comparison above takes place in the fixed
category \(\mathscr C_{\DD'}(\eta)\).
At \(t=1\), the ascending products of the cuspidal classes in
\eqref{eq:Qin-singleton-bosonic-comparison} equal the classes of the
descending PBW standard modules, since the ordinary Grothendieck
ring is commutative.  Hence these two families of standards within
\(\mathscr C_{\DD'}(\eta)\) have exactly the same simple
composition factors, with multiplicities.  Their common Serre
closure is \(\mathscr C_{\DD'}(\eta)\) by
Proposition~\ref{pro:lusztig-parameter-modules}.
\end{remark}

\begin{lemma}[Garside base case]
\label{lem:Garside-base-case}
Under the assumption~\ref{assumption}, let
\[
    \boldsymbol{\Delta}=(j_1,\ldots,j_\ell),
    \qquad
    \ell=\ell(w_0),
\]
be a reduced word for the positive lift of \(w_0\), and let
\(\kappa=(i_1,\ldots,i_r)\) be a positive braid word. Set
\begin{equation}\label{eq:Delta-shifted-duality-datum}
    \DD^\Delta
    :=
    \mathscr S_{j_\ell}\cdots
    \mathscr S_{j_1}(\DD).
\end{equation}
The composition in \eqref{eq:Delta-shifted-duality-datum} is read
from right to left: first apply \(\mathscr S_{j_1}\), then
\(\mathscr S_{j_2}\), and so on.

Then, under the identification of the \(k\)-th position of
\(\kappa\) with the position \(\ell+k\) of
\(\boldsymbol{\Delta}\kappa\), one has
\begin{equation}\label{eq:Delta-kappa-base-identification}
    \overrightarrow{\mathbf s}_t
    (\boldsymbol{\Delta}\kappa)
    \simeq
    \mathbf s_t(\kappa),
    \qquad
    \mathscr D_{\DD}(\boldsymbol{\Delta}\kappa)
    =
    \mathscr C_{\DD^\Delta}(\kappa).
\end{equation}
The seed identification includes the exchange matrix, the compatible
form, and the frozen vertex set. Moreover,
\(\mathscr D_{\DD}(\boldsymbol{\Delta}\kappa)\) is a based
categorification of
\(
    \mathcal U_t\left(
        \overrightarrow{\mathbf s}_t
        (\boldsymbol{\Delta}\kappa)
    \right).
\)
More precisely, its normalized localized simple classes correspond
bijectively to the common triangular basis.
\end{lemma}

\begin{proof}
The longest-element identity for reflected duality data gives
\begin{equation}\label{eq:longest-element-duality-shift}
    L_i^{\DD^\Delta}
    \simeq
    \mathcal D L_{i^*}^{\DD}
    \qquad(i\in I),
\end{equation}
where \(w_0(\alpha_i)=-\alpha_{i^*}\). Hence
\(\DD^\Delta\) is, up to the relabeling \(i\mapsto i^*\), the
complete strong duality datum associated with the one-level shift of
the \(Q\)-datum \(\mathcal Q\).
Thus the categorical operation in
\eqref{eq:longest-element-duality-shift} has the following explicit
order: relabel the color by the Dynkin involution
\(i\mapsto i^*\), and then apply the right dual \(\mathcal D\),
which raises the level by one.

By the recursive definition of the reflected duality data, for
\(1\leq k\leq r\) one has
\[
    \DD_{\boldsymbol{\Delta}\kappa}^{(\ell+k-1)}
    =
    (\DD^\Delta)_\kappa^{(k-1)}.
\]
Consequently,
\begin{equation}\label{eq:Delta-kappa-root-module-identification}
    C_{\ell+k}^{\DD,\boldsymbol{\Delta}\kappa}
    \simeq
    C_k^{\DD^\Delta,\kappa}.
\end{equation}
Likewise, the affine determinantial modules attached to corresponding
boxes satisfy
\begin{equation}\label{eq:Delta-kappa-determinantial-identification}
    M^{\DD,\boldsymbol{\Delta}\kappa}
        [\ell+a,\ell+b]
    \simeq
    M^{\DD^\Delta,\kappa}[a,b].
\end{equation}

The root modules arising from the reduced
\(\boldsymbol{\Delta}\)-block form the level-zero PBW block, whereas
the modules in
\eqref{eq:Delta-kappa-root-module-identification} belong to
\(\mathscr C_{\DD}[1,\infty)\). By uniqueness of the PBW
factorization with respect to the level decomposition, a simple
object of
\(\mathscr C_{\DD}(\boldsymbol{\Delta}\kappa)\) belongs to
\(\mathscr C_{\DD}[1,\infty)\) if and only if its PBW parameters at
the positions of the \(\boldsymbol{\Delta}\)-block vanish. Therefore
\[
    \mathscr D_{\DD}(\boldsymbol{\Delta}\kappa)
    =
    \mathscr C_{\DD}(\boldsymbol{\Delta}\kappa)
      \cap\mathscr C_{\DD}[1,\infty)
    =
    \mathscr C_{\DD^\Delta}(\kappa).
\]

Since \(\boldsymbol{\Delta}\) is reduced, none of its letters produces
a trivalent vertex in the right-inductive weave. Thus all trivalent
vertices of
\(\overrightarrow{\mathfrak W}
(\boldsymbol{\Delta}\kappa)\)
arise from the \(\kappa\)-block. By
\cite[Corollary~4.45]{casals2025cluster}, the corresponding ice
quiver is the word quiver of \(\kappa\). Hence the exchange matrices
and frozen vertex sets agree.

Equations
\eqref{eq:Delta-kappa-root-module-identification} and
\eqref{eq:Delta-kappa-determinantial-identification} identify the
cluster-variable modules attached to the corresponding vertices.
Their renormalized \(R\)-matrix degrees are equal, so their
quasi-commutation matrices agree. This proves the seed identification
in \eqref{eq:Delta-kappa-base-identification}, including the
compatible forms.

We next verify the strong-commutation condition for the frozen
objects. For each color \(i\) occurring in \(\kappa\), set
\[
    a_i:=\min\{k\mid i_k=i\},
    \qquad
    b_i:=\max\{k\mid i_k=i\}.
\]
The corresponding frozen module is
\[
    F_i=M^{\DD^\Delta,\kappa}[a_i,b_i].
\]
In the bi-infinite extension defining predecessor and successor
positions, one has
\[
    a_i^-<k<b_i^+
    \qquad(1\leq k\leq r).
\]
It follows from
\cite[Lemma~5.12]{kashiwara2025monoidal} that
\[
    \mathfrak d
    \bigl(C_k^{\DD^\Delta,\kappa},F_i\bigr)
    =
    0
    \qquad(1\leq k\leq r).
\]
Thus \(F_i\) strongly commutes with every affine cuspidal module.
Proposition~\ref{pro:lusztig-parameter-modules} and
Lemma~\ref{lem:commuting-with-PBW-heads} then imply that \(F_i\)
strongly commutes with every simple object of
\(\mathscr C_{\DD^\Delta}(\kappa)\).

Apply Lemma~\ref{lem:Qin-positive-braid-category-comparison}
to \(\DD'=\DD^\Delta\) and \(\eta=\kappa\).
Its identification \eqref{eq:word-common-triangular-global-basis}
shows that the normalized simple classes of the fixed category
\(\mathscr C_{\DD^\Delta}(\kappa)\) form the common triangular
basis of its word algebra. After inverting the frozen variables,
Qin's quantum \(\mathcal A=\mathcal U\) theorem for word seeds
\cite[Theorem~8.10(3)]{qin2024analogs} gives the based isomorphism
\[
 \mathcal K_t\bigl(\mathscr C_{\DD^\Delta}(\kappa)\bigr)_{\mathrm{loc}}
 \xrightarrow{\sim}
 \mathcal U_t\bigl(\mathbf s_t(\kappa)\bigr).
\]
This uses the standard-basis comparison in the fixed bosonic
realization and requires no equality with a category defined using
another spectral-parameter convention.

Transporting this isomorphism through
\eqref{eq:Delta-kappa-base-identification} yields
\[
    \mathcal K_t\bigl(
        \mathscr D_{\DD}(\boldsymbol{\Delta}\kappa)
    \bigr)_{\mathrm{loc}}
    \xrightarrow{\sim}
    \mathcal U_t\left(
        \overrightarrow{\mathbf s}_t
        (\boldsymbol{\Delta}\kappa)
    \right).
\]
Together with the preceding strong-commutation statement for the
frozen modules and
Theorems~\ref{thm:Lambda-admissible-seed} and
\ref{thm:word-seed-monoidal-categorification}, which supply the
completely \(\Lambda\)-admissible quantizable monoidal seed, this
proves that
\(\mathscr D_{\DD}(\boldsymbol{\Delta}\kappa)\) is a based
categorification of the stated upper quantum cluster algebra.
\end{proof}

\subsubsection{Main theorem}
We can now state the main monoidal categorification theorem. It
identifies the quantum cluster algebra of a right-inductive seed with
both a localized quantum Grothendieck ring and a localized subalgebra
of the bosonic extension algebra. We write
\(\mathbf A(\beta)_{\mathrm{loc}}\) for the localization of
\(\mathbf A(\beta)\) at the global basis elements corresponding to the
frozen variables of \(\overrightarrow{\mathbf s}_t(\beta)\).

\begin{theorem}\label{thm:main-monoidal-categorification}
Let \(\beta\) be a positive braid word. Under the assumption~\ref{assumption}, there is a canonical algebra isomorphism
\begin{equation}\label{eq:isoAt}
    \mathcal A_t\bigl(\overrightarrow{\mathbf s}_t(\beta)\bigr)
    \xrightarrow{\ \sim\ }
    \mathcal{K}_t\bigl(\mathscr D_{\DD}(\beta)\bigr)_{\mathrm{loc}}.
\end{equation}
Under this isomorphism, the category \(\mathscr D_\DD(\beta)\) is a
based categorification of
\(\mathcal A_t(\overrightarrow{\mathbf s}_t(\beta))\). In particular,
every quantum cluster monomial with nonnegative
frozen exponents corresponds to the normalized \((q,t)\)-character of a
real simple object of \(\mathscr D_{\DD}(\beta)\). 
\end{theorem}

\begin{proof}
Recall that \(\Delta\in\Br^+\) is the positive lift of the longest Weyl
group element \(w_0\). Its positive powers are the Garside elements used
below; the completion property says that every positive braid is a left
divisor of some \(\Delta^m\).

By \cite[Corollary~7.3]{oh2025pbw}, there exist a positive braid word
\(\gamma\) and an integer \(m>0\) such that \(\beta\gamma\) represents
the Garside power \(\Delta^m\). Fix a reduced word
\(\boldsymbol{\Delta}\) for \(\Delta\), and choose a positive word
\(\kappa\) representing \(\Delta^{m-1}\). Thus,
\(\boldsymbol{\Delta}\kappa\) and \(\beta\gamma\) represent the same
element of \(\Br^+\).

Lemma~\ref{lem:Garside-base-case} supplies the compatible seed,
category, frozen-localization, and basis identifications
\eqref{eq:Delta-kappa-base-identification}. In particular, it yields
the based categorification
\begin{equation}\label{eq:categorification-Delta-kappa}
    \mathcal A_t\bigl(
        \overrightarrow{\mathbf s}_t
        (\boldsymbol{\Delta}\kappa)
    \bigr)
    \xrightarrow{\ \sim\ }
    \mathcal{K}_t\bigl(
        \mathscr D_{\DD}(\boldsymbol{\Delta}\kappa)
    \bigr)_{\mathrm{loc}}.
\end{equation}
In this case, every frozen cluster-variable module strongly commutes
with every simple object of the category.

The words \(\boldsymbol{\Delta}\kappa\) and \(\beta\gamma\) are
related by a sequence of braid moves. Proposition
\ref{lem:mutation-right-inductive-seed}, together with the
braid-move invariance of the categorical PBW construction and the
corresponding based quantum Grothendieck rings
\cite{kashiwara2025monoidal}, transports
\eqref{eq:categorification-Delta-kappa} to a based isomorphism
\begin{equation}\label{eq:categorification-beta-gamma}
    \mathcal A_t\bigl(\overrightarrow{\mathbf s}_t(\beta\gamma)\bigr)
    \xrightarrow{\ \sim\ }
    \mathcal{K}_t\bigl(
        \mathscr D_{\DD}(\beta\gamma)
    \bigr)_{\mathrm{loc}}.
\end{equation}
The strong-commutation property of the frozen modules is also
preserved under these braid moves.

Write \(\gamma=(j_1,\ldots,j_s)\), and set
\(\beta_k:=\beta j_1\cdots j_k\) for \(0\leq k\leq s\). Starting with
\(\beta_s=\beta\gamma\), we apply Proposition
\ref{pro:induction-upper-cluster-categorification} successively to
the pairs \((\beta_{k-1},j_k)\), for \(k=s,s-1,\ldots,1\). At each
step, Theorem~\ref{thm:quantum-A-U-braid-seeds} identifies the quantum
cluster algebra with the corresponding upper quantum cluster algebra. By
descending induction, the proposition therefore gives
based isomorphisms
\begin{equation*}
    \mathcal A_t\bigl(\overrightarrow{\mathbf s}_t(\beta_k)\bigr)
    \xrightarrow{\ \sim\ }
    \mathcal{K}_t\bigl(\mathscr D_{\DD}(\beta_k)\bigr)_{\mathrm{loc}}
    \qquad (0\leq k\leq s).
\end{equation*}
Taking \(k=0\) proves \eqref{eq:isoAt}.

Although the argument used a Garside completion, a reduced expression
of \(\Delta\), and a braid-move path, the resulting isomorphism is
independent of these auxiliary choices. Indeed, every step above is
performed in the common based realization supplied by
\(\Phi_{\DD}\), and the final map is uniquely characterized by
\[
    X_v
    \longmapsto
    [B_v^\beta]_t
    \qquad(v\in V_\beta).
\]
The braid-move calculation in
Proposition~\ref{lem:mutation-right-inductive-seed} shows that this
assignment is unchanged under each elementary change of expression.
Thus the isomorphism is canonical relative to the fixed \(Q\)-datum,
right-inductive seed, and normalization stated in the theorem.

The categorification in
\eqref{eq:categorification-Delta-kappa} is based, and this property
is preserved under braid moves and at every application of
Proposition~\ref{pro:induction-upper-cluster-categorification}.
Hence every quantum cluster monomial with nonnegative frozen exponents
corresponds to the normalized \((q,t)\)-character of a real simple object
of \(\mathscr D_{\DD}(\beta)\); negative frozen exponents give the
localized simple classes described above.
\end{proof}

The preceding categorification yields the following realization of
the quantum cluster algebra associated with the braid variety inside
the bosonic extension algebra.

\begin{theorem}\label{thm:quantization-braid-variety}
In type \(ADE\), for any positive braid word $\beta$, there is a canonical algebra isomorphism
\begin{equation}\label{eq:braid-variety-quantization}
    \mathcal A_t\bigl(\overrightarrow{\mathbf s}_t(\beta)\bigr)
    \xrightarrow{\ \sim\ }
    \mathbf A(\beta)_{\mathrm{loc}}.
\end{equation}
Under this isomorphism, every quantum cluster monomial with nonnegative
frozen exponents corresponds to a normalized global-basis element of
\(\mathbf A(\beta)\).
\end{theorem}

\begin{proof}
The compatibility of \(\Phi_{\DD}\) with the level filtration gives
a based isomorphism
\begin{equation*}\mathcal{K}_t(\mathscr D_{\DD}(\beta))\simeq\mathbf A(\beta).\end{equation*}
Moreover, \(\Phi_{\DD}\) sends the class of each frozen
cluster-variable module to the corresponding frozen global basis
element. It therefore extends uniquely to an isomorphism
\begin{equation}\label{eq:localized-Phi-D}
    \mathcal{K}_t\bigl(\mathscr D_{\DD}(\beta)\bigr)_{\mathrm{loc}}
    \xrightarrow{\ \sim\ }
    \mathbf A(\beta)_{\mathrm{loc}}.
\end{equation}
Composing \eqref{eq:localized-Phi-D} with the isomorphism
\eqref{eq:isoAt} proves \eqref{eq:braid-variety-quantization}.
Finally, Theorem~\ref{thm:main-monoidal-categorification} identifies a
cluster monomial with nonnegative frozen exponents with the normalized
\((q,t)\)-character of a real simple object, while \(\Phi_{\DD}\) maps
this character to the corresponding normalized global-basis element.
Allowing negative frozen exponents produces precisely the stated localized
classes. This proves the final assertion.
\end{proof}

\begin{corollary}
\label{cor:classical-specialization-braid-variety}
In type \(ADE\), let $\beta$ be a positive braid word.
Let \(R:=\ZZ[q^{\pm1/2}]\), and regard \(\CC\) as an \(R\)-algebra
via \(q^{1/2}\mapsto1\). Then
\(\mathbf A(\beta)_{\mathrm{loc}}\) is free, and hence flat, over
\(R\). Moreover, specialization at
\(q^{1/2}=1\) induces a canonical \(\CC\)-algebra isomorphism
\begin{equation}
\label{eq:classical-specialization-braid-variety}
    \operatorname{sp}_\beta\colon
    \mathbf A(\beta)_{\mathrm{loc}}
    \otimes_R\CC
    \xrightarrow{\ \sim\ }
    \CC[X(\beta)].
\end{equation}
Consequently, \(\mathbf A(\beta)_{\mathrm{loc}}\), or equivalently
\(\mathcal A_t(\overrightarrow{\mathbf s}_t(\beta))\) under
\(t^{1/2}=q^{-1/2}\), is a flat quantum deformation of
\(\CC[X(\beta)]\). If \(\mathbf L_\beta\) is its common triangular
basis, then
\(
\{\operatorname{sp}_\beta(L)\mid L\in\mathbf L_\beta\}
\)
is a \(\CC\)-basis of \(\CC[X(\beta)]\).
\end{corollary}

\begin{proof}
By Theorems~\ref{thm:quantum-A-U-braid-seeds} and
\ref{thm:quantization-braid-variety},
\(\mathbf A(\beta)_{\mathrm{loc}}\) admits a common triangular basis
\(\mathbf L\) over \(R\). Therefore
\[
    \mathbf A(\beta)_{\mathrm{loc}}
    =
    \bigoplus_{L\in\mathbf L}RL,
\]
so it is free, and hence flat, over \(R\). After tensoring with
\(\CC\), the images of the elements of \(\mathbf L\) remain a
basis. The specialization isomorphism constructed below therefore
sends them to a \(\CC\)-basis of \(\CC[X(\beta)]\).

Let
\[
    \Theta_\beta\colon
    \mathcal A_t\bigl(
        \overrightarrow{\mathbf s}_t(\beta)
    \bigr)
    \xrightarrow{\sim}
    \mathbf A(\beta)_{\mathrm{loc}}
\]
be the isomorphism in
\eqref{eq:braid-variety-quantization}. Define
\[
    \operatorname{sp}_\beta
    :=
    \operatorname{sp}^{\mathrm{cl}}_\beta
    \circ
    \left(
        \Theta_\beta\otimes_R\CC
    \right)^{-1}.
\]
By Lemma~\ref{lem:specialization-cluster-variables},
\[
    \operatorname{sp}^{\mathrm{cl}}_\beta\colon
    \mathcal A_t\bigl(
        \overrightarrow{\mathbf s}_t(\beta)
    \bigr)\otimes_R\CC
    \xrightarrow{\sim}
    \CC[X(\beta)]
\]
is an isomorphism. Since
\(\Theta_\beta\otimes_R\CC\) is also an isomorphism,
their composition \(\operatorname{sp}_\beta\) is an isomorphism.
Moreover, it sends \(A_v\otimes1\), the bosonic realization of a
quantum cluster variable, to the corresponding classical cluster
variable \(x_v\). This proves
\eqref{eq:classical-specialization-braid-variety} and the final
assertion.
\end{proof}

\subsection{Generalized \texorpdfstring{\(T\)}{T}-systems}\label{sec:T-system}
We  establish a generalized \(T\)-system in this subsection. It extends the
determinantial \(T\)-systems studied in
\cite{qin2024analogs,kashiwara2025monoidal,bi2025cluster}.

\begin{theorem}[Generalized quantum \(T\)-system]
\label{thm:generalized-T-system}
Assume that the Cartan datum is of type \(ADE\). Let
\(\ddot{\mathbf s}\) and \(\ddot{\mathbf s}'\) be related by the
solid--special interchange
\[
    \ddot{\mathbf s}
    =(\boldsymbol\rho,p\mathrm L,j\mathrm R,\boldsymbol\tau),
    \qquad
    \ddot{\mathbf s}'
    =(\boldsymbol\rho,j\mathrm R,p\mathrm L,\boldsymbol\tau),
\]
where both displayed steps are solid and
\[
    \ell\bigl(s_p\delta(\beta_{\boldsymbol\rho})s_j\bigr)
    =\ell\bigl(\delta(\beta_{\boldsymbol\rho})\bigr).
\]
Let \(1\leq c\leq r-1\) be the affected depth, and let \(\beta\)
be the common final positive braid word. Suppressing the ambient
superscript \(\beta\), there exist \(A,B\in\frac12\ZZ\) such that
\begin{equation}\label{eq:generalized-T-system}
    D_{c,j}^{\ddot{\mathbf s}}D_{c,j}^{\ddot{\mathbf s}'}
    =q^A
      D_{c-1,j}^{\ddot{\mathbf s}}
      D_{c+1,j}^{\ddot{\mathbf s}}
     +q^B
      \bigodot_{i\neq j}
      \left(D_{c,i}^{\ddot{\mathbf s}}\right)^{\odot(-c_{ji})}.
\end{equation}
\end{theorem}

\begin{proof}
Put \(w=\delta(\beta_{\boldsymbol\rho})\). The solid--special
condition gives \(s_pw=ws_j<w\). The generalized-minor identity
used in the proof of \cite[Proposition~4.6]{galashin2026braidII}
depends only on this local relation; it therefore applies here
for general \(\delta(\beta)\). In our notation, the classical grid
minors satisfy
\begin{equation}
\label{eq:classical-grid-minor-T-system}
    \Delta_{c,j}^{\ddot{\mathbf s}}
    \Delta_{c,j}^{\ddot{\mathbf s}'}
    =
    \Delta_{c-1,j}^{\ddot{\mathbf s}}
    \Delta_{c+1,j}^{\ddot{\mathbf s}}
    +
    \prod_{i\neq j}
    \left(
        \Delta_{c,i}^{\ddot{\mathbf s}}
    \right)^{-c_{ji}}.
\end{equation}

Set
\[
    M_+
    :=
    D_{c-1,j}^{\ddot{\mathbf s}}
    \odot
    D_{c+1,j}^{\ddot{\mathbf s}},
    \qquad
    M_-
    :=
    \bigodot_{i\neq j}
    \left(
        D_{c,i}^{\ddot{\mathbf s}}
    \right)^{\odot(-c_{ji})}.
\]
By Definition~\ref{def:quantum-grid-minors} and Corollary~\ref{cor:right-weave-cluster-variables-level-one}, all the
quantum grid minors appearing here are normalized global-basis elements in \(\mathbf A(\beta)\).
In particular,
\[
    D_{c,j}^{\ddot{\mathbf s}},
    \quad
    D_{c,j}^{\ddot{\mathbf s}'},
    \quad
    M_+,
    \quad
    M_-
\]
belong to the based subalgebra \(\mathbf A(\beta)\), and \(M_+\) and
\(M_-\) are normalized global-basis elements.

We first check that \(M_+\neq M_-\). Let \(e\) be the position
in the final word \(\beta\) of the right insertion \(jR\) at
depth \(c+1\) of \(\ddot{\mathbf s}\). The degree formula in
Theorem~\ref{thm:quantum-grid-minor-factorization} gives coefficient
\(1\) at \(e\) for
\(\mathbf a^\beta(D_{c+1,j}^{\ddot{\mathbf s}})\): this is the
last, solid position of its truncated word and
\(\langle\varpi_j,\alpha_j^\vee\rangle=1\).
Every degree vector coming from depth at most \(c\) is supported
in the earlier consecutive block and has coefficient \(0\) at
\(e\). Additivity of parameters for normalized cluster products
therefore gives
\begin{equation}\label{eq:T-system-distinct-terms}
    \bigl(\mathbf a^\beta(M_+)\bigr)_e=1,
    \qquad
    \bigl(\mathbf a^\beta(M_-)\bigr)_e=0.
\end{equation}

The normalized global basis has positive structure constants by
\cite[Corollary~6.6]{kashiwara2025monoidal}.
This follows from the quantum Grothendieck-ring positivity theorem
of \cite[Theorem~4.3]{varagnolo2003perverse}. Thus
\begin{equation}\label{eq:T-system-global-basis-expansion}
    D_{c,j}^{\ddot{\mathbf s}}
    D_{c,j}^{\ddot{\mathbf s}'}
    =
    \sum_{\bb\in\widehat B(\infty)}
        c_{\bb}(q)\widetilde G(\bb),
    \qquad
    c_{\bb}(q)\in
    \ZZ_{\geq0}[q^{\pm1/2}],
\end{equation}
with only finitely many nonzero coefficients.

The left-hand side belongs to \(\mathbf A(\beta)\). It follows that
\[
    c_{\bb}(q)\neq0
    \quad\Longrightarrow\quad
    \widetilde G(\bb)\in\mathbf A(\beta).
\]
By Theorems~\ref{thm:main-monoidal-categorification} and
\ref{thm:quantization-braid-variety}, the normalized global-basis
elements belonging to \(\mathbf A(\beta)\) map to distinct
elements of the common triangular basis \(\mathbf L_\beta\) of
the localization. Only this based inclusion is needed here; no
characterization of arbitrary basis elements by frozen exponents
is used. Thus
\eqref{eq:T-system-global-basis-expansion} may be rewritten as
\begin{equation}\label{eq:T-system-triangular-basis-expansion}
    D_{c,j}^{\ddot{\mathbf s}}
    D_{c,j}^{\ddot{\mathbf s}'}
    =
    \sum_{L\in\mathbf L_\beta}
        c_L(q)L,
    \qquad
    c_L(q)\in\ZZ_{\geq0}[q^{\pm1/2}],
\end{equation}
where the sum is finite. In particular, both \(M_+\) and \(M_-\)
belong to \(\mathbf L_\beta\).

Apply the specialization isomorphism
\(\operatorname{sp}_\beta\) of
Corollary~\ref{cor:classical-specialization-braid-variety}.
Theorem~\ref{thm:quantum-grid-minor-factorization} gives
\[
\begin{aligned}
    \operatorname{sp}_\beta(M_+)
    &=
    \Delta_{c-1,j}^{\ddot{\mathbf s}}
    \Delta_{c+1,j}^{\ddot{\mathbf s}},\\
    \operatorname{sp}_\beta(M_-)
    &=
    \prod_{i\neq j}
    \left(
        \Delta_{c,i}^{\ddot{\mathbf s}}
    \right)^{-c_{ji}}.
\end{aligned}
\]
Hence, by
\eqref{eq:classical-grid-minor-T-system}, the specialization of
\eqref{eq:T-system-triangular-basis-expansion} is
\[
    \sum_{L\in\mathbf L_\beta}
        c_L(1)\operatorname{sp}_\beta(L)
    =
    \operatorname{sp}_\beta(M_+)
    +
    \operatorname{sp}_\beta(M_-).
\]

The elements
\(
    \left\{
        \operatorname{sp}_\beta(L)
        \,\middle|\,
        L\in\mathbf L_\beta
    \right\}
\)
are linearly independent by
Corollary~\ref{cor:classical-specialization-braid-variety}.
Together with \(M_+\neq M_-\), proved in
\eqref{eq:T-system-distinct-terms}, this implies
\[
    c_{M_+}(1)=c_{M_-}(1)=1
\]
and
\[
    c_L(1)=0
    \qquad
    \left(
        L\in\mathbf L_\beta
        \setminus\{M_+,M_-\}
    \right).
\]

Since every \(c_L(q)\) has nonnegative coefficients, the equality
\(c_L(1)=0\) implies \(c_L(q)=0\). Likewise, a Laurent polynomial in
\(\ZZ_{\geq0}[q^{\pm1/2}]\) whose value at \(q=1\) is \(1\) must be a
single Laurent monomial. Therefore, there exist
\(a_+,a_-\in\frac12\ZZ\) such that
\[
    c_{M_+}(q)=q^{a_+},
    \qquad
    c_{M_-}(q)=q^{a_-}.
\]
Thus the quantum identity is obtained from specialization only after
using positivity and the linear independence of the specialized
basis; it is not a formal consequence of setting \(q=1\).
It follows that
\[
    D_{c,j}^{\ddot{\mathbf s}}
    D_{c,j}^{\ddot{\mathbf s}'}
    =
    q^{a_+}M_+
    +
    q^{a_-}M_-.
\]

Finally, the normalized product \(M_+\) differs from the ordered
product
\(
    D_{c-1,j}^{\ddot{\mathbf s}}
    D_{c+1,j}^{\ddot{\mathbf s}}
\)
by a power of \(q^{1/2}\). Absorbing this power into \(q^{a_+}\) and
setting \(B:=a_-\) gives
\eqref{eq:generalized-T-system}.
\end{proof}

\begin{remark}[Computing the two coefficients]
\label{rem:T-system-coefficients-PBW}
The proof of Theorem~\ref{thm:generalized-T-system} also gives
a finite way to recover its exponents from PBW expansions. Write
\(F=D_{c,j}^{\ddot{\mathbf s}}D_{c,j}^{\ddot{\mathbf s}'}\)
and retain the normalized products \(M_+,M_-\) used there.
The distinguished coordinate in
\eqref{eq:T-system-distinct-terms} is their rightmost possible
nonzero coordinate; hence
\(\mathbf a^\beta(M_-)\prec_{\mathrm r}\mathbf a^\beta(M_+)\).
Using a common PBW basis, divide the leading coefficient of
\(F\) by that of \(M_+\) to obtain \(q^{a_+}\). Subtract
\(q^{a_+}M_+\) from \(F\); the ratio of the remaining leading
coefficient to that of \(M_-\) is \(q^{a_-}\).
These ratios are Laurent monomials by the theorem. Finally,
write
\(D_{c-1,j}^{\ddot{\mathbf s}}D_{c+1,j}^{\ddot{\mathbf s}}
=q^\rho M_+\), where \(\rho\) is obtained from
\eqref{eq:quantum-monomial-product}. Then
\(A=a_+-\rho\) and \(B=a_-\).
This computation uses the finite PBW expansions
of the minors, as well as the commutation form for the final
normalization; the commutation matrix alone is not asserted to
determine both coefficients.
\end{remark}

\begin{example}[A braid without a left \(\Delta\)-factor]
\label{ex:T-system-without-Delta-factor}
Work in type \(A_2\) and take \(\beta=(1,1,1,2)\). Since
\(\delta(\beta)=s_1s_2\neq w_0\), the positive braid represented
by \(\beta\) is not left divisible by \(\Delta\). Consider
\[
\begin{aligned}
 \ddot{\mathbf s}
   &=(1R^+,1L,1R,2R^+),\\
 \ddot{\mathbf s}'
   &=(1R^+,1R,1L,2R^+).
\end{aligned}
\]
Both strings have final word \(\beta\). Their middle two steps
give a solid--special interchange at depth \(c=2\), because
\(s_1\delta(1)s_1=s_1\).

Set
\[
 X:=\widetilde G(\beta,E_2),\qquad
 Y:=\widetilde G(\beta,E_3),\qquad
 Z:=\widetilde G(\beta,E_2+E_3).
\]
At depth \(2\), the truncated word is \((1,1)\) for both
strings. Its nonzero parameter is its second unit vector. The
later left insertion in \(\ddot{\mathbf s}'\) shifts this
coordinate to position \(3\) of the final word, whereas no such
shift occurs for \(\ddot{\mathbf s}\). At depth \(3\), the
truncated word is \((1,1,1)\), and its grid-minor parameter is
\(E_2+E_3\). Consequently,
\[
\begin{array}{c|c|c}
 \text{grid minor}&\text{value}&\beta\text{-Lusztig parameter}\\ \hline
 D_{1,1}^{\ddot{\mathbf s}}&1&0\\
 D_{2,2}^{\ddot{\mathbf s}}&1&0\\
 D_{2,1}^{\ddot{\mathbf s}}&X&E_2\\
 D_{2,1}^{\ddot{\mathbf s}'}&Y&E_3\\
 D_{3,1}^{\ddot{\mathbf s}}&Z&E_2+E_3
\end{array}
\]
by Theorem~\ref{thm:quantum-grid-minor-factorization}.

The scalar coefficients can be checked directly in the bosonic
algebra. The first three PBW root vectors are
\[
 P_1^\beta=q^{1/2}f_{1,0},\qquad
 P_2^\beta=q^{1/2}f_{1,1},\qquad
 P_3^\beta=q^{1/2}f_{1,2}.
\]
Thus \(X=f_{1,1}\) and \(Y=f_{1,2}\). Moreover,
\[
 Z=q\bigl(f_{1,2}f_{1,1}-1\bigr).
\]
Indeed, the expression on the right is bar invariant, has weight
zero, and differs from
\(P^\beta(E_2+E_3)=qf_{1,2}f_{1,1}\) by \(-q\).
The triangular characterization in
Proposition~\ref{prop:PBW-global-basis} therefore identifies it
with \(\widetilde G(\beta,E_2+E_3)\).
The bosonic relation now gives
\[
 \boxed{XY=qZ+1},\qquad
 YX=q^{-1}Z+1,\qquad XZ=q^2ZX.
\]
In the seed \((X,Z)\), the variable \(X\) is mutable and
\(Z\) is frozen. Its exchange and commutation matrices are
\[
 B_\beta=\begin{pmatrix}0\\1\end{pmatrix},\qquad
 \Lambda_\beta=\begin{pmatrix}0&-2\\2&0\end{pmatrix},
 \qquad B_\beta^T\Lambda_\beta=(2,0),
\]
where the toric frame uses \(t=q^{-1}\). Thus the same relation
is the quantum seed mutation at \(X\).
This is \eqref{eq:generalized-T-system} with \(A=1\) and
\(B=0\) for a braid outside the \(\Delta\)-left-divisible case.
It is a nontrivial rank-one cluster exchange; at \(q=1\), it
becomes \(xy=z+1\).
\end{example}

\begin{example}[A mixed-color exchange without a left \(\Delta\)-factor]
\label{ex:mixed-color-T-system}
Work in type \(A_3\), and write \(w_J=s_1s_2s_1=s_2s_1s_2\)
for the longest element of the parabolic subgroup on \(J=\{1,2\}\).
Take
\[
 \beta=(2,1,2,1,2,1,2,3).
\]
Its Demazure product is \(w_Js_3\), of length \(4\), whereas
\(\ell(w_0)=6\). Thus \(\beta\) has no left full \(A_3\)
\(\Delta\)-factor. All three colors occur in \(\beta\).

The first three steps of its right inductive weave are hollow,
positions \(4,5,6,7\) are solid, and the last step \(3R^+\)
is hollow. On colors \(1,2\), the reduced prefix \((2,1,2)\)
is the parabolic Garside word. Applying the word-seed construction
\eqref{eq:right-weave-variable-initial-minor} within this
\(A_2\) subsystem gives the four variables of the word
\(\eta=(1,2,1,2)\). Appending the final color-\(3\) strand
requires no further weave move or trivalent vertex and preserves
all cycle intersections and the frozen set. The resulting
classical seed is therefore the word seed of \(\eta\).
The consecutive-block embedding of the tail algebra into
\(\widehat{\cA}(\beta)\) is given by \(T_{w_J}\). It supplies
the quantum realization, whose compatible form is checked below.
This uses the Garside construction in the subsystem \(J\), without
assuming a full \(A_3\) Garside prefix.

Put
\[
 x:=\widetilde P_4^\beta,\quad
 y:=\widetilde P_5^\beta,\quad
 z:=\widetilde P_6^\beta,\quad
 u:=\widetilde P_7^\beta.
\]
Equivalently,
\[
 x=T_{w_J}(f_{1,0}),\qquad
 y=T_{w_J}T_1(f_{2,0}),\qquad
 z=T_{w_J}(f_{2,0}),\qquad
 u=T_{w_J}(f_{1,1}).
\]
The initial seed and its mutation in the first direction are
\[
\begin{array}{c|c|c}
 \text{variable}&\text{PBW expression}&\beta\text{-Lusztig parameter}
 \\ \hline
 X_1&x&E_4\\
 X_2&y&E_5\\
 X_3&q^{1/2}zx-qy&E_4+E_6\\
 X_4&q^{1/2}uy-qz&E_5+E_7\\
 Y&z&E_6
\end{array}
\]
Here each row is a normalized global-basis element. For example,
\[
 G(\beta,E_4+E_6)=P_6^\beta P_4^\beta-qP_5^\beta,
 \qquad
 G(\beta,E_5+E_7)=P_7^\beta P_5^\beta-qP_6^\beta.
\]
The lower terms are strictly smaller in the right lexicographic order.
By the braid symmetry formulas, the normalized expressions are bar
invariant. Proposition~\ref{prop:PBW-global-basis} therefore gives
the displayed identities.

With the variable order \((X_1,X_2,X_3,X_4)\), the mutable
indices are \(1,2\), and the frozen indices are \(3,4\).
Equations~\eqref{eq:B-beta} and \eqref{eq:Lambda-beta-word},
or a direct calculation with the displayed PBW expressions, give
\[
 B=
 \begin{pmatrix}
 0&1\\-1&0\\1&-1\\0&1
 \end{pmatrix},\qquad
 \Lambda=
 \begin{pmatrix}
 0&1&-1&-1\\
 -1&0&0&-1\\
 1&0&0&-1\\
 1&1&1&0
 \end{pmatrix},\qquad
 B^{\mathsf T}\Lambda=
 \begin{pmatrix}2&0&0&0\\0&2&0&0\end{pmatrix}.
\]
The convention is \(X_iX_j=t^{\lambda_{ij}}X_jX_i\), with
\(t=q^{-1}\). In particular, mutation at \(X_1\) yields
\begin{equation}\label{eq:mixed-color-quantum-exchange}
 \boxed{X_1Y=q^{1/2}X_3+q^{-1/2}X_2}.
\end{equation}
Indeed, \eqref{eq:braid-symmetry-T} gives
\[
 y=\frac{q^{1/2}zx-q^{-1/2}xz}{q-q^{-1}},
\]
which proves \eqref{eq:mixed-color-quantum-exchange} directly.
In the reverse order, \(YX_1=q^{-1/2}X_3+q^{1/2}X_2\).

To realize this exchange as one simultaneous grid-minor relation,
consider the double strings
\[
\begin{aligned}
 \ddot{\mathbf s}
  &=(1R^+,2R^+,1R^+,2R,2L,1R,2R,3R^+),\\
 \ddot{\mathbf s}'
  &=(1R^+,2R^+,1R^+,2R,1R,2L,2R,3R^+).
\end{aligned}
\]
Both have the literal final word \(\beta\). Their fifth and sixth
steps form the required solid--special interchange, because
\(s_2w_Js_1=w_J\); the affected depth is \(c=5\) and
\(j=1\). At the four relevant corners, the truncated words are
\[
\begin{array}{c|c}
 (\ddot{\mathbf s},4)&(1,2,1,2)\\
 (\ddot{\mathbf s},5)&(2,1,2,1,2)\\
 (\ddot{\mathbf s}',5)&(1,2,1,2,1)\\
 (\ddot{\mathbf s},6)&(2,1,2,1,2,1)
\end{array}
\]
The reduced prefix in each row has length \(3\). Reading the
remaining solid positions and inserting the zero coordinates of
later steps gives
\[
\begin{array}{c|c|c}
 \text{grid minor}&\text{value}&\beta\text{-Lusztig parameter}
 \\ \hline
 D_{4,1}^{\ddot{\mathbf s}}&1&0\\
 D_{5,1}^{\ddot{\mathbf s}}&X_1&E_4\\
 D_{5,1}^{\ddot{\mathbf s}'}&Y&E_6\\
 D_{6,1}^{\ddot{\mathbf s}}&X_3&E_4+E_6\\
 D_{5,2}^{\ddot{\mathbf s}}&X_2&E_5
\end{array}
\]
by Definition~\ref{def:delta-double-string} and
Theorem~\ref{thm:quantum-grid-minor-factorization}.
Consequently, \eqref{eq:mixed-color-quantum-exchange} is exactly
Theorem~\ref{thm:generalized-T-system} with
\(A=\tfrac12\) and \(B=-\tfrac12\). Its neighboring-color
term is the nonunit \(D_{5,2}^{\ddot{\mathbf s}}=X_2\).
The classical specialization has both coefficients equal to one.
For the coefficient extraction of Remark~\ref{rem:T-system-coefficients-PBW},
the leading coefficients of \(xz\) and \(X_3\) relative to the
descending monomial \(zx\) are \(q\) and \(q^{1/2}\). Their
ratio gives \(q^A=q^{1/2}\); subtracting \(q^{1/2}X_3\)
leaves \(q^{-1/2}X_2\), giving \(q^B=q^{-1/2}\).

The same exchange is visible in a genuine \(3\)-braid move:
\[
 \beta=(2,1,2,\underline{1,2,1},2,3)
 \longleftrightarrow
 \beta'=(2,1,2,\underline{2,1,2},2,3).
\]
Let \(E'_a\) denote the coordinate vectors for \(\beta'\).
Applying \eqref{eq:braid-move-PBW-transition} in positions
\(4,5,6\) gives
\[
\begin{array}{c|c|c}
 &\beta\text{-parameter}&\beta'\text{-parameter}\\ \hline
 Y&E_6&E'_4\\
 X_3&E_4+E_6&E'_5\\
 X_2&E_5&E'_4+E'_6\\
 X_4&E_5+E_7&E'_4+E'_6+E'_7
\end{array}
\]
These are precisely the initial interval parameters of the tail
word \((2,1,2,2)\). Hence its seed is
\((Y,X_3,X_2,X_4)\), obtained from the original seed by mutation
at \(X_1\) and interchange of the second and third labels;
its mutable indices are \(1,3\).
This also verifies the parameter transport for a mixed-color
mutation in an example with two mutable vertices.
\end{example}

\subsubsection{Determinantial
\texorpdfstring{\(T\)}{T}-systems for
\texorpdfstring{\(\boldsymbol\Delta\beta\)}{Delta beta}}
\label{sec:determinantial-T-system}

Let \(\beta=(i_1,\ldots,i_r)\), and fix a nontrivial
\(i\)-box \([a,b]\), so that \(a<b\) and \(i_a=i_b=i\).
Let \(a^+\) be the next occurrence of \(i\) after \(a\), and
let \(b^-\) be the previous occurrence before \(b\). For
\(j\neq i\), let \(a(j)^+\) and \(b(j)^-\) be the first and
last occurrences of \(j\) in \([a,b]\), respectively. An empty
interval contributes the element \(1\); in particular,
\(D^\beta[a^+,b^-]=1\) when \(a^+>b^-\).

We give one simultaneous grid-minor realization of all the terms
in the determinantial relation. Choose an auxiliary reduced word
\[
 \mathbf d=(u_1,\ldots,u_{\ell-1},i)
\]
for \(w_0\), where \(\ell=\ell(w_0)\), and set
\(\mathbf d^L=(i^*,u_1,\ldots,u_{\ell-1})\).
The relation \(s_{i^*}w_0=w_0s_i\) shows that
\(\mathbf d^L\) is also a reduced word for \(w_0\).
Crucially, one has the literal word identity
\[
 (i^*)\mathbf d=\mathbf d^L(i).
\]
Define
\[
\begin{aligned}
 \boldsymbol\rho
   &=(u_1R^+,\ldots,u_{\ell-1}R^+,iR^+,
        i_{a+1}R,\ldots,i_{b-1}R),\\
 \boldsymbol\tau
   &=(i_{a-1}^*L,\ldots,i_1^*L,
        i_{b+1}R,\ldots,i_rR),\\
 \ddot{\mathbf s}
   &=(\boldsymbol\rho,i^*L,iR,\boldsymbol\tau),\\
 \ddot{\mathbf s}'
   &=(\boldsymbol\rho,iR,i^*L,\boldsymbol\tau),
 \qquad c:=\ell+b-a,
\end{aligned}
\]
omitting empty blocks. Here the local colors are \(p=i^*\)
and \(j=i\) in Theorem~\ref{thm:generalized-T-system}.
All insertions after the initial \(\mathbf d\)-block are solid,
and
\[
 s_{i^*}\delta(\beta_{\boldsymbol\rho})s_i
   =s_{i^*}w_0s_i=w_0.
\]
Hence these strings satisfy precisely the hypotheses of that
theorem.

Write \(\gamma=(i_1^*,\ldots,i_{a-1}^*)\). The common actual
final word is
\[
 \omega=\gamma\mathbf d^L(i_a,\ldots,i_r).
\]
It represents the same positive braid as
\(\boldsymbol\Delta\beta\). In particular, the prefix
\(\gamma\mathbf d^L\) can be changed to
\(\boldsymbol\Delta(i_1,\ldots,i_{a-1})\) by braid moves
within its \(\ell+a-1\) positions.

For an interval \([h,k]\subseteq[a,b]\), put
\[
 \mathbf e_j[h,k]
   :=\sum_{\substack{h\leq t\leq k\\ i_t=j}}E_{\ell+t},
\]
with value zero for an empty interval. The truncated words at the
four relevant corners are
\[
\begin{array}{c|c}
 (\text{string},\text{depth})&\text{truncated word}\\ \hline
 (\ddot{\mathbf s},c-1)&\mathbf d(i_{a+1},\ldots,i_{b-1})\\
 (\ddot{\mathbf s},c)&\mathbf d^L(i_a,\ldots,i_{b-1})\\
 (\ddot{\mathbf s}',c)&\mathbf d(i_{a+1},\ldots,i_b)\\
 (\ddot{\mathbf s},c+1)&\mathbf d^L(i_a,\ldots,i_b)
\end{array}
\]
Each has a reduced prefix of length \(\ell\), followed only by
solid positions. After inserting the zero coordinates of its later
left and right blocks, Definition~\ref{def:delta-double-string}
therefore gives
\[
\begin{aligned}
 \boldsymbol\delta_{c-1,i}^{\ddot{\mathbf s}}
   &=\mathbf e_i[a+1,b-1],&
 \boldsymbol\delta_{c,i}^{\ddot{\mathbf s}}
   &=\mathbf e_i[a,b-1],\\
 \boldsymbol\delta_{c,i}^{\ddot{\mathbf s}'}
   &=\mathbf e_i[a+1,b],&
 \boldsymbol\delta_{c+1,i}^{\ddot{\mathbf s}}
   &=\mathbf e_i[a,b],\\
 \boldsymbol\delta_{c,j}^{\ddot{\mathbf s}}
   &=\mathbf e_j[a,b]
   &&(j\neq i).
\end{aligned}
\]
All these vectors vanish on the prefix \(\gamma\mathbf d^L\).
Consequently, the PBW transition to
\(\boldsymbol\Delta\beta\) fixes the displayed vectors, just
as in \eqref{eq:i-box-PBW-transition}.
Theorem~\ref{thm:quantum-grid-minor-factorization} and
Lemma~\ref{lem:global-basis-consecutive-block} now identify the
minors simultaneously:
\[
\begin{aligned}
 D_{c-1,i}^{\omega,\ddot{\mathbf s}}
   &=T_{w_0}\bigl(D^\beta[a^+,b^-]\bigr),\\
 D_{c,i}^{\omega,\ddot{\mathbf s}}
   &=T_{w_0}\bigl(D^\beta[a,b^-]\bigr),\\
 D_{c,i}^{\omega,\ddot{\mathbf s}'}
   &=T_{w_0}\bigl(D^\beta[a^+,b]\bigr),\\
 D_{c+1,i}^{\omega,\ddot{\mathbf s}}
   &=T_{w_0}\bigl(D^\beta[a,b]\bigr),\\
 D_{c,j}^{\omega,\ddot{\mathbf s}}
   &=T_{w_0}\bigl(D^\beta[a(j)^+,b(j)^-]\bigr)
     \qquad(j\neq i).
\end{aligned}
\]
This specifies one pair of double strings for the entire relation,
rather than a separate realization for each factor.

Apply Theorem~\ref{thm:generalized-T-system} to this pair. Its
left-hand side has the factors corresponding to \([a,b^-]\)
and \([a^+,b]\), in that order. Apply the bar anti-involution to
the relation, which reverses each ordered product and fixes every
normalized grid minor and normalized cluster monomial. Applying
\(T_{w_0}^{-1}\) then gives
\begin{equation}\label{eq:determinantial-T-system}
\begin{aligned}
 D^\beta[a^+,b]D^\beta[a,b^-]
  &=q^A D^\beta[a,b]D^\beta[a^+,b^-]\\
  &\quad+q^B
   \bigodot_{j\neq i}
   \left(D^\beta[a(j)^+,b(j)^-]\right)^{\odot(-c_{ij})},
 \qquad A,B\in\tfrac12\ZZ.
\end{aligned}
\end{equation}
Here \(A,B\) are the negatives of the exponents before applying
bar. This recovers the determinantial \(T\)-system of
\cite[Theorems~3.25 and~3.26]{bi2025cluster} as a special case,
with the word changes, common realization, and multiplication
order specified explicitly.\\

\noindent
\textbf{Declaration on the Use of Generative AI.}
The author used OpenAI's ChatGPT to improve the language and
presentation of the manuscript, to check normalization and logical
consistency, and to assist in formulating revisions to selected
proofs. The author is responsible for verifying all mathematical
arguments, statements, and references and for the final content of
the paper.

\end{document}